\documentclass[11pt]{article}

\usepackage[english]{babel}
\usepackage{amssymb}
\usepackage{amsmath}
\usepackage{enumerate}
\usepackage{amsthm}
\usepackage{amsfonts}
\usepackage{listings}
\usepackage{xypic}

\usepackage{MnSymbol}

\usepackage[numbers,sort&compress]{natbib}

\usepackage{hyperref}
\usepackage{color}
\usepackage[dvipsnames]{xcolor}
\usepackage{graphicx}

\usepackage{tikz}
\newcommand*\numcirc[1]{\tikz[baseline=(char.base)]{
            \node[shape=circle,draw,inner sep=2pt] (char) {#1};}}

\newcommand{\obsolete}[1]{}
\newcommand{\comment}[1]{}

\title{Six-Functor-Formalisms on Higher Stacks}
\date{April 6, 2022}
\author{Fritz H\"ormann\\ Mathematisches Institut, Albert-Ludwigs-Universit\"at Freiburg}

\newtheorem{SATZ}{Theorem}[section]
\newtheorem{HAUPTSATZ}[SATZ]{Main theorem}
\newtheorem{LEMMA}[SATZ]{Lemma}
\newtheorem{KEYLEMMA}[SATZ]{Key Lemma}
\newtheorem{DEF}[SATZ]{Definition}
\newtheorem{MC}[SATZ]{Main Construction}
\newtheorem{PROP}[SATZ]{Proposition}

\newtheorem{DEFLEMMA}[SATZ]{Definition/Lemma}
\newtheorem{BEISPIEL}[SATZ]{Example}

\newtheorem{KOR}[SATZ]{Corollary}

\newtheorem{BEM}[SATZ]{Remark}

\newtheoremstyle{bare}        
  {}            
  {}            
  {\normalfont}                 
  {}                            
  {\bfseries}                   
  {}                            
  {.0em}                           
  {\thmnumber{#2}#1. \thmnote{\normalfont\textsc{(#3)}} } 

\theoremstyle{bare}
\newtheorem{PAR}[SATZ]{}

\newcommand{\iso}{\stackrel{\sim}{\longrightarrow}}

\newcommand{\N}{ \mathbb{N} }
\newcommand{\HH}{ \mathbb{H} }
\newcommand{\DD}{ \mathbb{D} }

\newcommand{\PP}{ \mathbb{P} }
\newcommand{\SSS}{ \mathbb{S} }

\DeclareMathOperator{\CechDia}{\check{C}}
\DeclareMathOperator{\Cech}{\check{C}ech}
\DeclareMathOperator{\new}{new}
\DeclareMathOperator{\old}{old}
\DeclareMathOperator{\dia}{Dia}

\DeclareMathOperator{\Flip}{Flip}
\DeclareMathOperator{\Fib}{Fib}
\DeclareMathOperator{\Mor}{Mor}

\DeclareMathOperator{\Cat}{Cat}

\DeclareMathOperator{\id}{id}

\DeclareMathOperator{\Hom}{Hom}

\DeclareMathOperator{\Dia}{Dia}
\DeclareMathOperator{\cor}{cor}
\DeclareMathOperator{\loc}{loc}
\DeclareMathOperator{\lax}{lax}
\DeclareMathOperator{\oplax}{oplax}
\DeclareMathOperator{\cov}{cov}
\DeclareMathOperator{\comp}{comp}
\DeclareMathOperator{\hcor}{hcor}
\DeclareMathOperator{\holim}{holim}
\DeclareMathOperator{\hocolim}{hocolim}
\DeclareMathOperator{\op}{op}
\DeclareMathOperator{\cart}{cart}
\DeclareMathOperator{\cocart}{cocart}
\DeclareMathOperator{\pr}{pr}
\DeclareMathOperator{\Fun}{Fun}
\DeclareMathOperator{\Cor}{Cor}

\DeclareMathOperator{\Dir}{Dir}
\DeclareMathOperator{\Inv}{Inv}
\DeclareMathOperator{\Catf}{Catf}
\DeclareMathOperator{\Catlf}{Catlf}
\DeclareMathOperator{\Dirf}{Dirf}
\DeclareMathOperator{\Invf}{Invf}
\DeclareMathOperator{\Pos}{Pos}
\DeclareMathOperator{\Posf}{Posf}
\DeclareMathOperator{\Dirpos}{Dirpos}
\DeclareMathOperator{\Invpos}{Invpos}
\DeclareMathOperator{\Dirlf}{Dirlf}
\DeclareMathOperator{\Invlf}{Invlf}

\newcommand{\tm}[1]{ {{}^{\downarrow \downarrow} #1 }} 
\newcommand{\tw}[1]{ {{}^{\downarrow \uparrow} #1 }} 
 
\newcommand{\twwc}[1]{ {{}^{\downarrow \uparrow \uparrow \downarrow} #1 }} 
 
\newcommand{\twtw}[1]{ {{}^{\downarrow \downarrow \uparrow\uparrow} #1 }}

\begin{document}

\maketitle



\section*{Abstract}

In this article, it is shown that derivator six-functor-formalisms on any (classical) site canonically extend to higher geometric stacks as defined by To\"en-Vezzosi under some
natural locality conditions. As an application, it is shown that the six-functor-formalisms of Morel-Voevodsky-Ayoub extend to higher (Nisnevich\nobreakdash-)Artin stacks locally of finite type over some fixed base scheme.

\tableofcontents

\section{Introduction}

This  is the third article in a series of three \cite{Hor16, Hor17} concerning {\em derivator six-functor-formalisms}. Six-functor-formalisms were first introduced by Grothendieck, Verdier and Deligne \cite{Ver77, SGAIV1, SGAIV2, SGAIV3} and there 
has been increasing interest in them in various contexts in the last two decades \cite{FHM, Ayo07I, Ayo07II, LO08I, LO08II, CD19, LH09, Ayo14, ZL14, ZL14b, Zh10, BS15, Sch15, GR16, Sch17, Hoy17, Dre18, AGV20, DG20, Cho21, KR21, GHW22}.
When the author started this project quite a long time ago, it was not completely clear in which language such a theory should be developed. This work --- most notably influenced by the work of Ayoub \cite{Ayo07I, Ayo07II} and Cisinski \cite{Cis03, Cis04, Cis08} --- is written in the language of {\em derivators}.
Recently, many similar constructions have been made in the language of $\infty$-categories, and this is certainly the more modern way. Nevertheless, the author hopes that the language of derivators, which has definitely its own elegance, will not completely fade into obscurity. 

Recall that a symmetric {\em six-functor-formalism} is defined on a base category $\mathcal{S}$ with finite limits and specifies a (usually derived) category $\mathcal{D}_S$ for each object in $\mathcal{S}$, adjoint pairs of functors
\vspace{0.1cm}
\[
\begin{array}{lcrp{1cm}l}
f^* & &f_* &  & \text{ for each $f$ in $\Mor(\mathcal{S})$}  \\
f_! & &f^! & & \text{ for each $f$ in $\Mor(\mathcal{S})$} \\
\otimes & & \mathcal{HOM} & & \text{ in each fiber $\mathcal{D}_S$}
\end{array}
\]

and isomorphisms between the left adjoints (corresponding isomorphisms between the right adjoints follow formally --- see Proposition~\ref{PROPCONSEQUENCES}):
\begin{equation}\label{eqiso} \text{
\begin{tabular}{r|lc}
& isomorphisms \\ 
& between left adjoints \\
\hline
$(*,*)$ & $(fg)^* \iso g^* f^*$ & for composable morphisms $f, g$ \\
$(!,!)$ & $(fg)_! \iso f_! g_!$  &  ---''---   \\ 
$(!,*)$ & $g^* f_! \iso F_! G^*$ & where the morphisms $f, g, F, G$  \\
$(\otimes,*)$ & $f^*(- \otimes -) \iso f^*- \otimes f^* -$ &  form a Cartesian square \\
$(\otimes,!)$ & $f_!(- \otimes f^* -) \iso  (f_! -) \otimes -$  \\ 
$(\otimes, \otimes)$ &  $(- \otimes -) \otimes - \iso - \otimes (- \otimes -)$  
\end{tabular}}
\end{equation} 
as well as distinguished isomorphisms $f^! \cong f^*$ for isomorphisms $f$. 
For a detailed introduction to classical six-functor-formalisms the reader is referred to \cite{Hor16}. 
Without the tensor product we speak of a {\em $*,!$-formalism}. If $\mathcal{S}$ is small\footnote{We will always make this assumption. Of course, one can work with universes, and has to make sure that the universe used to define diagrams, pre-sheaves, the coproduct completion, etc.\@ is large enough such that $\mathrm{Ob}(\mathcal{S})$ is contained.}, and equipped with a Grothendieck topology, consider 
the model category of simplicial pre-sheaves $\mathcal{M} := \mathcal{SET}^{\mathcal{S}^{\op} \times \Delta^{\op}}$. It has a left Bousfield localization $\mathcal{M}_{\loc}$ at all \v{C}ech hypercovers. 
The objective of this article is to find necessary conditions under which the six-functor-formalism extends to $\mathcal{M}_{\loc}$ (or to a suitable subcategory) such that for $X \in \mathcal{M}_{\loc}$ the category 
$\mathcal{D}_X$ does  --- up to equivalence --- only depend on the weak equivalence class of $X$, and to construct isomorphisms between the six functors analogous to (\ref{eqiso}). It was explained in \cite{Hor21c} that, at least if there is a derivator enhancement of the six-functor-formalism
with stable and well-generated fibers, there are two potential ways of defining $\mathcal{D}_X$, one using cohomological descent for the $*$-functors, and one using homological descent for the $!$-functors. 
Informally, those categories are defined as follows: Since $\mathcal{S}$ is small, one can represent $X$ by a simplicial object $\widetilde{X}_\bullet \in \mathcal{S}^{\amalg, \Delta^{\op}}$. Then define
\begin{eqnarray} 
 \mathcal{D}_X^! &:=& \{ ( (\mathcal{E}_n)_n, (\mathcal{E}_m \cong \widetilde{X}(\alpha)^! \mathcal{E}_n)_{\alpha: \Delta_n \rightarrow \Delta_m}) \ | \   \mathcal{E}_n \in \mathcal{D}_{\widetilde{X}_n} \},   \label{eqcocartnaive1} \\
 \mathcal{D}_X^* &:=& \{ ( (\mathcal{E}_n)_n, (\widetilde{X}(\alpha)^* \mathcal{E}_n \cong \mathcal{E}_m)_{\alpha: \Delta_n \rightarrow \Delta_m}) \ | \   \mathcal{E}_n \in \mathcal{D}_{\widetilde{X}_n} \}.  \label{eqcocartnaive2}  
\end{eqnarray}
Of course, these definitions have to be understood in the sense of fibered derivators, i.e.\@ as {\em coherent} diagrams of shape $\Delta^{\op}$ (resp.\@ $\Delta$) which are Cartesian (resp.\@ coCartesian). For simplicity, we ignore the question of coherence for the moment, emphasizing however that it is precisely this problem that makes the objective of this article rather difficult and technically demanding.  

Fix a pre-topology on $\mathcal{S}$. 
There are three sufficient conditions for cohomological descent, namely for each covering $\{ f_i: U_i \rightarrow S \}$:
\begin{enumerate}
\item[(C1)]  $f_i^*$, as morphism of derivators, commutes with homotopy limits for all $i$;\footnote{If $\DD^*$ has well-generated but not compactly generated fibers then (C1) has to be replaced by a stronger axiom (C1'')}
\item[(C2)]  For each $i$ and each Cartesian square
\[ \xymatrix{ \ar[r]^{F_i} \ar[d]_G & \ar[d]^g \\ \ar[r]_{f_i} &  }\]
the base change morphism $f_{i}^* g_* \rightarrow G_* F_i^*$ is an isomorphism;
\item[(C3)]  The $f_i^*$ are jointly conservative. 
\end{enumerate}
Similarly there are three sufficient conditions for homological descent:
\begin{enumerate}
\item[(H1)]  $f_i^!$, as morphism of derivators, commutes with homotopy colimits for all $i$;
\item[(H2)]  For each $i$ and each Cartesian square
\[ \xymatrix{ \ar[r]^{F_i} \ar[d]_G & \ar[d]^g \\ \ar[r]_{f_i} &  }\]
the base change morphism $G_! F_i^! \rightarrow f_{i}^! g_! $ is an isomorphism;
\item[(H3)]  The $f_i^!$ are jointly conservative. 
\end{enumerate}

If these conditions hold, the Theorems of (co)homological descent \cite[Theorems 3.5.4.--3.5.5]{Hor15} combined with \cite[Main~Theorem~6.9]{Hor21c} imply that $\mathcal{D}_X^*$, and $\mathcal{D}_X^!$, respectively, depend only on the weak equivalence class of $X$ in $\mathcal{M}_{\loc}$. 
Furthermore for each morphism of stacks $X \rightarrow Y$ there are adjunctions
\[ \xymatrix{  \mathcal{D}_X^* \ar@/^10pt/[r]^{f_*} & \ar@/^10pt/[l]^{f^*}  \mathcal{D}_Y^* } \quad\text{resp.\@} \quad \xymatrix{  \mathcal{D}_Y^! \ar@/^10pt/[r]^{f^!} & \ar@/^10pt/[l]^{f_!}  \mathcal{D}_X^!. } \]
in which the functors $f^*$ and $f^!$ are induced by the homonymous functors in the given six-functor-formalism. However, their given adjoints $f_!$, $f_*$ do not preserve the (co)Cartesianity conditions in general and thus have to be composed with a (co)Cartesian projector to yield the above adjoints on $\mathcal{D}_X^!$, and $\mathcal{D}_X^*$, respectively. This is not constructive\footnote{not knowing the generators explicitly...}; the existence of these projectors is implied by the theory of well-generated triangulated categories. 
So far the interplay of the $!$-functors and the $*$-functors, i.e.\@ the base-change formula, did not play any role. 
Consider the following additional axiom: 
\begin{enumerate}
\item[(CH)]  For each Cartesian square
\[ \xymatrix{ X \ar[r]^{F} \ar[d]_G & X' \ar[d]^g \\ \prod_i Y_i \ar[r]_{\prod_i f_i} & \prod_i Y_i'  }\]
in which the $f_i: Y_i \rightarrow Y_i'$, $i=1, \dots, k$ \footnote{For a $*,!$-formalism we understand $k=1$ in axiom (CH) otherwise the statement does not make sense.} are part of a covering of $Y_i'$ for all $i$, 
the exchange $G_1^*f_1^!- \otimes \cdots \otimes G_k^* f_k^!- \rightarrow F^! (g^*_1 - \otimes \cdots \otimes g_k^* -)$ and for $k=1$ the exchange $F^*g^! \rightarrow G^! f^*$  (of the respective base-change)   are isomorphisms.
\end{enumerate}

The objective of this article is to show that, in the presence of the axioms (H1--H3), (C1--C3) and (CH), a given $*,!$-formalism (or a given six-functor-formalism) extends to the subcategory of $\mathcal{M}_{\loc}$ consisting of higher geometric stacks in the sense of To\"en-Vezzosi (w.r.t.\@ the class $C$ of morphisms induced by coverings in the chosen pre-topology, cf.\@ Remark~\ref{BEMTV} for a discussion). 
This means in particular: 
\begin{enumerate}
\item For a higher geometric stack $X$, the category $\mathcal{D}_X$ is equivalent to both the categories $\mathcal{D}_X^!$ and $\mathcal{D}_X^*$ described above in such a way that the $f_!$, $f^!$-functors, and the  $f_*$, $f^*$-functors agree (up to unique isomorphism) with those naturally defined on $\mathcal{D}_X^*$, and $\mathcal{D}_X^!$, respectively. 
\item For each {\em homotopy Cartesian} square of higher geometric stacks
\[ \xymatrix{ \ar[r]^{F} \ar[d]_G & \ar[d]^g \\ \ar[r]_{f} &  }\]
 we have a base-change isomorphism
\[ g^*  f_! \cong F_! G^*.  \]
\item For a weak equivalence $f: X \rightarrow Y$ between higher geometric stacks there is a canonical isomorphism $f^! \cong f^*$ and both functors are equivalences.
\item If the input was a six-functor-formalism (not just a $*,!$-formalism), then also a monoidal structure exists on $\mathcal{D}_X$ in such  a way that the $f^*$ are monoidal functors and the projection formula holds true. 
\end{enumerate}

We emphasize that the construction of the categories $\mathcal{D}_X$ and of the $*,!$-functors does not involve the tensor product. Nevertheless, assuming the existence of the tensor product, we can state stronger axioms in a more compact form which imply the conditions (C1--C3), (H1--H3) and (CH): We call a symmetric six-functor-formalism {\em strongly local} w.r.t.\@ the pre-topology, if for each covering $\{ f_i: U_i \rightarrow S \}$: 
 
 \begin{enumerate}
\item[(O1)] For each $i$ the object $f^!_i 1_S$ is tensor-invertible\footnote{i.e.\@ the functor $\mathcal{E} \mapsto \mathcal{E} \otimes (f^! 1_S)$ is an equivalence};
\item[(O2)] For each $i$ the natural morphism (exchange of the projection formula)
\[ f^!_i1_S \otimes (f_i^* -) \rightarrow  f_i^! -  \]
is an isomorphism;
 \item[(O3)] The $f^*_i$ or, equivalently\footnote{assuming (O1) and (O2)}, the $f_i^!$ are jointly conservative; 
 \item[(O4)] For each $i$ and each Cartesian square 
 \[ \xymatrix{
X' \ar[r]^{F_i} \ar[d]_G & S' \ar[d]^g \\
X \ar[r]_{f_i} & S
} \]
 the natural morphism (exchange of the base-change) 
\[  G^*f^!_i 1_S \rightarrow F^!_i g^* 1_S = F^!_i 1_{S'}  \]
is an isomorphism. 
 \end{enumerate}
It is (almost) an easy exercise (cf.\@ Proposition~\ref{PROPSTRONGLYLOCAL}) to prove that, if $F$ is strongly local, then the previous axioms (C1--C3), (H1--H3) and (CH) hold true. 
Of course, six-functor-formalisms are often strongly local, the object $f^!_i 1_S$ being the (relative) {\em dualizing object}. 

\vspace{0.2cm}
{\em Coherence of the six functors:}
\vspace{0.2cm}

Of course, to form a valid six-functor-formalism, the isomorphisms (\ref{eqiso}) have to fulfill compatibilities as, for example, the pentagon axiom, and many more. 
In \cite{Hor15, Hor15b} it was explained that using the language of (op)fibrations of 2-multicategories one can package all this information into a neat definition:
Let $\mathcal{S}^{\mathrm{cor}}$ be the symmetric 2-multicategory whose objects are the objects of $\mathcal{S}$ and in which a 1-morphism $\xi \in \Hom(S_1, \dots, S_n; T)$ is a multicorrespondence
\begin{equation}\label{excor}
 \vcenter{ \xymatrix{ 
 &&&  \ar[llld]_{g_1} A \ar[ld]^{g_n} \ar[rd]^{f} &\\
 S_1 & \cdots & S_n & ; &  T   } }
 \end{equation}
The composition of 1-morphisms is given by forming fiber products and the 2-morphisms are the isomorphisms of such multicorrespondences. 

\vspace{0.2cm}

{\bf Definition. }{\em 
A (symmetric) {\bf six-functor-formalism} on $\mathcal{S}$ is a 1-bifibration and 2-bifibration of (symmetric) 2-multicategories with 1-categorical fibers}
\[ p: \mathcal{D} \rightarrow \mathcal{S}^{\mathrm{cor}}. \]
Such a fibration can also be seen as a pseudo-functor of 2-multicategories
\[ \mathcal{S}^{\cor} \rightarrow \mathcal{CAT} \]
with the property that all multivalued functors in the image have right adjoints w.r.t.\@ all slots. 
Note that $\mathcal{CAT}$, the ``category''\footnote{having, of course, a {\em higher class} of objects} of categories, has naturally the structure of a symmetric 2-``multicategory'' where the 1-multimorphisms are functors of several variables.

The pseudo-functor maps the correspondence (\ref{excor}) to a functor isomorphic to 
\[ f_! ((g_1^* -) \otimes_A \cdots \otimes_A (g_n^* -)) \]
where $\otimes_A$, $f_!$, and $g^*_i$, are the images of  the following correspondences
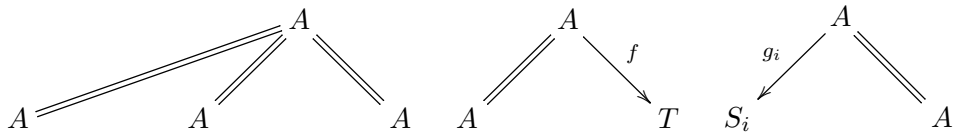
\[
 \vcenter{ \xymatrix{
&&&  A \ar@{=}[rd] \ar@{=}[ld] \ar@{=}[llld] \\
A &  & A  & &  A
} } \quad
 \vcenter{ \xymatrix{
& A \ar[rd]^f \ar@{=}[ld]  \\
A   & &  T
} } \quad
 \vcenter{ \xymatrix{
&  A \ar@{=}[rd] \ar[ld]_{g_i}  \\
S_i  & &  A
} }
\]

\vspace{0.2cm}
{\em Coherence of diagrams:}
\vspace{0.2cm}

For constructions like the categories $\mathcal{D}^!_X$ and $\mathcal{D}^*_X$ above, it is essential to be able to speak of coherent diagrams of objects involving the four functors --- as opposed to just diagrams in the derived category. The language of fibered derivators provides a suitable framework:
A {\em derivator six-functor-formalism} $\DD \rightarrow \SSS^{\cor}$ (cf.\@ \ref{PARDER6FU}) specifies not only a (derived) category for each object in $\mathcal{S}$ but
also a (derived) category for each diagram $I \rightarrow \mathcal{S}^{\cor}$ of correspondences in $\mathcal{S}$ for each small category $I$.
Over {\em constant diagrams} with value $S \in \mathcal{S}$ such datum gives back the derivator enhancement of the derived category $\mathcal{D}_S$.  
Objects in these categories $\DD(I)_{p^*S}$ should be seen as coherent versions of diagrams $I \rightarrow \mathcal{D}_S$, e.g.\@
commutative diagrams of actual complexes (or commutative diagrams in a model category) up to point-wise weak-equivalences as opposed to diagrams commuting only up to homotopy.
The extension to stacks envisioned above is a fibered multiderivator $\DD^{(\infty)} \rightarrow \SSS^{\hcor, (\infty)}$ where $\SSS^{\hcor, (\infty)}$ is the 2-pre-multiderivator of
 multicorrespondences between higher geometric stacks (cf.\@ Definition~\ref{DEFSHCOR}) in which composition is defined using {\em homotopy} fiber products.
 
 \vspace{0.2cm}
 {\bf Main Theorem \ref{HAUPTSATZ}. }{\em Let 
\[ \DD \rightarrow \SSS^{\cor} \]
be a derivator six-functor-formalism (i.e.\@ fibered multiderivator)
with stable, well-generated fibers which is local w.r.t.\@ the pre-topology on $\mathcal{S}$.
There is a derivator six-functor-formalism
\[ \DD^{(\infty)} \rightarrow \SSS^{\hcor, (\infty)} \]
extending $\DD  \to \SSS^{\cor}$ such that, for a higher geometric stack $X \in \SSS^{\hcor, (\infty)} (\cdot)$ represented by a simplicial object $\widetilde{X}_\bullet \in \mathcal{S}^{\amalg, \Delta^{\op}}$, there are equivalences:
\[ \DD^{(\infty)}(\cdot)_{X}  \cong \DD(\Delta^{\op})_{\widetilde{X}_\bullet}^{\cart} \cong \DD(\Delta)_{\widetilde{X}_\bullet^{\op}}^{\cocart}.  \] 
}

Here $\DD(\Delta^{\op})_{\widetilde{X}_\bullet}^{\cart}$ and $\DD(\Delta)_{\widetilde{X}_\bullet^{\op}}^{\cocart}$ are the correct derivator versions of the naive definitions (\ref{eqcocartnaive1}), and (\ref{eqcocartnaive2}), respectively.  
The extension $\DD^{(\infty)}$ is the union of analogous extensions $\DD^{(k)}$ to $k$-geometric stacks. 
The construction of  $\DD^{(k)}$  is done by induction on $k$. For a detailed explanation of the induction step the reader is referred to \ref{PARIDEA}.

\vspace{0.2cm}
{\em Example:} \nopagebreak
\vspace{0.2cm}

Let $S$ be a base scheme and let $\mathcal{S}$ be the category of quasi-affine schemes of finite type over $S$ (or quasi-projective schemes over $S$ ---  that amounts to the same class of higher Artin stacks). 
In \cite[Section~25]{Hor17} we showed that there is a symmetric derivator six-functor-formalism with domain $\Cat$
\[ \mathbb{SH}_{\mathcal{M}}^T \rightarrow  \mathcal{S}^{\op} \]
(for any choice of objects as in \ref{PARSETTINGAYOUB})
such that:
\begin{enumerate}
\item for a diagram $F: I \rightarrow \mathcal{S}^{\op}$ of quasi-affine schemes  of finite type over $S$ (embedded via the inclusion $\SSS^{\op} \rightarrow \SSS^{\cor}$), we have an equivalence of monoidal categories
\[ \mathbb{SH}_{\mathcal{M}}^T(I)_F \cong \mathbb{SH}_{\mathcal{M}}^T(F^{\op}, I^{\op}) \]
where the right hand side is the ``algebraic derivator'' defined by Ayoub \cite[D\'efinition~4.2.21]{Ayo07II};
\item the push-forward along a multicorrespondence 
\[ \xymatrix{ & & & A  \ar[llld]_{g_1}\ar[ld]^{g_n}\ar[rd]^f \\
S_1 & \cdots & S_n & ; & T 
 }\]
 in $\SSS^{\cor}(\cdot)$ is given, up to unique isomorphism, by
\[ f_! ( L g^*_1 - \overset{L}{\otimes} \cdots \overset{L}{\otimes} Lg_n^* -) \]
with the functors $f_!$ of \cite[Proposition~1.6.46]{Ayo07I} and the $g_i^*$ of \cite[Th\'eor\`eme~4.5.23]{Ayo07II}, cf.\@ also \cite[Scholie~1.4.2]{Ayo07I}.
\end{enumerate}

It follows from Ayoub's work \cite{Ayo07I,Ayo07II} (cf.\@ Section~\ref{SECTIONEXAMPLE}) that $\mathbb{SH}_{\mathcal{M}}^T$ is strongly local w.r.t.\@ to the Nisnevich-smooth\footnote{in which coverings are jointly surjective collections of smooth morphisms that are refined by a Nisnevich covering} pre-topology.  Sometimes it is even strongly local w.r.t.\@ the smooth\footnote{in which coverings are jointly surjective collections of smooth morphisms} pre-topology --- of course, the decisive question being here whether axiom (O3), i.e.\@ conservativity, holds. For example, the derivator six-functor-formalism $\mathbb{DA}^{et}$ of etale motives (without transfers) satisfies conservativity for etale covers and thus it is strongly local w.r.t.\@ the smooth pre-topology.

Let $\SSS^{\hcor, (\infty)}$ be the 2-pre-multiderivator of multicorrespondences of higher geometric stacks on $\mathcal{S}$\footnote{i.e.\@ higher Artin stacks that are locally of finite type over $S$} w.r.t.\@ the Nisnevich-smooth pre-topology (or, if appropriate, w.r.t.\@ the smooth pre-topology), cf.\@ Definition~\ref{DEFSHCOR}.
As a corollary of the fact that Ayoub's six-functor-formalism is strongly local, we thus get

\vspace{0.2cm}
{\bf Theorem~\ref{SATZEXTAYOUB}. }{\em There is a symmetric derivator six-functor-formalism 
\[ \mathbb{SH}_{\mathcal{M}}^{T, (\infty)} \rightarrow \SSS^{\hcor, (\infty)} \]
extending $\mathbb{SH}_{\mathcal{M}}^T  \to \SSS^{\cor}$ such that, given a higher geometric stack $X \in \SSS^{\hcor, (\infty)}(\cdot)$ represented by a simplicial object $\widetilde{X}_\bullet \in \mathcal{S}^{\amalg, \Delta^{\op}}$, we have equivalences:
\[ \mathbb{SH}_{\mathcal{M}}^{T, (\infty)}(\cdot)_{X}  \cong \mathbb{SH}_{\mathcal{M}}^T(\Delta^{\op})_{\widetilde{X}_\bullet}^{\cart} \cong \mathbb{SH}_{\mathcal{M}}^T(\Delta)_{\widetilde{X}_\bullet^{\op}}^{\cocart}.  \] 
}%
Note that $\mathbb{SH}_{\mathcal{M}}^T(\Delta)_{\widetilde{X}_\bullet^{\op}}$ (not the full subcategory of coCartesian objects) is just a value of the algebraic derivator defined by Ayoub. 

\section{Preliminaries and Notation}

Recall the following from \cite[7.3]{Hor16}:

\begin{PAR}\label{PARTW}
Let $I$ be a diagram, $n$ a natural number and let $\Xi = (\Xi_1, \dots, \Xi_n) \in \{ \uparrow, \downarrow \}^n$ be a sequence of arrow directions. We define a diagram
\[ {}^\Xi I \]
whose objects are sequences of $n$ objects and $k-1$ morphisms in $I$
\[ \xymatrix{
i_1 \ar[r] & i_2 \ar[r] & \cdots \ar[r]  & i_n
} \]
and whose morphisms are commutative diagrams
\[ \xymatrix{
i_1 \ar[r] \ar@{<->}[d] &i_2 \ar[r] \ar@{<->}[d] & \cdots \ar[r]  & i_n \ar@{<->}[d] \\
i_1' \ar[r] &i_2' \ar[r] & \cdots \ar[r]  & i_n' \\
} \]
in which the $j$-th vertical arrow goes in the direction indicated by $\Xi_j$. 
We call a morphism {\bf of type $j$} if at most the morphism $i_j \rightarrow i_j'$ is {\em not} an identity. 

For a diagram $I$ and an object $i \in I$ we adopt the convention that $i$ also denotes the subcategory of $I$ consisting only of $i$ and its identity. 
In coherence with this convention ${}^\Xi i$ denotes the subcategory of ${}^\Xi I$ consisting of the sequence $i = \cdots = i$ and its identity. 

Examples: ${}^{\downarrow \downarrow} I = I \times_{/I} I$ is the comma category, ${}^{\uparrow} I = I^{\op}$, and $\tw I$ is the twisted arrow category. 
\end{PAR}

\begin{PAR}\label{PARTW2}
For any ordered subset $\{i_1, \dots, i_m\} \subseteq \{1, \dots, n\}$, denoting $\Xi'$ the restriction of $\Xi$ to the subset, we get an obvious restriction functor
\[ \pi_{i_1, \dots, i_m}:\  {}^\Xi I \rightarrow {}^{\Xi'} I. \]

If $\Xi = \Xi' \circ \Xi'' \circ \Xi'''$ is a concatenation, then the projection
\[ \pi_{1,\dots,n'}: \ {}^{\Xi}  I \rightarrow {}^{\Xi'} I   \]
is a {\em fibration} if the last arrow of $\Xi'$ is $\downarrow$, and an {\em opfibration} if the last arrow of $\Xi'$ is $\uparrow$ while the projection
\[ \pi_{n-n'''+1,\dots,n}: \  {}^{\Xi} I \rightarrow {}^{\Xi'''} I   \]
is an {\em opfibration} if the first arrow of $\Xi'''$ is $\downarrow$ and a {\em fibration} if the first arrow of $\Xi'''$ is $\uparrow$.
\end{PAR}

\begin{PAR}\label{MORLIST}
Let $\mathcal{C}$ be a multicategory. For each pair of ordered set of objects $\mathcal{E}:=(\mathcal{E}_1, \dots, \mathcal{E}_n)$ and $\mathcal{F}:=(\mathcal{F}_1, \dots, \mathcal{F}_m)$ in $\mathcal{C}$ we define the set of {\bf morphisms} from $\mathcal{E}$ to $\mathcal{F}$ to be
a sequence of integers $0 \le n_1 \le \dots \le n_{m-1} \le n$ and multimorphisms
\[ \mathcal{E}_1, \dots, \mathcal{E}_{n_1} \rightarrow \mathcal{F}_1; \ 
 \mathcal{E}_{n_1+1}, \dots, \mathcal{E}_{n_2} \rightarrow \mathcal{F}_2; \ 
 \dots; \ 
 \mathcal{E}_{n_{m-1}+1}, \dots , \mathcal{E}_{n} \rightarrow \mathcal{F}_m  \]
The integers $n_i$ may be equal and also $n=0$ is allowed.  If $n=m=0$ we understand there to be exactly one morphism. 
\end{PAR}

\begin{PAR}\label{DEFTREE}
Recall that a {\bf tree} is a finite connected multicategory freely generated by a set of multimorphisms such that each
object occurs at most once as a source and at most once as a destination of one of these generating multimorphisms. 
The generating multimorphisms are allowed to be 0-ary. 
Each tree has a finite number of source objects (which do not occur as the destination of any non-identity morphism) and a unique destination object 
(which does not occur as the source of any non-identity morphism).

{\em Examples:}
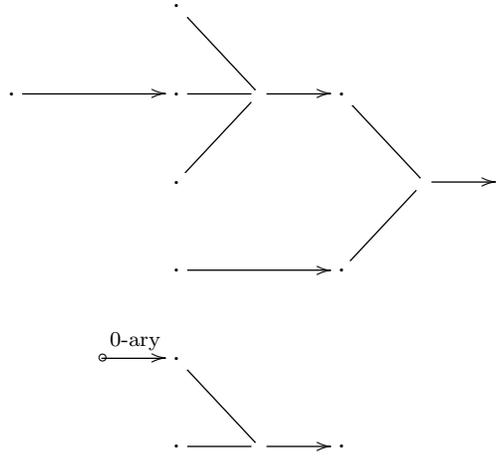
\[ \xymatrix{
&&\cdot \ar@{-}[dr] &   \\ 
\cdot \ar[rr] & & \cdot \ar@{-}[r] & \ar[r] & \cdot  \ar@{-}[rd] & &  \\
&&\cdot \ar@{-}[ur] & & & \ar@{->}[r] & \cdot \\
&&\cdot \ar[rr] & & \cdot \ar@{-}[ur] \\
 & \ar@{o->}[r]^{\text{0-ary}} & \cdot \ar@{-}[dr]&   \\
 && \cdot \ar@{-}[r]& \ar[r] & \cdot  \\
} \]

A {\bf symmetric tree} $\tau^S$ is obtained from a tree $\tau$ adding images (in the most free way possible) of the multimorphisms (not only the generating ones) under the respective symmetric groups. Observe that there is an obvious composition turning a symmetric tree into a symmetric multicategory. Giving a functor (of multicategories) from a tree to a symmetric multicategory $\mathcal{S}$ is the same as giving a functor (of symmetric multicategories) from its symmetric variant to $\mathcal{S}$. 

The most basic tree is $\Delta_{1,n}$ 
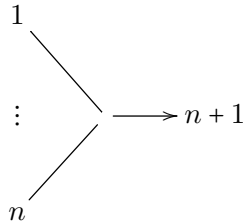
\[ \xymatrix{
1 \ar@{-}[rd]  &  \\
\vdots & \ar[r] & n+1 \\
 n \ar@{-}[ur]
} \]
consisting of $n+1$ objects and one $n$-ary morphism connecting them. 
Each tree has a well-defined destination object and a number (possibly zero) of source objects. Two trees $\tau_1$ and $\tau_2$ can by concatenated to a tree $\tau_2 \circ_i \tau_1$ choosing any source object $i$ of the tree $\tau_2$. 
\end{PAR}

\begin{PAR}\label{PARTWMULTI}
There exists the following multidiagram-variant of \ref{PARTW}.
Let $\Xi \in \{ \downarrow, \uparrow \}^l$ be sequence of arrow directions and let $M$ be a multidiagram (i.e.\@ a small multicategory like $\Delta_{1,n}$).
If $\Xi_l = \downarrow$, we define a small multicategory ${}^\Xi M$ and 
if $\Xi_l = \uparrow$, we define a small opmulticategory ${}^\Xi M$. We concentrate on the case $\Xi_l = \downarrow$ for definiteness. 
Objects are sequences
\[ \xymatrix{ [S_{1,1}, \dots, S_{1,n_1}] \ar[r] & [S_{2,1}, \dots, S_{2,n_2}] \ar[r] & \cdots \ar[r] &   [S_{l,1}] &   } \]
of $l$ lists of objects (can be empty) and morphisms in the sense of \ref{MORLIST} between them, where however the $l$-th list consist of exactly one object. 
Multimorphisms $S^{(1)}, \dots, S^{(k)} \rightarrow T$ are diagrams
\[ \xymatrix{ [S_{1,1}^{(1)}, \dots, S_{1,n_1}^{(1)}, \dots] \ar[r] \ar@{<->}[d] & [S_{2,1}^{(1)}, \dots, S_{2,n_2}^{(1)}, \dots] \ar[r]\ar@{<->}[d] & \cdots \ar[r] &   [S_{l,1}^{(1)}, \dots, S^{(k)}_{l,1}] \ar@{->}[d] &   \\
 [T_{1,1}, \dots, T_{1,n_1}] \ar[r] & [T_{2,1}, \dots, T_{2,n_2}] \ar[r] & \cdots \ar[r] &   [T_{l,1}] &   } \]
where the arrow direction in the $i$-th column is determined by $\Xi_i$. 
Such a morphism is called of {\bf type $i$} if all vertical morphisms except the $i$-th one are identities of lists. There are thus only $n$-ary morphisms for $n\not=1$ of type $l$ and not of any other type.

{\em Example:} For the tree $\Delta_{1,n}$ 
the multidiagram ${}^{\uparrow \downarrow} (\Delta_{1,n})$ is 
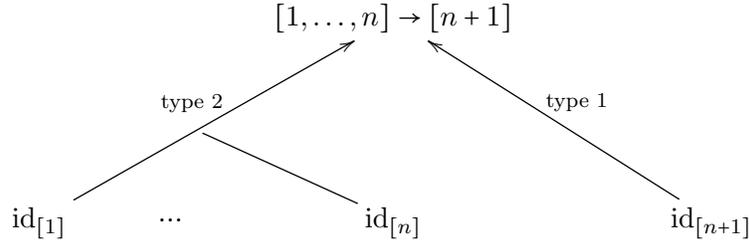
\[ \xymatrix{
&& [1,\dots,n] \rightarrow [n+1]  \\
& {\phantom{xxx}} \\ 
\id_{[1]} \ar[urur]^-{\text{type 2}}  & \cdots &  \ar@{-}[lu] \id_{[n]}  &  & \id_{[n+1]}  \ar[ulul]_{\text{type 1}}
} \]
\end{PAR}

\begin{DEF}\label{DEFDIAGRAMCAT}
A {\bf diagram category} is a full sub-2-category $\Dia \subset \mathrm{Cat}$,
satisfying the following axioms:
\begin{enumerate}
\item[(Dia1)] The empty category $\emptyset$, the final category $\cdot$ (or $\Delta_0$), and $\Delta_1$ are objects of $\Dia$.
\item[(Dia2)] $\Dia$ is stable under taking finite coproducts and fibered products.
\item[(Dia3)] All comma categories 
$I \times_{/J} K$ for functors $I \rightarrow J$ and $K \rightarrow J$ in $\Dia$ are in $\Dia$.
\end{enumerate}

A diagram category $\Dia$ is called {\bf infinite}, if it satisfies in addition:
\begin{itemize}
\item[(Dia5)] $\Dia$ is stable under taking arbitrary coproducts.
\end{itemize}
\end{DEF}

Under the term {\bf diagram} we understand a small category in the sequel.

\begin{BEISPIEL}\label{EXDIA}
We have the following diagram categories\footnote{We alert the reader that the notions $\Dir$ and $\Inv$ unfortunately are mixed up at several places in \cite{Hor15}  including \cite[2.1.2]{Hor15}.}:
\begin{itemize}
\item[$\Cat$:] the category of all {\bf diagrams}. 
\item[$\Dir$:] the category of {\bf directed diagrams} $C$, i.e.\@ small categories $C$ such that there exists a functor $C \rightarrow \N_0$ with the property that
the preimage of an identity consists of identities\footnote{In some sources $\N_0$ is replaced by any ordinal.}. An example is the injective simplex category $\Delta^\circ$:
\[ \xymatrix{ \cdots & \ar@<-1.5ex>[l]\ar@<-0.5ex>[l]\ar@<0.5ex>[l]\ar@<1.5ex>[l]    \cdot & \ar@<1ex>[l]\ar@<0ex>[l]\ar@<-1ex>[l] \cdot & \ar@<0.5ex>[l] \ar@<-0.5ex>[l]  \cdot  } \]
\item[$\Inv$:] the category of {\bf inverse diagrams} $D$, i.e.\@ small categories such that $D^{\op}$ is inverse. An example is the opposite of the injective simplex category $(\Delta^\circ)^{\op}$:
\[ \xymatrix{ \cdots \ar@<-1.5ex>[r]\ar@<-0.5ex>[r]\ar@<0.5ex>[r]\ar@<1.5ex>[r]  &   \cdot \ar@<1ex>[r]\ar@<0ex>[r]\ar@<-1ex>[r] & \cdot \ar@<0.5ex>[r] \ar@<-0.5ex>[r] & \cdot  } \]
\item[$\Catf$,] $\Dirf$, and $\Invf$ are defined as before but consisting of {\bf finite diagrams}.  We have obviously $\Dirf = \Invf$.
\item[$\Pos$,] $\Posf$, $\Dirpos$, and $\Invpos$: the categories of {\bf posets, finite posets, directed posets,} and {\bf inverse posets}. 
We have $\Posf \subset \Dirf = \Invf \subset \Catf$. 
\item[$\Catlf$:] the category of {\bf locally finite diagrams}, i.e.\@ in which morphisms factor only in a finite number of ways into non-identity morphisms. 
\end{itemize}
\end{BEISPIEL}

\begin{PAR}For a category $\mathcal{S}$ we denote by $\Cat(\mathcal{S})$ the 2-category of diagrams in $\mathcal{S}$ whose objects are pairs $(I, S)$ consisting of a diagram $I$, and a functor $S \in \mathcal{S}^I$. 1-Morphisms $(I,S) \rightarrow (J, T)$ are pairs $(\alpha, f)$ with $\alpha: I \rightarrow J$ and $f: S \rightarrow \alpha^*T$. 2-morphisms are natural transformations between the $\alpha$ satisfying an obvious compatibility with the $f$. Similarly
$\Cat^{\op}(\mathcal{S})$ is defined with the same objects and 1-morphisms $(I,S) \rightarrow (J, T)$ being pairs $(\alpha, f)$ with $f: \alpha^*T \rightarrow S$, instead. 
There is an equivalence $\op: \Cat(\mathcal{S})^{2-\op} \rightarrow \Cat^{\op}(\mathcal{S}^{\op})$ mapping $(I, S)$ to $(I^{\op}, S^{\op})$. 
\end{PAR}

\begin{PAR}\label{PARINTNABLA}
We denote by $\mathcal{S}^{\amalg} \subset \Cat(\mathcal{S})$ the free coproduct completion of $\mathcal{S}$ whose objects are pairs $(S_0, S)$ consisting of a set $S_0$ (considered as a discrete category) together with a functor (i.e.\@ collection indexed by $S_0$) $S \in \mathcal{S}^{S_0}$. 
Recall that there is a natural functor extending the Grothendieck construction $\int: \mathcal{SET}^{\Delta^{\op}} \rightarrow \Cat$ to 
\[ \int^{\amalg}: \mathcal{S}^{\amalg, \Delta^{\op}} \rightarrow \Cat(\mathcal{S}) \]
and a similar functor
\[ \nabla^{\amalg}:= \op \circ (\int^{\amalg}): \mathcal{S}^{\amalg, \Delta^{\op}} \rightarrow \Cat^{\op}(\mathcal{S}^{\op}). \]
\end{PAR}

\begin{PAR}\label{PARCWFO}
We say that a triple $(\mathcal{M}, \Fib, \mathcal{W})$ consisting of a category $\mathcal{M}$ and classes $\Fib$ (fibrations) and $\mathcal{W}$ (weak equivalences) of morphisms in $\mathcal{M}$ is a {\bf category of fibrant objects} if the following axioms are satisfied
\begin{enumerate}
\item $\mathcal{M}$ has finite products (in particular, a terminal object). 
\item $\mathcal{W}$ satisfies 2-out-of-3 and contains the isomorphisms. 
\item $\Fib$ is closed under composition, contains the isomorphisms, and the morphisms to the terminal object (``all objects are fibrant''). 
\item $\Fib$ is stable under pull-back, i.e.\@ the pullback of a fibration along an arbitrary morphism exists and is again a fibration.
\item $\Fib \cap \mathcal{W}$ is stable under pull-back.
\item (Existence of path objects) The diagonal $X \rightarrow X \times X$ can be factored as
\[ \xymatrix{ X \ar[r]^-{\in \mathcal{W}} &  X^I \ar[r]^-{\in \Fib} &  X \times X. } \]
We assume that this can be made functorial, i.e.\@ there is a functor $X \mapsto X^I$, such that the above sequence is natural in $X$. 
\end{enumerate}
We call $(\mathcal{M}, \Fib, \mathcal{W})$ a {\bf simplicial category of fibrant objects} if axioms 1--5 are satisfied and instead of axiom 6 the following holds
\begin{enumerate}
\item[6'.] There is a functor
\[ \Hom(\mathcal{SET}^{\Delta^{\op}}_{fin}, \mathcal{M}) \rightarrow \mathcal{M} \]
such that for a cofibration $f: X \rightarrow Y$ of finite simplicial sets, and, a fibration $g: X' \rightarrow Y'$  the morphism
\[  \Hom(Y,X') \overset{\boxdot \Hom (f, g)}{\longrightarrow} \lim \left( \vcenter{ \xymatrix{ & \Hom(Y,Y') \ar[d] \\ \Hom(X, X') \ar[r] & \Hom(X, Y') } } \right) \]
exists and is a fibration. It is a weak equivalence if $f$ or $g$ are. 
\end{enumerate}
Axiom 6' is a convenient generalization of axiom 6 (including functoriality). This additional axiom is satisfied for many categories of fibrant objects, in particular for those coming from simplicial model categories. 
This allows for the constructions in Section~\ref{SECTMC} without having to choose (Reedy) fibrant replacements of diagrams. 

In a category of fibrant objects, there is a notion of homotopy of morphisms and the homotopy categories $\mathcal{C}[\mathcal{W}^{-1}]$ can be constructed
using a calculus of fractions after identifying homotopic morphisms. Furthermore, pull-backs along fibrations are homotopy pull-backs. 
\end{PAR}

\begin{PAR}\label{PARHSF}
Let  $(\mathcal{M}, \Fib, \mathcal{W})$ be a category with fibrant objects. We say that a full subcategory $\mathcal{M}'$ is {\bf homotopy strictly full} if with an object $X \in \mathcal{M}'$ also every weakly equivalent object $X'$ is in $\mathcal{M}'$. If $\mathcal{M}'$ is furthermore closed under homotopy pull-back $(\mathcal{M}', \Fib', \mathcal{W}')$ is again a category with fibrant objects, where $\Fib'$ and $\mathcal{W}'$ are the intersections of the corresponding classes with $\mathcal{M}'$. Furthermore $\mathcal{M}'[(\mathcal{W}')^{-1}]$ is a strictly full subcategory of $\mathcal{M}[\mathcal{W}^{-1}]$. 
\end{PAR}

\begin{PAR}\label{PARHFP}
Recall that in a category with fibrant objects for every solid diagram in $\mathcal{M}[\mathcal{W}^{-1}]$
\[ \xymatrix{ Y \widetilde{\times}_W Z \ar@{.>}[r] \ar@{.>}[d] & Y \ar[d] \\
Z \ar[r]  & W  \\
} \]
there is a {\bf homotopy fiber product} extending the given diagram, well-defined up to (non-unique!) isomorphism. It can be computed (up to isomorphism in the homotopy category) as the honest homotopy fiber product of any
diagram of shape $\righthalfcup$ in $\mathcal{M}$ that represents the given one. 
\end{PAR}

\begin{DEF}\label{DEFCARTMORPH}
Let $I$ be a diagram.
A morphism $g: X \rightarrow Y$  in $\mathcal{M}^{I}$ is called a  {\bf homotopy Cartesian} morphism, if all squares
\[ \xymatrix{ X(i) \ar[r]^{g(i)} \ar[d] & Y(i) \ar[d] \\
X(j) \ar[r]_{g(j)} & Y(j)  \\
} \]
for morphisms $i \rightarrow j$ are homotopy Cartesian. 
\end{DEF}

\begin{PAR}\label{PARFPSOBJ}
Let $(\mathcal{M}, \Fib, \mathcal{W})$ be a simplicial category with fibrant objects.
Let $X \in \mathcal{M}$ be an object. 
 We define an (essentially constant) simplicial object $\delta(X)_\bullet \in \mathcal{M}^{\Delta^{\op}}$
depending functorially on $X$  by $n \mapsto \delta(X)_n:=\Hom(\Delta_n, X)$. It comes equipped with a point-wise weak equivalence $\iota: X \rightarrow \delta(X)_\bullet$ from the constant simplicial object with value $X$. 
\end{PAR}

\begin{DEF}\label{DEFCECH}
Let $(\mathcal{M}, \Fib, \mathcal{W})$ be a simplicial category with fibrant objects. 
Let $f: U \rightarrow X$ be a morphism in $\mathcal{M}$. We define a simplicial object, the \v{C}ech nerve
\[ \Cech(f) \in \mathcal{M}^{\Delta^{\op}} \]
\[ \Cech(f)_n := \lim \CechDia(f)_n \]
where  $\CechDia(f)_n \in \Cat^{\op}(\mathcal{M})$ is defined by
\[ \CechDia(f)_n := \nabla \left( \vcenter{ \xymatrix{  & \coprod_{0..n} (\cdot, U)   \\
(\cdot, \Hom(\Delta_n, X)) & \ar[l] \coprod_{0..n} (\cdot, X) \ar[u] }} \right)  \cdot   \]
Note that the limit exists because $\Hom(\Delta_n, X) \rightarrow X^n$ is a fibration by assumption. 
\end{DEF}

\begin{LEMMA}\label{LEMMAPROPCECH}
\begin{enumerate}
\item Definition~\ref{DEFCECH} yields a well-defined functor
\[ \Cech:  \mathcal{M}^{\rightarrow} \rightarrow \mathcal{M}^{\Delta^{\op}}  \]
which maps weak equivalences to weak equivalences and hence descends to a functor of homotopy categories.
\item There is a morphism of simplicial objects
\[ \Cech(f) \rightarrow \delta(X). \]
\item 
There is an isomorphism in the homotopy category
\[ U \widetilde{\times}_X  \cdots  \widetilde{\times}_X  U \cong \Cech(f)_n.  \]
\item For a homotopy Cartesian diagram
\[ \xymatrix{ U' \ar[d]_{F} \ar[r] & U \ar[d]^f \\
Y \ar[r] & X  } \]
there is a homotopy Cartesian square in the homotopy category of $\mathcal{M}^{\Delta^{\op}}$
\[ \xymatrix{ \Cech(F) \ar[r] \ar[d] & \Cech(f)  \ar[d] \\ 
Y \ar[r] & X  } \]
in which  $Y$ and $X$ are considered constant. 
\end{enumerate}
\end{LEMMA}
\begin{proof}1.\@ is clear. 

2.\@ The morphism is given point-wise by the projection 
\[ \lim \CechDia(f)_n \rightarrow \Hom(\Delta_n, X). \]

3.\@ $\Hom(\Delta_n, X) \rightarrow X^n$ is a fibration hence $\lim \CechDia(f)_n \cong \holim \CechDia(f)_n \cong U \widetilde{\times}_X U \widetilde{\times}_X \cdots  \widetilde{\times}_X  U$.

4.\@ The given homotopy Cartesian square in $\mathcal{M}$ yields 
by 1.\@  a square
\[ \xymatrix{ \Cech(F) \ar[r] \ar[d] & \Cech(f)  \ar[d] \\ 
\Cech(\id_Y) \ar[r] & \Cech(\id_X)  } \]
which is point-wise homotopy Cartesian by 3.\@ and standard properties. But $\Cech(\id_Y) \cong Y$ and $\Cech(\id_X) \cong X$ in the homotopy category of $\mathcal{M}^{\Delta^{\op}}$. 
\end{proof}

\begin{DEF}\label{DEFFPS}
Let $(\mathcal{M}, \Fib, \mathcal{W})$ be a simplicial category with fibrant objects.
Let $X \rightarrow Z  \leftarrow Y$ be a diagram in $\mathcal{M}$. We define a specific homotopy fiber product by the following diagram with Cartesian squares
\[ \xymatrix{  X \widetilde{\times}_{Z} Y \ar[r] \ar[d] & \Box \ar[r] \ar[d] & Y \ar[d]  \\
 \Box \ar[r] \ar[d] &  \Hom(\Delta_1, Z) \ar[r]^-{\Fib} \ar[d]^-{\Fib}  & Z \\ 
 X \ar[r] & Z  } \]
(This is the usual fiber product construction in categories with fibrant objects if $\Hom(\Delta_1, Z)$ is understood to be the path object $Z^I$.)
\end{DEF}
It induces a well-defined commutative square 
\[ \xymatrix{ X \widetilde{\times}_Z Y  \ar[r]  \ar[d]  & Y \ar[d] \\
X \ar[r] & Z } \]
in the homotopy category of $\mathcal{M}$ which is homotopy Cartesian.

\begin{DEF}\label{DEFFPSOBJ}
Let $(\mathcal{M}, \Fib, \mathcal{W})$ be a simplicial category with fibrant objects.
Let $X_\bullet, Y_\bullet$ be objects in $\mathcal{M}^{\Delta^{\op}}$ and $Z \in \mathcal{M}$ equipped with morphisms (in $\mathcal{M}^{\Delta^{\op}}$) $X_\bullet \rightarrow \delta(Z)_\bullet$ and $Y_\bullet \rightarrow \delta(Z)_\bullet$.

For $m,n \in \N_0$ define an object $(X \widetilde{\times}_{Z} Y)_{n,m} \in \mathcal{M}$ by the following commutative diagram with Cartesian squares
\[ \xymatrix{  (X \widetilde{\times}_{Z} Y)_{n,m} \ar[r] \ar[d] & \Box \ar[r] \ar[d] & Y_m \ar[d]  \\
 \Box \ar[r] \ar[d] &  \Hom(\Delta_n \star \Delta_m, Z) \ar[r]^-{\Fib} \ar[d]^-{\Fib}  & \Hom(\Delta_m, Z) \\ 
 X_n \ar[r] & \Hom(\Delta_n, Z)  } \]
 \end{DEF}
 Be aware that not only $X_n$ and $Y_m$ depend on the indices $m$ and $n$ but also 
 the fiber product construction. 
  
 \begin{LEMMA}\label{LEMMAHOFP}
 The construction in Definition~\ref{DEFFPSOBJ} defines a bi-simplicial set $(X \widetilde{\times}_Z Y)_{\bullet,\bullet}: (\Delta^{\op})^2 \rightarrow \mathcal{M}$ such that 
 \begin{enumerate} 
 \item If $X$ and $Y$ are constant, there is a (point-wise) weak equivalence of bisimplicial objects
 \[ \gamma: X \widetilde{\times}_{Z} Y  \rightarrow (\delta(X) \widetilde{\times}_{Z} \delta(Y))_{n,m}  \]
  where the left hand side is considered constant. 
\item There are  well-defined homotopy Cartesian squares in the homotopy category of $\mathcal{M}^{\Delta^{\op}}$
\[ \xymatrix{ (X \widetilde{\times}_Z Y)_{n,\bullet}  \ar[r]  \ar[d]  & Y_{\bullet} \ar[d] \\
X_n \ar[r] & Z } \quad \xymatrix{ (X \widetilde{\times}_Z Y)_{\bullet,m}  \ar[r]  \ar[d]  & Y_m \ar[d] \\
X_\bullet \ar[r] & Z } \]
in which $Z, X_n,$ and $Y_m$ are considered constant.  
 \item The association $X, Y, Z \mapsto (X \widetilde{\times}_Z Y)_{\bullet,\bullet}$ maps (point-wise) homotopy Cartesian squares in each variable to point-wise homotopy Cartesian squares. 
  \item
 A square of the form
\[ \xymatrix{ 
 (X \widetilde{\times}_Y Z)_{n,m} \ar[r] \ar[d] &  (X \widetilde{\times}_Y Z)_{n,m'} \ar[d]  \\
 (X \widetilde{\times}_Y Z)_{n',m} \ar[r] &  (X \widetilde{\times}_Y Z)_{n',m'} } \]
is homotopy Cartesian. 
 \end{enumerate} 
 \end{LEMMA}
\begin{proof}
 1. The morphism is induced by the obvious morphism
 {\footnotesize
\[ \left(\vcenter{  \xymatrix{   & & Y \ar[d]  \\
  &  \Hom(\Delta_1, Z) \ar[r]^-{\Fib} \ar[d]^-{\Fib}  & Z \\ 
 X \ar[r] & Z  } } \right) \rightarrow 
 \left(\vcenter{ \xymatrix{   & & \Hom(\Delta_n, Y) \ar[d]  \\
 &  \Hom(\Delta_m \star \Delta_n, Z) \ar[r]^-{\Fib} \ar[d]^-{\Fib}  & \Hom(\Delta_n, Z) \\ 
\Hom(\Delta_m, X) \ar[r] & \Hom(\Delta_m, Z)  } }  \right) \]
 }
 
 2.--4. are clear. 
\end{proof}

\section{Fibered multiderivators over 2-categorical bases}\label{SECTFIBDER}

In this section, we recall from \cite{Hor16} the notions of 2-pre-multiderivator and fibered multiderivator (over 2-categorical bases). We will not use 
the more general notions of lax or oplax 2-pre-multiderivator.
Let $\Dia$ be a diagram category (cf.\@ Definition~\ref{DEFDIAGRAMCAT}).

\begin{DEF}[{\cite[Definition~2.1]{Hor16}}]\label{DEF2PREMULTIDER}
A {\bf 2-pre-multiderivator} with domain $\Dia$ is a functor $\SSS: \Dia^{1-\op} \rightarrow \text{2-$\mathcal{MCAT}$}$ which is strict in 1-morphisms (functors) and pseudo-functorial in 2-morphisms (natural transformations). 
More precisely, it associates with a diagram $I$ a 2-multicategory $\SSS(I)$, with a functor $\alpha: I \rightarrow J$ a strict functor
\[ \SSS(\alpha):  \SSS(J) \rightarrow \SSS(I) \] 
denoted also $\alpha^*$ if $\SSS$ is understood, and with a natural transformation $\mu: \alpha \Rightarrow \alpha'$ a pseudo-natural transformation
\[ \SSS(\eta): \alpha^* \Rightarrow (\alpha')^* \]
such that the following holds:
\begin{enumerate}
\item The association
\[ \Fun(I, J) \rightarrow \Fun^{\mathrm{strict}}(\SSS(J), \SSS(I)) \]
given by $\alpha \mapsto \alpha^*$, resp.\@ $\mu \mapsto \SSS(\mu)$, is
a pseudo-functor (this involves, of course, the choice of further data). Here $\Fun^{\mathrm{strict}}(\SSS(J), \SSS(I))$ is the 2-category of strict 2-functors, pseudo-natural transformations, and modifications. 
\item (Strict functoriality w.r.t.\@ compositons of 1-morphisms) For functors $\alpha: I \rightarrow J$ and $\beta: J \rightarrow K$, we have 
an {\em equality} of pseudo-functors $\Fun(I, J) \rightarrow \Fun^{\mathrm{strict}}(\SSS(I), \SSS(K))$
\[ \beta^* \circ \SSS(-) = \SSS(\beta \circ -).   \]
\end{enumerate}

A {\bf symmetric, resp.\@ braided 2-pre-multiderivator} is given by the structure of strictly symmetric (resp.\@ braided) 2-multicategory on $\SSS(I)$ such that
the strict functors $\alpha^*$ are equivariant w.r.t.\@ the action of the symmetric groups (resp.\@ braid groups). 
\end{DEF}

\begin{DEF}[{\cite[Definition~2.2]{Hor16}}]\label{DEF2PREMULTIDERSTRICTMOR}
A strict morphism $p: \DD \rightarrow \SSS$ of 2-pre-multiderivators  is given by a collection of strict 2-functors
\[ p(I): \DD(I) \rightarrow \SSS(I) \]
for each $I \in \Dia$ such that we have $\SSS(\alpha) \circ p(J) = p(I) \circ \DD(\alpha)$ and $\SSS(\mu) \ast p(J) = p(I) \ast \DD(\mu)$ 
 for all functors $\alpha: I \rightarrow J$, $\alpha': I \rightarrow J$ and natural transformations $\mu: \alpha \Rightarrow \alpha'$ as illustrated by the following diagram:
\[ \xymatrix{
\DD(J) \ar[rr]^{p(J)} \ar@/_15pt/[dd]_{\DD(\alpha)}^{\phantom{x}\overset{\DD(\mu)}{\Rightarrow}} \ar@/^15pt/[dd]^{\DD(\alpha')} && \SSS(J) \ar@/_15pt/[dd]_{\SSS(\alpha)}^{\phantom{x}\overset{\SSS(\mu)}{\Rightarrow}} \ar@/^15pt/[dd]^{\SSS(\alpha')} \\
\\
\DD(I) \ar[rr]^{p(I)} && \SSS(I)
} \]
\end{DEF}

\begin{PAR}\label{PARDER12}
As with usual pre-multiderivators we consider the following axioms: 

\begin{itemize}
\item[(Der1)] For $I, J \in \Dia$, the natural functor $\DD(I \coprod J) \rightarrow \DD(I) \times \DD(J)$ is an equivalence of 2-multicategories. Moreover $\DD(\emptyset)$ is not empty.
\item[(Der2)]
For $I \in \Dia$ the `underlying diagram' functor
\[ \dia: \DD(I) \rightarrow \Fun(I, \DD(\cdot)) \]
is 2-conservative (this means that it is conservative on 2-morphisms and that a 1-morphism $\alpha$ is an equivalence if $\dia(\alpha)$ is an equivalence).
\end{itemize}

A 2-pre-multiderivator with domain $\Dia$ is called {\bf infinite} if $\Dia$ is infinite (i.e.\@ closed under arbitrary coproducts)  and we have
\begin{itemize}
\item[(Der1${}^\infty$)] For $\{I_i\}_{i \in \mathcal{I}}$ a (possibly infinite) family with $I_i \in \Dia$, the natural functor $\DD(\coprod_{i} I_i) \rightarrow \prod_{i} \DD(I_i)$ is an equivalence of 2-multicategories. 
\end{itemize}
\end{PAR}

\begin{PAR}
Consider the following axioms on a strict morphism $p: \DD \rightarrow \SSS$ of 2-pre-multiderivators (where (FDer0 left) is assumed for (FDer3--5 left) and similarly for the right case):
\begin{itemize}
\item[(FDer0 left)]
For each $I$ in $\Dia$ the morphism $p$ specializes to an 1-opfibered 2-bifibered 2-multicategory with 1-categorical fibers. 
Moreover any {\em fibration} $\alpha: I \rightarrow J$ in $\Dia$ induces a  diagram
\[ \xymatrix{
\DD(J) \ar[r]^{\alpha^*} \ar[d] & \DD(I) \ar[d]\\
\SSS(J) \ar[r]^{\alpha^*} & \SSS(I) 
}\]
of 1-opfibered and 2-bifibered 2-multicategories, i.e.\@ the top horizontal functor maps coCartesian 1-morphisms to coCartesian 1-morphisms and (co)Cartesian 2-morphisms to (co)Cartesian 2-morphisms.

We assume that corresponding push-forward functors between the fibers have been chosen and those will be denoted by $(-)_\bullet$.

\item[(FDer3 left)]
For each functor $\alpha: I \rightarrow J$ in $\Dia$ and $S \in \SSS(J)$ the functor
$\alpha^*$ between fibers (which are 1-categories by (FDer0 left))
\[ \DD(J)_{S} \rightarrow \DD(I)_{\alpha^*S} \]
has a left adjoint $\alpha_!^{(S)}$.

\item[(FDer4 left)]
For each functor $\alpha: I \rightarrow J$ in $\Dia$, and for any object $j \in J$, and for the 2-commutative square
\[ \xymatrix{  I \times_{/J} j \ar[r]^-\iota \ar[d]_{\alpha_j} \ar@{}[dr]|{\Swarrow^\mu} & I \ar[d]^\alpha \\
\{j\} \ar@{^{(}->}[r]^j & J \\
} \]
 the induced natural transformation of functors \[ \alpha_{j,!}^{(j^*S)} \SSS(\mu)(S)_\bullet \iota^* \rightarrow j^* {\alpha}_!^{(S)} \] is an isomorphism for all $S \in \SSS(J)$.
\item[(FDer5 left)] For any {\em opfibration} $\alpha: I \rightarrow J $ in $\Dia$, and for any 1-morphism $\xi \in \Hom(S_1, \dots, S_n; T)$ in $\SSS(J)$ for some $n\ge 1$, the natural transformations of functors
\[ \alpha_! (\alpha^*\xi)_\bullet (\alpha^*-, \cdots, \alpha^*-,\ \underbrace{-}_{\text{at }i}\ , \alpha^*-, \cdots, \alpha^*-) \cong  \xi_\bullet (-, \cdots, -,\ \underbrace{\alpha_!-}_{\text{at }i}\ , -, \cdots, -) \]
are isomorphisms for all $i=1, \dots,  n$.
\end{itemize}
\end{PAR}

and dually: 

\begin{enumerate}
\item[(FDer0 right)]
For each $I$ in $\Dia$ the morphism $p$ specializes to a 1-fibered 2-bifibered 2-multicategory with 1-categorical fibers.
Furthermore, any {\em opfibration} $\alpha: I \rightarrow J$
in $\Dia$  induces a diagram
\[ \xymatrix{
\DD(J) \ar[r]^{\alpha^*} \ar[d] & \DD(I) \ar[d]\\
\SSS(J) \ar[r]^{\alpha^*} & \SSS(I) 
}\]
of 1-fibered and 2-bifibered multicategories, i.e.\@ the top horizontal functor maps Cartesian 1-morphisms w.r.t.\@ the $i$-th slot to Cartesian 1-morphisms w.r.t.\@ the $i$-th slot for any $i$ and maps (co)Cartesian 2-morphisms to (co)Cartesian 2-morphisms.

We assume that corresponding pull-back functors between the fibers have been chosen and those will be denoted by $(-)^{\bullet,i}$.

\item[(FDer3 right)]
For each functor $\alpha: I \rightarrow J$ in $\Dia$ and $S \in \SSS(J)$ the functor
$\alpha^*$ between fibers (which are 1-categories by (FDer0 right))
\[ \DD(J)_{S} \rightarrow \DD(I)_{\alpha^*S} \]
has a right adjoint $\alpha_*^{(S)}$.
\item[(FDer4 right)]
For each morphism $\alpha: I \rightarrow J$ in $\Dia$, and for any object $j \in J$, and for the 2-commutative square
\[ \xymatrix{  j \times_{/J} I \ar[r]^-\iota \ar[d]_{\alpha_j} \ar@{}[dr]|{\Nearrow^\mu} & I \ar[d]^\alpha \\
\{j\} \ar@{^{(}->}[r]^j & J \\
} \]
 the induced natural transformation of functors \[  j^* \alpha_*^{(S)} \rightarrow \alpha_{j,*}^{(j^*S)} \SSS(\mu)(S)^\bullet \iota^* \] is an isomorphism for all $S \in \SSS(J)$.

\item[(FDer5 right)] For any {\em fibration} $\alpha: I \rightarrow J$ in $\Dia$, and for any 1-morphism $\xi \in \Hom(S_1, \dots, S_n; T)$ in $\SSS(J)$ for some $n\ge 1$, the natural transformations of functors
\[ \alpha_* (\alpha^*\xi)^{\bullet,i} (\alpha^*-, \overset{\widehat{i}}{\cdots}, \alpha^*-\ ;\ -) \cong  \xi^{\bullet,i} (-, \overset{\widehat{i}}{\cdots}, -\ ;\ \alpha_*-) \]
are isomorphisms for all $i= 1, \dots, n$.
\end{enumerate}

\begin{DEF}\label{DEFFIBDER}
A strict morphism $p: \DD \rightarrow \SSS$ of 2-pre-multiderivators is called a  {\bf  left (resp.\@ right) fibered multiderivator}
if $\DD$ and $\SSS$ both satisfy (Der1) and (Der2) if (FDer0 left/right) and (FDer3--5 left/right) hold true. We say ``fibered'' for ``left and right fibered''.
\end{DEF}

\begin{BEM}
One can show that the axioms imply that the second part of (FDer0 left) and (FDer5 right) --- which are adjoint to each other ---  hold true for any functor $\alpha: I \rightarrow J$. 
Similarly {\em for 1-ary morphisms} (but not necessarily for $n$-ary morphisms) the second part of (FDer0 right) and (FDer5 left) hold true for any functor $\alpha: I \rightarrow J$. 
\end{BEM}

\begin{DEF}\label{DEFSTABLE}For a fibered multiderivator $p: \DD \rightarrow \SSS$ over a 2-pre-multiderivator and an object $S \in \SSS(I)$ the association
\[ \DD_{I,S}: J \mapsto \DD(I \times J)_{\pr_1^*S}  \]
is a usual derivator, called its {\bf fiber} above $(I, S)$. We say that $p$ has {\bf stable fibers} if $\DD_{I,S}$ is stable for all $I$ and for all $S \in \SSS(I)$. In fact, it suffices to require this for $I=\{\cdot\}$. 
\end{DEF}

We include strongness into the definition of stable. Strongness can be more generally defined for 2-pre-multiderivators:

\begin{DEF}
A 2-pre-multiderivator is called {\bf strong} if the partial underlying diagram functors
\[ \dia_{\Delta_1}: \SSS(\Delta_1 \times I) \rightarrow \Fun(\Delta_1, \SSS(I)) \]
are essentially surjective (in the 2-categorial sense), and the induced functor on morphism categories is essentially surjective and full. 
\end{DEF}

\begin{PAR}
In \cite[Definition~4.1]{Hor16} an alternative definition of a fibered multiderivator over $\SSS$ was given: a morphism of 2-pre-multiderivators 
$\SSS \rightarrow \DD$ such that both satisfy (Der1) and (Der2) and such that
\[ \Dia^{\cor}(\DD) \rightarrow \Dia^{\cor}(\SSS) \]
is a 1-bifibration and 2-fibration of 2-multicategories with 1-categorical fibers.  Above we gave the equivalent patchwork definition because the axioms are anyway the ones to be checked. 
However, we will nevertheless define the 2-multicategory $\Dia^{\cor}(\SSS)$ and also a refinement $\widehat{\Dia}^{\cor}(\SSS)$ which are convenient to understand duality questions. 
The discussion only plays a minor role for this article and can be skipped on a first reading. 
\end{PAR}

\begin{PAR}\label{PARCORDIA}
Let $I_1, \dots, I_n, J \in \Dia$. 
Recall from e.g.\@ \cite[Section 4]{Hor15b} the 2-category $\Cor(I_1, \dots, I_n; J)$ whose objects are multicorrespondences
\begin{equation}\label{eqcor}
 \xi := \left( \vcenter{ \xymatrix{ & & & A \ar[llld]_{\alpha_1} \ar[ld]^{\alpha_n} \ar[rd]^{\beta} \\
I_1  & \cdots & I_n  &  & J  }    } \right)
\end{equation}
with $A \in \Dia$, and whose 1-morphisms are diagrams
\begin{equation}\label{eqcormor}
\vcenter{ \xymatrix{ & & & A \ar[llld]_{\alpha_1} \ar[dd]^{\gamma} \ar[ld]^{\alpha_n} \ar[rd]^\beta \\
I_1  & \cdots & I_n  &  & J  \\
& & & A' \ar[lllu]^{\alpha_1'} \ar[lu]_{\alpha_n'} \ar[ru]_{\beta'}
} }
\end{equation}
together with natural transformations
\[ \nu_i: \alpha_i \Rightarrow  \alpha_i' \gamma \quad 
 \mu:  \beta' \gamma  \Rightarrow \beta. \]
 2-Morphisms are natural transformations $\gamma \Rightarrow \gamma'$ compatible with the $\nu$'s and $\mu$'s.
\end{PAR}

\begin{PAR}\label{PARCORFIB}
Let $\SSS$ be a 2-pre-multiderivator satisfying (Der1) and (Der2). There is an analogue for the construction in \ref{PARCORDIA} for fibered situations over $\SSS$. 
Let $S_i \in \SSS(I_i)$ and $T \in \SSS(T)$.  The 2-category
 \[ \Cor_{\SSS}((I_1, S_1), \dots, (I_n, S_n); (J, T)) \] 
 is obtained by applying the Grothendieck construction to the pseudo-functor
\begin{equation}\label{eqcorfunctor} \Cor(I_1, \dots, I_n; J)  \rightarrow \mathcal{CAT} 
\end{equation}
mapping a multicorrespondence (\ref{eqcor}) to the category
\[ \Hom_{\SSS(A)}(\alpha_1^*I_1,  \dots, \alpha_n^* I_n; \beta^* T). \]

\begin{PAR}
In \cite[Definition~3.8]{Hor16} we defined a 2-multicategory 
\[ \Dia^{\cor}(\SSS) \]
whose objects are pairs $(I,S)$ with $I \in \Dia$ and $S \in \SSS(I)$ and whose morphism categories
are (equivalent to)
\[ \tau_1(\Cor_{\SSS}((I_1, S_1), \dots, (I_n, S_n); (J, T))) \]
and it is shown in \cite[Theorem~4.2]{Hor16} that a morphism $\DD \rightarrow \SSS$ of 2-pre-multiderivators satisfying (Der1) and (Der2)
 is a left (resp.\@ right) fibered multiderivator if and only if 
\[ \Dia^{\cor}(\DD) \rightarrow \Dia^{\cor}(\SSS)\]
is a 1-opfibration (resp.\@ 1-fibration) and 2-fibration with 1-categorical fibers. 
\end{PAR}
Equivalently, a left fibered multiderivator induces a pseudo-functor
\[ \Dia^{\cor}(\SSS) \rightarrow \mathcal{CAT} \]
mapping $(I, S)$ to $\DD(I)_S$ and mapping a multi-correspondence $\xi$ as in  (\ref{eqcor}) together with 
\[ f \in \Hom_{\SSS(A)}(\alpha_1^*I_1,  \dots, \alpha_n^* I_n; \beta^* T) \] to the multivalued functor
\[ (\xi, f)_\bullet := \beta_!^{(T)} f_\bullet(\alpha^*_1 -, \dots, \alpha_n^* -). \]
The brute force truncation $\tau_1$ is only a first approximation useful to encode and characterize the axioms (FDer3--5) but insensitive to (Der1) and (Der2). 
The resulting 2-multicategory is better behaved, if we take instead the localization
\[ \Cor_{\SSS}((I_1, S_1), \dots, (I_n, S_n); (J, T))[\mathcal{W}^{-1}] \]
(forgetting the 2-morphisms) where $\mathcal{W}$ is defined as follows.
\end{PAR}

\begin{DEF}
Let $\mathcal{W}$ be the class of morphisms $(\xi, f) \rightarrow (\xi', f')$ in $\Cor_{\SSS}((I_1, S_1), \dots, (I_n, S_n); (J, T))$ which have the property that 
$(\xi, f)_\bullet \rightarrow (\xi', f')_\bullet$  is an isomorphism 
for all fibered multiderivators $\DD \rightarrow \SSS$.
\end{DEF}

\begin{LEMMA}
If $\SSS$ is strong\footnote{Actually, it suffices that the functor (induced by $\dia_{\Delta_1}$) on morphism categories is full. In particular, if $\SSS$ is a usual 1-pre-multiderivator, no condition is required.} then 
we have an induced functor 
\[ \tau_1(\Cor_{\SSS}((I_1, S_1), \dots, (I_n, S_n); (J, T))) \rightarrow \Cor_{\SSS}((I_1, S_1), \dots, (I_n, S_n); (J, T))[\mathcal{W}^{-1}] \]
in which, on the right hand side, $\Cor_{\SSS}((I_1, S_1), \dots, (I_n, S_n); (J, T))$ is considered as a 1-category forgetting the 2-morphisms. 
\end{LEMMA}

\begin{proof}
2-morphisms $(\gamma, \nu_1, \dots, \nu_n, \mu, \Xi) \rightarrow (\gamma', \nu_1', \dots, \nu_n', \mu', \Xi')$ in $\Hom((\alpha_1, \dots, \alpha_n, \beta, f), (\alpha_1', \dots, \alpha_n', \beta', f'))$ 
 given by $\delta: \gamma \Rightarrow \gamma'$ can be encoded by a diagram
\[ \xymatrix{ & & & A \times \Delta_1 \ar[llld]_{\alpha_1 \pr_1} \ar[dd]^{\delta} \ar[ld]^{\alpha_n \pr_1} \ar[rd]^{\beta \pr_1} \\
I_1  & \cdots & I_n  &  & J  \\
& & & A' \ar[lllu]^{\alpha_1'}  \ar[lu]_{\alpha_n'} \ar[ru]_{\beta'}
}  \]
and appropriate
\[ \widetilde{\nu}_i:  \alpha_i  \pr_1 \Rightarrow  \alpha_i' \delta   \quad 
 \widetilde{\mu}:    \beta'  \delta \Rightarrow  \beta  \pr_1
 \]
point-wise in $\Delta_1$ equal to $\nu_i$, $\nu_i'$, resp.\@ to $\mu$, $\mu'$, 
and a 2-morphism $\widetilde{\Xi}$ in
\[ \xymatrix{  \delta^* (\alpha_1')^* S_1 ,\dots, \delta^* (\alpha_n')^* S_n  \ar@{<-}[d] \ar[r]^-{\delta^* f'} &  \delta^* (\beta')^* T \ar[d]  \\
\pr_1^* \alpha_1^* S_1, \dots, \pr_1^*  \alpha_n^* S_n \ar[r]^-{\pr_1^* f} \ar@{}[ru]|{\Uparrow^{\widetilde{\Xi}}} & \pr_1^*  \beta^*  T
 }   \]
 point-wise in $\Delta_1$ equal to $\Xi$ and $\Xi'$. The 2-morphism $\widetilde{\Xi}$ is constructed in the following way: Let $F$ and $F'$ be the two compositions 
 $\pr_1^* \alpha_1^* S_1, \dots, \pr_1^*  \alpha_n^* S_n \rightarrow \pr_1^*  \beta^*  T$. 
 $\Xi'$ and $\Xi$ together with the compatibility of $\delta$ with them \cite[3.5.3]{Hor16} implies that we have 
 a 2-morphism in $\Fun(\Delta_1, \SSS(A))$
 \[ \dia_{\Delta_1}(F) \Rightarrow \dia_{\Delta_1}(F'). \]
 By strongness, it lifts to a morphism
 \[ \widetilde{\Xi}: F \Rightarrow F'. \]

 Furthermore, one checks that the projection (where the two multicorrespondences are equipped with $\pr_1^* f$, and $f$, respectively)
\[ \xymatrix{ & & & A \times \Delta_1 \ar[llld]_{\alpha_1 \pr_1} \ar[dd]^{\pr_1} \ar[ld]^{\alpha_n \pr_1} \ar[rd]^{\beta \pr_1} \\
I_1  & \cdots & I_n  &  & J  \\
& & & A \ar[lllu]_{\alpha_1}  \ar[lu]_{\alpha_n} \ar[ru]_{\beta}
}  \]
is in $\mathcal{W}$. This shows $(\gamma, \Xi) = (\gamma', \Xi')$ in $\Cor_{\SSS}((I_1, S_1), \dots, (I_n, S_n); (J, T))[\mathcal{W}^{-1}]$. 
\end{proof}

\begin{PAR}\label{PARDIACORHAT}
We denote by $\widehat{\Dia}^{\cor}(\SSS)$ the corresponding 2-multicategory\footnote{strictified as in \cite[Definition~3.8]{Hor16} using only multicorrespondences in which $\alpha_1 \times \dots \times \alpha_n$ is a fibration and $\beta$ is an opfibration and a choice of associative fiber products, cf.\@ also \cite[Remark~5.3]{Hor16}} whose morphism categories are (equivalent to) $\Cor_{\SSS}((I_1, S_1), \dots, (I_n, S_n); (J, T))[\mathcal{W}^{-1}]$. Tautologically, the pseudo-functor induced by a left fibered multiderivator $\DD \rightarrow \SSS$ factors 
\[ \Dia^{\cor}(\SSS) \rightarrow \widehat{\Dia}^{\cor}(\SSS) \rightarrow \mathcal{CAT}. \]

Recall that we have strict functors
\[ \xymatrix{  \Fun(I_1^{\op} \times \cdots \times I_n^{\op} \times J, \Dia) \ar@<3pt>[rr]^-\Xi && \ar@<3pt>[ll]^-\Pi \Cor(I_1, \dots, I_n; J)  } \]
equipped with strict natural transformations
\[ \kappa: \Xi \circ  \Pi \Rightarrow  \id_{\Cor(I_1, \dots, I_n; J)} \quad r:  \Pi \circ  \Xi  \Rightarrow \id_{\Fun(I_1^{\op} \times \cdots \times I_n^{\op} \times J, \Dia)}.  \]

The functor $\Xi$ maps a functor $F$ to $\int \nabla F$ which is equipped with a fibration to $I_1 \times \cdots \times I_n$ and an opfibration to $J$. 
The functor $\Pi$ maps a correspondence $\xi$ as in (\ref{eqcor}) to the functor $i_1, \dots, i_n, j \mapsto (i_1, \dots, i_n) \times_{/I_1 \times \cdots \times I_n} A \times_{/J} j$ contravariant in $I_1, \dots, I_n$ and covariant in $J$.
The 2-morphisms $r$ and $\kappa$ will be recalled later. 

We will show below that the first is in $\mathcal{W}$ and the second in $\Xi^{-1}(\mathcal{W})$. 
We have thus an equivalence (using that $\mathcal{W}$ satisfies 2-out-of-3 by definition)
\[  \Fun(I_1^{\op} \times \cdots \times I_n^{\op} \times J, \Dia)[\Xi^{-1}(\mathcal{W})^{-1}] \cong \Cor(I_1, \dots, I_n; J)[\mathcal{W}^{-1}].  \] 
In \cite[Proposition~4.6]{Hor15b} it was erroneously claimed that this is already an equivalence of the level of $\tau_1$-categories. 
We proceed to prove the (corrected) statement more generally for fibered situations. 

For this purpose define  $\Fun_{\SSS; S_1, \dots, S_n; T}(I_1^{\op} \times \cdots \times I_n^{\op} \times J, \Dia)$ as the (strict) pull-back of 
\[ \Cor_{\SSS}((I_1, S_1), \dots, (I_n, S_n); (J, T)) \] along 
$\Xi$. Equivalently, it is the 2-category  obtained by the Grothendieck construction from the composition
\[ \xymatrix{ \Fun(I_1^{\op} \times \cdots \times I_n^{\op} \times J, \Dia) \ar[r]^-{\Xi} & \Cor(I_1, \dots, I_n; J) \ar[r]^-{\text{(\ref{eqcorfunctor})}} & \mathcal{CAT} } \]
 \end{PAR}

\begin{PROP}\label{PROPALTCOR}
There are strict functors
\begin{equation}\label{eqaltcor} \xymatrix{  \Fun_{\SSS; S_1, \dots, S_n; T}(I_1^{\op} \times \cdots \times I_n^{\op} \times J, \Dia) \ar@<3pt>[rr]^-{\widetilde{\Xi}} && \ar@<3pt>[ll]^-{\widetilde{\Pi}} \Cor_{\SSS}((I_1, S_1), \dots, (I_n, S_n); (J, T))  } \end{equation}
equipped with strict natural transformations  
\[  \widetilde{r}: \widetilde{\Pi} \circ  \widetilde{\Xi} \Rightarrow \id  \qquad \widetilde{\kappa}: \widetilde{\Xi} \circ  \widetilde{\Pi} \Rightarrow \id .     \]
Both morphisms are point-wise in $\mathcal{W}$, resp.\@ in $\widetilde{\Xi}^{-1}(\mathcal{W})$. 
We thus get an equivalence of categories with weak equivalences
\[  (\Fun_{\SSS; S_1, \dots, S_n; T}(I_1^{\op} \times \cdots \times I_n^{\op} \times J, \Dia), \widetilde{\Xi}^{-1}(\mathcal{W})) \cong (\Cor_{\SSS}((I_1, S_1), \dots, (I_n, S_n); (J, T)), \mathcal{W}). \]
Furthermore, $\widetilde{\Xi}^{-1}(\mathcal{W})$ consists of those (1-)morphisms $(\mu, \Xi): (F, f_F) \Rightarrow (G, f_G)$ defined by 
 $\mu: F \Rightarrow G$ and
 \[ \xymatrix{  (\int \nabla \gamma)^* \pi_{1,G}^* S_1 ,\dots, (\int \nabla \gamma)^* \pi_{n,G}^* S_n  \ar@{=}[d] \ar[rr]^-{(\int \nabla \gamma)^* f_G}  && (\int \nabla \gamma)^* \pi_{J,G}^* T \ar@{=}[d]  \\
\pi_{1,F}^* S_1, \dots,  \pi_{n,F}^* S_n \ar[rr]_-{f_F} \ar@{}[rru]|{\Uparrow^{\Xi}} && \pi_{J,F}^*  T
 }   \]
 where $F, G:  I_1^{\op} \times \cdots \times I_n^{\op} \times J \to \Dia$ are functors equipped with $f_F: \pi_{1,F}^* S_1, \dots \pi_{n,F}^* S_n \rightarrow \pi_{J,F}^* T$, and  $f_G: \pi_{1,G}^* S_1, \dots \pi_{n,G}^* S_n \rightarrow \pi_{J,G}^* T$, respectively, such that for all $i_1, \dots, i_n, j$ the functor $\mu_f: F_f \rightarrow G_f$  (value of $\mu$ at $i_1, \dots, i_n, j$) together with $\Xi|_{F_f}$
 induce an isomorphism
\[ p_{F_f,!} f_{F_f,\bullet}(p_{F_f}^* - , \dots, p_{F_f}^* -) \rightarrow  
p_{G_f,!} f_{G_f,\bullet}(p_{G_f}^* - , \dots, p_{G_f}^* -)  \]
for every fibered multiderivator over $\SSS$. 
\end{PROP}

\begin{proof}
Since the l.h.s.\@ of (\ref{eqaltcor}) is defined as pull-back we have a commutative diagram
\[ \xymatrix{
\Fun_{\SSS; S_1, \dots, S_n; T}(I_1^{\op} \times \cdots \times I_n^{\op} \times J, \Dia) \ar[r]^{\widetilde{\Xi}} \ar[d] &  \Cor_{\SSS}((I_1, S_1), \dots, (I_n, S_n); (J, T)) \ar[d] \\
\Fun(I_1^{\op} \times \cdots \times I_n^{\op} \times J, \Dia) \ar[r]^{\Xi} & \Cor(I_1, \dots, I_n; J)
}\]
The vertical functors are actually in each case a 1-fibration and 2-opfibration by definition. 
Recall the strict functor $\Pi$ between the categories in the bottom line. It sends a correspondence
\begin{equation} \label{eqcor2}
\vcenter{ \xymatrix{ & & & A \ar[llld]_{\pi_1} \ar[ld]^{\pi_n} \ar[rd]^{\pi_J} \\
I_1  & \cdots & I_n  &  & J  } }
\end{equation}
to the functor
\[ i_1, \dots, i_n, j \mapsto (i_1, \dots, i_n) \times_{/I_1 \times \cdots \times I_n} A \times_{/J} j. \]
We have strict natural transformations
\[ \xymatrix{ \Xi \circ \Pi \ar@<2pt>[r]^-\kappa & \ar@<2pt>[l]^-\tau \id } \]
given by the obvious 2-commutative diagram
\[ \xymatrix{ & & & I_1 \times \cdots \times I_n \times_{/I_1 \times \cdots \times I_n} A \times_{/J} J \ar[llld]_{p_1} \ar[ld]^{p_n} \ar@/^10pt/[dd]^{\kappa} \ar@{<-}@/_10pt/[dd]_{\tau} \ar[rd]^{p_J} \\
I_1  & \cdots & I_n  &  & J   \\
 & & & A  \ar[lllu]^{\pi_1} \ar[lu]_{\pi_n} \ar[ru]_{\pi_J} 
}  \]
We can thus define a pseudo-functor $\widetilde{\Pi}$ on $\Cor_{\SSS}((I_1, S_1), \dots, (I_n, S_n); (J, T))$ mapping a correspondence $\xi$ as in (\ref{eqcor2}) together with $f \in \Hom(\pi_1^*S_1, \dots, \pi_n^*S_n; \pi_J^* T)$ 
to the functor
\[ i_1, \dots, i_n, j \mapsto (i_1, \dots, i_n) \times_{/I_1 \times \cdots \times I_n} A \times_{/J} j \]
together with the morphism (actually the pull-back along $\kappa$)
\[ p_1^* S_1, \dots, p_n^* S_n \rightarrow \kappa^* \pi_1^* S_1, \dots, \kappa^* \pi_n^* S_n \rightarrow \kappa^* \pi_J^* T \rightarrow p_J^* T.  \]
and it follows from the axioms of a fibered  multiderivator that this is in $\mathcal{W}$. 

$\kappa$ and $\tau$ induce strict 1-morphisms
\[ \xymatrix{ \widetilde{\Xi} \circ \widetilde{\Pi} \ar@<2pt>[r]^-{\widetilde{\kappa}} & \ar@<2pt>[l]^-{\widetilde{\tau}} \id }.  \]
We have $\widetilde{\kappa} \circ \widetilde{\tau} = \id$ and $\widetilde{\tau} \circ \widetilde{\kappa}$ is related to $\id$ by a chain of 2-morphisms. Thus 
$\widetilde{\kappa}$ and $\widetilde{\tau}$ are point-wise in $\mathcal{W}$. 

The composition $\widetilde{\Pi} \circ \widetilde{\Xi}$ applied to a functor $F$ together with $\pi_1^*S_1, \dots, \pi_n^*S_n \rightarrow \pi_J^* T$ 
is the following functor
\[ i_1, \dots, i_n, j \mapsto (i_1, \dots, i_n) \times_{/I_1 \times \cdots \times I_n} \left( \int \nabla F \right) \times_{/J} j \] 
together with the morphism
\[ p_1^* S_1, \dots, p_n^* S_n \rightarrow (\kappa \ast \Xi)^* \pi_1^* S_1, \dots, (\kappa \ast \Xi)^* \pi_n^* S_n \rightarrow (\kappa \ast \Xi)^* \pi_J^* T \rightarrow p_J^* T.  \]

We have a morphism $r:  \Pi \circ \Xi \to \id$ mapping $o \in F(i_1', \dots, i_n', j')$ together with $\alpha_i: i_1 \rightarrow i_1'$ and $\beta: j' \rightarrow j$ to
$F(\alpha_1, \dots, \alpha_n, \beta)\sigma \in F(i_1, \dots, i_n, j)$. 
Contrary to the false claim in \cite[Proposition~4.6]{Hor15b} there is no (strict or non-strict) natural transformation in the other direction in general.
We claim the existence of a chain of 2-morphisms
\[ \widetilde{\Xi} \ast \widetilde{r} \Rightarrow \widetilde{\rho} \Leftarrow \widetilde{\kappa} \ast \widetilde{\Xi}  \]
thus showing that $r$ is point-wise in $\widetilde{\Xi}^{-1}(\mathcal{W})$ because $\widetilde{\kappa}$ is point-wise in $\mathcal{W}$. 

On the level of diagrams the three 1-morphisms
\[ \xymatrix{ 
 & & & I_1 \times \cdots \times I_n \times_{/I_1 \times \cdots \times I_n} \int \nabla F \times_{/J} J  \ar[llld]^{p_1} \ar[ld]^{p_n} \ar[rd]^{p_J} 
 \ar@/_15pt/[dd]_{\Xi \ast r}  \ar[dd]^\rho  \ar@/^15pt/[dd]^{\kappa \ast \Xi} 
 \\
I_1  & \cdots & I_n  &  & J   \\
& & & \int \nabla F \ar[lllu]^{\pi_1} \ar[lu]_{\pi_n}\ar[ru]^{\pi_J} 
}  \]
 map a tupel $(\{ \alpha_k: i_k \rightarrow i_k'\}_k, \beta: j' \rightarrow j, o \in F(i_1', \dots, i_n', j'))$ respectively to
\[ (i_1, \dots, i_n, j, F(\alpha_1, \dots, \alpha_n, \beta)o)  \quad (i_1', \dots, i_n', j, F(\beta)o) \quad  (i_1', \dots, i_n', j', o)   \]
and noting that the relevant squares (cf.\@ \cite[3.5.2.]{Hor16}) are commutative on the nose already.

It remains to show the point-wise characterization of $\widetilde{\Xi}^{-1} \mathcal{W}$. By (Der2) the morphism
$(\xi, f)_\bullet \rightarrow (\xi', f')_\bullet$ is an isomorphism if it is point-wise the case and, using the right adjoints for each slot and (Der2) again, if it is an isomorphism on tupels of objects of the form 
$i_{1,!}\mathcal{E}_1, \dots,  i_{n,!}\mathcal{E}_n$. I.e.\@ it is an isomorphism if 
\[ j^*(\xi, f)_\bullet(i_{1,!}^{(S_1)}\mathcal{E}_1, \dots, i_{n,!}^{(S_n)}\mathcal{E}_n) \rightarrow j^*(\xi', f')_\bullet(i_{1,!}^{(S_1)}\mathcal{E}_1, \dots, i_{n,!}^{(S_n)}\mathcal{E}_n)   \]
is an isomorphism for all $i_1, \dots, i_n, j$. 
The functors $j^*$ and $i_{k,!}^{(S_k)}$ are given as push-forward along the correspondences
\[ [i_k^{(S_{k,i_k})}]: \vcenter{ \xymatrix{ & i_k \times_{/I} I  \ar[ld]_{p} \ar[rd]^{\iota} \\
i_1   & ; & I 
} } \quad
 [j^{(T_{j})}]':  \vcenter{ \xymatrix{ & J \times_{/J} j  \ar[ld]_{\iota} \ar[rd]^{p} \\
J   & ; & j
} } \]
equipped with the obvious 1-morphisms
\[  p^*S_{k,i_k} \rightarrow \iota^*S_k \quad \iota^*S  \rightarrow p^*T_{j}.   \]
The multicorrespondence  
\[ [(j)^{(T_{j})}]' \circ (\xi, f) \circ( [(i_1)^{(S_{j})}], \dots, [(i_n)^{(S_{i_n})}])    \]
is nothing else but
\[ \xymatrix{ & & & i_1 \times \cdots \times i_n \times_{/I_1 \times \cdots \times I_n} \int \nabla F \times_{/J} j \ar[llld]^{} \ar[ld]^{}\ar[rd]^{} \\
i_1  & \cdots & i_n  &  & j   
}  \]
together with
\[ p^*S_{1,i_1}, \dots, p^*S_{n,i_n} \rightarrow (\kappa \ast \Xi)^* \pi^*_1S_1, \dots, (\kappa \ast \Xi)^*  \pi_n^* S_n \rightarrow (\kappa \ast \Xi)^* \pi_J^* T  \rightarrow p^*T_j \]
 and we have 1-morphisms 
\[ \xymatrix{ & & & i_1 \times \cdots \times i_n \times_{/I_1 \times \cdots \times I_n} \int \nabla F \times_{/J} j \ar[llld]^{} \ar[ld]^{}\ar[rd]^{} \ar@/^15pt/[dd]^R \ar@{<-}@/_15pt/[dd]_L   \\
i_1  & \cdots & i_n  &  & j   \\
 & & & F(i_1, \dots, i_n,j)  \ar[lllu]^{} \ar[lu]^{} \ar[ru]^{}  
}  \]
where $R$ is defined similar to $r$ above.
One checks that $R$ and $L$ promote to morphisms $\widetilde{R}$ and $\widetilde{L}$ where in the bottom multicorrespondence the morphism
\[ p^*S_1, \dots, p^*S_n \rightarrow p^*T \]
is chosen given by the restriction of $f$ to the fiber above $i_1, \dots, i_n, j$.
One checks also that we have
$\widetilde{R}\widetilde{L} = \id$ and $\widetilde{L}\widetilde{R} \sim \id$ up to a chain of 2-morphisms. 
The claim follows.  
\end{proof}

\begin{KOR}
For $\SSS=\{\cdot\}$ the final pre-multiderivator, and $\Dia=\Cat$, the class $\Xi^{-1}(\mathcal{W})$ consists of those morphisms $\mu: F \Rightarrow G$ in $\Fun(I_1^{\op} \times \cdots \times I_n^{\op} \times J, \Cat)$ such that $\mu$ evaluated at $i_1, \dots, i_n, j$, 
denoted $\mu_f: F_f \rightarrow G_f$, induces a weak equivalence
 $N(F_f) \rightarrow N(G_f)$ for all $i_1, \dots, i_n$ and $j$. Thus we have an equivalence 
 \[ \Fun(I_1^{\op} \times \cdots \times I_n^{\op} \times J, \Dia)[ (\Xi^{-1}(\mathcal{W}))^{-1}]  \cong \HH(I_1^{\op} \times \cdots \times I_n^{\op} \times J)   \]
 where $\HH$ is the derivator of spaces.
\end{KOR}
\begin{proof}
A fibered multiderivator over $\SSS=\{\cdot\}$ is just a closed monoidal derivator and the morphism 
\[ p_{F(i_1, \dots, i_n, j),!} (p_{F(i_1, \dots, i_n, j)}^* \mathcal{E}_1 \otimes \cdots \otimes p_{F(i_1, \dots, i_n, j)}^*\mathcal{E}_n) \rightarrow  
p_{G(i_1, \dots, i_n, j),!} (p_{G(i_1, \dots, i_n, j)}^* \mathcal{E}_1\otimes \cdots \otimes p_{G(i_1, \dots, i_n, j)}^*\mathcal{E}_n)  \]
is, using (FDer0 left), isomorphic to 
\[ p_{F(i_1, \dots, i_n, j),!} p_{F(i_1, \dots, i_n, j)}^* (\mathcal{E}_1 \otimes \cdots \otimes \mathcal{E}_n) \rightarrow  
p_{G(i_1, \dots, i_n, j),!} p_{G(i_1, \dots, i_n, j)}^* (\mathcal{E}_1\otimes \cdots \otimes \mathcal{E}_n)  \]
Hence the condition is independent of the monoidal structure and the statement follows from Cisinski's theorem \cite{Cis04}.
\end{proof}

The description of the morphism categories in $\widehat{\Dia}^{\cor}$ as $\HH(I_1^{\op} \times \cdots \times I_n^{\op} \times J)$ is very convenient for duality questions. 
We immediately get that $\widehat{\Dia}^{\cor}$ is 1-bifibered over the final multicategory $\{ \cdot \}$ or, in other words, it is closed monoidal with the 
product given by $I \boxtimes J = I \times J$ and internal hom given by $\mathbf{HOM}(I, J) = I^{\op} \times J$. Indeed, we have canonical equivalences
\[  \Hom_{\widehat{\Dia}^{\cor}}(I, J; K) \cong \HH(I^{\op} \times J^{\op} \times K) \cong  \Hom_{\widehat{\Dia}^{\cor}}(I; J^{\op} \times K) \cong  \Hom_{\widehat{\Dia}^{\cor}}(I \times J; K)  \]  
and one has to verify the following
\begin{LEMMA}
The above isomorphisms are induced by composition with a coCartesian morphism\footnote{in fact the preimage of the identity under $\Hom_{\widehat{\Dia}^{\cor}}(I, J; I \times J) \cong \Hom_{\widehat{\Dia}^{\cor}}(I \times J; I \times J)$} in 
\[  \Hom_{\widehat{\Dia}^{\cor}}(I, J; I \times J) \]
resp.\@ by composition with a Cartesian (w.r.t.\@ the first slot) morphism\footnote{in fact the preimage of the identity under $\Hom_{\widehat{\Dia}^{\cor}}(I \times J^{\op}, J;  I) \cong \Hom_{\widehat{\Dia}^{\cor}}(I \times J; I \times J)$} in
\[ \Hom_{\widehat{\Dia}^{\cor}}(I \times  J^{\op}, J^{\op}, I). \]
\end{LEMMA}
Actually, this will be part of the more general Theorem~\ref{SATZDUALITY} involving derivator six-functor-formalisms.

We will later need the following

\begin{LEMMA}\label{LEMMAKAN3}
Let $\SSS$ be a 2-pre-derivator and
let $\DD \rightarrow \SSS$ be a fibered derivator, $\alpha: I \rightarrow J$ a fibration or opfibration whose fibers have a final or initial object.
Then $\alpha^*$ is an equivalence 
\[ \xymatrix{ \DD(J)_{S} \ar[r]^-{\alpha^*} & \DD(I)_{\alpha^* S}^{\alpha-\cart} = \DD(I)_{\alpha^* S}^{\alpha-\cocart}   } \]
Both $\alpha_!^{(S)}$ and $\alpha_*^{(S)}$ restricted to $\alpha$-Cartesian objects are quasi-inverse functors to this equivalence. 
\end{LEMMA}
Note that the statement is also true in the seemingly skewed cases: ``fibration + fiberwise final object'' and ``opfibration + fiberwise initial object''\footnote{Using (co)homological descent (Section~\ref{SECTIONNAIVE}), under additional assumptions (existence of (co)Cartesian projectors) this can be generalized to {\em contractible} fibers. We prefer to give an elementary proof for this special case.}.

\begin{proof}It suffices to see that on $\alpha$-Cartesian objects the natural morphism (induced by the unit)
\[ j^* \alpha_!^{(S)} = i^* \alpha^* \alpha_!^{(S)} \leftarrow i^*  \]
or the natural morphism (induced by the counit)
\[ j^* \alpha_*^{(S)} = i^* \alpha^* \alpha_*^{(S)} \to i^*   \]
is an isomorphism for any $i \in I, j \in J$, with $j=\alpha(i)$, because this easily implies that $\alpha_!^{(S)}$, resp.\@ $\alpha_*^{(S)}$ is quasi-inverse to $\alpha^*$ on $\alpha$-Cartesian objects. If this is the case then also the other one is a quasi-inverse to $\alpha^*$ on $\alpha$-Cartesian objects (in an equivalence both functors are left and right adjoint). 

{\bf Step 1:} Assume first that $J=\{ j \}$. If $I$ has an initial object then 
\[ j^* \alpha_*^{(S)} \cong (i')^*.  \]
If $I$ has a final object then 
\[ j^* \alpha_!^{(S)} \cong (i')^*.  \]
In both cases, on Cartesian objects, the same holds with $i'$ replaced by any $i$.

{\bf Step 2:} Assume now, that $\alpha$ is a fibration (the other case is dual). Then we have
\[  j^* \alpha_*^{(S)} \cong p_{I_j,*}^{(j^*S)} \iota_{j}^*  \]
where 
\[ \iota_j: I_j \rightarrow I \]
is the inclusion of the fiber. By Step 1 $p_{I_j,*}$ is (isomorphic to) evaluation at $i$. The statement follows. 
\end{proof}

\section{Multicategories of homotopy multicorrespondences}

\begin{PAR}\label{PARDER6FU}
Let $\mathcal{S}$ be a category with finite limits. 
Recall from e.g.\@ \cite[Section 5]{Hor16} the construction of the 2-multicategory $\mathcal{S}^{\cor}$ of multicorrespondences in $\mathcal{S}$. Its objects are the objects of $\mathcal{S}$, its 1-multimorphisms are multicorrespondences 
\[ \xymatrix{ & & & A \ar[llld] \ar[ld] \ar[rd] \\
X_{1} & \dots & X_{n} & ; & Y } \]
in $\mathcal{S}$ which are composed forming fiber products, and the 2-morphisms are isomorphisms of such multicorrespondences.
This 2-multicategory was used in \cite[Definition~5.8]{Hor16} to define
a derivator six-functor-formalism as a fibered multiderivator over $\SSS^{\cor}$, the 2-pre-multiderivator represented  by $\mathcal{S}^{\cor}$, i.e.\@ defined by 
$\SSS^{\cor}(I):= \Fun(I, \mathcal{S}^{\cor})$. In this article, derivator six-functor-formalisms over more general (homotopy) bases, which are itself modeled by a category with weak equivalences $(\mathcal{M}, \mathcal{W})$, will be discussed. In this section, the definition of a 2-pre-multiderivator $\HH^{\cor}(\mathcal{M})$ of {\bf homotopy multicorrespondences}  is given, which is {\em not} represented by a 2-multicategory unless $\mathcal{W}$ is the class of isomorphisms. If this is the case, we reobtain the previously defined 2-pre-multiderivator, i.e.\@ we have 
$\HH^{\cor}(\mathcal{M}) = \mathbb{M}^{\cor}$ if $\mathcal{W}$ is the class of isomorphisms. To dispose of a nice construction of homotopy fiber products we have to assume that the datum is part of a {\em category with fibrant objects} $(\mathcal{M}, \Fib, \mathcal{W})$ as in \ref{PARCWFO}. 
\end{PAR}

\begin{DEF}[Informal]
Let $(\mathcal{M}, \Fib, \mathcal{W})$ be a category of fibrant objects.  

We define the 2-category $\HH^{\cor}(\mathcal{M})$ of homotopy multicorrespondences in $\mathcal{M}$ as the following 2-multicategory:
\begin{enumerate}
\item Objects are the objects of $\mathcal{M}$.
\item 1-multimorphisms are the multicorrespondences 
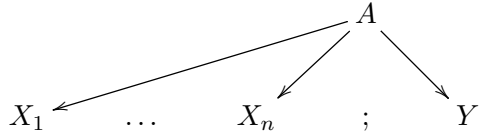
\[ \xymatrix{ & & & A \ar[llld] \ar[ld] \ar[rd] \\
X_{1} & \dots & X_{n} & ; & Y } \]
in $\mathcal{M}$.
\item 2-morphisms are the isomorphisms in the homotopy category of such multicorrespondences.
\end{enumerate}
The composition is given by homotopy fiber products. 
\end{DEF}
The rest of this section is devoted to making this definition precise 
 and enhance it to a 2-pre-multiderivator using a device for constructing 2-multicategories from \cite[Appendix~A]{Hor17}, cf.\@ also the construction of $\SSS^{\cor}$ in \cite[Section~6]{Hor17}.

\begin{PAR}\label{PARCATFIBOBJ}
Let $(\mathcal{M}, \Fib, \mathcal{W})$ be a category of fibrant objects with functorial path object. This implies that, in particular, all $\mathcal{M}^I := \Fun(I, \mathcal{M})$ are also categories of fibrant objects defining fibrations and weak equivalences point-wise. 
We consider $\mathcal{M}$ as a symmetric {\em opmulticategory} setting: 
\begin{equation}\label{eqopmult} \Hom(X; Y_1, \dots, Y_n):= \Hom(X, Y_1) \times \cdots \times  \Hom(X, Y_n)  \end{equation}
with the obvious action of the symmetric group. 
\end{PAR}

\begin{DEF}\label{DEFADMISSIBLE}
Let $I$ be a multidiagram (i.e.\@ a small multicategory). Recall the construction \ref{PARTWMULTI}.
 We denote $\mathcal{M}^{\tw I} := \Fun(\tw I, \mathcal{M})$ (functors of opmulticategories).
We say that a diagram $X \in \mathcal{M}^{\tw I}$ is {\bf admissible} if in every square in the opmulticategory $\tw I$
\[ \xymatrix{ i \ar[r] \ar[d] & j_1, \dots, j_n \ar[d] \\ i' \ar[r] & j_1', \dots, j_n'  }\]
in which the horizontal morphisms are of type 2 and the vertical morphisms are of type 1 
is mapped by $X$ to a homotopy Cartesian diagram. More concretely, this means that the diagram
\[ \xymatrix{ X(i) \ar[r] \ar[d] & \prod_{i=1}^n X(j_i) \ar[d] \\ X(i') \ar[r] & \prod_{i=1}^n X(j_i')  }\]
is homotopy Cartesian. Here $\prod$ is the usual product, which exists, and is a homotopy product because $\mathcal{M}$ is a category with fibrant objects. 
We call a multimorphism $X \Rightarrow Y_1, \dots, Y_n$ with $X, Y_1, \dots, Y_n \in \mathcal{M}^{\tw I}$ {\bf type-1-admissible}, if
\[ \xymatrix{ X(i) \ar[r] \ar[d] & \prod_{k=1}^n Y_k(i) \ar[d] \\ X(i') \ar[r] & \prod_{k=1}^n Y_k(i')  }\]
is homotopy Cartesian for any type-1 morphism $i \rightarrow i'$. 
We call a morphism $X \Rightarrow Y$ with $X, Y \in \mathcal{M}^{\tw I}$ {\bf type-2-admissible}, if
\[ \xymatrix{ X(i) \ar[r] \ar[d] & X(i') \ar[d] \\ Y(i) \ar[r] &  Y(i')  }\]
is homotopy Cartesian for any type-2 morphism $i \rightarrow i'$. 
\end{DEF}

\begin{PAR}\label{PARCIRC}
For the {\em particular} opmulticategory structure (\ref{eqopmult}) on $\mathcal{M}$, we can also describe the functors of opmulticategories $\mathcal{M}^M := \Fun(M, \mathcal{M})$ for a multidiagram $M$ as follows:  For an opmulticategory $M$ define a usual category $M^\circ$ replacing all multimorphisms in $\Hom_M(j; i_1, \dots, i_n)$ by a set of 1-ary morphisms $j \rightarrow i_1$, \dots,  $j \rightarrow i_n$ \footnote{that means, in particular, forgetting all $0$-ary morphisms}. Then
a functor of opmulticategories $M \rightarrow \mathcal{M}$ is the same as a functor between usual categories $M^\circ \rightarrow \mathcal{M}$.
\end{PAR}

\begin{DEF}\label{DEFCOR}
Let $\tau$ be a tree with symmetrization $\tau^S$, $I$ a (multi)diagram. 
We define a category with weak equivalences 
\[ \Cor_I(\tau^S) \]
whose objects are admissible objects $X \in \mathcal{M}^{\tw (\tau \times I)}$ (which are functors of opmulticategories) and the morphisms are the point-wise weak equivalences as are the weak equivalences (i.e.\@ all morphisms are weak equivalences). 
This defines a strict functor
\[ \Cor_I: \Delta_S \rightarrow \mathcal{CATW} \]
from symmetric trees to categories with weak equivalences. Note that any automorphism of $\tau^S$ acts on $\mathcal{M}^{\tw (\tau \times I)}$ because $\mathcal{M}$ is symmetric. Equivalently, observe
that functors $\tw (\tau \times I) \rightarrow \mathcal{M}$ are the same as functors of usual categories $(\tw (\tau \times I))^\circ \rightarrow \mathcal{M}$ (cf.\@ \ref{PARCIRC}) and that 
every functor $\tau^S \rightarrow (\tau')^S$ induces an obvious functor $(\tw (\tau \times I))^\circ \rightarrow (\tw (\tau' \times I))^\circ$.
\end{DEF}

\begin{BEISPIEL}\label{BEISPIELMULTICOR}
The category $\Cor_I(\Delta_{1,n}^S)$ is just the category of multicorrespondences
\[ \xymatrix{ & & & A \ar[llld]_{g_1} \ar[ld]^{g_n} \ar[rd]^f \\
X_{1} & \dots & X_{n} & ; & Y } \]
where all objects $X_1, \dots, X_n, A, Y$ are admissible objects in $\mathcal{M}^{\tw I}$ and
where the multimorphism $(g_1, \dots, g_n)$ is type-1-admissible and $g$ is type-2-admissible. 
Weak equivalences are point-wise weak equivalences between such multicorrespondences. The action of $S_n$ on $\Delta_{1,n}^S$ permutes the morphisms $A \rightarrow X_i$.
\end{BEISPIEL}

\begin{LEMMA}\label{LEMMAPROPCONSTRSYMMULTI1}
The strict functor $\tau \mapsto \Cor_I(\tau)$ satisfies the axioms 1 and 2 of \cite[Proposition~A.3]{Hor17}.
\end{LEMMA}
\begin{proof}
Axiom 1 holds by construction. 
Axiom 2: We have to show that for $\{X_o\}_{o \in \tau}$ a family of admissible $X_o \in \mathcal{M}^{\tw I}$  the functor
\begin{equation}\label{eqaxiom2} \Cor_{I, \{X_o\}_{o \in \tau}}(\tau)[\mathcal{W}^{-1}_{\{X_o\}}] \rightarrow \prod_m  \Cor_{I, \{X_o\}_{o \in m}}(\Delta_{1,n_m}^S)[\mathcal{W}^{-1}_{\{X_o\}}]  
\end{equation}
is an equivalence, where the product runs over the generating morphisms of $\tau$. 
A right adjoint to the restriction
\[ \Cor_{I, \{X_o\}_{o \in \tau}}(\tau) \rightarrow \prod_m  \Cor_{I, \{X_o\}_{o \in m}}(\Delta_{1,n_m}^S) \]
exists on those tupels of objects in $\prod_m  \Cor_{I, \{X_o\}_{o \in m}}(\Delta_{1,n_m}^S)$ where in each entry 
\[ \xi_m = \left( \vcenter{ \xymatrix{ & & & A_m \ar[llld] \ar[ld] \ar[rd] \\
X_{m,1} & \dots & X_{m,n_m} & ; & Y_m } } \right) \]
the morphism $A_m \rightarrow Y_m$ is a point-wise fibration and is just given by right Kan extension along 
\[ \bigcup_m \tw(m \times I) \hookrightarrow \tw(\tau \times I) \]
which point-wise computes an iterated fiber product along these fibrations. 
The right adjoint can thus be derived using the following replacement functor
\[ \id \Rightarrow R\]
with $R$ defined by
\[ R(\xi_m) := \left( \vcenter{\xymatrix{ & & & Y_m^I \times_{Y_m} A_m \ar[llld] \ar[ld] \ar[rd] \\
X_{m,1} & \dots & X_{m,n_m} & ; & Y_m } } \right) \]
in which the morphism to the right is the composition $Y_m^I \times_{Y_m} A_m \rightarrow Y_m^I \rightarrow Y_m$ and is a fibration, and the morphisms to the $X_{m,j}$ factor through the projection to $A_m$. 
On checks that the composition of this replacement functor with the right Kan extension preserves weak equivalences (point-wise the composition just computes the usual derived fiber product using the path objects in categories with fibrant objects) and is adjoint to the restriction. 

The derived left and right adjoints preserve the categories of admissible diagrams. 
On checks that unit and counit are point-wise weak equivalences on the full subcategories of admissible objects, hence (\ref{eqaxiom2})
is an equivalence.  
\end{proof}

\begin{DEF}\label{DEFHCOR1}
Let $I$ be a (multi)diagram. We define $\HH^{\cor}(\mathcal{M})(I)$ to be the symmetric 2-multicategory of \cite[Proposition~A.3]{Hor17} constructed from the functor $C_I$.
\end{DEF}

We can read off the construction that 
\begin{itemize}
\item the objects of $\HH^{\cor}(\mathcal{M})(I)$ are the admissible objects $X \in \mathcal{M}^{\tw I}$,
\item every 1-multimorphism $X_1, \dots, X_n \rightarrow Y$ is 2-isomorphic to one of the form 
\[ \xymatrix{ & & & A \ar[llld] \ar[ld] \ar[rd] \\
X_{1} & \dots & X_{n} & ; & Y. }  \]
as in \ref{BEISPIELMULTICOR}.
\item 2-morphisms between the previous multimorphisms are given by 
isomorphisms in the homotopy category of such multicorrespondences (with fixed $X_i$ and $Y$),
\item composition is given (up to unique 2-isomorphism) by forming homotopy fiber products,
\item the action of the symmetric group is given by permuting the morphisms $A \rightarrow X_i$.
\end{itemize}

\begin{LEMMA}\label{LEMMAPF}
Let $I$ and $J$ be (multi)diagrams. 
\begin{enumerate}
\item Any diagram $D \in \HH^{\cor}(\mathcal{M})(J \times I)$ gives rise to a canonical pseudo-functor of 2-multicategories
\[ \Dia(D): J \rightarrow \HH^{\cor}(\mathcal{M})(I) \]
defined on objects by $j \mapsto D|_{\tw I \times (\tw j)}$. 
\item If $\mathcal{W}$ consists of isomorphisms the association $D \mapsto \Dia(D)$ yields an equivalence of 2-multicategories. 
\end{enumerate}
\end{LEMMA}
\begin{proof}
1.\@ We sketch the non-multi variant here, and leave the multi-case to the reader. 
Each morphism $\alpha: j \rightarrow j'$ in $J$ gives rise to a functor (pull-back of $D$)
\[ \alpha': \tw \Delta_{1} \rightarrow \mathcal{M}^{\tw I }\]
with values in admissible objects, more precisely, to a diagram of the form considered in \ref{BEISPIELMULTICOR}, i.e.\@ to a 1-morphism in $\HH^{\cor}(\mathcal{M})(I)$.
A composition $\beta \circ \alpha$ in $J$ gives  rise to  a functor
\[ \tw \Delta_2  \rightarrow\mathcal{M}^{\tw I } \]
and $(\beta \circ  \alpha)'$ is the pullback along 
\[ e_{02}: \tw \Delta_1 \rightarrow \tw \Delta_2.  \]
By construction, we get a 2-isomorphism 
\[ \beta' \circ \alpha' \Rightarrow (\beta \circ \alpha)'. \]
The analogous reasoning with a composition of three morphisms shows that this construction yields a pseudo-functor. 

2.\@ is left to the reader. 
\end{proof}

\begin{DEF}\label{DEFHCOR2}
We define a symmetric 2-pre-multiderivator $\HH^{\cor}(\mathcal{M})$. Let $I$ be a diagram\footnote{not a multidiagram, because in our definition of (pre-)multiderivator we still take only usual diagrams as input which is, of course, an omission.}. The 2-multicategory $\HH^{\cor}(\mathcal{M})(I)$ has been defined in Definiton~\ref{DEFHCOR1}.
Those 2-multicategories are equipped with strict pull-back functors
\[ \alpha^*: \HH^{\cor}(\mathcal{M})(J) \rightarrow \HH^{\cor}(\mathcal{M})(I). \]
for each functor $\alpha: I \rightarrow J$. For a natural transformation $\mu: \alpha \Rightarrow \beta$ we get a pseudo-natural transformation 
\[ \alpha^* \Rightarrow \beta^* \]
as follows: $\mu$ might be seen as a functor $I \times \Delta_1 \rightarrow J$ and so each admissible $X \in \mathcal{M}^{\tw J}$ gives rise to an admissible
 $\mu^* X \in \mathcal{M} ^{\tw (\Delta_1 \times I)}$ which, by definition, constitutes a 1-morphism from $\alpha^* X$ to $\beta^* X$. For each (1-ary) 1-morphism $\xi: X \rightarrow Y$ the 
 square
 \[ \xymatrix{ \alpha^*X \ar[r]^{\alpha^* \xi} \ar[d]_{\mu^* X} & \alpha^*Y \ar[d]^{\mu^*Y} \\
 \beta^*X \ar[r]_{\beta^* \xi} & \beta^*Y } 
  \]
  commutes up to a uniquely determined 2-isomorphism using the pseudo-functor
  \[ \Delta_1^2 \rightarrow \HH^{\cor}(\mathcal{M})(I) \] 
  obtained from $\mu^* \xi \in \mathcal{M} ^{\tw (\Delta_1^2 \times I)}$ via Lemma~\ref{LEMMAPF}, 1. 
Similarly, the pseudo-naturality for $n$-ary morphisms is constructed. One also checks that this defines a pseudo-functor:
\[ \Fun(I, J) \rightarrow \Fun^{\mathrm{strict}}(\HH^{\cor}(\mathcal{M})(J), \HH^{\cor}(\mathcal{M})(I))\]
where $\Fun^{\mathrm{strict}}$ is the 2-category of strict 2-functors, pseudo-natural transformations, and modifications.
\end{DEF}

\begin{BEM}If $\mathcal{W}$ consists of isomorphisms then it follows from Lemma~\ref{LEMMAPF}, 2.\@ that $\HH^{\cor}(\mathcal{M})$ is equivalent to the 2-pre-multiderivator represented by the 2-multicategory $\mathcal{M}^{\cor}$.
\end{BEM}

\begin{LEMMA}
$\HH^{\cor}(\mathcal{M})$ satisfies the axioms (Der1${}^\infty$) and (Der2) of \ref{PARDER12}.
\end{LEMMA}
\begin{proof}
(Der1${}^\infty$) is clear. For (Der2) use the fact that equivalences in $\HH^{\cor}(\mathcal{M})(I)$ are precisely those correspondences (cf.\@ \ref{BEISPIELMULTICOR}) for $n=1$ in which both $f$ and $g$ are weak equivalences (by an easy application of the fact that $\mathcal{W}$ is saturated and thus satisfies 2-out-of-6).  
\end{proof}

The following Lemma follows from the definitions in a straightforward way:

\begin{LEMMA}\label{LEMMAFUNCTHOCOR}
Let $(\mathcal{M}_i, \Fib_i, \mathcal{W}_i)$, $i=1,2$, be two categories with fibrant objects. 
Every functor $F: (\mathcal{M}_1, \mathcal{W}_1) \rightarrow (\mathcal{M}_2, \mathcal{W}_2)$ of categories with weak equivalences, which preserves homotopy fiber products, induces a morphism of 2-pre-multiderivators
\[  \HH^{\cor}(F): \HH^{\cor}(\mathcal{M}_1) \rightarrow \HH^{\cor}(\mathcal{M}_2).  \]
If $F$ is an equivalence of categories with weak equivalences, in the sense that there is a functor $G: (\mathcal{M}_2, \mathcal{W}_2) \rightarrow (\mathcal{M}_1, \mathcal{W}_1)$ of categories with weak equivalences such that $F \circ G$ and $G \circ F$ are weakly equivalent to the identity, then  $\HH^{\cor}(F)$ is an equivalence of 2-pre-multiderivators.
\end{LEMMA}

\begin{PAR} \label{EMBEDDINGM}
If $\mathbb{M}$ is the   (usual 1-pre-derivator) associated with  $\mathcal{M}$ then we have obvious (non-strict) morphisms
\[ \mathbb{M} \rightarrow \HH^{\cor}(\mathcal{M})  \]
and
\[ \mathbb{M}^{\op} \rightarrow \HH^{\cor}(\mathcal{M}).  \]
The latter is even a morphism of symmetric 2-pre-{\em multi}\,derivators if on $\mathcal{M}^{\op}$ one chooses the symmetric multicategory structure
\[ \Hom_{\mathcal{M}^{\op}}(X_1, \dots, X_n; Y) := \Hom_\mathcal{M}(Y, X_1) \times \cdots \times \Hom_\mathcal{M}(Y, X_n)\]
i.e.\@ the opposite of the symmetric opmulticategory structure (\ref{eqopmult}).  
\end{PAR}

Generalizing the case of a usual base category (cf.\@ \cite[Definition~5.8]{Hor16}) we define:

\begin{DEF}\label{DEF6FU} Let $(\mathcal{M}, \Fib, \mathcal{W})$ be a category with fibrant objects as in \ref{PARCATFIBOBJ}. 
A (symmetric) {\bf derivator six-functor-formalism} over $(\mathcal{M}, \Fib, \mathcal{W})$ is
a (symmetric) fibered multiderivator 
\[ \DD \rightarrow \HH^{\cor}(\mathcal{M}). \]
\end{DEF}

In this article we abstain from discussing any generalization of {\em proper}, or {\em etale}, six-functor-formalisms in the sense of \cite[Definition~5.8]{Hor16}.

\begin{PAR}\label{PAREXTR6FU}
Given a derivator six-functor-formalism as in Definition~\ref{DEF6FU},
a morphism $f: X \rightarrow Y$ in $\mathcal{M}$ yields adjunctions
\[ \xymatrix{ \DD(\cdot)_X \ar@/^10pt/[r]^{f_*} & \ar@/^10pt/[l]^{f^*}  \DD(\cdot)_Y } \quad \xymatrix{ \DD(\cdot)_Y \ar@/^10pt/[r]^{f^!} & \ar@/^10pt/[l]^{f_!}  \DD(\cdot)_X }  \]
as pull-back and push-forward along the correspondences in $\HH^{\cor}(\mathcal{M})(\cdot)$:
\[ \vcenter{ \xymatrix{  & X \ar[ld]_f \ar@{=}[rd]  \\
Y & & X } } \quad \text{and} \quad \vcenter{ \xymatrix{  & X \ar@{=}[ld] \ar[rd]^f  \\
X & & Y. } } \]
We get a tensor product
\[ \otimes: \DD(\cdot)_X \times \DD(\cdot)_X \rightarrow \DD(\cdot)_X \]
with right adjoints w.r.t.\@ both variables as push-forward and pull-backs along the multicorrespondence
\[ \xymatrix{ &&X \ar@{=}[lld] \ar@{=}[ld] \ar@{=}[rd]  \\
X & X & ; & X } \]
\end{PAR}
\begin{PROP}\label{PROPCONSEQUENCES}
Given a symmetric (for simplicity) derivator six-functor-formalism as in Definition~\ref{DEF6FU}, there exist canonical isomorphisms
for the functors extracted in \ref{PAREXTR6FU}:

\begin{center}
\begin{tabular}{rlll}
& left adjoints & right adjoints \\
\hline
$(*,*)$ & $(fg)^* \iso g^* f^*$ & $f_* g_* \iso (fg)_*$ &\\
$(!,!)$ & $(fg)_! \iso f_! g_!$ & $ g^! f^!\iso (fg)^!$ &\\ 
$(!,*)$ 
& $g^* f_! \iso F_! G^*$ & $G_* F^! \iso f^! g_*$ & \\
$(\otimes,*)$ & $f^*(- \otimes -) \iso f^*- \otimes f^* -$ & $f_* \mathcal{HOM}(f^*-, -) \iso \mathcal{HOM}(-, f_*-)$  & \\
$(\otimes,!)$ & $f_!(- \otimes f^* -) \iso  (f_! -) \otimes -$ & $\mathcal{HOM}(f_! -, -) \iso f_* \mathcal{HOM}(-, f^!-) $ & \\ 
& & $f^!\mathcal{HOM}(-, -) \iso \mathcal{HOM}(f^* -, f^!-)$ & \\
$(\otimes, \otimes)$ &  $(- \otimes -) \otimes - \iso - \otimes (- \otimes -)$ &  $\mathcal{HOM}(-, \mathcal{HOM}(-, -)) \iso \mathcal{HOM}(- \otimes -, -)$ & 
\end{tabular}
\end{center}
Here $f, g, F, G$ are morphisms in $\mathcal{M}$ which, in the $(!,*)$-row, are related by a {\em homotopy Cartesian} diagram
\[ \xymatrix{  \ar[r]^G \ar[d]_F  &  \ar[d]^f \\  \ar[r]_g &  } \]

Furthermore, if $f \in \mathcal{W}$, there is a canonical isomorphism
\[ f^! \cong f^* \]
and both functors are equivalences. 
\end{PROP}
\begin{proof}See \cite[Lemma~A.2.19]{Hor15} or \cite[Proposition~3.9]{Hor15b}  for the analogous statement over usual categories. The proof is the same. 
 For the additional statement we proceed analogously to \cite[Proposition~8.3]{Hor15b}:
 Consider the correspondences
 \[ \xymatrix{  & X \ar[ld]_f \ar@{=}[rd]  \\
Y & & X } \quad \xymatrix{  & X \ar@{=}[ld] \ar[rd]^f  \\
X & & Y } \]
Those are adjoint equivalences in the 2-category $\HH^{\cor}(\mathcal{M})(\cdot)$ by means of the following unit and couint:
 \[ \xymatrix{  & X \ar[ld]_f \ar[rd]^f \ar[dd]^f  \\
Y & & Y  \\
& Y  \ar@{=}[lu] \ar@{=}[ru] } \quad \xymatrix{  & X \ar@{=}[ld] \ar@{=}[rd] \ar[dd]^{\Delta_f}  \\
X  & & X \\
 & X \widetilde{\times}_Y X \ar[lu] \ar[ru] } \]
where $X \widetilde{\times}_Y X$ denotes the homotopy fiber product chosen in the definition of $\HH^{\cor}(\mathcal{M})$. By definition both 2-morphisms are invertible. 
Hence the push-forward functors $f_!$ and $f^*$ are adjoint equivalences as well, whence a canonical isomorphism
$f^! \cong f^*$.
\end{proof}

\begin{PAR}
The following is only needed for a certain construction in Proposition~\ref{PROPSTRONGLYLOCAL} and can be skipped on a first reading. 
It is, however, of independent interest. As explained in \ref{PARDIACORHAT}, a derivator six-functor-formalism on $\mathcal{M}$ gives rise to a pseudo-functor of 2-multicategories
\[ \widehat{\Dia}^{\cor}(\HH^{\cor}(\mathcal{M})) \rightarrow \mathcal{CAT} \]
such that all functors in the image have right adjoint w.r.t.\@ all slots. 
\end{PAR}
\begin{PAR}\label{PARNOTATIONDERCORHCOR}
Note that a 1-morphism in $\widehat{\Dia}^{\cor}(\HH^{\cor}(\mathcal{M}))$ can be (up to 2-isomorphism) presented by a diagram of the form
\[ \xymatrix{ & & & (A, U) \ar[llld]_-{(\alpha_1,g_1)} \ar[ld]^-{(\alpha_n,g_n)} \ar[rd]^-{(\beta, f)} \\
(I_1,S_1) & \cdots & (I_n, S_n) & ; & (J, T) }  \]
where $S_i: \tw I \rightarrow \mathcal{M}$, $U: \tw A \rightarrow \mathcal{M}$, and $T: \tw J \rightarrow \mathcal{M}$ are admissible diagrams, 
$(g_1, \dots, g_n): U \rightarrow (\tw \alpha_1)^*S_1, \dots, (\tw \alpha_n)^*S_n$ is type-1-admissible and $f: U \rightarrow (\tw \beta)^*T$ is type-2-admissibile.
The image of this 1-morphism under the above pseudo-functor is given by the following functor in $n$-variables:
\[ \beta_!^{(T)} f_!( (g_1^* \alpha_1^* -) \otimes \cdots \otimes (g_n^* \alpha_n^* -)  ) \]
in which $f_!$ is a push-forward in the six-functor-formalism and $\beta_!^{(T)}$ is the relative homotopy Kan extension of the fibered derivator. 
As promised, we have the following statement about the 2-multicategory $\widehat{\Dia}^{\cor}(\HH^{\cor}(\mathcal{M}))$:
\end{PAR}

\begin{SATZ}\label{SATZDUALITY}
Let $\mathcal{M}$ be a category with fibrant objects and functorial path object. 
The functor of symmetric 2-multicategories
\[ \widehat{\Dia}^{\cor}(\HH^{\cor}(\mathcal{M})) \rightarrow  \{ \cdot\} \]
is 1-bifibered (in particular $\widehat{\Dia}^{\cor}(\HH^{\cor}(\mathcal{M}))$ is closed monoidal) with the monoidal product given by
\[ (I, S) \boxtimes (J, T) = (I \times J, S \times T) \]
(product of $S$ and $T$ formed in $\mathcal{M}$), unit given by $(\cdot, \cdot)$,
and internal hom given by
\[ \mathbf{HOM}((J, T), (I, S)) = (J^{\op} \times I, T^{\op} \times S) \]
where $T^{\op} \in \HH^{\cor}(\mathcal{M})(I^{\op})$ is obtained from $T \in \HH^{\cor}(\mathcal{M})(I)$ by flipping all correspondences\footnote{For the precise definition of $T^{\op}$ observe that we have a canonical isomorphism $\tw I \cong \tw (I^{\op})$ which interchanges type-1 and type-2 morphisms. }.  
\end{SATZ}

\begin{PAR}\label{PAREXTHOM}
This has the following useful consequence for a symmetric derivator six-functor-formalism (which is a generalization of a well-known behaviour for closed monoidal derivators, a.k.a.\@ fibered multiderivators over $\{\cdot\}$). 
For objects $\mathcal{E} \in \DD(I)_S$ and  $\mathcal{F} \in \DD(J)_T$ there is an absolute Hom object
\[ \mathbf{HOM}(\mathcal{F}; \mathcal{E}) \quad  \text{ on }\quad  (I \times J^{\op}, S \times T^{\op})  \]
defined as the pull-back  $\xi^{\bullet,1}(\mathcal{F}$; $\mathcal{E})$ along 
the Cartesian (w.r.t.\@ the first slot) morphism
\[ \xi: (I \times J^{\op},  S \times T^{\op}), (J, T) \to (I, S).  \]
Furthermore  this is ``computed point-wise'', i.e.\@ for $(i,j) \in I \times J^{\op}$ the natural exchange morphism
\[ (i,j)^* \mathbf{HOM}(\mathcal{F}; \mathcal{E}) \rightarrow \mathbf{HOM}(j^*\mathcal{F}, i^*\mathcal{E}) = \mathcal{HOM}_{S_i \times T_j}(\pr_{T_j}^* \mathcal{F}_j, \pr_{S_i}^! \mathcal{E}_i )  \]
is an isomorphism. 
\end{PAR}

\begin{proof}[Proof of the point-wise computation of $\mathbf{HOM}$.]
One checks that composition of  
\[ \xymatrix{ & (\cdot, S(i) \times T(j)) \ar[ld] \ar[rd] \\
(\cdot,S(i)\times T(j)) & ; & (I \times J^{\op}, S \times T^{\op}) }  \]
with the coCartesian morphism (see proof) is isomorphic to:
\[ \xymatrix{ & & (\cdot, S(i) \times T(j)) \ar[lld] \ar[ld] \ar[rd] \\
(\cdot,S(i) \times T(j)) & (J, T) & ; & (I, S) }  \]
\end{proof}

This also implies that the 2-multicategory  $\widehat{\Dia}^{\cor}(\HH^{\cor}(\mathcal{M}))$ is self-dual, i.e.\@ the functor
\[ (I, S) \mapsto \mathbf{HOM}((I, S); (\cdot, \cdot)) \cong (I^{\op}, S^{\op})   \]
is an equivalence $\widehat{\Dia}^{\cor}(\HH^{\cor}(\mathcal{M})) \cong \widehat{\Dia}^{\cor}(\HH^{\cor}(\mathcal{M}))^{1-\op}$ (as 2-categories). 
\begin{proof}[Proof of Theorem~\ref{SATZDUALITY}.]
We will use the description (cf.\@ Proposition~\ref{PROPALTCOR}) of the homomorphism categories in $\widehat{\Dia}^{\cor}(\HH^{\cor}(\mathcal{M}))$ as localizations of
\[  \Fun_{\HH^{\cor}(\mathcal{M}); S, T; U}(I^{\op} \times  J^{\op} \times K, \Dia)  \]
considered as usual 1-category. 
We will establish an equivalence of its localization with the localization of the following categories
\[  \Fun_{\HH^{\cor}(\mathcal{M}); S \times T; U}((I \times J)^{\op} \times  J^{\op} \times K, \Dia)  \]
\[  \Fun_{\HH^{\cor}(\mathcal{M}); S; T^{\op} \times U}(I^{\op} \times  (J^{\op} \times K), \Dia).  \]
The three categories are obtained by applying the Grothendieck construction to the functors (cf.\@ \ref{PARDIACORHAT})
\begin{eqnarray*}
1.&\xymatrix{  \Fun(I^{\op} \times  J^{\op} \times K, \Dia) \ar[rr]^-{\Xi} &&  \Cor(I, J; K) \ar[r] & \mathcal{CAT}  } &\quad \text{ induced by } (S, T; U) \\
2.&\xymatrix{  \Fun(I^{\op} \times  J^{\op} \times K, \Dia) \ar[rr]^-{\Xi'} &&  \Cor(I \times J; K) \ar[r] & \mathcal{CAT}  } &\quad \text{ induced by } (S \times T; U) \\
3.&\xymatrix{  \Fun(I^{\op} \times  J^{\op} \times K, \Dia) \ar[rr]^-{\Xi''} &&  \Cor(I; K \times J^{\op}) \ar[r] & \mathcal{CAT}  } &\quad \text{ induced by } (S, U  \times T^{\op})
\end{eqnarray*}

Actually, it suffices to consider the left hand side categories as usual 1-categories, forgetting the 2-morphisms, as they are taken care of by the localization. 
We proceed to construct an equivalence between the localizations of the Grothendieck constructions of 1.\@ and 3. The other case is simpler and left to the reader. 
Define a functor
\[  \Theta:  \Fun_{\HH^{\cor}(\mathcal{M}); S, T; U}(I^{\op} \times  J^{\op} \times K, \Dia)  \rightarrow \Fun_{\HH^{\cor}(\mathcal{M}); S; T^{\op} \times U}(I^{\op} \times  (J^{\op} \times K), \Dia)  \]
as follows: 
An object on the l.h.s. is given by a pair $(F, A)$ with 
\[ F: I^{\op} \times J^{\op} \times K \rightarrow \Dia  \quad  A \in \mathcal{M}^{\tw (\int \nabla F)} \]
equipped with type-1 (resp.\@ type-2) admissible morphisms
\[ A \rightarrow \pi_I^*S \times \pi_J^*T \qquad A \rightarrow \pi_K^* U. \]
We send it to the functor\footnote{Actually we have $\Theta(F) \cong F$ in $\HH(I^{\op} \times J^{\op} \times K) = \Fun(I^{\op} \times J^{\op} \times K, \Dia)[(\Xi^{-1}(\mathcal{W}))^{-1}] $ but it seems not easily possible to transfer the correspondence $A$.}
\begin{eqnarray*} \Theta(F):  I^{\op} \times J^{\op} \times K &\rightarrow& \Dia \\
 i, j, k& \mapsto& \{ (j \rightarrow \widetilde{j}, o) \ \mid\ \widetilde{j} \in J, o \in F(i, \widetilde{j}, k) \}
\end{eqnarray*}
together with 
\[ \Theta(A) \in \mathcal{M}^{\tw (\int \nabla' \Theta(F))} \]
given by
\[  \Theta(A) \left( \vcenter{ \xymatrix{
i \ar[d]^\alpha & j \ar[r]^\kappa  \ar@{<-}[d]^{\beta_2} & \widetilde{j} \ar[d]^{\beta_3} & k \ar[d]^\gamma & o \in F(i,\widetilde{j},k) \ar@{.>}[d] \\
i'  & j' \ar[r]_{\kappa'}  & \widetilde{j}' & k & o' \in F(i',\widetilde{j}',k') 
}} \right)   := A(i \to i', j' \to \widetilde{j}', k \to k', F(\beta_2) F(\kappa) o \rightarrow o').     \]
One checks that $\Theta(A)$ is admissible and comes equipped with type-1 (resp.\@ type-2) admissible morphisms
\[ \Theta(A) \rightarrow \pi_I^* S \qquad \Theta(A) \rightarrow \pi_{J^{\op} \times K}^*(T^{\op} \times U). \]

Similarly, we define a functor 
\[  \Upsilon:  \Fun_{\HH^{\cor}(\mathcal{M}); S; T^{\op} \times U}(I^{\op} \times  (J^{\op} \times K), \Dia) \rightarrow  \Fun_{\HH^{\cor}(\mathcal{M}); S, T; U}(I^{\op} \times  J^{\op} \times K, \Dia).  \]
An object on the l.h.s. is given by a pair $(G, B)$ with 
\[ G: I^{\op} \times J^{\op} \times K \rightarrow \Dia  \quad  B \in \mathcal{M}^{\tw (\int \nabla' G)} \]
equipped with type 1- (resp.\@ type-2) admissible morphisms
\[ B \rightarrow  \pi_I^* S \qquad B \rightarrow \pi_{J^{\op}\times K}^* (T^{\op} \times U).  \]
We send it to the functor
\begin{eqnarray*} \Upsilon(G):  I^{\op} \times J^{\op} \times K &\rightarrow& \Dia \\
 i, j, k& \mapsto& \{ (j \rightarrow \widetilde{j}, o) \ \mid\ \widetilde{j} \in J^{\op}, o \in G(i, \widetilde{j}, k) \}
\end{eqnarray*}
together with 
\[ \Upsilon(B) \in \mathcal{M}^{\tw (\int \nabla \widetilde{F})} \]
given by
\[  \Upsilon(B) \left( \vcenter{ \xymatrix{
i \ar[d]^\alpha & j \ar[r]^\tau  \ar[d]^{\beta_1} & \widetilde{j} \ar@{<-}[d]^{\beta_2} & k \ar[d]^\gamma & o \in G(i,\widetilde{j},k) \ar@{.>}[d] \\
i'  & j' \ar[r]_{\tau'}  & \widetilde{j}' & k & o' \in G(i',\widetilde{j}',k') 
}} \right)   := B(i \to i', \widetilde{j} \to j, k \to k',  o \rightarrow F(\beta_1) F(\tau') o').     \]
One checks that $\Upsilon(B)$ is admissible and comes equipped with type-1 (resp.\@ type-2) admissible morphisms
\[ \Upsilon(B) \rightarrow \pi_{I \times J}^*(S \times T) \qquad \Upsilon(B) \rightarrow \pi_J^* U. \]

We claim that there are isomorphisms
\[  \Upsilon \Theta \rightarrow \id \quad \Theta \Upsilon \rightarrow \id \]
in the localization modulo $\widetilde{\Xi}^{-1}(\mathcal{W})$ and $(\widetilde{\Xi}')^{-1}(\mathcal{W})$, respectively, and therefore descend to equivalences of the corresponding localizations. 
We will describe the second morphism (the first is similar):
Let $(G, B)$ be in $\Fun_{\HH^{\cor}(\mathcal{M}); S; T^{\op} \times U}(I^{\op} \times  (J^{\op} \times K), \Dia)$. Inserting the definitions we have
\[ \Theta \Upsilon(G): i, j, k \mapsto \{ (j \to \widetilde{j} \to \widetilde{\widetilde{j}}, o) \ \mid \ \widetilde{j} \in J, \widetilde{\widetilde{j}} \in J^{\op}, o \in F(i, \widetilde{\widetilde{j}}, k) \} \]
and
\[ \Theta \Upsilon(B) = \epsilon^*B \] where $\epsilon$ is the functor
\begin{eqnarray*} \epsilon: \tw \int \nabla' ( \Theta \Upsilon(G)) &\rightarrow& \tw \int \nabla' G \\
 \left( \vcenter{ \xymatrix{
i \ar[d]^\alpha & j \ar[r]^\kappa  \ar@{<-}[d]^{\beta_1} &\widetilde{j} \ar[r]^\tau  \ar[d]^{\beta_2} & \widetilde{\widetilde{j}} \ar@{<-}[d]^{\beta_3} & k \ar[d]^\gamma & o \in G(i,\widetilde{\widetilde{j}},k) \ar@{.>}[d] \\
i'  & j' \ar[r]_{\kappa'}  & \widetilde{j}' \ar[r]_{\tau'}  & \widetilde{\widetilde{j}}' & k & o' \in G(i',\widetilde{\widetilde{j}}',k') 
}} \right) &\mapsto& (i \to i', \widetilde{\widetilde{j}} \to j', k \to k',  o \rightarrow  F(\kappa') F(\tau') o')
\end{eqnarray*}

We have also a natural transformation 
\begin{eqnarray*}
 \rho: \Theta \Upsilon (G)& \rightarrow& G \\
 (\xymatrix{ j \ar[r]^{\kappa} & \widetilde{j} \ar[r]^\tau & \widetilde{\widetilde{j}} }, o)   & \mapsto &  G(\kappa) G(\tau) o 
\end{eqnarray*}
yielding 
\begin{eqnarray*}  \tw(\int \nabla \rho): \tw \int \nabla' ( \Theta \Upsilon(G)) &\rightarrow& \tw \int \nabla' G \\
\footnotesize \left( \vcenter{ \xymatrix{
i \ar[d]^\alpha & j \ar[r]^\kappa  \ar@{<-}[d]^{\beta_1} &\widetilde{j} \ar[r]^\tau  \ar[d]^{\beta_2} & \widetilde{\widetilde{j}} \ar@{<-}[d]^{\beta_3} & k \ar[d]^\gamma & o \in G(i,\widetilde{\widetilde{j}},k) \ar@{.>}[d] \\
i'  & j' \ar[r]_{\kappa'}  & \widetilde{j}' \ar[r]_{\tau'}  & \widetilde{\widetilde{j}}' & k & o' \in G(i',\widetilde{\widetilde{j}}',k') 
}} \right) &\mapsto& (i \to i', j \to j', k \to k',  F(\tau \kappa) o \rightarrow  F(\tau' \kappa') o')
\end{eqnarray*}
There is an obvious natural transformation $\alpha:  \epsilon \Rightarrow \tw (\int \nabla' \rho)$ giving 
a diagram
\[. \xymatrix{ & \Theta \Upsilon(B) = \epsilon^* B \ar[dd]^{\alpha(B)} \ar[ld]  \ar[rd] \\
\pi_I^* S & & \pi_{J^{\op} \times K}^*(T^{\op} \times U)   \\
& (\tw (\int \nabla' \rho)))^* B \ar[lu] \ar[ru]
} \]
in which $\alpha(B)$ is a point-wise weak equivalence because the morphism $B \rightarrow \pi_I^* S$ is type-1 admissible.

We can factor for any right fibered multiderivator
\[ \Xi'(\Theta \Upsilon(G), \epsilon^*B)^\bullet \leftarrow \Xi'(\Theta \Upsilon(G), \tw (\int \nabla' \rho)^* B)^\bullet \leftarrow \Xi'(G, B)^\bullet,   \]
or for any left fibered multiderivator
\[ \Xi'(\Theta \Upsilon(G), \epsilon^*B)_\bullet \rightarrow \Xi'(\Theta \Upsilon(G), \tw (\int \nabla' \rho)^* B)_\bullet \rightarrow \Xi'(G, B)_\bullet.   \]
The first natural transformation is an isomorphism because $\alpha(B)$ is a weak equivalence 
 and the second is also an isomorphism. Indeed this can be checked point-wise
 in $i, j, k$ (cf.\@ Proposition~\ref{PROPALTCOR}). It suffices to show that for the functor
 \[ \rho_{i,j,k}: \{ j \rightarrow \widetilde{j} \to \widetilde{\widetilde{j}} \ \mid o \in G(i, \widetilde{\widetilde{j}}, k)\} \rightarrow G(i,j,k)  \]
 we have that $\id \rightarrow \rho_{i,j,k,*}^{(X)} \rho^*_{i,j,k}$ (resp.\@ $\rho_{i,j,k,!}^{(X)} \rho^*_{i,j,k} \rightarrow \id$) is an isomorphism for any $X$. But $\rho_{i,j,k}$ has a left adjoint 
 \[ \delta: G(i,j,k)  \rightarrow \{ (j \rightarrow \widetilde{j} \to \widetilde{\widetilde{j}}, o) \ \mid o \in G(i, \widetilde{\widetilde{j}}, k)\}  \]
 mapping $o \in G(i,j,k)$ to $(j = j = j, o)$. 
Hence $\delta^*$ is right adjoint to $\rho_{i,j,k}^*$  and thus for a right fibered multiderivator  $\rho_{i,j,k,*}^{(X)} \cong \delta^*$. Thus $\rho_{i,j,k,*}^{(X)} \rho^*_{i,j,k} \cong \delta^* \rho^*_{i,j,k} \cong \id$
and for a left fibered multiderivator $\delta_!^{(\rho_{i,j,k}^*X)} \cong \rho_{i,j,k}^*$ and thus $\id = \rho_{i,j,k,!}^{(X)} \delta_!^{(\rho_{i,j,k}^*X)} =   \rho_{i,j,k,!}^{(X)} \rho_{i,j,k,!}^*$. 

We have to show that the transformation $\Upsilon$ is induced by composition with a multicorrespondence (which is thus (weakly) Cartesian w.r.t.\@ the first slot), in fact with the following multicorrespondence
\[ \mathbf{CART}^1 =  \left( \vcenter{ \xymatrix{ & & ((\tw J) \times K, C)  \ar[lld] \ar[ld] \ar[rd] \\
(J^{\op} \times K, T^{\op} \times U) & (J,T) & ; & (K,U) } } \right)  \]
where $C$ maps an object in $\tw((\tw J) \times I)$
\[ \xymatrix{ j_2 \ar[d] \ar[r] & j_1 \ar@{<-}[d] & k \ar[d]   \\
j_2'  \ar[r]  & j_1' & k' 
}\] 
to $U(k \to k') \times T(j_2 \to j_1)$. 
Consider a composition
\[ \xymatrix{ 
&& D  \ar[ld] \ar[rrd] \\
& \pr_1^*B  \ar[ld] \ar[rd] & & & \pr_2^*C \ar[lld] \ar[ld] \ar[rd] \\
\pi_{I}^*S & & \pi_{J^{\op} \times K}^*(T^{\op} \times U) &  \pi_J^*T & ; & \pi_K^*U }  \]
on $(\int \nabla' F) \times_{J^{\op}} \tw J$. We see that $D$ evaluated at an object 
\[ \xymatrix{ i \ar[d]^\alpha & j_2  \ar[d]_{\beta_1} \ar[r]^{\tau} & j_1  \ar@{<-}[d]^{\beta_2} & k \ar[d]^\gamma & l \in F(i, j_1, k) \ar@{.>}[d]^{\rho}  \\
i'  & j'_2 \ar[r]_{\tau'} & j_1' & k'  & l' \in F(i', j_1', k') 
}\] 
in $\tw((\int \nabla' F) \times_{J^{\op}} \tw J)$ sits in a homotopy Cartesian square
\[ \xymatrix{  D(i \to i', j_2 \to j'_2, j_1 \rightarrow j_1',  k \to k', l \to l') \ar[r]  \ar[d] & T(j_2 \to j_1) \times U(k \to k')   \ar[d] \\
B(i \to i' ,j_1 \to j_1',k \to k', l \to l' )  \ar[r] & T(j_1' \to j_1) \times U(k \to k').   
  }\]
Because of the type-2 admissibility of the morphism $B \rightarrow \pi_{J^{\op} \times K}^*(T^{\op} \times U)$ we have an isomorphism in the homotopy category
\[ D(i \to i', j_2 \to j'_2, j_1 \rightarrow j_1',  k \to k', l \to l') \cong B(i \to i' ,j_1 \to j_2,k \to k', l \to F(\beta_1)F(\tau')l' ).    \]
Thus $\Xi'(\Upsilon(G), \Upsilon(B)) \cong ((\int \nabla' F) \times_{J^{\op}} \tw J, D)$ and
the equivalence established in the first part of the proof, proves that the multicorrespondence is (weakly) coCartesian w.r.t.\@ the first slot. 
One can then either show that the composition of weakly coCartesian morphisms is weakly coCartesian or, more easily, show that it is actually coCartesian by
introducing  additional arguments. 
\end{proof}

\section{Naive extensions and (co)homological descent}\label{SECTIONNAIVE}

\begin{PAR}\label{PARSETTINGGENERAL}
Let $\mathcal{S}$ be a small category with finite limits and Grothendieck pre-topology
and choose a category with fibrant objects $(\mathcal{M}, \Fib, \mathcal{W})$ with functorial path object representing the \v{C}ech localization of simplicial pre-sheaves on $\mathcal{S}$ such that coproducts of representables are in $\mathcal{M}$. In other words, such that there is an
embedding $\mathcal{S}^{\amalg} \hookrightarrow \mathcal{M}$. 
\end{PAR}

\begin{DEF}\label{DEFSHCOR}
In the situation of \ref{PARSETTINGGENERAL}, we denote
\[ \SSS^{\hcor} := \HH^{\cor}(\mathcal{M}) \]
the symmetric 2-pre-multiderivator of homotopy multicorrespondences in $\mathcal{S}$ (cf.\@ Definition~\ref{DEFHCOR2}) with domain $\Cat$.
\end{DEF}

We get a strict morphism $\iota: \SSS^{\amalg, \cor} \rightarrow \SSS^{\hcor}$.

\begin{PAR}
The goal of the article is to extend a derivator six-functor-formalism on $\mathcal{S}$, i.e.\@ a fibered multiderivator over $\SSS^{\cor}$, to
a fibered multiderivator over the sub-2-pre-multiderivator $\SSS^{\hcor, (\infty)}$ of $\SSS^{\hcor}$ coming from higher geometric stacks (cf.\@ Section~\ref{SECTHIGHERGEOM}).  

In this section, we recall that there is always a (naive) extension of the restriction $\DD^{!}$ of $\DD$ to
$\SSS$ using homological descent, and an extension of the restriction $\DD^{*}$ of $\DD$ to
$\SSS^{\op}$ using cohomological descent. In other word, a naive extension of the formalism of the $!$-functors and of the formalism of the $*$-functors (the latter together with $\otimes, \mathcal{HOM}$ if desired). It is not clear a priori that the two extensions are even equivalent on objects. 
\end{PAR}

\begin{PAR}\label{PARLOCAL}Let $\Dia$ be a diagram category (Definition~\ref{DEFDIAGRAMCAT}).
Recall that a fibered derivator with domain $\Dia$
\[ \DD^! \rightarrow \SSS  \]
is called {\bf local} for the pretopology on $\mathcal{S}$, if  (writing, as usual, $f^!, f_!$ for $f^\bullet, f_\bullet$):
\begin{enumerate}
\item[(H1'')] For all $I \in \Dia$ and each morphism $f: S \rightarrow T$ in $\mathcal{S}^I$ such that $S_i \rightarrow T_i$ is part of a covering, 
the morphism of derivators $f^!: \DD^!_{(I,T)} \rightarrow \DD^!_{(I,S)}$ has a right adjoint.
\item[(H2)] For each covering $\{f_i: U_i \rightarrow X\}$, and for each $i$, $f_i^!$ satisfies base change, i.e.\@ for each Cartesian square in $\mathcal{S}$
\[ \xymatrix{  \ar[r]^{F_i} \ar[d]_G &  \ar[d]^g \\ 
 \ar[r]_{f_i} & 
} \] 
 the natural exchange $G_! F_i^! \rightarrow f^!_i g_!$ is an isomorphism. 
\item[(H3)] For each covering $\{f_i: U_i \rightarrow X\}$, the $f_i^!$ are jointly conservative. 
\end{enumerate}
Also consider the following variants of (H1''). For (H1') assume that $\DD^! \rightarrow \SSS$ has stable fibers. 
\begin{enumerate}
\item[(H1)] For each covering $\{f_i: U_i \rightarrow X\}$, and for each $i$, $f_i^!$, as morphism of derivators, commutes with homotopy colimits. 
\item[(H1')] For each covering $\{f_i: U_i \rightarrow X\}$, and for each $i$, $f_i^!$ is exact and commutes with arbitrary (homotopy) coproducts. 
\end{enumerate}
We obviously always have $(H1'') \Rightarrow (H1) \Rightarrow (H1')$. 
If we assume that $\DD^! \rightarrow \SSS$ is infinite (i.e.\@ satisfies (Der1${}^\infty$)) then 
homotopy coproducts in the derivator (fiber) $\DD^!_X$ are the same as coproducts in the underlying category $\DD^!(\cdot)_X$. 
If $\DD^! \rightarrow \SSS$ has stable fibers, then we have always $(H1') \Rightarrow (H1)$, and if $\DD^!$ is moreover infinite, and has perfectly generated fibers we have
$(H1) \Rightarrow (H1'')$ using Brown representability \cite[Theorem~4.2.1]{Hor15}. In case all assumptions are met  (infinite + stable and perfectly generated fibers) all three properties (H1''), (H2), (H3) and hence the locality of $\DD^!$ can thus be checked
in terms of the {\em triangulated} categories $\DD^!(\cdot)_X$ alone. 
\end{PAR}

\begin{PAR}
Recall that the cofibrant replacement functor $Q$ in the projective model category structure on $\mathcal{SET}^{\mathcal{S}^{\op} \times \Delta^{\op}}$ has values in the essential image of
$\mathcal{S}^{\amalg, \Delta^{\op}}$ (coproducts of representables, using that $\mathcal{S}$ has finite limits and is thus idempotent complete). This yields a functor
\[ \int^{\amalg} Q:  \mathcal{SET}^{\mathcal{S}^{\op} \times \Delta^{\op}} \rightarrow \Cat(\mathcal{S}). \]
and also 
\[ \nabla^{\amalg} Q:  \mathcal{SET}^{\mathcal{S}^{\op} \times \Delta^{\op}} \rightarrow \Cat^{\op}(\mathcal{S}^{\op}). \]
For the notation $\int^{\amalg}$ and $\nabla^{\amalg}$ see \ref{PARINTNABLA}.
For a diagram $D=(I, S) \in \Cat(\mathcal{S})$ and a fibered derivator $\DD^! \rightarrow \SSS$ we use the notation $\DD^!(D):= \DD^!(I)_S$ for the fiber above $S$ of the specialization $\DD^!(I) \rightarrow \SSS(I)=\mathcal{S}^I$.
\end{PAR}

\begin{SATZ}[Homological descent]\label{SATZHODESC}
Let $\DD^! \rightarrow \SSS$ be an infinite local fibered derivator with stable, well-generated fibers with domain $\Cat$.
If $f: X \rightarrow Y$ is a \v{C}ech weak equivalence of simplicial pre-sheaves in $\mathcal{SET}^{\mathcal{S}^{\op} \times \Delta^{\op}}$ then 
\[ f^!: \DD^!(\int^{\amalg} Q Y)^{\cart} \rightarrow \DD^!(\int^{\amalg} Q X)^{\cart}  \]
is an equivalence.  
\end{SATZ}
\begin{proof}
In \cite[Main Theorem 3.5.4]{Hor15} it is shown that the strong $\DD^!$-equivalences, i.e.\@ the class of morphisms $\alpha: D_1 \rightarrow D_2$ in $\Cat(\mathcal{S})$  such that 
\[ \alpha^!: \DD^!(D_2)^{\cart} \rightarrow \DD^!(D_1)^{\cart} \]
 is an equivalence, form a localizer. In \cite[Theorem 6.9]{Hor21c} it is shown that the smallest such localizer
is the class of morphisms in $\Cat(\mathcal{S})$ such that the nerve is a \v{C}ech weak equivalence. Furthermore, for $Y \in \mathcal{M}$ the simplicial object $N( \int^{\amalg} Q Y)$ is (even globally) weakly equivalent to $Y$ (cf.\@ \cite[Theorem 6.4]{Hor21c}).
\end{proof}

\begin{PAR}\label{PARCOLOCAL}
There is a dual version: 
A fibered derivator with domain $\Dia$
\[ \DD^* \rightarrow \SSS^{\op}  \]
is called {\bf colocal} for the pretopology on $\mathcal{S}$ if for each covering $\{f_i: U_i \rightarrow X\}$, we have that (writing, as usual, $f^*, f_*$ for $(f^{\op})_\bullet, (f^{\op})^\bullet$):
\begin{enumerate}
\item[(C1'')] For all $I \in \Dia$ and for each morphism $f: S \rightarrow T$ in $\mathcal{S}^{I^{\op}}$ such that $S_i \rightarrow T_i$ is part of a covering for all $i$, 
the morphism of derivators $f^*: \DD^*_{(I,T^{\op})} \rightarrow \DD^*_{(I,S^{\op})}$ has a left adjoint.
\item[(C2)] For each covering $\{f_i: U_i \rightarrow X\}$, and for each $i$, $f_i^*$ satisfies base change, i.e.\@ for each Cartesian square in $\mathcal{S}$
\[ \xymatrix{  \ar[r]^{F_i} \ar[d]_G & \ar[d]^g \\ 
 \ar[r]_{f_i} & 
} \] 
 the natural exchange $f_i^* g_* \rightarrow G_* F_i^*$ is an isomorphism. 
\item[(C3)] For each covering $\{f_i: U_i \rightarrow X\}$, the $f_i^*$ are jointly conservative. 
\end{enumerate}
Also consider the following variants of (C1). For $(C1')$ assume that $\DD^* \rightarrow \SSS^{\op}$ has stable fibers. 
\begin{enumerate}
\item[(C1)] For each covering $\{f_i: U_i \rightarrow X\}$, and for each $i$, $f_i^*$, as morphism of derivators, commutes with homotopy limits. 
\item[(C1')] For each covering $\{f_i: U_i \rightarrow X\}$, and for each $i$, $f_i^*$ is exact and commutes with infinite (homotopy) products. 
\end{enumerate}
The analogous statements about the relation of (C1), (C1'), and (C1'') hold true except that for the implication 
$(C1) \Rightarrow (C1'')$, the fibered derivator $\DD^* \rightarrow \SSS^{\op}$ has to have {\em compactly generated} fibers (using, this time,  Brown representability for the dual \cite[Theorem~4.2.2]{Hor15}). However, we will never assume compactly generated fibers.
In \cite[Definition~2.5.6]{Hor15} ``colocal'' has been defined using axiom (C1) instead of (C1''). However, the above definition has the advantage that the theorem of cohomological descent is valid without assuming compactly generated fibers:
\end{PAR}

\begin{SATZ}[Cohomological descent]\label{SATZCOHODESC}
Let $\DD^* \rightarrow \SSS^{\op}$ be an infinite colocal fibered derivator with stable, well-generated generated fibers  with domain $\Cat$.
If $f: X \rightarrow Y$ is a \v{C}ech weak equivalence of simplicial pre-sheaves in $\mathcal{SET}^{\mathcal{S}^{\op} \times \Delta^{\op}}$ then 
\[ f^*: \DD^*(\nabla^{\amalg} Q Y)^{\cocart} \rightarrow \DD^*(\nabla^{\amalg} Q X)^{\cocart}  \]
is an equivalence. 
\end{SATZ}
\begin{proof}
In \cite[Main Theorem 3.5.5]{Hor15} it is shown that the class of morphisms $\alpha: D_1 \rightarrow D_2$ in $\Cat(\mathcal{S})$  such that 
\[ (\alpha^{\op})^*: \DD^*(D_2^{\op})^{\cocart} \rightarrow \DD^*(D_1^{\op})^{\cocart} \]
 is an equivalence (i.e.\@ the strong $\DD^*$-equivalences) is a localizer. Then proceed as in the proof of Theorem~\ref{SATZHODESC}.
 In \cite[Main Theorem 3.5.5]{Hor15} compactly generated fibers are assumed. Observe that the only place where compactly generated fibers are used in the proof of \cite[Main Theorem 3.5.5]{Hor15} is in \cite[Lemma 3.5.10]{Hor15}.
 If (C1'') holds for the morphism $U \rightarrow S$ in the statement of the Lemma --- as we do in the definition of colocal in this article --- no assumptions are needed. The Lemma is applied in the proof of \cite[Main Theorem 3.5.5]{Hor15} only for morphisms  that are (point-wise) part of a covering. 
\end{proof}

\begin{SATZ}\label{SATZEXTNAIVELEFT}
Let $\DD^! \rightarrow \SSS^{\op}$  be an infinite fibered derivator with stable, well-generated fibers  with domain $\Cat$.
Then $\DD^!$  has a {\bf naive extension}, a fibered derivator
\[ \DD^{!,h} \to \mathbb{M} \]
with domain $\Cat$, with stable, well-generated fibers, such that for $S: I \rightarrow \mathcal{M}$ 
\[ \DD^{!,h}(I)_S = \DD^!(\int_I \int^{\amalg} QS)^{\pi-\cart} \]
where 
\[ \pi: \int_{I \times \Delta^{\op}} (QS)_0 \rightarrow I  \]
is the projection\footnote{and $X_0$ for an object in $X: I \rightarrow \mathcal{S}^{\amalg, \Delta^{\op}}$ denotes the underlying $I \times \Delta^{\op} \rightarrow \mathcal{SET}$}. 
\end{SATZ}
\begin{proof}
This is \cite[Proposition 8.10]{Hor21c}. There the extension is constructed on $\mathbb{CAT}(\mathcal{S})$, the associated pre-derivator of the category $\Cat(\mathcal{S})$ (neglecting 2-morphisms). The pull-back along 
\[ \int^{\amalg} Q: \mathbb{M} \rightarrow \mathbb{CAT}(\mathcal{S})  \] 
yields the naive extension $\DD^{!,h}$ in the Theorem. It has well-generated fibers by the argument in the proof of \cite[Theorem~8.1]{Hor21c}, i.e.\@ by \cite[Lemma~8.12]{Hor21c}.
\end{proof}
There is also a dual variant, which also includes a monoidal structure if desired.  For this note that $\SSS^{\op}$ is a symmetric pre-multiderivator represented by $\mathcal{S}^{\op}$, considered as a symmetric multicategory setting
\[ \Hom_{\mathcal{S}^{\op}}(X_1, \dots, X_n; Y) := \Hom(Y, X_1) \times \cdots \times \Hom(Y, X_n)  \]
with the obvious action of the symmetric groups. 
\begin{SATZ}\label{SATZEXTNAIVERIGHT}
Let  $\DD^* \rightarrow \SSS^{\op}$ be an infinite fibered (symmetric) (multi)derivator with stable, well-generated fibers  with domain $\Cat$.
Then $\DD^*$  has a {\bf naive extension}, a fibered (symmetric) (multi)derivator 
\[ \DD^{*,h} \to \mathbb{M}^{\op} \]
with domain $\Cat$, with stable, well-generated fibers, such that for $S: I \rightarrow \mathcal{M}^{\op}$ 
\[  \DD^{*,h}(I)_S = \DD^*(\nabla^{\amalg} Q S^{\op})^{\pi-\cocart} \]
where 
\[  \pi:  \nabla (Q S^{\op})_0 \rightarrow I    \]
is the projection\footnote{and $X_0$ denotes the underlying $I^{\op} \times \Delta^{\op} \rightarrow \mathcal{SET}$}. 
\end{SATZ}
\begin{proof}
This is dual to Proposition~\ref{SATZEXTNAIVELEFT} except for the multi-aspect. However, the same reasoning as in \cite[Proposition 8.10]{Hor21c} works. 
\end{proof}

Note that these Theorems do not use (co)homological descent --- only (co)Cartesian projectors are needed. Indeed, the adjective ``naive'' indicates that the homotopical structure  of $\mathbb{M}$, i.e.\@ the weak equivalences, has not yet been exploited. If the conditions for (co)homological descent are satisfied we have, however, the statement that if $w: S \rightarrow S'$ is a point-wise weak equivalence in $\mathcal{M}^I$ (resp.\@ in $\mathcal{M}^{I^{\op}}$). then 
\[ w^!:=w^\bullet: \DD^{!,h}(I)_{S'} \rightarrow \DD^{!,h}(I)_{S} \quad  \text{ resp. }\quad w^*:=(w^{\op})_\bullet: \DD^{*,h}(I)_{(S')^{\op}} \rightarrow \DD^{*,h}(I)_{S^{\op}} \]
are equivalences. This does {\em not} imply, however, that $\DD^{!,h}$ is a fibered derivator over the usual homotopy 1-pre-derivator $I \mapsto \mathcal{M}^I[\mathcal{W}^{-1}_I]$.
As in \cite[Proposition 8.11]{Hor21c}, one can instead construct a fibered derivator $\DD^{!,h} \rightarrow \HH^2(\mathcal{M})$, where $\HH^2(\mathcal{M})$ is the homotopy 2-pre-derivator associated with $\mathbb{M}$. We will not use this because a similar extension can also be extracted from an extended $*,!$-formalism.

(Co)homological descent implies that the restriction of $\DD^{!,h}$ along $\SSS \rightarrow \mathbb{M}$ is equivalent to $\DD^!$
and that the restriction of $\DD^{*,h}$ along $\SSS^{\op} \rightarrow \mathbb{M}^{\op}$ is equivalent to $\DD^*$.
 
More generally, homological descent takes the following neat form: 
\begin{PROP}[Homological descent in the extension] \label{PROPHODESC} With the assumptions of Theorem~\ref{SATZHODESC},
the maps $\alpha: D_1 \rightarrow D_2$ in $\Dia(\mathcal{M})$ 
which induce an isomorphism
\[ \hocolim D_1 \rightarrow \hocolim D_2 \]
in the homotopy category of $\mathcal{M}$ are strong $\DD^{!,h}$-equivalences, i.e.\@
\[ \alpha^!: \DD^{!,h}(D_2)^{\cart} \rightarrow \DD^{!,h}(D_1)^{\cart} \]
is an equivalence. 
\end{PROP}
\begin{proof}
By definition we have
\[ \DD^{!,h}(I)_S^{\cart} = \DD( \int_I \int^\amalg QS )^{\cart}.  \]
Let $\mathcal{W}_{\infty} \subset \Mor(\Cat(\mathcal{S}))$ be the smallest localizer w.r.t.\@ the Grothendieck pre-topology on $\mathcal{S}$, and let $\mathcal{W}_{\loc} \subset \Mor(\mathcal{SET}^{\mathcal{S}^{\op} \times \Delta^{\op}})$
be the class of \v{C}ech weak equivalences. By \cite[Theorem 6.9]{Hor21c} the functor
\[ \int^{\amalg} Q: (\mathcal{SET}^{\mathcal{S}^{\op} \times \Delta^{\op}}, \mathcal{W}_{\loc}) \rightarrow  (\Dia(\mathcal{S}), \mathcal{W}_{\infty})   \]
is an equivalence of categories with weak equivalences with quasi-inverse given by the nerve, 
and by e.g.\@ \cite[Proposition A.6]{Hor21c} the homotopy colimit for the pair $(\Dia(\mathcal{S}), \mathcal{W}_{\infty})$ is given by the Grothendieck construction. 
For a morphism $\alpha: (I, S) \rightarrow (J,T)$ in $\Dia(\mathcal{M})$ with the property that  
\[ \hocolim_I S \rightarrow \hocolim_J T \]
is a weak equivalence, it implies that 
\[ \int_I \int^{\amalg} Q S  \rightarrow \int_J \int^{\amalg} Q T  \]
is in $\mathcal{W}_{\infty}$. 
By homological descent, Theorem~\ref{SATZHODESC}, it being in $\mathcal{W}_{\infty}$, implies that 
\[ \alpha^!: \DD^!(\int_J \int^{\amalg} Q T)^{\cart} \rightarrow \DD^!(\int_I \int^{\amalg} Q S)^{\cart} \]
is an equivalence. 
\end{proof}
There is a corresponding dual statement for $\DD^{*}$:
\begin{PROP}[Cohomological descent in the extension]\label{PROPCOHODESC}With the assumptions of Theorem~\ref{SATZCOHODESC},
the maps $\alpha: D_1 \rightarrow D_2$ in $\Dia^{\op}(\mathcal{M}^{\op})$ which induce an isomorphism
\[ \hocolim D_1^{\op} \rightarrow \hocolim D_2^{\op} \]
in the homotopy category of $\mathcal{M}$ are strong $\DD^{*,h}$-equivalences, i.e.\@
\[ \alpha^*: \DD^{*,h}(D_2)^{\cocart} \rightarrow \DD^{*,h}(D_1)^{\cocart} \]
is an equivalence. 
\end{PROP}

\begin{PAR}\label{PARHOMOTOPYCAT}
It also follows that if $f = g$ in the homotopy category, there is a (non-canonical) isomorphism $f^! \cong g^!$ (resp.\@ $f^* \cong g^*$).
\end{PAR}

We will later need the following 
\begin{LEMMA}\label{LEMMAEXACTCOLIMS}
In the category with weak equivalences $(\mathcal{SET}^{\mathcal{S}^{\op} \times \Delta^{\op}}, \mathcal{W}_{\loc})$ homotopy colimits commute with homotopy fiber products. 
\end{LEMMA}
\begin{proof}
This is clear for simplicial pre-sheaves looking, for instance, at the Bousfield-Kan formula for homotopy colimits. But the left Bousfield localization at the \v{C}ech weak equivalences preserves homotopy colimits and is exact (i.e.\@ preserves homotopy fiber products) hence the property holds also
true for the localization. More generally, this property holds true in any $\infty$-topos. 
\end{proof}

We can, in particular, apply this to the restrictions $\DD^!$ to $\SSS$, and $\DD^*$ to $\SSS^{\op}$, of a (symmetric) six-functor-formalism (or just of a $*,!$-formalism) $\DD \rightarrow \SSS^{\cor}$. If $\DD \rightarrow \SSS^{\cor}$ is a (symmetric) six-functor-formalism, $\DD^*$ is even a (symmetric) fibered {\em multi}\,derivator. To construct these extensions the interplay of the $*$ and $!$ functors did not play any role. 
The objective is to extend {\em the whole formalism} $\DD \rightarrow \SSS^{\cor}$  to $\HH^{\cor}(\mathcal{M})$ such that the pullbacks along (cf.\@ \ref{EMBEDDINGM})
\[ \mathbb{M} \rightarrow \HH^{\cor}(\mathcal{M}) \qquad  \mathbb{M}^{\op} \rightarrow \HH^{\cor}(\mathcal{M}) \]
are equivalent to the naive extensions. This would imply the following:
\begin{enumerate}
\item There is a canonical equivalence
\[ \DD^!(\cdot)_S \cong \DD^*(\cdot)_S \]
for $S \in \mathcal{M}$, in other words the fibers of the extension constructed by means of the $!$-functors are canonically equivalent to the fibers of the extension constructed by means of the $*$-functors. 
\item The formulas of a six-functor-formalism  (cf.\@ Proposition~\ref{PROPCONSEQUENCES}) continue to hold for the extension to $\mathcal{M}$ (with base change for {\em homotopy} Cartesian squares). 
\item We can make sense of categories of {\em coherent} diagrams involving the $*,!$-functors also for any diagram of correspondences in $\mathcal{M}$.
\end{enumerate}

It is probably not reasonable to expect that such an extension to the whole $\SSS^{\hcor}$ (i.e.\@ to {\em arbitrary} fibrant simplicial pre-sheaves) exists.
The objective of this article is to construct such an extension to {\em higher geometric  stacks} in the sense of To\"en-Vezzosi \cite{TV08}. 
For this we need to assume that the six-functor-formalism (or just $*,!$-formalism) $\DD \rightarrow \SSS^{\cor}$ is {\em local} w.r.t.\@ the Grothendieck pre-topology in the following sense: 

\begin{DEF}\label{DEFLOCAL6FU}
A fibered (multi)derivator
\[ \DD \rightarrow \SSS^{\cor} \]
is called {\bf local} w.r.t.\@ the Grothendieck pretopology, if $\DD^!$ is local (\ref{PARLOCAL}), i.e.\@ if (H1''), (H2) and (H3) hold, $\DD^*$ is colocal (\ref{PARCOLOCAL}), i.e.\@ if (C1''), (C2) and (C3) hold, and the following axiom holds true:
\begin{enumerate}
\item[(CH)]  For each Cartesian square
\[ \xymatrix{ X \ar[r]^{F} \ar[d]_G & X' \ar[d]^g \\ Y_1, \dots, Y_k \ar[r]_{(f_i)} & Y'_1, \dots, Y_k'  }\]
in which the $f_i: Y_i \rightarrow Y_i'$ are part of a covering of $Y_i'$ for all $i$, 
the exchange $G^* (f_1^!-, \dots, f_k^!-) \rightarrow F^! g^*(-, \dots, -)$ and for $k=1$ the exchange $F_i^*g^! \rightarrow G^! f_i^*$  (of the respective base-change)   are isomorphisms.
\end{enumerate}
If $\DD \rightarrow \SSS^{\cor}$ is just a fibered derivator (not multiderivator, i.e.\@ a $*,!$-formalism) then axiom (CH) is required for $k=1$ only. 
\end{DEF}

\section{A simple criterion for locality}

Let $\mathcal{S}$ be a small category with finite limits and Grothendieck pre-topology. 
Recall Definition~\ref{DEFLOCAL6FU} of {\em locality} for a symmetric derivator six-functor-formalism $\DD \rightarrow \SSS^{\cor}$ w.r.t.\@ the Grothendieck pre-topology. 
In this section, we will define a simpler but stronger notion, which involves the tensor product and is actually the one
that is usually checked in examples. We abstain from trying to state a similar definition for non-symmetric derivator six-functor-formalisms.

\begin{DEF}\label{DEFSTRONGLYLOCAL}
A  symmetric derivator six-functor-formalism $\DD \rightarrow \SSS^{\cor}$ is called {\bf strongly local} w.r.t.\@ the Grothendieck pre-topology on $\mathcal{S}$, 
if for each covering $\{ f_i: X_i \rightarrow S \}$ in $\mathcal{S}$, 
 \begin{enumerate}
\item[(O1)] $f_i^! 1_S$ is tensor-invertible\footnote{i.e.\@ the functor $\mathcal{E} \mapsto \mathcal{E} \otimes (f_i^! 1_S)$ is an equivalence.} for all $i$;
\item[(O2)] the natural morphism (exchange of the projection formula)
\[ f_i^!1_S \otimes (f_i^* -) \rightarrow  f_i^! -  \]
is an isomorphism for all $i$;
 \item[(O3)] The $f_i^*$ or, equivalently\footnote{assuming (O1) and (O2)}, the $f^!_i$ are jointly conservative; 
 \item[(O4)] for each $i$ and a Cartesian diagram with $g$ arbitrary
 \[ \xymatrix{
X' \ar[r]^{F_i} \ar[d]_G & S' \ar[d]^g \\
X \ar[r]_{f_i} & S
} \]
 the natural morphism (exchange of the base-change) 
\[  G^*f_i^! 1_S \rightarrow F_i^!g^* 1_S = F_i^!1_{S'}  \]
is an isomorphism. 
 \end{enumerate}
\end{DEF}
Note that all four conditions only concern the underlying classical six-functor-formalism $\DD(\cdot) \rightarrow \mathcal{S}^{\cor}$ and 
are thus independent of the derivator enhancement. They involve, however, in contrast to ``local'', the monoidal (i.e.\@ multi-categorical) structure. 

In the following, $\DD \rightarrow \SSS^{\cor}$ will thus denote a symmetric derivator six-functor-formalism.
The goal of this section (cf.\@ Proposition~\ref{PROPSTRONGLYLOCAL}) is to show that ``strongly local'' implies ``local''.
We begin to draw some immediate consequences of the axioms (O1--O4).

\begin{LEMMA}\label{LEMMAO24}
The combination of (O2) and (O4) is equivalent to the following axiom (the first half of axiom (CH)):
\begin{enumerate}
\item[(CH left)]  For each Cartesian square
\[ \xymatrix{ X \ar[r]^{F} \ar[d]_G & X' \ar[d]^g \\ Y_1, \dots, Y_k \ar[r]_{(f_i)} & Y'_1, \dots, Y_k'  }\]
in which the $f_i: Y_i \rightarrow Y_i'$ are part of a covering of $Y_i'$ for all $i$, 
the exchange 
\[ G^* (f_1^!-, \dots, f_k^!-) \rightarrow F^! g^*(-, \dots, -) \] 
is an isomorphism.
\end{enumerate}
\end{LEMMA}
\begin{proof}
First we assume (O2) and (O4) and proceed to show (CH left). It suffices to show this for $k=1$ and to show that the exchange  
\begin{equation}\label{eqboxtimesfshriek} - \boxtimes f^! - \rightarrow (\id \times f)^! (- \boxtimes -) \end{equation}
is an isomorphism for morphisms $f: Y_2 \rightarrow Y_2'$ which are part of a covering, which is the statement  (CH left) for 
\[ \xymatrix{ Y_1 \times Y_2 \ar[r]^{\id \times f} \ar[d] & Y_1 \times Y_2'   \ar[d] \\ Y_1, Y_2 \ar[r]_{(\id, f)} & Y_1, Y_2'  }\]

For $k=1$ consider a Cartesian diagram as in the statement. From (O2) and (O4), we get a commutative diagram (we leave the check of commutativity to the reader)
\[ \xymatrix{
(G^* f^* -) \otimes F^! 1_X  \ar[d]^-{p.fun.}_-\sim \ar[r]^-\sim &  (G^* f^* -) \otimes (G^* f^! 1)  \ar[r]^-\sim & G^* ((f^* -) \otimes f^! 1) \ar[r]^-\sim &  G^* f^! -\ar[d]^{\text{exch.}}   \\
(F^* g^* - ) \otimes F^! 1_X \ar[rrr]^-\sim  & & &  F^! g^* -
} \]
in which all morphisms except possibly the right vertical one are isomorphisms. Therefore also the right vertical one is an isomorphism. 
Furthermore, we get a commutative diagram (we leave the check of commutativity to the reader)
\[ \tiny \xymatrix{
\pr_1^* -  \otimes (\pr_2^* -  \otimes   (\id,f)^! 1)  \ar[d]^-{ass.}_-\sim \ar[r]^-\sim &   \pr_1^* -   \otimes (( (\id,f)^*  \pr_2^* -)  \otimes   (\id,f)^! 1) \ar[r]^-\sim & \pr_1^* -  \otimes  ((\id,f)^! \pr_2^* - ) \ar[r]^-\sim   &     \pr_1^*  -  \otimes \pr_2^*f^! - \ar[d]^{\text{exch.}}   \\
(\pr_1^* - \otimes   \pr_2^* - ) \otimes  (\id,f)^! 1  \ar[rrr]^-\sim  & & &  (\id, f)^! (\pr_1^* - \otimes \pr_2^* - ) 
} \]
in which all morphisms except possibly the right vertical one are isomorphisms. Therefore also the right vertical one (which is \ref{eqboxtimesfshriek}) is an isomorphism.

Conversely, (O2) is (CH left) for the Cartesian diagram
\[ \xymatrix{ X  \ar[r]^{f} \ar[d]_{(f, \id)} & Y   \ar[d]^{(\id, \id)} \\ Y,X \ar[r]_{(\id,f)} & Y, Y  }\]
specialized to $(1_Y, \mathcal{E}) \in \DD(\cdot)_Y \times \DD(\cdot)_Y$ and (O4) is (CH left) for the Cartesian diagrams 
 \[ \xymatrix{
X' \ar[r]^{F_i} \ar[d]_G & S' \ar[d]^g \\
X \ar[r]_{f_i} & S
} \]
specialized to $1_S \in \DD(\cdot)_S$.
\end{proof}

\begin{LEMMA}\label{LEMMATENSORINV}
Let $\mathcal{E}$, $\mathcal{F}$ be objects in $\DD(\cdot)_S$. If either
$\mathcal{E}$ or $\mathcal{F}$ is tensor invertible then 
 the  exchange morphism
\[ \mathcal{HOM}(\mathcal{F}; -) \otimes \mathcal{E} \cong \mathcal{HOM}(\mathcal{F}; - \otimes \mathcal{E}).   \]
is an isomorphism.
\end{LEMMA}
\begin{proof}
The exchange in question is the exchange of the following 2-commutative diagram
\[ \xymatrix{ \DD(\cdot)_X \ar[r]^{\mathcal{F} \otimes -}  \ar[d]_{- \otimes \mathcal{E}} & \DD(\cdot)_X  \ar[d]^{- \otimes \mathcal{E}} \\
\DD(\cdot)_X \ar[r]^{\mathcal{F} \otimes -} & \DD(\cdot)_X    } \]
in which either the vertical or horizontal functors are equivalences by assumption. 
\end{proof}

\begin{LEMMA}If (O1--O2) holds, and $f$ is part of a covering then 
 the natural morphism
\[ f^* \mathcal{HOM}(-; -) \rightarrow \mathcal{HOM}(f^* -; f^*-)   \]
is an isomorphism. 
In particular, if (O3) holds in addition and $\{ f_i: U_i \rightarrow X\}$ is a covering, an object $\mathcal{E}$ is tensor invertible if and only if $f_i^*\mathcal{E}$ are tensor invertible for all $i$. 
\end{LEMMA}
\begin{proof}
We have
\[ \begin{array}{rcll}
 \mathcal{HOM}(f^* -; f^!1 \otimes f^* -) &\cong& \mathcal{HOM}(f^* -; f^!-)  &  \text{using (O2)} \\
 &\cong& f^! \mathcal{HOM}(-; -) &  \text{adjoint of the projection formula} \\ 
 & \cong & (f^! 1) \otimes (f^* \mathcal{HOM}(-; -))  &  \text{using (O2)} 
\end{array} \]
One checks that this composition is the exchange considered in Lemma~\ref{LEMMATENSORINV}. Thus
the first statement follows from that Lemma using (O1).
$X$ is tensor invertible, if and only if the counit 
\[ X \otimes \mathcal{HOM}(X; 1) \rightarrow 1\]
is an isomorphism. Taking $f^*_i$ of this equation yields (using the previous part)
\[ f^*_iX \otimes \mathcal{HOM}(f_i^*X; 1) \rightarrow 1. \]
The statement follows thus from the joint conservativity of the $f^*_i$, axiom (O3).
\end{proof}

\begin{LEMMA}\label{LEMMAO2DUAL}
If (O1--O2) hold, and $f$ is part of a covering, the natural morphism
\[ f^* \to \mathcal{HOM}(f^! 1; f^! -)  \]
is an isomorphism. 
\end{LEMMA}
\begin{proof}
The morphism in question is the composition
\[ f^* X \to  \mathcal{HOM}(f^! 1; f^! 1 \otimes f^* X) \to  \mathcal{HOM}(f^! 1; f^! X)   \]
in which the two morphisms are isomorphisms by (O1), and  (O2), respectively. 
\end{proof}
\begin{BEM} The morphism may also be described as the composition
\[ f^*- \cong f^*\mathcal{HOM}(1;  -)  \to f^* \mathcal{HOM}(f_!f^! 1;  -) \cong f^* f_* \mathcal{HOM}(f^! 1;  f^! -)  \to \mathcal{HOM}(f^! 1;  f^! -) \]
involving the counits of the $f^*,f_*$ and the $f_!,f^!$-adjunction. 
\end{BEM}

\begin{LEMMA}\label{LEMMAADJ}
If (O1--O2) holds, and $f$ is part of a covering, the functor 
$f^!$ has the right adjoint
\[ f_? = f_*  \mathcal{HOM}(f^! 1;  -)  \]
and the functor $f^*$ has the left adjoint
\[ f_\# = f_!(f^! 1 \otimes -).  \]
\end{LEMMA}
\begin{proof}
This follows directly from (O2), and from the statement of Lemma~\ref{LEMMAO2DUAL}, respectively.
\end{proof}

\begin{PAR}
There is a diagrammatic version of Lemma~\ref{LEMMAADJ} as well, which is more complicated because  --- except in very special situations --- $f_\#$ will not be computed point-wise, i.e.\@
for a morphism $f: S \rightarrow T$ in $\mathcal{S}^{I^{\op}}$ such that the functor $f^*$ has a left adjoint $f_\#$, and for $i \in I$ the functor $f_i^*$ has a left adjoint $f_{i,\#}$, the exchange
\[  f_{i,\#} i^* \rightarrow  i^* f_\# \]
 will in general not be an isomorphism. We will give a construction which assumes also (O4). 
Let $I$ be a small category and $f: S \rightarrow T$ be a morphism in $\mathcal{S}^{I^{\op}}$.
We denote by $f^{\op}, S^{\op}, T^{\op}$ also their images under $\SSS^{\op} \rightarrow \SSS^{\cor}$. 
We proceed to define two multicorrespondences $F: (I, T^{\op}), ({}^{\uparrow \downarrow} I, \Box) \rightarrow (I, S^{\op})$ and
$ \xi:  ({}^{\uparrow \downarrow} I, \Box) \rightarrow (I, S^{\op}) $
as follows (using the notation \ref{PARNOTATIONDERCORHCOR}):
\end{PAR}

We define an object  $\Box \in \mathcal{S}^{\cor}({}^{\uparrow \downarrow} I)$  by the functor $\tw  ({}^{\uparrow \downarrow} I) \rightarrow \mathcal{S}$ given by
\[ \Box \left( \vcenter{ \xymatrix{ i_2 \ar[r] \ar@{<-}[d]  &  i_3 \ar[d]  \\
 i_2' \ar[r] & i_3' } } \right) := T(i_3') \times_{T(i_2)} S(i_2) \]
 (note the different indexing)
 which one checks to be admissible, 
and  define a correspondence 
 \[ \xi:= \left(\vcenter{\xymatrix{ & ({}^{\uparrow \downarrow} I, \Box) \ar@{=}[ld] \ar[rd]^{} \\
 ({}^{\uparrow \downarrow} I, \Box) & & ({}^{\uparrow \downarrow} I, \pi_3^*T^{\op})   }} \right) \]
 in $\Dia^{\cor}(\SSS^{\cor})$ given point-wise by a correspondence completely of $!$-type:
 \[  \Box(i_2 \rightarrow i_3') =  \Box(i_2 \rightarrow i_3')  \rightarrow  T(i_3').    \]
By definition of $\Box$ the morphism to $\pi_3^* T^{\op}$ is type-2-admissible. 

We define a multicorrespondence 
\[ \Xi:= \left( \vcenter{ \xymatrix{ & & ({}^{\downarrow \uparrow \downarrow \downarrow} I, \pi_{24}^*\Box) \ar[lld]_{\pi_1} \ar[ld]^{\pi_{23}} \ar[rd]^{\pi_4} \\
(I,S^{\op}) & ({}^{\uparrow \downarrow} I, \Box) & ; & (I,T^{\op})
} } \right)\]
in $\Dia^{\cor}(\SSS^{\cor})$
 mapping 
\[ \xymatrix{ i_1 \ar[r] \ar[d]  & i_2 \ar[r] \ar@{<-}[d]  & i_3 \ar[r] \ar[d]  &  i_4 \ar[d]  \\
 i_1' \ar[r] &i_2' \ar[r] & i_3' \ar[r] &i_4'  } \]
 to
 \[ S(i_1') \times \Box(i_2 \rightarrow i_3') \leftarrow \Box(i_2 \rightarrow i_4') \rightarrow T(i_4') .   \]
 One checks that the morphism to the left is type-1-admissible and the one to the right type-2-admissible. 
 In the description of Proposition~\ref{PROPALTCOR} the multicorrespondence $\Xi$ (on the level of diagrams) is given by the functor
 \begin{eqnarray*} 
 F: I^{\op} \times ({}^{\uparrow \downarrow} I)^{\op} \times I &\rightarrow& \mathcal{SET} \\
  i_1, (i_2 \to i_3), i_4 &\mapsto& \Hom(i_1, i_2) \times \Hom(i_3, i_4) 
\end{eqnarray*}
 because $\int \nabla F = {}^{\downarrow \uparrow \downarrow \downarrow} I$.
 
\begin{LEMMA}\label{LEMMAADJDIA}
We have an isomorphism
\[ (f^* -)  \cong \Xi^{\bullet,1}(\xi^\bullet 1; -). \]
In particular, $f^*$ has the left adjoint
\[ \Xi_{\bullet}(-, \xi^\bullet 1). \]
\end{LEMMA}
There is a corresponding dual construction for the $!$-functor which we leave to the reader. It will not be needed in this article because for the construction of
right adjoints we can and will always  use the theory of well-generated triangulated categories. 
\begin{proof}[Proof (Sketch): ]
$f^*$ may be described as $\varphi_\bullet$ for the obvious correspondence
\[ \varphi:= \left( \vcenter{ \xymatrix{ & (I, S^{\op})  \ar[ld] \ar@{=}[rd] \\
(I,T^{\op}) & ; & (I,S^{\op})
} } \right)\]
in $\Dia^{\cor}(\SSS^{\cor})$. 

We proceed to describe a morphism $c: \Xi_\bullet(\varphi_\bullet -, \xi^\bullet 1) \rightarrow \id$ which, by adjunction, yields a morphism 
\[ \varphi_\bullet \rightarrow \Xi^{\bullet,1}(\xi^\bullet 1; -) \]
which will be shown to be an isomorphism. 
One calculates that 
\[ \Xi_\bullet(\varphi_\bullet -,  -) \cong \Pi_\bullet( -, \xi_\bullet -) \]
for
\[ \Pi:= \left( \vcenter{ \xymatrix{ & & ({}^{\downarrow \uparrow \downarrow \downarrow} I, \pi_{4}^*T^{\op}) \ar[lld]_{\pi_1} \ar[ld]^{\pi_{23}} \ar[rd]^{\pi_4} \\
(I,T^{\op}) & ({}^{\uparrow \downarrow} I, \pi_3^* T^{\op}) & ; & (I,T^{\op})
} } \right).\]
Note that
\[\Pi(-, 1_{({}^{\uparrow \downarrow} I, \pi_3^* T^{\op})}) = \left( \vcenter{ \xymatrix{ &  ({}^{\downarrow \uparrow \downarrow \downarrow} I, \pi_{4}^*T^{\op}) \ar[ld]_{\pi_1} \ar[rd]^{\pi_4} \\
 (I, T^{\op}) & ; & (I,T^{\op})
} } \right)\]
and one checks that this correspondence is isomorphic to $\id_{(I,T^{\op})}$ in the localization 
\[ \Cor((I,T^{\op}); (I, T^{\op}))[\mathcal{W}^{-1}] \]
(i.e.\@ in in $\widehat{\Dia}^{\cor}(\SSS^{\cor})$), in particular $\Pi_\bullet(-,  1) \cong \id$ for all derivator six-functor-formalisms. 
 This yields the required morphism: 
 \[  c: \Xi_\bullet(F_\bullet -, \xi^\bullet 1) \cong  \Pi_\bullet(-,  \xi_\bullet \xi^\bullet 1) \rightarrow \Pi_\bullet(-,  1) \cong \id. \]

The multicorrespondence $\Xi$ corresponds, by duality (cf.\@ Theorem~\ref{SATZDUALITY}), to the following correspondence 
\[ \Xi' = \left( \vcenter{ \xymatrix{ &  ({}^{\downarrow  \downarrow \uparrow \downarrow} I, \Box') \ar[ld]_{\pi_1} \ar[rd]^{\pi_{234}} \\
(I,S^{\op}) & ; & (\tw I \times I ,\Box^{\op} \times T^{\op})
}} \right) \]
(obviously $\int \nabla' F = {}^{\downarrow  \downarrow \uparrow \downarrow} I$ and the transition to $\Theta(F)$, cf.\@ proof of Theorem~\ref{SATZDUALITY}, is not necessary in this case) defined point-wise at

\[ \xymatrix{ i_1 \ar[r] \ar[d]  & i_2 \ar[r] \ar[d]  & i_3 \ar[r] \ar@{<-}[d]  &  i_4 \ar[d]  \\
 i_1' \ar[r] &i_2' \ar[r] & i_3' \ar[r] &i_4'  } \]
 by 
  \[ \xymatrix{ S(i_1') & \ar[l]_-{p_1} \Box(i_2' \rightarrow i_4') \ar[r]^-{p_2} & \Box(i_2' \rightarrow i_3)  \times T(i_4'). }   \]
 
 Similarly, the multicorrespondence $\Pi$ corresponds, by duality, to the following correspondence
\[ \Pi' = \left( \vcenter{ \xymatrix{ &  ({}^{\downarrow  \downarrow \uparrow \downarrow} I, \pi_4^* T^{\op}) \ar[ld]_{\pi_1} \ar[rd]^{\pi_{234}} \\
(I,T^{\op}) & ; & (\tw I \times I , \pi_2^*T \times T^{\op})
}} \right) \]
defined point-wise at
\[ \xymatrix{ i_1 \ar[r] \ar[d]  & i_2 \ar[r] \ar[d]  & i_3 \ar[r] \ar@{<-}[d]  &  i_4 \ar[d]  \\
 i_1' \ar[r] &i_2' \ar[r] & i_3' \ar[r] &i_4'  } \]
 by 
  \[ \xymatrix{ T(i_1') & \ar[l]_-{q_1} T(i_4') \ar[r]^-{q_2} & T(i_3)  \times T(i_4'). }   \]

$c$ induces the composition
\[ \Xi^{\bullet,1}(\xi^\bullet 1; -) \leftarrow \varphi_\bullet \varphi^\bullet \Xi^{\bullet, 1}(\xi^\bullet 1; -) \cong \varphi_\bullet \Pi^{\bullet, 1} (\xi_\bullet \xi^\bullet 1; -) \leftarrow \varphi_\bullet \Pi^{\bullet, 1}(1; -) \cong \varphi_\bullet    \]
which inserting the Cartesian morphism (cf.\@ Theorem~\ref{SATZDUALITY}) w.r.t.\@ the first slot
\[ \mathbf{CART}^{1}: (\tw I \times I, \Box^{\op} \times T^{\op}),  ({}^{\uparrow \downarrow} I, \Box )  \rightarrow  (I, T^{\op})  \]
 is the same as
\begin{gather}\label{theeq} (\Xi')^{\bullet} \mathbf{CART}^{1,\bullet} (\xi^\bullet 1; -) \leftarrow \varphi_\bullet \varphi^\bullet (\Xi')^{\bullet}\mathbf{CART}^{1,\bullet}(\xi^\bullet 1; -) \\
\cong \varphi_\bullet (\Pi')^\bullet \mathbf{CART}^{1,\bullet} (\xi_\bullet \xi^\bullet 1; -) \leftarrow \varphi_\bullet (\Pi')^{\bullet} \mathbf{CART}^{1,\bullet} (1; -) \cong \varphi_\bullet  \nonumber
\end{gather}

The advantage is that the pull-back along $\mathbf{CART}^{1}$ is the external $\mathbf{HOM}$ which is computed point-wise (cf.\@ \ref{PAREXTHOM}).
We see that (\ref{theeq}) is $\pi_{1,*}$ applied to 
\begin{gather}\label{theeq2}
p_{1,*} p_2^! \pi_{234}^* \mathbf{HOM}( \xi^!1; -) \leftarrow 
f^* f_* p_{1,*} p_2^! \pi_{234}^* \mathbf{HOM}( \xi^!1; -) \\
\nonumber
\leftarrow  f^* q_{1,*} q_2^! \pi_{234}^* \mathbf{HOM}( \xi_! \xi^! 1; -) 
\leftarrow   f^* q_{1,*} q_2^! \pi_{234}^* \mathbf{HOM}( 1; -) \cong   f^* q_{1,*} \pi_{4}^*
\end{gather}

We will investigate this morphism at an object $(\xymatrix{i_1 \ar[r]^{\alpha} & i_2  \ar[r]^{\beta} & i_3  \ar[r]^{\gamma} & i_4})$ denoting
 \[ \xymatrix{  \Box(i_2 \rightarrow i_4) \ar[r]^{\xi_4} \ar[d]^{\rho}  \ar@/_30pt/[ddd]_{P} & T(i_4) \ar[d]^{T(\gamma)} \\
 \Box(i_2 \rightarrow i_3)   \ar[d]  \ar[r]^{\xi_3} & T(i_3) \ar[d]^{T(\beta)} \\
 S(i_2) \ar[d]  \ar[r] & T(i_2) \ar[d]^{T(\alpha)}   \\
 S(i_1) \ar[r]^{f_1} & T(i_1).     }
 \]
 Note that, in this diagram, the upper two squares are Cartesian. 
Consider the following diagram in which the lower horizontal composition is (\ref{theeq2}) evaluated at $(i_1 \rightarrow i_2  \rightarrow i_3  \rightarrow i_4)$. 
{ \footnotesize \[ \xymatrix{ 
P_* \xi_4^* T(\gamma)_* -  &  
f_1^* \underbrace{f_{1,*} P_*}_{T(\gamma \beta \alpha)_* \xi_{4,*} } \xi_4^*  - \ar[l]_-{\numcirc{1}}  -   
\\
P_* \mathcal{HOM}(  \xi_4^! 1; \xi_4^!  - ) \ar[u]^-{\numcirc{3}}_-\sim  &  
f_1^* \underbrace{f_{1,*} P_*}_{T(\gamma \beta \alpha)_* \xi_{4,*} } \mathcal{HOM}( \xi_4^! 1; \xi_4^!  -) \ar[l]_-{\numcirc{1}}   \ar[u]^-{\numcirc{3}}_-\sim  &
f_1^* T(\gamma \beta \alpha)_* \mathcal{HOM}( \xi_{4,!} \xi_4^! 1; -) \ar[l]_-\sim &
f_1^* T(\gamma \beta \alpha)_*  -  \ar[l]^-{\numcirc{5}} \ar[llu]_-{\numcirc{4}} \\
 P_* \mathcal{HOM}( \rho^* \xi_3^! 1; \xi_4^! - ) \ar[u]_-{\sim}^-{\numcirc{2}}  &  
f_1^* \underbrace{f_{1,*} P_*}_{T(\gamma \beta \alpha)_* \xi_{4,*} }  \mathcal{HOM}( \rho^* \xi_3^! 1; \xi_4^!  -) \ar[l]_-{\numcirc{1}}  \ar[u]_-{\sim}^-{\numcirc{2}}  &  
f_1^* T( \beta \alpha)_* \mathcal{HOM}( \xi_{3,!} \xi_3^! 1; T(\gamma)_* -) \ar[l]_-\sim &
f_1^* T(\gamma \beta \alpha)_*  -  \ar[l]^-{\numcirc{5}} \ar@{=}[u]
 } \] }
Here the arrows
\begin{itemize}
\item[$\numcirc{1}$] are induced by the counit $f_1^* f_{1,*} \to \id$;
\item[$\numcirc{2}$] are induced by the exchange $\rho^* \xi_3^! 1 \to  \xi_4^! T(\gamma)^* 1$ which is an isomorphism by (O4). Note that $\xi_3$ is also part of a covering as base-change of such;
\item[$\numcirc{3}$] are induced by the isomorphisms $\mathcal{HOM}(\xi_4^!1; \xi_4^!-) \cong \xi_4^*$ established in Lemma~\ref{LEMMAO2DUAL}. Note that $\xi_4$ is also part of a covering as base-change of such;
\item[$\numcirc{4}$] is induced by the unit $\id \to \xi_{4,*} \xi_4^*$;
\item[$\numcirc{5}$] are induced by the counits $\xi_{3,!} \xi_3^! \to \id$, and $\xi_{4,!} \xi_4^! \to \id$, respectively. 
\end{itemize}
 
A tedious check shows that the diagram is commutative. 
This shows that (up to isomorphism) the underlying diagram 
of {\em all entries} in the sequence (\ref{theeq2}) is independent of $i_3$, i.e.\@ the corresponding objects are Cartesian in the direction of $i_3$. Therefore we claim that the
limit over the fiber ${}^{ \downarrow \uparrow \downarrow }(i_1 \times_{/I} I)$ is isomorphic evaluation at $i_1=i_1=i_1=i_1$. This implies that (\ref{theeq}) is an isomorphism because at $i_1=i_1=i_1=i_1$
the top row is just the unit/counit equation for $f_{1}^*, f_{1,*}$ and thus the identity. 

Let us elaborate a bit on the claim. Consider the functor
\[ \xymatrix{ {}^{ \downarrow \uparrow \downarrow }(i_1 \times_{/I} I) \ar[r]^-{\pi_{2}} & i_1 \times_{/I} I }  \]
It is also a fibration. Hence we can first compute the limit over its fibers. Since $i_1 \times_{/I} I$ has an initial object, the limit (over the base) is the evaluation at $i_1=i_1$. 
This shows that $\pi_{1,*}$ is the homotopy limit over ${}^{ \uparrow \downarrow }(i_1 \times_{/I} I)$ (interpreted as $i_1=i_1 \to i_3 \to i_4$). Now 
\[ \xymatrix{ {}^{ \uparrow \downarrow }(i_1 \times_{/I} I) \ar[r]^-{\pi_{4}} & i_1 \times_{/I} I  } \]
is an opfibration whose fiber has an initial object. Since the objects in question are Cartesian in the direction of $i_3$,  
an application of Lemma~\ref{LEMMAKAN3} shows that $\pi_{4,*}$ is nevertheless point-wise in $i_4$ isomorphic to the evaluation at $i_1=i_1 \to i_4 = i_4$. 
The limit over $i_1 \times_{/I} I$ (4th entry) is again evaluation at $i_1=i_1$ because this is an initial object.
\end{proof}

\begin{PROP}\label{PROPSTRONGLYLOCAL}
Strongly local implies local. 
\end{PROP}
\begin{proof}
We first show the locality of the restriction $\DD^!$ of  $\DD$ to $\SSS$. 

(H3) follows immediately from (O3).

(H2) Let 
\[ \xymatrix{ \ar[r]^F \ar[d]_G & \ar[d]^g \\ 
 \ar[r]_f & } \]
 be a Cartesian diagram in $\mathcal{S}$ in which $f$ is part of a covering.  From (O1--O4) we get a commutative diagram (we leave the check of commutativity to the reader)
\[ \xymatrix{
(G_! F^* -) \otimes f^! 1_S \ar[d]^-{b.c.}_-\sim \ar[r]^-\sim &  G_! (F^* - \otimes G^* f^! 1_S) \ar[r]^-\sim &  G_! ((F^* -) \otimes F^! 1_{S'}) \ar[r]^-\sim &  G_! F^! -  \ar[d]^-{\text{exch.}}   \\
(f^* g_! - ) \otimes f^! 1_S \ar[rrr]^-\sim  & & &  f^! g_! -
} \]
in which all morphisms except possibly the right vertical one are isomorphisms. Therefore also the right vertical morphism is an isomorphism. 

(H1'') By a ``dual version'' of Lemma~\ref{LEMMAADJDIA} the functor $f^!$ has a right adjoint for $f: S \rightarrow T$ in $\mathcal{S}^{I}$. 
Actually, if $\DD \rightarrow \SSS^{\cor}$ has stable, well-generated fibers, (H1) and (H1'') are equivalent and thus Lemma~\ref{LEMMAADJ} is sufficient, because
the functor in the statement of the Lemma promotes clearly to a morphism of derivators for the fiber over $S$, resp.\@ $T$, and thus (H1) follows. 
 

Next, we show the colocality of the restriction $\DD^*$ of  $\DD$ to $\SSS^{\op}$. 

(C3) follows immediately from (O3).

(C2) Consider again a Cartesian diagram as in the proof of (H2) above. From (O1--O4), using Lemma~\ref{LEMMAO2DUAL}, we get a commutative diagram (we leave the check of commutativity to the reader)
\[ \xymatrix{
 f^* g_* -  \ar[d]^{exch.}  \ar[rrr]^-\sim  & & &  \mathcal{HOM}(f^! 1,  f^! g_* -)   \ar[d]^{\text{b.c.}}_-\sim  \\
G_* F^* - \ar[r]^-\sim &  G_* \mathcal{HOM}(F^! 1,  F^! -)  \ar[r]^-\sim &  G_* \mathcal{HOM}(G^*f^! 1 , F^!-)\ar[r]^-\sim     & \mathcal{HOM}(f^! 1 , G_*F^!-)  
} \]
in which all morphisms except possibly the left vertical one are isomorphisms. Therefore also the left vertical one is an isomorphism. 

(C1'') is Lemma~\ref{LEMMAADJDIA}.

It remains to show axiom (CH): The statement (CH left) was already shown in Lemma~\ref{LEMMAO24}. 
Using Lemma~\ref{LEMMAO2DUAL} we get a commutative diagram (we leave the check of commutativity to the reader)
\[ \xymatrix{
F^* g^! -  \ar[d]^{exch.} \ar[rrr]^-\sim & & &   \mathcal{HOM}(F^! 1, F^! g^!-   )  \ar[d]_-\sim^{\text{p.fun.}}   \\
G^! f^* - \ar[r]^-\sim  &  G^! \mathcal{HOM}(f^! 1,  f^! -) \ar[r]^-\sim &  \mathcal{HOM}(G^* f^! 1, G^! f^! - )  \ar[r]^-\sim &  \mathcal{HOM}(F^! 1,  G^! f^! -  ) 
} \]
in which all morphisms except possibly the left vertical one are isomorphisms. Therefore also the left vertical one is an isomorphism. 
\end{proof}

\section{Higher geometric stacks}\label{SECTHIGHERGEOM}

\begin{PAR}\label{PARSETTINGTV}
Let $\mathcal{S}$ be a small category with finite limits and Grothendieck pre-topology. 
Let $(\mathcal{M}, \Fib, \mathcal{W})$ with $\mathcal{M} \subset \mathcal{SET}^{\mathcal{S}^{\op} \times \Delta^{\op}}$ be any category of fibrant objects which models the \v{C}ech localization of simplicial pre-sheaves on $\mathcal{S}$ containing the constant ones in $\mathcal{S}^{\amalg}$ (coproducts of representables).
Roughly following To\"en-Vezzosi \cite[Definition 1.3.3.1]{TV08} we give the following recursive definition of {\bf $k$-geometric stacks} w.r.t.\@ the chosen pre-topology: 
Those are objects of homotopy strictly full (\ref{PARHSF}) subcategories $\mathcal{M}^{(k)} \subset \mathcal{M}$ such that the objects of $\mathcal{M}^{(0)}$ (i.e.\@ 0-geometric stacks) are those weakly equivalent to objects in $\mathcal{S}^{\amalg}$.  We need to make the technical assumption  that the site $\mathcal{S}$ is {\em geometric}, see Appendix~\ref{APPENDIXGEOMETRIC}.
\end{PAR}

\begin{DEF}\label{DEFC}
Define the class $C$  consisting of those
morphisms
\[ f: (X, F) \rightarrow (Y, G) \]
in $\mathcal{S}^{\amalg}$  such that  for any $y \in Y$ the family
\[ \{ F(x) \rightarrow G(y) \}_{x \in f^{-1}(y)} \]
is a covering. 

By abuse of notation, we also denote by $C$ the class of morphisms in $\mathcal{M}^{(0)}$ that are weakly equivalent to a morphism in $\mathcal{S}^{\amalg}$ which lies in $C$.
\end{DEF}

By Lemma~\ref{LEMMAWEM0}, 1., for a geometric  site, every morphism in $\mathcal{M}^{(0)}$ is weakly equivalent to a morphism in $\mathcal{S}^{\amalg}$ and $\mathcal{M}^{(0)}$ is closed under homotopy pull-back.
Note that the inclusion of $\mathcal{S}^{\amalg}$ into $\mathcal{SET}^{\mathcal{S}^{\op} \times \Delta^{\op}}$ (and thus the inclusion into $\mathcal{M}$) sends fiber products to homotopy fiber products. 
Therefore $C$  is stable under homotopy pull-back along morphisms in $\mathcal{M}^{(0)}$. 
\begin{DEF}\label{DEFKGEOMETRIC}
\begin{enumerate}
\item A morphism $X \rightarrow Y$ in $\mathcal{M}$ is {\bf $k$-representable}, if for any $S \in \mathcal{M}^{(0)}$ and morphism $S \rightarrow Y$ the homotopy pullback $X \widetilde{\times}_Y S$ is in $\mathcal{M}^{(k)}$. 
\item A morphism $X \rightarrow Y$ in $\mathcal{M}$ is {\bf in $0-C$} if it is 0-representable and for any $S \in \mathcal{M}^{(0)}$ and morphism $S \rightarrow Y$ the morphism $X \widetilde{\times}_Y S \rightarrow S$ is  in $C$.
\item For $k \ge 1$, a morphism $U \rightarrow X$ in $(k-1)-C$ with $U \in \mathcal{M}^{(0)}$ is called a {\bf $k$-atlas}.
\item For $k \ge 1$, an object $X$ is in $\mathcal{M}^{(k)}$, i.e.\@ {\bf $k$-geometric}, if the diagonal $X \rightarrow X \times X$ is $(k-1)$-representable and there is a $k$-atlas $U \rightarrow X$. 
\item For $k \ge 1$, a morphism $X \rightarrow Y$ is {\bf in $k-C$} if it is $k$-representable and if for any $S \in \mathcal{M}^{(0)}$ and morphism $S \rightarrow Y$ there is a $k$-atlas $U \rightarrow X \widetilde{\times}_Y S$ such that $U \rightarrow S$ is in $C$. 
\end{enumerate}
\end{DEF}

We have the following elementary properties:
\begin{PROP}\label{PROPGEOMETRIC}
\begin{enumerate}
\item $k$-representable morphisms, and $k-C$ morphisms, are stable under homotopy pull-backs, composition, and isomorphism (in the homotopy category). 
\item $(k-1)-C$ implies $k-C$, and $(k-1)$-representable implies $k$-representable.
\item For a homotopy Cartesian diagram
\[ \xymatrix{ 
W \ar[r] \ar[d] & Z \ar[d] \\
X \ar[r] & Y   } \]
if $Y$ is $k$-geometric and $Z$ and $X$ are $(k-1)$-geometric then $W$ is $(k-1)$-geometric. 
\end{enumerate}
\end{PROP}
\begin{proof}
For 1.\@ and 2.\@ see \cite[Proposition~1.3.3.3]{TV08}. 

3. 
We have the following homotopy Cartesian diagram: 
\[ \xymatrix{ 
W \ar[r] \ar[d] & Z \times X \ar[d] \\
Y \ar[r]_-{\Delta} & Y \times Y   } \]
Since the diagonal $\Delta$ is $(k-1)$-respresentable by assumption, the morphism $W \rightarrow Z \times X$ is $(k-1)$-representable. 
Since $Z \times X$ is $(k-1)$-geometric also $W$ is $(k-1)$-geometric. 
\end{proof}

\begin{BEM}
The notions of $k$-representable, $k-C$, and $k$-geometric, depend only on the isomorphism class of the morphism (resp.\@ object) in the homotopy category. In particular, the subcategory $\mathcal{M}^{(k)} \subset \mathcal{M}$ of $k$-geometric stacks is homotopy strictly full (\ref{PARHSF}). Using \ref{PARHFP} one could even phrase the definition entirely in terms of morphisms and objects in the homotopy category. 
\end{BEM}

\begin{BEM}\label{BEMTV}
The setting here is more restrictive that the setting in \cite{TV08} in two ways: 

Firstly, the base category $\mathcal{S}$ is just a category and not a model category. The constructions in 
 this article could be easily adopted to the more general setting. The author abstained from doing so in order to keep the exposition as simple as possible, in view that 
the only available {\em construction} of derivator six-functor-formalisms \cite{Hor17} is in this setting so far.

Secondly, it is tacitly assumed that the information about the class $\mathbf{P}$ of \cite[1.3.2]{TV08} can be incorporated into the choice of pre-topology.
This is a slight restriction. According to \cite[Assumption~1.3.2.11]{TV08} the class  $\mathbf{P}$ contains the morphisms in the coverings of some chosen pre-topology $\tau$. Hence one can define a pre-topology $\tau'$ adding all families $\{ U_i \rightarrow X \}$ consisting of morphisms in $\mathbf{P}$ which have the property that there is a covering $\{U_j' \rightarrow X\}$ in $\tau$ refining $\{ U_i \rightarrow X \}$. This might give a weaker notion of $n$-geometric stacks, though.
Nevertheless, it is still important to choose {\em some} pre-topology! The notion of $k$-geometric stack as stated above is not a property of the topology generated by $\tau$. 
\end{BEM}

\begin{BEISPIEL}The index $k$ reflects the {\em construction steps} for the stack in question and not directly the vanishing or non-vanishing of the associated sheaves of homotopy groups. 
For example, if $\mathcal{S}$ is affine schemes  with the Zariski pre-topology, then schemes are 1-geometric stacks, or if $\mathcal{S}$ is the category of affine schemes with the \'etale pre-topology then algebraic spaces are 2-geometric stacks and usual (1-)Deligne-Mumford stacks are thus only 3-geometric in general. Similarly, if $\mathcal{S}$ is the category affine schemes with the smooth pre-topology then usual (1-)Artin stacks are 3-geometric in general.
Cf.\@ the distinction between $n$-geometric stack and geometric $n$-stack \cite[Section~2.1.1]{TV08}.
\end{BEISPIEL}

\begin{LEMMA}\label{LEMMALOCALEPI}For any $k \in \N_0$,
the class $k-C$ consists of homotopy local epimorphisms, i.e.\@ for a morphism $f: X \rightarrow Y$ in $k-C$ and a
morphism $U \rightarrow Y$ in the homotopy category with $U \in \mathcal{S}$ there is a covering $\{U_i \rightarrow U\}_i$ and lifts
\[ \xymatrix{ & U \widetilde{\times}_Y X \ar[r] \ar[d] & X \ar[d]^f \\
 U_i \ar[r] \ar@{.>}[ru] & U \ar[r] & Y   } \]
 in the homotopy category. 
\end{LEMMA}
\begin{proof}
Let $f: X \rightarrow Y$ a morphism in $k-C$. 
By induction on $k$. For $k=0$ and a morphism $U \rightarrow Y$ with $U \in \mathcal{S}$, the homotopy fiber product $U \widetilde{\times}_Y X$ is in 
$\mathcal{M}^{(0)}$ hence w.l.o.g.\@ in $\mathcal{S}^{\amalg}$ and the morphism to $U$ is in $C$. Hence a covering and lift can be read off by definition. 
For $k=1$ look at the following diagram in the homotopy category
\[ \xymatrix{ 
& U' \ar[d] &  \\
& U \widetilde{\times}_Y X \ar[r] \ar[d] & X \ar[d]^f \\
& U \ar[r] & Y   } \]
where $U' \rightarrow  U \widetilde{\times}_Y X$ is the $k$-atlas existing by definition and 
the morphism $U' \rightarrow  U$ is in $C$. W.l.o.g.\@ $U' \in \mathcal{S}^{\amalg}$ and thus a covering and the lift can be read off.
\end{proof}

\begin{DEF}\label{DEFSHCOR}
Denote by $(\mathcal{M}^{(k)}, \Fib^{(k)}, \mathcal{W}^{(k)})$ the  category with fibrant objects whose objects are the $k$-geometric stacks in $\mathcal{M}$ and whose fibrations snd weak equivalences are the same as in $\mathcal{M}$.
Denote 
\[ \SSS^{\hcor, (k)} := \HH^{\cor}(\mathcal{M}^{(k)}) \]
the symmetric 2-pre-multiderivator of multicorrespondences of $k$-geometric stacks in $\mathcal{S}$.
We also denote by
\[ \SSS^{\hcor, (\infty)}  \]
the union of the $\SSS^{\hcor, (k)}$. Be aware that for infinite diagrams $I$ the objects in $\SSS^{\hcor, (\infty)}(I)$ do lie in some $\SSS^{\hcor, (k)}(I)$, i.e.\@ the ``degree'' of the stacks is globally bounded. 
\end{DEF}

\begin{PAR}
Pridham \cite{Pri13} shows  that $k$-geometric stacks according to this definition are essentially the same as those representable by $(k,C)$-hypergroupoids (see \cite[Definition 3.1]{Pri13}).
In a previous version of the article the author tried to use these for the constructions. However, he did not understand the precise relation of
 the category of $(k,C)$-hypergroupoids as a category with fibrant objects with $(\mathcal{M}^{(k)}, \Fib^{(k)}, \mathcal{W}^{(k)})$ constructed above. For example, it is not clear whether they
 are equivalent as categories with weak equivalences. 
\end{PAR}

\begin{PAR}
The main objective of this article is to extend a derivator six-functor-formalism (or $*,!$-formalism) with domain $\Cat$
\[ \DD \rightarrow \SSS^{\cor} \]
to higher geometric stacks, i.e.\@ to a fibered (multi)derivator 
\[ \DD^{(\infty)} \rightarrow \SSS^{\hcor, (\infty)}. \]
\end{PAR}

\begin{PAR}\label{PAREXTAMALG}
The first, and rather trivial, step is to extend $\DD$ to $\SSS^{\amalg, \cor}$, the 2-pre-multiderivator of multicorrespondences in the free coproduct completion $\mathcal{S}^{\amalg}$, as follows:
Let $I$ be a diagram, and let $S_0: \tw I \rightarrow \mathcal{SET}$ be an admissible diagram in the sense of Definition~\ref{DEFADMISSIBLE}. We define a diagram
\[ \widetilde{\int_I} S_0  \]
whose objects are pairs $(i, x)$ with $i \in I$ and $x \in S_0(i = i)$, and whose morphisms $(i, x) \rightarrow (i', x')$ are morphisms $i \rightarrow i'$ in $I$ together with $y \in S_0(i \rightarrow i')$ such that  the morphisms
\[ \xymatrix{ & S(i \rightarrow i') \ar[ld] \ar[rd] & \\
S(i = i) & & S(i' = i') } \]
map $y$  to $x$, and to $x'$, respectively. The composition of two such morphisms $(i, x) \rightarrow (i', x')$ given by $y$, and $(i', x') \rightarrow (i'', x'')$ given by $y'$, is $(i \to i'', y'')$ where 
$y'' \in S(i \to i'')$ is the unique element mapping to $y$, and $y'$, respectively. For the uniqueness observe that the square in the diagram
\[ \xymatrix{ & & S( i \to i'') \ar[ld] \ar[rd]\\
& S(i \rightarrow i') \ar[ld] \ar[rd] & & S(i' \rightarrow i'') \ar[ld] \ar[rd] & \\
S(i = i) & & S(i' = i')& & S(i'' = i'') } \]
is Cartesian. This comes equipped with a functor $\widetilde{\pi}: \widetilde{\int_I} S_0 \rightarrow I$ which is neither a fibration nor an opfibration in general. There is an isomorphism of categories
\begin{equation} \label{eqisoint} \tw (\widetilde{\int_I} S_0)  \rightarrow \int_{\tw I} S_0  \end{equation}
mapping $(i \to i', y \in S(i \to i'))$ to $(i \to i', y)$.
This construction allows to define  a functor of 2-multicategories
\[ \Cat^{\cor}((\mathcal{S}^\amalg)^{\cor}) \rightarrow  \Cat^{\cor}(\mathcal{S}^{\cor})   \]
given on objects by mapping $(I, S)$ given by $S_0: \tw I \rightarrow \mathcal{SET}$ and $S' \in \mathcal{S}^{\int S_0}$ 
 to $(\widetilde{\int} S_0, S'')$, 
 In which $S'': \tw (\widetilde{\int} S_0) \rightarrow \int S_0 \rightarrow \mathcal{S}$ is the composition of $S'$ with the isomorphism (\ref{eqisoint}). 
 A  multicorrespondence (i.e.\@ a generating 1-morphism) in the notation \ref{PARNOTATIONDERCORHCOR}
 \[ \xymatrix{ & & & (A, U) \ar[llld]_-{} \ar[ld]^-{} \ar[rd]^-{} \\
(I_1,S_1) & \cdots & (I_n, S_n) & ; & (J, T) }  \]
is mapped to 
\[ \xymatrix{ & & & (\widetilde{\int} U_0, U'') \ar[llld]_-{} \ar[ld]^-{} \ar[rd]^-{} \\
(\widetilde{\int} S_{1,0},S_1'') & \cdots & (\widetilde{\int} S_{n,0}, S_n'') & ; & (\widetilde{\int} T_0, T'') }  \]

A derivator six-functor-formalism (or $*,!$-formalism) with domain $\Cat$
\[ \DD \rightarrow \SSS^{\cor} \]
gives rise to a strict functor of 2-(multi)categories
\[ \Cat^{\cor}(\DD)  \rightarrow  \Cat^{\cor}(\SSS^{\cor})  \]
which is a 1-bifibration and 2-bifibration with 1-categorical fibers.
We can construct the strict pull-back 
\[ \xymatrix{ \Box \ar[r] \ar[d]&  \Cat^{\cor}(\DD) \ar[d] \\
\Cat^{\cor}((\SSS^{\amalg})^{\cor}) \ar[r] & \Cat^{\cor}(\SSS^{\cor})
 }  \]
 In which both vertical functors are strict. From the left vertical morphism a 2-pre-multiderivator $\DD^{\amalg}$  can be extracted (defining $\DD^{\amalg}(I)$ as the fiber of $I$ under the composition $\Box \rightarrow \Cat^{\cor}((\SSS^{\amalg})^{\cor})  \rightarrow \Cat^{\cor}$)  in such a way that $\Cat^{\cor}(\DD^{\amalg}) = \Box$. We leave the details to the reader. 
\end{PAR}

\begin{DEF}\label{DEFADMEXT}
Let $(\mathcal{M}, \Fib, \mathcal{W})$ be a category of fibrant objects with $\mathcal{M} \subset \mathcal{SET}^{\mathcal{S}^{\op} \times \Delta^{\op}}$ containing $\mathcal{S}^{\amalg}$ and whose
weak equivalences are \v{C}ech weak equivalences. 
Let $\DD \rightarrow \SSS^{\cor}$ be a fibered (multi)derivator with domain $\Cat$. We say that a fibered (multi)derivator $\DD' \rightarrow \HH^{\cor}(\mathcal{M})$ with domain $\Cat$ is an {\bf admissible extension} of 
$\DD$ if
\begin{enumerate}
\item The pull-backs of $\DD'$ via
\[ \mathbb{M} \rightarrow \HH^{\cor}(\mathcal{M}) \quad \text{resp.\@} \quad \mathbb{M}^{\op} \rightarrow \HH^{\cor}(\mathcal{M}) \]
are equivalent to the naive extensions (Theorem~\ref{SATZEXTNAIVELEFT}) of $\DD^!$, and $\DD^*$, respectively. In particular, we have equivalences
\[ \DD_S' \cong \DD_S^{!,h} \cong \DD_S^{*,h}  \]
of fibers for every $S \in \mathcal{M}$. 
\item The pull-back of $\DD$ via
\[ \SSS^{\amalg, \cor} \rightarrow \HH^{\cor}(\mathcal{M}) \]
is equivalent to $\DD^{\amalg}$ (defined in \ref{PAREXTAMALG}). 
\end{enumerate}
\end{DEF}

The reader may check that the restrictions of the naive extensions to $\SSS^{\amalg}$ and $(\SSS^{\amalg})^{\op}$ are equivalent to the restrictions of
$\DD^{\amalg}$ (constructed in \ref{PAREXTAMALG}) to $\SSS^{\amalg}$, and $(\SSS^{\amalg})^{\op}$, respectively. Thus 1.\@ and 2.\@ are compatible a priori.

The main result of the article is the following

\vspace{0.3cm}
{\bf Main Theorem~\ref{HAUPTSATZ}.} {\em Let $\mathcal{S}$ be a small category with finite limits and Grothendieck pre-topology. Assume that $\mathcal{S}$ is geometric (cf.\@ Definition~\ref{DEFGEOMETRIC}).

Let $\DD \rightarrow \SSS^{\cor}$ be an infinite (symmetric) fibered  (multi)derivator with domain $\Cat$, with stable, well-generated fibers which is local w.r.t.\@ the Grothendieck pre-topology (cf.\@ Definition~\ref{DEFLOCAL6FU}).

There exists a (symmetric) fibered (multi)derivator 
\[ \DD^{(\infty)} \rightarrow \SSS^{\hcor, (\infty)}  \]
with domain $\Cat$ with stable, well-generated fibers which is an {\em admissible extension} in the sense of Definition~\ref{DEFADMEXT}.
}

\section{Locality in extensions}

Let $\mathcal{S}$ be a small category with finite limits and Grothendieck pre-topology which is geometric (Definition~\ref{DEFGEOMETRIC}).
Choose a simplicial category with fibrant objects $(\mathcal{M}, \Fib, \mathcal{W})$ (cf.\@ \ref{PARCWFO}) which is a full subcategory (not homotopy strictly full) of the \v{C}ech localization of simplicial pre-sheaves on $\mathcal{S}$ containing the constant ones in $\mathcal{S}^{\amalg}$ (i.e.\@ the coproducts of representables).

In this section, we {\em assume} the existence of an appropriate extension of a given fibered derivator over $\SSS$, $\SSS^{\op}$, or $\SSS^{\cor}$, to $\mathbb{M}, \mathbb{M}^{\op}$, and $\HH^{\cor}(\mathcal{M})$, respectively, and show that the axioms involved in the notion of being (co)local (\ref{PARLOCAL}, \ref{PARCOLOCAL}, \ref{DEFLOCAL6FU}) are inherited by morphisms that are $k-C$ in the sense of Definition~\ref{DEFKGEOMETRIC}. 

\begin{PAR}\label{PARRESOLUTION}
Let $X \in \mathcal{M}$. 
In $\mathcal{SET}^{\mathcal{S}^{\op} \times \Delta^{\op}}$ there is a morphism $\widetilde{X} \rightarrow X$ (cofibrant replacement in the projective model structure) with $\widetilde{X} \in \mathcal{S}^{\amalg, \Delta^{\op}}$.
This yields a morphism\footnote{Where $\widetilde{X}_\bullet$ denotes the corresponding simplicial object of constant objects and $\delta(X)_n = \Hom(\Delta_n, X)$.} $(\Delta^{\op}, \widetilde{X}_\bullet) \rightarrow (\Delta^{\op}, \delta(\widetilde{X})) \rightarrow (\Delta^{\op}, \delta(X))$. Even if $\widetilde{X}$ is not in $\mathcal{M}$, the composition
\[ (\Delta^{\op}, \widetilde{X}_\bullet) \rightarrow (\Delta^{\op}, \delta(X)) \]
is in $\Cat(\mathcal{M})$ because the $\widetilde{X}_i$ are by assumption.
\end{PAR}

\begin{PAR}\label{PARSETTINGLOCALITYLEFT}
Let $\DD^! \rightarrow \SSS$ be an infinite local fibered derivator with domain $\Cat$ (for instance, the restriction of a $*,!$-formalism or six-functor-formalism $\DD \rightarrow \SSS^{\cor}$, but not necessarily) with stable and well-generated fibers, and consider the following statements w.r.t.\@ the naive extension (cf.\@ Theorem~\ref{SATZEXTNAIVELEFT})
$\DD^{!,h} \rightarrow \mathbb{M}$ and (denoting, as usual, $f^! := f^\bullet$)
\begin{enumerate}
\item[(H1$(f)$)] for  a morphism $f$ in $\mathcal{M}$: $f^!$, as morphism of derivators, commutes with homotopy colimits. 
\item[(H1''$(f)$)] for a morphism $f: (I, S) \rightarrow (I, T)$ in $\Cat(\mathcal{M})$: $f^!$ has a right adjoint (as morphism of derivators\footnote{equivalently: the adjoints are computed point-wise on {\em constant} diagrams})
\item[(H2$(f)$)] for  a morphism $f$ in $\mathcal{M}$: $f^!$ satisfies base change, i.e.\@ for each homotopy Cartesian square in $\mathcal{M}$
\[ \xymatrix{  \ar[r]^F \ar[d]_G &  \ar[d]^g \\ 
 \ar[r]_f & 
} \] 
 the natural transformation $G_! F^! \rightarrow f^! g_!$ is an isomorphism. 
\item[(H3$(f)$)] for  a morphism $f$ in $\mathcal{M}$: $f^!$ is conservative. 
\end{enumerate}
Note that (H1''$(f)$) implies (H1$(f)$) for a morphism $f$ in $\mathcal{M}$ and, under the assumption, for a morphism $f: (I, S) \rightarrow (I, T)$ of diagrams in $\mathcal{M}$, 
(H1$(f_i)$) for all $i \in I$ implies (H1''$(f)$) by Brown representability. 
\end{PAR}

\begin{PAR}\label{PARSETTINGLOCALITYRIGHT}
Dually, let $\DD^* \rightarrow \SSS^{\op}$ be an infinite colocal fibered derivator with domain $\Cat$ (for instance, the restriction of a $*,!$-formalism or six-functor-formalism $\DD \rightarrow \SSS^{\cor}$, but not necessarily) with stable and well-generated fibers, and consider the following statements w.r.t.\@ the naive extension (cf.\@ Theorem~\ref{SATZEXTNAIVERIGHT})
$\DD^{*,h} \rightarrow \mathbb{M}$ (denoting, as usual, $f^* := (f^{\op})_\bullet$):
\begin{enumerate}
\item[(C1$(f)$)] for a morphism $f$ in $\mathcal{M}$: $f^*$, as morphism of derivators, commutes with homotopy limits.
\item[(C1''$(f)$)] for a morphism $f: (I, S) \rightarrow (I, T)$  in $\Cat^{\op}(\mathcal{M}^{\op})$: $f^*$ has a left adjoint (as morphism of derivators) 
\item[(C2$(f)$)] for a morphism $f$ in $\mathcal{M}$: $f^*$ satisfies base change, i.e.\@ for each homotopy Cartesian square in $\mathcal{M}$
\[ \xymatrix{  \ar[r]^F \ar[d]_G &  \ar[d]^g \\ 
 \ar[r]_f & 
} \] 
 the natural transformation $ f^* g_* \rightarrow G_* F^*$ is an isomorphism. 
\item[(C3$(f)$)] for a morphism $f$ in $\mathcal{M}$: $f^*$ is conservative. 
\end{enumerate}
Note that (C1''$(f)$) implies (C1$(f)$) for a morphism $f$ in $\mathcal{M}$ and, if $\DD^*$ would have compactly generated fibers, then for a morphism $f: (I, S) \rightarrow (I, T)$ of diagrams in $\mathcal{M}$, 
(C1$(f_i)$) for all $i \in I$ would imply (C1''$(f)$) by Brown representability for the dual. 
\end{PAR}

\begin{PAR}\label{PARSETTINGLOCALITY2}
Let $\DD \rightarrow \SSS^{\cor}$ now be an infinite local (as in Definition~\ref{DEFLOCAL6FU}) fibered (multi)derivator with domain $\Cat$ (i.e.\@ $*,!$-formalism or six-functor-formalism) with stable and well-generated fibers, and let $\DD' \rightarrow \HH^{\cor}(\mathcal{M})$ be an {\em admissible extension} of $\DD \rightarrow \SSS^{\cor}$ in the sense of Definition~\ref{DEFADMEXT}. Consider the following statement w.r.t.\@ $\DD'$:
\begin{enumerate}
\item[(CH$(f_1,\dots,f_l)$)] for morphisms $f_1, \dots, f_l$ in $\mathcal{M}$: for each homotopy Cartesian square in $\mathcal{M}$
\[ \xymatrix{ X \ar[r]^F \ar[d]_G & X' \ar[d]^g \\ 
Y_1, \dots, Y_l \ar[r]_{(f_i)} &  Y_1', \dots, Y_l'
} \] 
the natural exchange $G^* (f_1^!-, \dots, f_l^!-) \rightarrow F^! g^*(-, \dots, -)$ and for $l=1$, the natural exchange $F^* g^! \rightarrow G^! f^*$ are isomorphisms. 
\end{enumerate}
\end{PAR}

\begin{LEMMA}\label{LEMMAAXIOMS}
The validity of the axioms in \ref{PARSETTINGLOCALITYLEFT}--\ref{PARSETTINGLOCALITY2} depends only on the isomorphism class of $f$ in the homotopy category of $\mathcal{M}^{\rightarrow}$ (resp.\@ of $\mathcal{M}^{I \times \rightarrow}$ for (H1'') and (C1'')).
The statements of (H2), (C2) and (CH) depend only on the isomorphism class of the square in the homotopy category of $\mathcal{M}^{\Box}$ (resp.\@ in $\mathcal{M}^{\Delta_1 \times \Delta_{1,k}^{\op}}$).
\end{LEMMA}
\begin{proof}
The reader may verify the following statement (and its dual):
For a 2-commutative cube of categories and functors
\[ \xymatrix{  & \mathcal{A} \ar[ld]_{F_A^*}  \ar[rr]^{F^*} \ar[dd]^(.7){G^*} & & \mathcal{B} \ar[dd]^{g*} \ar[ld]^{F_B^*} \\
 \mathcal{A}' \ar[rr]_(.3){(F')^*} \ar[dd]^{(G')^*}  & & \mathcal{B}' \ar[dd]^(.3){(g')^*} \\
  & \mathcal{C} \ar[rr]^(.3){f^*}  \ar[ld]^-{F_C^*} & & \mathcal{D} \ar[ld]^-{F_D^*} \\
  \mathcal{C}' \ar[rr]_{(f')^*} & & \mathcal{D}'  
} \]
such that all functors $(\cdots)^*$ have right adjoints $(\cdots)_*$ and such that $F_A^*, F_B^*, F_C^*, F_D^*$ are equivalences, then 
the exchange 
\[ g^* f_* \rightarrow F_* G^* \]
is an isomorphism if and only if the exchange
\[ (g')^* f'_* \rightarrow F'_* (G')^* \]
is an isomorphism. 
Noting that the pull-backs $f^!$ and $f^*$ along a morphism $f$ which 
is a (point-wise) weak equivalence are equivalences, and
using the above statement for appropriate cubes the assertion follows for (H1), (H2), (C1), (C2), and (CH). 
For (H1''), (C1''), (H3), and (C3) the statement is clear.  
\end{proof}

\begin{LEMMA}\label{LEMMAAXIOMSDIRINV}
Axiom {\em (H1''$(f)$)}  for all $f \in \mathcal{M}^I$ point-wise in $k-C$ and $I \in \Invlf$ with finite matching diagrams implies Axiom {\em (H1''$(f)$)} for all $I$ and $f$ point-wise in $k-C$. 
Axiom {\em (C1''$(f^{\op})$)} for all $f \in \mathcal{M}^{\op,I}$ point-wise in $k-C$ and $I \in \Dirlf$ with finite latching diagrams implies Axiom {\em (C1''$(f)$)} for all $I$ and $f$ point-wise in $k-C$. 
\end{LEMMA}
\begin{proof}
This is an easy consequence of the fact \cite{Hor17b} that the theory of left (resp.\@ right) fibered derivators with domain $\Cat$ is in a certain sense equivalent to its restriction to $\Dirlf$ (resp.\@ to $\Invlf$) with finite latching (resp.\@ matching) diagrams. 
More precisely (restricting to (H1''), the case of (C1'') being dual): 
Let $I \in \Cat$ be a diagram, and $f: S \rightarrow T$ a morphism in $\SSS(I)$, point-wise in $k-C$.  By \cite[Proposition 3.6]{Hor17b} (applying indeed the {\em right} case), there is a functor $\pi: N(I) \rightarrow I$ with $N(I) \in \Invlf$ with finite matching diagrams such that $\pi^*$ induces an equivalence
\[ \DD^{!,h}(I)_S \cong \DD^{!,h}(N(I))_{\pi^*S}^{\pi-\cart}.  \]
for any $S \in \SSS(I)$. However, by assumption $(\pi^* f)^!$ has a right adjoint $(\pi^* f)_{?}$. Therefore $\pi^{(T)}_* (\pi^* f)_{?} \pi^*$ is a right adjoint of $f^!$.
Since also $\pi^{(T)}_*$ and $\pi^*$ promote to a morphism of derivators (fiber) this is the case also for the composition. 
\end{proof}

\begin{LEMMA}\label{LEMMAAXIOMSM0}
\begin{enumerate}
\item 
Every morphism in $\mathcal{M}^{(0)}$ which is in $C$ (cf.\@ Definition~\ref{DEFC})
satisfies (in the appropriate situation) {\em (H1$(f)$--H3$(f)$), (C1$(f)$--C3$(f)$)}, and {\em (CH$(f)$)} for the restrictions\footnote{i.e.\@ Axioms (C2$(f)$), (H2$(f)$), (CH$(f)$) hold for homotopy Cartesian diagrams lying completely in $\mathcal{M}^{(0)}$} of $\DD^{!,h}$,  $\DD^{*,h}$, and $\DD'$,  to $\mathbb{M}^{(0)}$, $\mathbb{M}^{(0), \op}$, and $\HH^{\cor}(\mathcal{M}^{(0)})$, respectively. 

\item 
Every morphism
$(I, S) \rightarrow (I, T)$ in $\Cat(\mathcal{M}^{(0)})$ that is point-wise in $C$, satisfies {\em (H1''$(f)$)} and   
every morphism $(I, S) \rightarrow (I, T)$ in $\Cat^{\op}(\mathcal{M}^{(0), \op})$ that is point-wise in $C$, satisfies {\em (C1''$(f)$)}. 
\end{enumerate}
\end{LEMMA}
\begin{proof}
1.\@ for morphisms in $\mathcal{S}^{\amalg}$ is just an unraveling of the definitions
 using (Der1--2).  For $\mathcal{M}^{(0)}$ observe that every morphism, resp.\@ every homotopy Cartesian square, in $\mathcal{M}^{(0)}$ is weakly equivalent to one in $\mathcal{S}^{\amalg}$ by Lemma~\ref{LEMMAWEM0}.
 It follows then from Lemma~\ref{LEMMAAXIOMS} using that $f^!$ and $f^*$ are equivalences for weak equivalences $f$. 

 2.\@ 
 We show (H1''($f$)) for a morphism $f: S \rightarrow T$ between diagrams $S$, $T: I \rightarrow \mathcal{S}^\amalg$, the other case is dual. 
  Let $f_0: S_0 \Rightarrow T_0: I \rightarrow \mathcal{SET}$ be the corresponding natural transformation of the functors of underlying sets. 
The Grothendieck construction yields a morphism of diagrams which may be factored as follows
\[ \xymatrix{ (\int S_0, S') \ar[r]^-{\widetilde{f}} & (\int S_0, (\int f_0)^* T')  \ar[r]^-{\int f_0} &  (\int T_0, T').  } \]
We have a 2-commutative diagram (by Definition of $\DD^{!,h}$ and homological descent) in which the horizontal functors are equivalences
\[ \xymatrix{ \DD^{!,h}(I)_T \ar[r]^-\sim \ar[dd]^{f^!} &  \DD(\int T_0)_{T'} \ar[d]^{(\int f_0)^*} \ar[d]  \\
&  \DD(\int S_0)_{(\int f_0)^* T'} \ar[d]^{\widetilde{f}^!} \\ 
 \DD^{!,h}(I)_S \ar[r]^-\sim & \DD(\int S_0)_{S'}  } \]  
 Since the functors on the right hand side have right adjoints (which promote to a morphism of derivators (fiber)) by axiom (H1''), and (FDer3--4 right), respectively, the statement follows. 
For $\mathcal{M}^{(0)}$ the statement follows from Lemma~\ref{LEMMAWEM0}, 2.\@ and Lemma~\ref{LEMMAAXIOMS} again, if $I$ is in $\Invlf$ with finite matching diagrams, or in $\Dirlf$ with finite latching diagrams, respectively.
The general case follows from Lemma~\ref{LEMMAAXIOMSDIRINV}.
\end{proof}

\begin{LEMMA}\label{LEMMAWEIRD}
Consider a 2-commutative diagram of categories of functors
\[ \xymatrix{ \mathcal{C} \ar[r]^{F} \ar@{}[rd]|{\Nearrow^\mu} \ar@{^{(}->}[d]_g &  \mathcal{D} \ar@{^{(}->}[d]^G  \\
\overline{\mathcal{C}} \ar[r]_{\overline{F}} & \overline{\mathcal{D}} 
 }\]
with 2-isomorphism $\mu$. 
 Assume that $\overline{F}$ has a right adjoint $\overline{X}$ and that $G$ and $g$ are fully-faithful and have left adjoints $H$, and $h$, respectively. 
 Assume that the exchange of $\mu$
 \[   H\overline{F} \rightarrow Fh \]
 is an isomorphism. 
 Then $X:=h \overline{X} G$ is a right adjoint of $F$.  
\end{LEMMA} 

Of course, there is a corresponding dual statement, which we leave to the reader to formulate. 
The assumption implies a posteriori that the adjoint 
 \[ gX \rightarrow  \overline{X} G   \]
 of the exchange above  is also an isomorphism and thus that $\overline{X}$ preserves the strictly full subcategories given by the images of $g$, and $G$, respectively.

\begin{proof}
We construct the unit and counits
\[ \xymatrix{
\id & \ar[l]_-{\sim} hg \ar[r] & h \overline{X} \overline{F} g \ar[r]^-\mu & h \overline{X} G F  = XF
}\]

\[ \xymatrix{
FX = Fh \overline{X} G & \ar[l]^-{\sim}_-{\text{ex.}} H\overline{F}\overline{X} G \ar[r]  & HG \ar[r] & \id 
}\]
in which unnamed transformations are (co)units for the various involved adjunctions. 

 We have to show the (co)unit equations. 
The first equation concerns the first line in the following commutative diagram
\[ \xymatrix{
F & \ar[l]^{\sim} Fhg \ar[r]  \ar@{<-}[rd]_-{\text{ex.}}^-\sim & Fh \overline{X} \overline{F} g \ar[r]^-\mu \ar@{<-}[rd]_-{\text{ex.}}^-\sim & Fh \overline{X} G F & \ar[l]_-{\text{ex.}}^-\sim \ar@{<-}[ld]^-\mu H \overline{F} \overline{X} G F  \ar[r]  & H G F \ar@{<-}[ld]^-\mu  \ar[r] & F \\
 && H \overline{F} g  \ar[r] &  H \overline{F} \overline{X} \overline{F} g \ar[r] & H \overline{F} g 
}\]
Since the second line composes to the identity the statement follows from the commutativity of the left hand side diagram in
\[ \xymatrix{ 
H\overline{F}g \ar[r]_-\sim^-{\text{ex.}} \ar[d]_\sim^\mu & Fhg \ar[d]^\sim \\
HGF \ar[r] & F 
} \quad \xymatrix{ 
\overline{F} \ar[r] \ar[d] & GH\overline{F} \ar[d]_-\sim^-{\text{ex.}} \\
\overline{F}gh \ar[r]^-\sim_-\mu & GFh 
} \]
The second equation concerns the first line in the following commutative diagram (where the middle square is induced by the right hand side diagram above)
\[ \xymatrix{
h\overline{X}G & \ar[l]^{\sim} hgh \overline{X} G \ar[r] & h \overline{X} \overline{F} gh \overline{X} G \ar[r]^-\mu_-\sim & h \overline{X} GFh \overline{X} G \ar@{<-}[r]^-{\text{ex.}}_-\sim & h \overline{X} GH \overline{F} \overline{X} G \ar[r] & h \overline{X} GHG \ar[r]^\sim & h \overline{X} G \\
& \ar@{=}[lu]  \ar[u] h \overline{X} G \ar[r] & h \overline{X} \overline{F} \overline{X} \ar[u] \ar@{=}[rr] & &  h \overline{X} \overline{F} \overline{X}\ar[u]  \ar[r] & h \overline{X} G \ar[u] \ar@{=}[ru]
}\]
Its composition is therefore also the identity. 
\end{proof}

\begin{LEMMA}\label{LEMMACARTPROJ}
\begin{enumerate}
\item 
Let $g_\bullet: (\Delta^{\op}, X_\bullet) \rightarrow (\Delta^{\op}, Y_\bullet)$ in $\Cat(\mathcal{M})$ be a homotopy Cartesian (cf.\@ Defintion~\ref{DEFCARTMORPH}) morphism of simplicial objects in $\mathcal{M}$.
Assume that {\em (H1''$(g_\bullet)$)} holds and {\em (H1''$(g_n)$)} and {\em (H2$(g_n)$)} hold for all $n$. 
Then $g_\bullet^!$ intertwines the left Cartesian projectors \cite[4.3]{Hor15}
\[ \Box_!:  \DD^{!,h}(\Delta^{\op})_{X_\bullet} \rightarrow \DD^{!,h}(\Delta^{\op})_{X_\bullet}^{\cart} \quad \text{ and } \quad \Box_!:  \DD^{!,h}(\Delta^{\op})_{Y_\bullet} \rightarrow \DD^{!,h}(\Delta^{\op})_{Y_\bullet}^{\cart}.      \]

\item 
Let $g_\bullet: (\Delta, X_\bullet) \rightarrow (\Delta, Y_\bullet)$ in $\Cat^{\op}(\mathcal{M}^{\op})$ be a homotopy Cartesian morphism of cosimplicial objects in $\mathcal{M}^{\op}$.
Assume that {\em (C1''$(g_\bullet)$)} holds and {\em (C1''$(g_n)$)} and {\em (C2$(g_n)$)} hold for all $n$. 
Then $g_\bullet^*$ intertwines the right coCartesian projectors  \cite[4.3]{Hor15}
\[ \Box_*: \DD^{*,h}(\Delta)_{X_\bullet} \rightarrow \DD^{*,h}(\Delta)_{X_\bullet}^{\cocart} \quad \text{ and } \quad   \Box_*: \DD^{*,h}(\Delta)_{Y_\bullet} \rightarrow \DD^{*,h}(\Delta)_{Y_\bullet}^{\cocart}.  \]
\end{enumerate}
\end{LEMMA}
\begin{proof}We show the first statement, the second is dual. 

This is an argument similar to \cite[Lemma~3.5.9]{Hor15}. By assumption
$g_n^!$ and $g_\bullet^!$ have right adjoints $g_{n,?}$, and $g_{\bullet, ?}$, respectively.
We claim that for all $n \in \N_0$ the natural transformation
\[ e_n^* g_{\bullet, ?} \rightarrow g_{n,?} e_n^*  \]
where $e_n: \{\Delta_n\} \hookrightarrow \Delta^{\op}$ is the inclusion, is an isomorphism. In other words, $g_{\bullet, ?}$ is computed point-wise. 
Equivalently, we have to show that 
\[  g_{\bullet}^! e_{n,!}^{(Y_\bullet)} \rightarrow  e_{n,!}^{(X_\bullet)} g_{n}^!  \]
is an isomorphism. This can be shown point-wise:
\[  e_m^* g_{\bullet}^! e_{n,!}^{(Y_\bullet)} = g_{m}^! e_m^* e_{n,!}^{(Y_\bullet)}   \rightarrow e_m^*  e_{n,!}^{(X_\bullet)} g_{n}^!  \]
has to be an isomorphism. By (FDer4 left) this a coproduct over the morphisms $\alpha: \Delta_m \rightarrow \Delta_n$ of
\[   g_{n}^! Y(\alpha^{\op})_!  \rightarrow X(\alpha^{\op})_! g_{m}^!  \]
which are isomorphisms by (H2$(g_n)$) and the fact that $g_\bullet$ is homotopy Cartesian. Also the adjoint
\[   Y(\alpha^{\op})^! g_{n,?}   \rightarrow  g_{m,?} X(\alpha^{\op})^!  \]
is an isomorphism which shows, together which what was just obtained, that $g_{\bullet,?}$ preserves Cartesian objects,
i.e.\@ we have a 2-commutative square
\[ \xymatrix{
\DD^{!,h}(\Delta^{\op})_{X_{\bullet}}^{\cart}  \ar@{^{(}->}[r] \ar[d]^{g_{\bullet,?}} & \DD^{!,h}(\Delta^{\op})_{X_\bullet} \ar[d]^{g_{\bullet,?}}  \\ 
\DD^{!,h}(\Delta^{\op})_{Y_{\bullet}}^{\cart}  \ar@{^{(}->}[r]  & \DD^{!,h}(\Delta^{\op})_{Y_\bullet} \\ 
}  \]
Its left adjoint is
\[ \xymatrix{
\DD^{!,h}(\Delta^{\op})_{X_{\bullet}}^{\cart}  \ar@{<-}[r]^-{\Box_!} \ar@{<-}[d]^{g_{\bullet}^!} & \DD^{!,h}(\Delta^{\op})_{X_\bullet} \ar@{<-}[d]^{g_{\bullet}^!}  \\ 
\DD^{!,h}(\Delta^{\op})_{Y_{\bullet}}^{\cart}  \ar@{<-}[r]^-{\Box_!}  & \DD^{!,h}(\Delta^{\op})_{Y_\bullet} \\ 
}  \]
and thus also 2-commutative.
\end{proof}

Recall Definition~\ref{DEFKGEOMETRIC} of $k-C$ morphism in $\mathcal{M}$ with $k \in \N_0$.

\begin{PROP}\label{PROPHEREDITYLOC1}
\begin{enumerate}
\item In the situation of \ref{PARSETTINGLOCALITYLEFT}, every morphism $(I, S) \rightarrow (I, T)$ in $\Cat(\mathcal{M})$ point-wise in $k-C$ (with $k$ independent of $i \in I$) and such that $T \in (\mathcal{M}^{(0)})^I$
 satisfies {\em (H1''$(f)$)} w.r.t.\@ $\DD^{!, h}$.

\item In the situation of \ref{PARSETTINGLOCALITYRIGHT}, every morphism $(I, S) \rightarrow (I, T)$ in $\Cat^{\op}(\mathcal{M}^{\op})$ point-wise in $k-C$ (with $k$ independent of $i \in I$) and such that $T(i) \in \mathcal{M}^{(0)}$
 satisfies {\em (C1''$(f)$)} w.r.t.\@ $\DD^{*, h}$.
\end{enumerate}
\end{PROP}

\begin{proof}
By induction on $k$. If $k=0$ also 
$X \in \mathcal{M}^{(0),I}$ and the morphism is point-wise in $C$ and thus the statement follows from Lemma~\ref{LEMMAAXIOMSM0}.

Let $k \ge 1$. By Lemma~\ref{LEMMAAXIOMSDIRINV},  we may assume w.l.o.g.\@ that $I \in \Invlf$ with finite matching diagrams. Then apply 
 Propositions~\ref{PROPLIMDIRECTED} and \ref{PROPCOVERDIRECTED}  to the triple $(\mathcal{S}, \mathcal{S}_0, \overline{C})$ in which 
 the full subcategory $\mathcal{S} \subset \mathcal{M}^\rightarrow$ contains the objects of the form $X \rightarrow Y$ in $k-C$ with $Y \in \mathcal{M}^{(0)}$ (in particular $X \in \mathcal{M}^{(k)}$), $\mathcal{S}_0$ is the full subcategory of objects of the form $U \rightarrow Y$ in $\mathcal{M}^{(0), \rightarrow}$ lying in $C$, and $\overline{C}$ is the class of morphisms of the form $(U \to Y) \to (X \to Y)$ in which $U \rightarrow X$ is a fibration in $(k-1)-C$ and the second morphism is the identity. The class $\overline{C}$ is closed under pull-back because $(k-1)-C$ is closed under homotopy pull-back. 
 By definition of $k-C$ (and fibrant replacement), for each object $X \rightarrow Y$ in $\mathcal{S}$ there 
 is a morphism $(U \to Y) \to (X \to Y)$ in $\overline{C}$ such that $(U \to Y) \in \mathcal{S}_0$.
Therefore we get a morphism $g: (I,U) \rightarrow (I,X)$ point-wise in $(k-1)-C$ with $U$ point-wise in $\mathcal{M}^{(0)}$, such that the composition  
\[ F: (I,U) \rightarrow (I,X)  \rightarrow (I,Y) \]
is point-wise in $C$.

Now consider $F_\bullet: (I \times \Delta^{\op}, \Cech(g)) \rightarrow (I \times \Delta^{\op}, \delta(Y))$. 
We have the 2-commutative diagram
\[ \xymatrix{  \DD^{!,h}(I \times \Delta^{\op})_{\Cech(g)}^{2-\cart}  \ar@{<-}[d]^-{\gamma^!}_-\sim \ar@{<-}[r]^-{F_\bullet^!} & \DD^{!,h}(I \times \Delta^{\op})_{\delta(Y)}^{2-\cart}  \ar[dd]^{\sim}  \\
 \DD^{!,h}(I \times \Delta^{\op})_{\delta(X)} \ar[d]^\sim  \ar@{<-}[ur]^-{\delta(f)^!}  &     \\
 \DD^{!,h}(I)_{X} \ar@{<-}[r]_{f^!} &  \DD^{*,h}(I)_{Y}     } \]
 in which the vertical functors are equivalences by homological descent. 
By induction $F_\bullet^!$ ($F_\bullet$ being point-wise $(k-1)-C$ with target in $\mathcal{M}^{(0)}$) has a right adjoint $F_{\bullet,?}$ (neglecting the coCartesianity condition). Hence also $f^!$ has the right adoint $\pr_{1,*} F_{\bullet,?} \gamma^!$ (omitting some equivalences). The constructed adjoints promote to a morphism of derivators (fiber). 
\end{proof}

\begin{SATZ}\label{SATZHEREDITYLOC2}
\begin{enumerate}
\item In the situation of \ref{PARSETTINGLOCALITYLEFT}, every morphism $f \in \mathcal{M}$ in $k-C$ satisfies {\em (H1$(f)$--H3$(f)$)}, w.r.t.\@ $\DD^{!, h}$. 

\item In the situation of \ref{PARSETTINGLOCALITYRIGHT}, every morphism $f \in \mathcal{M}$ in $k-C$ satisfies {\em (C1$(f)$--C3$(f)$)},  w.r.t.\@ $\DD^{*, h}$. 
\end{enumerate}
\end{SATZ}

\begin{proof}
Recall Definitions~\ref{DEFFPS}--\ref{DEFFPSOBJ}.
We will  prove 1.\@ the other statement being dual.

{\bf Step 1:}
Consider a $k-C$ morphism $f: X \rightarrow Y$ in $\mathcal{M}$ with $Y \in \mathcal{M}^{(0)}$. We will prove (H1$(f)$--H3$(f)$) by induction on $k$. 

If $k=0$ also 
$X \in \mathcal{M}^{(0)}$ and the morphism is in $C$.
If $k \ge 1$ there is a $k$-atlas (in particular in $(k-1)-C$) $f': U \rightarrow X$, w.l.o.g.\@ a fibration, such that the composition 
\[ F: U \rightarrow X  \rightarrow Y \]
is in $C$.

{\em Conservativity} (H3$(f)$):
If $k=0$ the statement follows from Lemma~\ref{LEMMAAXIOMSM0}.
If $k\ge 1$ the functors $(f')^!$ and $F^!$ are conservative by the induction hypothesis. Thus also $f^!$ is conservative. 

{\em Commutation with homotopy colimits} (H1$(f)$):
If $k=0$ the statement follows from Lemma~\ref{LEMMAAXIOMSM0}.
If $k\ge 1$ we are left to show that the upper square in the following diagram
 \[ \xymatrix{
\DD^{!,h}(I)_{p^*Y}^{} \ar[r]^-{p_!} \ar[d]_{(p^*f)^!} &   \DD^{!,h}(\cdot)_{Y}^{} \ar[d]^{f^!} \\
\DD^{!,h}(I)_{p^*X}^{} \ar[r]_-{p_!} \ar[d]_{(p^*f')^!}& \DD^{!,h}(\cdot)_{X}^{} \ar[d]^{(f')^!}  \\
\DD^{!,h}(I)_{p^*U}^{} \ar[r]_-{p_!}  &   \DD^{!,h}(\cdot)_{U}^{} 
} \] 
is 2-commutative. The fact that $(f')^!$ is conservative, and the fact that $(f')^!$ and the composition $F^!$ commute with homotopy colimits, establishes the fact. 

{\em Base change} (H2$(f)$):
Consider a homotopy Cartesian diagram:
\[  \xymatrix{ W \ar[r] \ar[d]  &  Z \ar[d]^g  \\
X \ar[r]_{f} &    Y  }
 \]
 If $k=0$ the morphism $f: X \rightarrow Y$ is in $\mathcal{M}^{(0)}$ and in $C$, but $g$ is arbitrary. Choose a replacement $\widetilde{Z} \in \mathcal{S}^{\amalg, \Delta^{\op}}$ of $Z$ with morphism 
 $(\Delta^{\op}, \widetilde{Z}) \rightarrow (\Delta^{\op}, \delta(Z))$ as in \ref{PARRESOLUTION}.
 By homological descent it suffices to show that the composite square is 2-commutative (via the natural exchanges): 
 \[ \xymatrix{ \DD^{!,h}(\Delta^{\op})_{(\delta(X) \widetilde{\times}_Y  \widetilde{Z})_{0,\bullet}}^{\cart} \ar@{<-}[r]^{F^!} \ar[d]_{\widetilde{G}_!} &   \DD^{!,h}(\Delta^{\op})_{(\delta(Y) \widetilde{\times}_Y  \widetilde{Z})_{0,\bullet}}^{\cart} \ar[d]^{\widetilde{g}_!} \\
\DD^{!,h}(\Delta^{\op})_{p^*X} \ar@{<-}[r]^{(p^*f)^!} \ar[d]_{p_!} &   \DD^{!,h}(\Delta^{\op})_{p^*Y}  \ar[d]^{p_!} \\
\DD^{!,h}(\cdot)_{X} \ar@{<-}[r]^{f^!} &   \DD^{!,h}(\cdot)_{Y} } \]

However, the lower square and upper square are 2-commutative because $f$ is in $C$.

Assume now that $k \ge 1$ and thus $f$ is in $k-C$ with $Y \in \mathcal{M}^{(0)}$.
Consider the diagram with homotopy Cartesian squares
\[  \xymatrix{ U \times_X  W \ar[r] \ar[d] & W \ar[r]^-{F} \ar[d]_G  & Z \ar[d]^g  \\
U \ar[r]_{f'} &  X \ar[r]_-{f}   & Y  }
 \] 
After applying $(f')^!$ which is conservative (induction hypothesis) the base-change morphism becomes
\[    (f')^! G_! F^! \rightarrow (f')^! f^! g_!   \]
and using base change for $f'$ (which is $(k-1)-C$, induction hypothesis) and for $f' f$ which is even in $C$ (Lemma~\ref{LEMMAAXIOMSM0}), we see that this is an isomorphism.

{\bf Step 2:}
Consider a general $k-C$ morphism $f: X \rightarrow Y$ in $\mathcal{M}$. 
Choose a replacement $\widetilde{Y} \in \mathcal{S}^{\amalg, \Delta^{\op}}$ of $Y$ with morphism 
 $(\Delta^{\op}, \widetilde{Y}) \rightarrow (\Delta^{\op}, \delta(Y))$ as in \ref{PARRESOLUTION}.
Consider the commutative square  in $\Cat(\mathcal{M})$ 
\[ \xymatrix{ (\cdot, X) \ar[r]^-f \ar@{<-}[d]^{G} & (\cdot, Y) \ar@{<-}[d]^{g}  \\
 (\Delta^{\op}, (\widetilde{Y} \widetilde{\times}_Y \delta(X))_{\bullet,0}) \ar[r]_-{f_n} & (\Delta^{\op}, (\widetilde{Y} \widetilde{\times}_Y \delta(Y))_{\bullet,0}) 
 }\]
 in which $(\widetilde{Y} \widetilde{\times}_Y \delta(Y))_{n,0}$ is weakly equivalent to $\widetilde{Y}_n$ and thus is in $\mathcal{M}^{(0)}$. 
In the resulting 2-commutative diagram 
\[ \xymatrix{ \DD^{!,h}(\cdot)_{X} \ar@{<-}[r]^-{f^!} \ar[d]_-{G^!} & \DD^{!,h}(\cdot)_{Y} \ar[d]^-{g^!}  \\
 \DD^{!,h}(\Delta^{\op})_{(\widetilde{Y} \widetilde{\times}_Y \delta(X))_{\bullet,0}}^{\cart} \ar@{<-}[r]_{f_\bullet^!} & \DD^{!,h}(\Delta^{\op})_{(\widetilde{Y} \widetilde{\times}_Y \delta(Y))_{\bullet,0}}^{\cart} 
 }\]
the functors $g^!$, and $G^!$ are equivalences by homological descent (Proposition~\ref{PROPHODESC}). 

{\em Conservativity} (H3$(f)$): The axioms (H3$(f_n)$) which hold by Step 1 imply that $f_\bullet^!$ is conservative. Thus also $f^!$ is conservative. 

{\em Commutation with homotopy colimits} (H1$(f)$): 
Consider the diagram
\[ \xymatrix{ \DD^{!,h}(I)_{p^*Y} \ar@<5pt>[r]^{p_!} \ar[d]^{(\pr_2^*g)^!} & \ar@<5pt>[l]^{p^*} \DD^{!,h}(\cdot)_{Y} \ar[d]^{g^!} \\
    \DD^{!,h}(I \times \Delta^{\op})_{\pr_2^*(\widetilde{Y}  \widetilde{\times}_Y \delta(Y))_{\bullet,0}}^{\pr_1-\cart} \ar@<5pt>[r]^-{\Box_! \pr_{2,!}} & \ar@<5pt>[l]^-{\pr_2^*} \DD^{!,h}(\Delta^{\op})_{(\widetilde{Y} \widetilde{\times}_Y \delta(Y))_{\bullet,0}}^{\cart}  } \]
    in which the square with the functors pointing to the left is 2-commutative. 
    Since $g^!$ and $(\pr_2^*g)^!$ are equivalences, we have that the exchange
   \begin{equation}\label{eqh1} \Box_! \pr_{2,!} (\pr_2^*g)^! \cong g^! p_!  \end{equation}
   is an isomorphism and similarly for $G$. 
   
By (H1$(f_n)$) established in Step 1, we have
\[  \pr_{2,!} (\pr_2^*f_\bullet)^! \cong (f_\bullet)^! \pr_{2,!}  \]
and thus   
\[  \Box_! \pr_{2,!} (\pr_2^*f_\bullet)^! (\pr_2^*g)^! \cong  \Box_! (f_\bullet)^! \pr_{2,!} (\pr_2^*g)^!   \]
and using the commutation of $f_\bullet^!$ with left Cartesian projectors (Lemma~\ref{LEMMACARTPROJ} applied using (H1''$(f_\bullet)$) and (H1''($f_n$)) which hold by Proposition~\ref{PROPHEREDITYLOC1}, and (H2$(f_n)$) which holds by Step 1) we get 
\[  \Box_! \pr_{2,!} (\pr_2^*G)^! (p^*f)^! \cong   (f_\bullet)^!  \Box_! \pr_{2,!} (\pr_2^*g)^!.   \]
Using the exchange isomorphisms (\ref{eqh1}), we get
\[ G^! p_! (p^*f)^! \cong   (f_\bullet)^!  g^! p_!   \]
and thus
\[ G^! p_! (p^*f)^! \cong   G^! f^! p_!   \]
and using that $G^!$ is an equivalence, we get (H1$(f)$):
\[ p_! (p^*f)^! \cong  f^! p_!.   \]

{\em Base change} (H2$(f)$):
Let 
\[ \xymatrix{ W \ar[r]^F \ar[d]_G & Z  \ar[d]^g \\ 
X \ar[r]_f & Y
} \] 
be a homotopy Cartesian square in $\mathcal{M}$.
By homological descent it suffices to see that the square 
\[  \xymatrix{ \DD^{!,h}(\Delta^{\op})^{\cart}_{(\widetilde{Y} \widetilde{\times}_Y \delta(W))_{\bullet,0}} \ar@{<-}[r]^-{F_\bullet^!} \ar[d]_{\Box_!G_{\bullet,!}}  & \DD^{!,h}(\Delta^{\op})^{\cart}_{(\widetilde{Y} \widetilde{\times}_Y \delta(Z))_{\bullet,0}} \ar[d]^{\Box_!g_{\bullet,!}}  \\
\DD^{!,h}(\Delta^{\op})^{\cart}_{(\widetilde{Y} \widetilde{\times}_Y \delta(X))_{\bullet,0}} \ar@{<-}[r]_-{f_\bullet^!}   & \DD^{!,h}(\Delta^{\op})^{\cart}_{(\widetilde{Y} \widetilde{\times}_Y \delta(Y))_{\bullet,0}}  }
 \]
where, $F_\bullet, f_\bullet, G_\bullet, g_\bullet$ denote the pull-backs, is 2-commutative via the natural exchange morphism.
By (H2$(f_n)$) established in Step 1, the natural exchange morphism
 \[   G_{\bullet,!} F_\bullet^! \rightarrow f_\bullet^! g_{\bullet, !}   \]
 is an isomorphism and hence
 \[  \Box_! G_{\bullet, !} F^!_\bullet \rightarrow  \Box_! f^!_\bullet g_{\bullet,!}   \]
 is an isomorphism. 
 Since $\Box_!$ commutes with $f^!_\bullet$ (Lemma~\ref{LEMMACARTPROJ} applied using (H1''$(f_\bullet)$) and (H1''($f_n$)) which hold by Proposition~\ref{PROPHEREDITYLOC1}, and (H2$(f_n)$) which holds by Step 1) we get that
 \[  \Box_! G_{\bullet, !} F^!_\bullet \rightarrow   f^!_\bullet \Box_! g_{\bullet,!}   \]
 is an isomorphism. 
\end{proof}

\begin{SATZ}\label{SATZHEREDITYLOC3}
\begin{enumerate}
\item In the situation of \ref{PARSETTINGLOCALITYLEFT}, every morphism $f: (I, S) \rightarrow (I, T)$ in $\Cat(\mathcal{M})$ point-wise in $k-C$ (with $k$ independent of $i \in I$)  satisfies {\em (H1''$(f)$)} w.r.t.\@ $\DD^{!, h}$.

\item In the situation of \ref{PARSETTINGLOCALITYRIGHT}, every morphism $f: (I, S) \rightarrow (I, T)$ in $\Cat^{\op}(\mathcal{M}^{\op})$ point-wise in $k-C$ (with $k$ independent of $i \in I$) satisfies {\em (C1''$(f)$)} w.r.t.\@ $\DD^{*, h}$.
\end{enumerate}
\end{SATZ}

\begin{proof}We prove 1.\@ the other case being dual. 
Recall Definitions~\ref{DEFFPS}--\ref{DEFFPSOBJ}.
Consider a $k-C$ morphism $f: X \rightarrow Y$ in $\mathcal{M}$. Choose a replacement $\widetilde{Y} \in \mathcal{S}^{\amalg, \Delta^{\op}}$ of $Y$ with morphism 
 $(\Delta^{\op}, \widetilde{Y}) \rightarrow (\Delta^{\op}, \delta(Y))$ as in \ref{PARRESOLUTION}.
Consider the commutative square  in $\Cat(\mathcal{M})$ 
\[ \xymatrix{ (I, X) \ar[r]^-f \ar@{<-}[d]_{G} & (I, Y) \ar@{<-}[d]^{g}  \\
 (I \times \Delta^{\op}, (\widetilde{Y} \widetilde{\times}_Y \delta(X))_{\bullet,0}) \ar[r]_-{f_{\bullet}} & (I \times \Delta^{\op}, (\widetilde{Y} \widetilde{\times}_Y \delta(Y))_{\bullet,0}) 
 }\]
 in which $(\widetilde{Y} \widetilde{\times}_Y \delta(Y))_{n,0}$ is weakly equivalent to $\widetilde{Y}_n$ and thus is in $(\mathcal{M}^{(0)})^I$. 
In the resulting 2-commutative diagram 
\[ \xymatrix{ \DD^{!,h}(I)_{X} \ar@{<-}[r]^-{f^!} \ar[d]^-{G^!} & \DD^{!,h}(I)_{Y} \ar[d]^-{g^!}  \\
 \DD^{!,h}(I \times \Delta^{\op})_{(\widetilde{Y} \widetilde{\times}_Y \delta(X))_{\bullet,0}} \ar@{<-}[r]_{f_\bullet^!} & \DD^{!,h}(I \times \Delta^{\op})_{(\widetilde{Y} \widetilde{\times}_Y \delta(Y))_{\bullet,0}} 
 }\]
the functors $g^!$, and $G^!$ are fully-faithful by homological descent (Proposition~\ref{PROPHODESC}) and have a left adjoint. 
By Proposition~\ref{PROPHEREDITYLOC1} (H1''$(f_\bullet)$) holds true. Thus (H1''$(f)$) holds by Lemma~\ref{LEMMAWEIRD} 
using that 
\[ G_! f_\bullet^! \rightarrow f^! g_! \]
is an isomorphism by (H1$(f)$) and (H2$(f)$) which hold by Theorem~\ref{SATZHEREDITYLOC2}. 
\end{proof}

\begin{LEMMA}\label{LEMMACART}
In the situation of \ref{PARSETTINGLOCALITY2},
let $X_{i,\bullet}$, $i=1, \dots, l$ be simplicial objects in $\mathcal{M}$ with morphisms $X_{i,\bullet} \rightarrow \pi_{\Delta^{\op}}^*X_i$ inducing an isomorphism $\hocolim X_{i,\bullet} \cong X_i$ in the homotopy category of $\mathcal{M}$, and let $g: Y \rightarrow X_1, \dots, X_l$ be a collection of fibrations.
Denote $X_\bullet := \prod_i X_{i, \bullet}$, $X := \prod_i X_{i}$,  and $g_\bullet: X_\bullet \times_X Y \rightarrow X_{1,\bullet}, \dots, X_{l,\bullet}$ the pull-back. 
Then the diagram
\[ \xymatrix{ \DD'(\Delta^{\op})_{X_{1,\bullet}}^{\cart} \times \cdots \times \DD'(\Delta^{\op})_{X_{l,\bullet}}^{\cart} \ar[r]^-{\Box_! g^*_\bullet}  \ar@{<-}[d]^{(f_i^!)}_\sim & \DD'(\Delta^{\op})_{X_\bullet \times_X Y}^{\cart}  \ar@{<-}[d]^{f^!}_\sim  \\
 \DD'(\cdot)_{X_1} \times \cdots \times \DD'(\cdot)_{X_l} \ar[r]_-{g^*} &  \DD'(\cdot)_{Y}    }  \]
 is 2-commutative via the natural exchange morphism. 
If $\DD' \rightarrow \HH^{\cor}(\mathcal{M})$ is a fibered derivator (not a multiderivator) then $l=1$ is understood. 
\end{LEMMA}
\begin{proof}

The diagram 
\[ \xymatrix{ \DD'(\Delta^{\op})_{X_{1,\bullet}}^{} \times \cdots \times \DD'(\Delta^{\op})_{X_{l,\bullet}}^{} \ar[r]^-{g^*_\bullet}  \ar[d]^{(f_{i,!})} & \DD'(\Delta^{\op})_{X_\bullet \times_X Y}^{}  \ar[d]^{f_!}  \\
 \DD'(\cdot)_{X_1} \times \cdots \times \DD'(\cdot)_{X_l} \ar[r]_-{g^*} &  \DD'(\cdot)_{Y}    }  \]
 is 2-commutative, i.e.\@
 \[ f_! g^*_\bullet(-, \dots, -) \to  g^*(f_{i,!}-, \dots, f_{l,!}-) \]
 is an isomorphism.  We have $f_! \Box_! \cong f_!$ because $f^!$ has values in Cartesian objects, therefore also  
 \begin{equation}\label{eqbcbc} f_! \Box_! g^*_\bullet(-, \dots, -) \to  g^*(f_{i,!}-, \dots, f_{l,!}-) \end{equation}
 is an isomorphism.  Furthermore the $f_{i,!}$ and $f_!$ are equivalences when restricted to Cartesian objects because these restrictions are adjoint to the equivalences $f^!_i$, and $f^!$, respectively.
 Therefore also the exchange of (\ref{eqbcbc}) is an isomorphism. 
 \end{proof}

\begin{SATZ}\label{SATZLOCALBC}
In the situation of \ref{PARSETTINGLOCALITY2},
every tupel\footnote{If $\DD' \rightarrow \HH^{\cor}(\mathcal{M})$ is a fibered derivator (not a multiderivator) then $l=1$ is understood. 
} of morphisms $f_1, \dots, f_l \in \mathcal{M}$ in $k-C$ satisfies {\em (CH$(f_1, \dots, f_l)$)}. 
\end{SATZ}
 
\begin{proof}
Let 
\[ \xymatrix{  W \ar[r]^F \ar[d]_G & Z \ar[d]^g \\ 
 X_1, \dots, X_l \ar[r]_{(f_i)} & Y_1, \dots, Y_l
} \] 
be a homotopy Cartesian diagram in which the $f_i$ are in $k-C$. We will show that  
\[ G^*( f_1^!-, \dots, f_k^!-) \rightarrow F^! g^*(-, \dots, -) \] 
is an isomorphism. The other statement is dual (for $l=1$ only). 

{\bf Step 1.} We first prove the special case in which the $f_i: X_i \rightarrow Y_i$ are morphisms in $\mathcal{M}^{(0)}$ which lies in $C$. 
Choose a replacement $\widetilde{Z} \in \mathcal{S}^{\amalg, \Delta^{\op}}$ of $Z$ with morphism 
 $(\Delta^{\op}, \widetilde{Z}) \rightarrow (\Delta^{\op}, \delta(Z))$ as in \ref{PARRESOLUTION}.
 
We may assume (cf.\@ Lemma~\ref{LEMMAAXIOMS}) that the above square is of the form  (setting $X:=\prod_i X_i$ and $Y:= \prod_i Y_i$)
\[  \xymatrix{ X \times_{Y} Z \ar[r]^-F \ar[d]_-G  & Z  \ar[d]^g  \\
X_1, \dots, X_l \ar[r]_{(f_i)} &    Y_1, \dots, Y_l  }
 \]
with  $f_i$ fibrations. 
Consider the diagram
 \[ \xymatrix{ 
 \DD'(\Delta)_{\delta(X) \times_{\delta(Y)}  \widetilde{Z}^{\op}}^{\cocart} \ar@{<-}[r]^-{R^*}_-\sim \ar@{<-}[d]^{F_\bullet^!} &  \ar@{<-}[d]^{\delta(F)^!} \DD'(\Delta)_{\delta(X \times_Y Z)^{\op}}^{\cocart} \ar@{<-}[r]^-{\delta(G)^*} &   \DD'(\Delta)_{\delta(X_1)^{\op}}^{\cocart} \times \cdots \times \DD'(\Delta)_{\delta(X_l)^{\op}}^{\cocart}   \ar@{<-}[d]^{(\delta(f_i)^!)}  \\
 \DD'(\Delta)_{\widetilde{Z}^{\op}}^{\cocart}  \ar@{<-}[r]_-{r^*}^-\sim  &   \DD'(\Delta)_{\delta(Z)^{\op}}^{\cocart}  \ar@{<-}[r]_-{\delta(g)^*} &  \DD'(\Delta)_{\delta(Y_1)^{\op}}^{\cocart} \times \cdots \times \DD'(\Delta)_{\delta(Y_l)^{\op}}^{\cocart} } \]
 in which the right hand side square is clearly equivalent to the square whose 2-commutativity is in question. 
 $F_\bullet^!$ preserves coCartesian objects by Lemma~\ref{LEMMAAXIOMSM0}, because the $F_n$ are morphisms in $\mathcal{M}^{(0)}$ and in $C$. 
 Since $R^*$ is an equivalence by cohomological descent it suffices to show that 
 \begin{equation} \label{eqbcleft} R^* \delta(G)^* \delta(f)^! \rightarrow R^* \delta(F)^! \delta(g)^* \end{equation}
 is an isomorphism. 
  Now the  left hand side square is 2-commutative because $R^*$ and $r^*$ are equivalences. 
  And the composite square is 2-commutative by applying Lemma~\ref{LEMMAAXIOMSM0} point-wise. Hence (\ref{eqbcleft}) is an isomorphism. 

{\bf Step 2.}  We continue proving the special case in which the $Y_i$ are in $\mathcal{M}^{(0)}$ but the $X_i$ are arbitrary. If $k=0$ also the $X_i$ are in $\mathcal{M}^{(0)}$ and we are in the special case considered in Step 1.
If $k \ge 1$ consider $k$-atlases $f'_i: U_i \rightarrow X_i$ (w.l.o.g.\@ a fibration) such $U_i \rightarrow Y_i$ is in $C$ and the diagram (setting $X:=\prod_i X_i$ and $U:=\prod_i U_i$): 
\[ \xymatrix{ 
W \times_X U \ar[r]^-{G'} \ar[d]_{F'} & U_1, \dots, U_l \ar[d]^{(f'_i)} \\
W \ar[r]^-G \ar[d]_F & X_1, \dots, Y_l \ar[d]^{(f_i)} \\
Z \ar[r]_-g  & Y_1, \dots, Y_l \\
 }\]
  Since $(F')^!$ is conservative (Theorem~\ref{SATZHEREDITYLOC2}) it suffices to show that 
 \[ (F')^! G^* (f_1^!-, \dots, f_l^!-) \rightarrow (F')^! F^! g^* (-, \dots, -) \]
 is an isomorphism. 
 Using the induction hypothesis this means that
 \[ (G')^* ((f'_1)^! f_1^! -, \dots, (f'_l)^! f_l^! -)  \rightarrow (F')^! F^! g^* (-, \dots, -) \]
 has to be an isomorphism, which is true by Step 1. 

{\bf Step 3.}  It remains to see the general case. W.l.o.g.\@ the $f_i$ are fibrations in $k-C$. Choose a replacement $\widetilde{Y}_i \in \mathcal{S}^{\amalg, \Delta^{\op}}$ of $Y_i$ with morphism 
 $(\Delta^{\op}, \widetilde{Y}_i) \rightarrow (\Delta^{\op}, \delta(Y_i))$ as in \ref{PARRESOLUTION}.
 
Consider the (point-wise homotopy Cartesian) square in $\mathcal{M}^{\Delta^{\op}}$ setting $\widetilde{Y} := \prod \widetilde{Y}_i$ and $Y := \prod Y_i$: 
\[  \xymatrix{ \widetilde{Y}\times_{\delta(Y)} \delta(W) \ar[r]^-{\widetilde{F}} \ar[d]_{\widetilde{G}}  & \widetilde{Y} \times_{\delta(Y)} \delta(Z)  \ar[d]^{\widetilde{g}}  \\
 \widetilde{Y}_1 \times_{\delta(Y_1)} \delta(X_1), \dots,  \widetilde{Y}_l \times_{\delta(Y_l)} \delta(X_l)  \ar[r]_-{(\widetilde{f}_i)}   & \widetilde{Y}_1, \dots, \widetilde{Y}_l  }
 \]
  Using Lemma~\ref{LEMMACART} and homological descent it suffices to show that the following square is 2-commutative 
 \[ \xymatrix{ \DD'(\Delta^{\op})_{ \widetilde{Y}\times_{\delta(Y)} \delta(W)}^{\cart} \ar@{<-}[d]_-{\Box_! \widetilde{G}^*} \ar@{<-}[r]^-{\widetilde{F}^!} &  
  \DD'(\Delta^{\op})_{\widetilde{Y} \times_{\delta(Y)} \delta(X)}^{\cart}  \ar@{<-}[d]^-{\Box_! \widetilde{g}^*}  \\
 \DD'(\Delta^{\op})_{\widetilde{Y}_1 \times_{\delta(Y_1)} \delta(Z)}^{\cart} \times \cdots \times  \DD'(\Delta^{\op})_{\widetilde{Y}_l \times_{\delta(Y_l)} \delta(Z)}^{\cart}
\ar@{<-}[r]_-{(\widetilde{f}_i^!)_i}   &   \DD'(\Delta^{\op})_{\widetilde{Y}_1}^{\cart}  \times \cdots \times  \DD'(\Delta^{\op})_{\widetilde{Y}_l}^{\cart}  } \]
 via the natural exchange, i.e.\@ that
  \[ \Box_! \widetilde{G}^* ( \widetilde{f}^!_1-, \dots, \widetilde{f}^!_l- )  \rightarrow  \widetilde{F}^!   \Box_! \widetilde{g}^* (-, \dots, -) \]
  is an isomorphism. 
 Since $\widetilde{F}^!$  commutes with the left Cartesian projector (Lemma~\ref{LEMMACARTPROJ} using Theorems~\ref{SATZHEREDITYLOC2}--\ref{SATZHEREDITYLOC3}) this follows from the fact that point-wise 
  \[ \widetilde{G}_n^* (\widetilde{f}_{1,n}^!-, \dots, \widetilde{f}_{l,n}^!-) \rightarrow \widetilde{F}_n^!  \widetilde{g}_{n}^*(-, \dots, -) \]
  is an isomorphism by Step 2. 
\end{proof}

\section{Coherent atlases }\label{SECTCOHATLAS}

\begin{PAR}\label{PARSETTINGEXT}
Let $(\mathcal{M}, \Fib, \mathcal{W})$ be a simplicial category with fibrant objects (\ref{PARCWFO}). 
For the construction of extensions of $*,!$-formalisms, or of six-functor-formalisms, we consider the following setting
\begin{itemize}
\item $\mathcal{M}^{(\old)} \subset \mathcal{M}$  a homotopy strictly full subcategory, closed under homotopy pull-backs;
\item $\mathcal{M}^{(\new)} \subset \mathcal{M}$  a homotopy strictly full subcategory, closed under homotopy pull-backs;
\item a class $C$ of morphisms in $h(\mathcal{M})$.
\end{itemize}
such that 
\begin{enumerate}
\item[(A1)] $C$ contains the isomorphisms, is stable under composition and homotopy pull-back;
\item[(A2)]
For a homotopy Cartesian diagram in $\mathcal{M}$
\[ \xymatrix{ W \ar[r]  \ar[d] & Y \ar[d] \\
Z \ar[r] & K }\]
 with $K \in \mathcal{M}^{(\new)}$ and $Y, Z \in \mathcal{M}^{(\mathrm{old})}$ also $W \in \mathcal{M}^{(\old)}$;
\item[(A3)] For any $X \in \mathcal{M}^{(\new)}$ there is a morphism $U \rightarrow X$ in $\mathcal{M}$, called {\bf atlas}, in $C$ with $U \in \mathcal{M}^{(\old)}$.
\end{enumerate}
\end{PAR}

Note that it is {\em not} assumed that $\mathcal{M}^{(\old)} \subset \mathcal{M}^{(\new)}$ and, in fact, the construction will be also applied in a case in which this is not the case (cf.\@ Example~\ref{MAINEXEXT2}).
By abuse of notation, we also denote by $C$ the preimage of $C$ in $\mathcal{M}$.  The importance the class $C$ of will not become clear until \ref{PARAXIOMSEXT}.
The construction will be applied exclusively in the following two settings: 

\begin{BEISPIEL}\label{MAINEXEXT1}
Let $(\mathcal{M}, \Fib, \mathcal{W})$ be as in \ref{PARSETTINGTV} and let $k \ge 1$ be an integer. 

Take $\mathcal{M}^{(\old)} := \mathcal{M}^{(k-1)}$ $((k-1)$-geometric stacks$)$ and $\mathcal{M}^{(\new)} := \mathcal{M}^{(k)}$ $(k$-geometric stacks$)$ and $C$ is
the class $(k-1)-C$.

(A1) is Proposition~\ref{PROPGEOMETRIC}, 1.

(A2) follows from Proposition~\ref{PROPGEOMETRIC}, 3.

(A3) is true (even with $U \in \mathcal{M}^{(\mathrm{0})}$) by definition of $k$-geometric. 
\end{BEISPIEL}

\begin{BEISPIEL}\label{MAINEXEXT2}With the notation as in Appendix~\ref{APPENDIXGEOMETRIC}, 
take $\mathcal{M} := \mathcal{SET}^{\mathcal{S}^{\op}}$ be the category of  ordinary pre-sheaves (with {\em trivial} structure as simplicial category of fibrant objects),
$\mathcal{M}^{(\old)} := h(\mathcal{S}^{\amalg})$ and $\mathcal{M}^{(\new)} := \iota(\mathcal{SH}^{(0)})$ and $C$ is
the class of local isomorphisms.

(A1) is clear.

(A2) By the geometricity of the site (cf.\@ Definition~\ref{DEFGEOMETRIC}) for each $Y, Z \in \mathcal{S}$ and $X \in \mathcal{SH}^{(0)}$  the fiber product $(h(Y \times Z)) \times_{X \times X} X \cong h(Y) \times_X  h(Z)$ is representable. 
Since in pre-sheaves fiber products commute with arbitrary coproducts the statement follows. 

(A3) By definition each object in $\iota(\mathcal{SH}^{(0)})$ is of the form $\iota L h(X)$ for some $X \in  \mathcal{S}^{\amalg}$. Thus we may take the unit $h(X) \rightarrow \iota L h(X)$.
\end{BEISPIEL}

Let $\overline{C} := C \cap \Fib$. This class is closed under (ordinary) pullback in the sense that the pullback of the morphisms in $\overline{C}$ along arbitrary morphisms exist and are in $\overline{C}$ again.
Furthermore, combining (A3) with fibrant replacement, for each $X \in \mathcal{M}^{(\new)}$ there is morphism $f: Y \rightarrow X$ in $\overline{C}$ with $Y \in \mathcal{M}^{(\old)}$. 

\begin{DEF}\label{DEFATLAS}Let $I$ be an inverse diagram with finite matching diagrams. 
Let $X \in (\mathcal{M}^{(\new)})^I$.
We call a morphism $U \Rightarrow X$ in $\mathcal{M}^{I}$ an {\bf atlas}, if $U \in (\mathcal{M}^{(\old)})^{I}$ and it is a $\overline{C}$-morphism in the sense of Definition~\ref{DEFCMORPH}. 
\end{DEF}

\begin{PAR}Let $I$ be as in Definition~\ref{DEFATLAS}.
Let $X \in (\mathcal{M}^{(\new)})^I$.
We will show that exists an  atlas of $X$ by
 applying the construction of Proposition~\ref{PROPLIMDIRECTED} to the class $\overline{C}$. 
\end{PAR}

\begin{PROP}\label{PROPCOV}
Let the situation be as in \ref{PARSETTINGEXT}. Let $I$ be an inverse diagram with finite matching diagrams then 
\begin{enumerate}
\item for every $X \in  (\mathcal{M}^{(\new)})^{I}$ there exists an atlas in the sense of Definition~\ref{DEFATLAS}.  It is point-wise  in $\overline{C}$. 
\item an atlas $U' \Rightarrow \iota^* X$ on any final subdiagram $\iota: I' \hookrightarrow I$ can be extended to an altas $U \Rightarrow X$;
\item any two atlases $U \Rightarrow X$, $U' \Rightarrow X$ can be dominated by a third.
\end{enumerate}
\end{PROP}
\begin{proof}
Apply Propositions~\ref{PROPLIMDIRECTED} and \ref{PROPCOVERDIRECTED} to the categories $\mathcal{S}:=\mathcal{M}^{(\new)}$, $\mathcal{S}_0:=\mathcal{M}^{(\old)}$ and the class $\overline{C}$. 
\end{proof}

\begin{PAR}
We will define a 2-pre-multiderivator $\HH^{\cor,\cov}(\mathcal{M}^{(\new)})$ of homotopy multicorrespondences, as before, equipped with a fixed atlas in the sense of Definition~\ref{DEFATLAS}.
This is almost literally the same construction as the construction of $\SSS^{\cor, \comp}$ in \cite[Section~6]{Hor17}, where multicorrespondences are equipped with compactifications instead of atlases. 

More precisely: 
\end{PAR}

\begin{DEF}\label{DEFCCOV}
Let $\tau$ be a tree with symmetrization $\tau^S$, $I$ a (usual) diagram. 
We define a category with weak equivalences 
\[ \Cor_I^{\cov}(\tau^S) \]
whose objects are atlases $U \rightarrow X$ in $\mathcal{M}^{\tw (\tau \times I)}$, with $X \in (\mathcal{M}^{(\new)})^{\tw (\tau \times I)}$ admissible 
and with $U \in (\mathcal{M}^{(\old)})^{\tw (\tau \times I)}$ ($U$ does not have to be admissible!)
and the morphisms are diagrams
\begin{equation}\label{eqmorphatlas}  \vcenter{ \xymatrix{ U \ar[r] \ar[d] & X \ar[d] \\
U' \ar[r]  & X'  } }
\end{equation}
in which $X \rightarrow X'$ is a point-wise weak equivalence and $U \rightarrow U'$ is {\em arbitrary}. All morphisms are weak equivalences.

This defines a functor
\[ \Cor_I^{\cov}: \Delta_S \rightarrow \mathcal{CATW} \]
from symmetric trees to categories with weak equivalences. 
Note that any automorphism of $\tau^S$ acts on $\mathcal{M}^{\tw (\tau \times I)}$ because $\mathcal{M}$ is symmetric. Equivalently, observe
that functors $\tw (\tau \times I) \rightarrow \mathcal{M}$ are the same as functors of usual categories $(\tw (\tau \times I))^\circ \rightarrow \mathcal{M}$ (cf.\@ \ref{PARCIRC}) and that 
every functor $\tau^S \rightarrow (\tau')^S$ induces an obvious functor $(\tw (\tau \times I))^\circ \rightarrow (\tw (\tau' \times I))^\circ$.
\end{DEF}

We will need the following   
\begin{LEMMA}\label{LEMMACATLF}
\begin{enumerate}
\item
If $I$ is in $\Catlf$ (locally finite diagrams, cf.\@ Definition~\ref{DEFDIAGRAMCAT}) then $\tw I$ is inverse and locally finite and has finite matching diagrams. 
\item
If $I$ is in $\Dirlf$ then also ${}^{\uparrow \uparrow \downarrow}I$ and $\twwc{I}$ are in $\Dirlf$.
\end{enumerate}
\end{LEMMA}
\begin{proof}\cite[Lemma~6.1]{Hor17}.
\end{proof}

\begin{LEMMA}\label{LEMMAPROPCONSTRSYMMULTI2}
 Let $I \in \Catlf$.
\begin{enumerate}
\item The forgetful functor
\[  \Cor_I^{\cov}(\tau^S)_{(U_o \rightarrow X_o)_{o \in \tau}}[\mathcal{W}^{-1}_{(U_o \rightarrow X_o)_{o \in \tau}}] \rightarrow  \Cor_I(\tau^S)_{(X_o)_{o \in \tau}}[ \mathcal{W}^{-1}_{(X_o)_{o \in \tau}}] \]
(cf.\@ Definition~\ref{DEFCOR}) is an equivalence. 
\item The strict functor $\tau^S \mapsto \Cor_I^{\cov}(\tau^S)$ satisfies the axioms 1 and 2 of \cite[Proposition~A.3]{Hor17}.
\end{enumerate}
\end{LEMMA}
\begin{proof}
1.\@ By definition of weak equivalences on the left hand side it suffices to see that we have an equivalence
\[  \Cor_I^{\cov}(\tau^S)_{(\overline{X}_o \rightarrow X_o)_{o \in \tau}}[(\mathcal{W}')^{-1}] \rightarrow  \Cor_I(\tau^S)_{(X_o)_{o \in \tau}} \]
where $\mathcal{W}'$ is the preimage of the identities, i.e.\@ morphisms (\ref{eqmorphatlas}) in which the morphism $X \rightarrow X'$ is the identity. 
We will apply Lemma~\ref{LEMMALOC} below and have to show that the above functor is surjective on objects and morphisms and satisfies the axioms (i) and (ii). 

 The subdiagram \[ \bigcup_{\sigma \in \tau} \tw I \hookrightarrow \tw (\tau \times I)^\circ \] 
 is final (cf.\@ \ref{PARCIRC} for the notation). 
Hence
 any partial atlas on the left hand side can be extended by Proposition~\ref{PROPCOV}. Note that by Lemma~\ref{LEMMACATLF} the diagram $\tw I$ and thus also $(\tw (\tau \times I))^\circ$
 is an inverse diagram with finite matching diagrams for any tree or symmetric tree $\tau$. Thus the functor is surjective on objects. The same argument for
 \[ \bigcup_{o \in \tau} (\tw I) \times \Delta_1 \hookrightarrow (\tw (\tau \times I))^\circ \times \Delta_1 \]
shows that the functor is surjective on morphisms.  
The same argument for 
 \[ (\bigcup_{o \in \tau} (\tw  I) \times \Box) \cup ((\tw (\tau \times I))^\circ \times \righthalfcup) \hookrightarrow (\tw (\tau \times I))^\circ \times \Box \]
 shows axiom (i). 
The same argument for
 \[ (\bigcup_{o \in \tau} (\tw  I)  \times \mathbin{\rotatebox[origin=c]{45}{$\boxtimes$}}) \cup ( (\tw (\tau \times I))^\circ \times \mathbin{\rotatebox[origin=c]{45}{$\Box$}}) \hookrightarrow (\tw (\tau \times I))^\circ \times \mathbin{\rotatebox[origin=c]{45}{$\boxtimes$}} \]
 shows axiom (ii). Here $\mathbin{\rotatebox[origin=c]{45}{$\Box$}}$ and $\mathbin{\rotatebox[origin=c]{45}{$\boxtimes$}}$ denote the diagrams
  
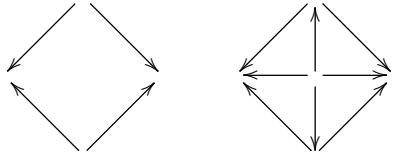
\[ \xymatrix{ &  \ar[ld] \ar[rd] \\
   &  &  \\
  &  \ar[lu] \ar[ru] } \qquad 
 \xymatrix{ &  \ar[ld] \ar[rd] \\
   &  \ar[u] \ar[d] \ar[l] \ar[r] &  \\
  &  \ar[lu] \ar[ru] }.\]
  Note that $F(f) = F(f')$ implies that the underlying diagram (without atlas) can be extended from $\mathbin{\rotatebox[origin=c]{45}{$\Box$}}$ to $\mathbin{\rotatebox[origin=c]{45}{$\boxtimes$}}$in a trivial way. 

2.\@ Using the equivalence of 1.\@ and the surjectivity on objects of the forgetful functor just established this follows from Lemma~\ref{LEMMAPROPCONSTRSYMMULTI1}.
\end{proof}

\begin{LEMMA}\label{LEMMALOC}
Let $F: \mathcal{C} \rightarrow \mathcal{D}$ be a functor which is surjective on objects and morphisms. 
Denote by $\mathcal{W}$ the class of morphisms in $\mathcal{C}$ that are mapped to an identity in $\mathcal{D}$. 
Assume that 
\begin{enumerate}
\item[(i)] A solid diagram of the form 
\[ \xymatrix{  \ar@{.>}[r]^{w'} \ar@{.>}[d]_{f'} &  \ar[d]^f \\
  \ar[r]_{w} &  }\]
with $w \in \mathcal{W}$ can always be completed to a commutative diagram as indicated such that $w' \in \mathcal{W}$. 
\item[(ii)] A solid diagram of the form 
 \[ \xymatrix{ &  \ar[ld]_{w} \ar[rd]^f \\
   &  \ar@{.>}[u]|{w''} \ar@{.>}[d]|{w'''} \ar@{.>}[l] \ar@{.>}[r] &  \\
  &  \ar[lu]^{w'} \ar[ru]_{f'} } \]
 in which $w, w' \in \mathcal{W}$ and such that $F(f) = F(f')$ can always be completed to a commutative diagram as indicated.
\end{enumerate}
Then the functor $F$ induces an equivalence $\mathcal{C}[\mathcal{W}^{-1}] \cong \mathcal{D}$. 
\end{LEMMA}

 \begin{proof} \cite[Lemma~6.19]{Hor17}.
 \end{proof}

\begin{DEF}\label{DEFHCORCOV}
We define a symmetric 2-pre-multiderivator $\HH^{\cor, \cov}(\mathcal{M}^{(\new)})$ with domain $\Catlf$. Let $I \in \Catlf$ be a locally finite diagram. The 2-multicategory $\HH^{\cor, \cov}(\mathcal{M}^{(\new)})(I)$ is the symmetric 2-multicategory of \cite[Proposition~A.3]{Hor17}  constructed from the functor $C_I^{\cov}$ (cf.\@ Definition~\ref{DEFCCOV}).
Those 2-multicategories are equipped with strict pull-back functors
\[ \alpha^*: \HH^{\cor, \cov}(\mathcal{M}^{(\new)})(J) \rightarrow \HH^{\cor, \cov}(\mathcal{M}^{(\new)})(I). \]
for each functor $\alpha: I \rightarrow J$. For a natural transformation $\mu: \alpha \Rightarrow \beta$ we get a pseudo-natural transformation 
\[ \alpha^* \Rightarrow \beta^* \]
as follows: $\mu$ might be seen as a functor $I \times \Delta_1 \rightarrow J$ and so each admissible object with atlas $U \rightarrow X$ in $\mathcal{M}^{\tw I} $ gives rise to an admissible object with altlas
 $\mu^*U \rightarrow \mu^* X$ in $\mathcal{M} ^{\tw (\Delta_1 \times I)}$ which constitutes a 1-morphism from $\alpha^* X$ to $\beta^* X$. For each 1-morphism $\xi: X \rightarrow Y$ the 
 square
 \[ \xymatrix{ \alpha^*X \ar[r]^{\alpha^* \xi} \ar[d]_{\mu^* X} & \alpha^*Y \ar[d]^{\mu^*Y} \\
 \beta^*X \ar[r]_{\beta^* \xi} & \beta^*Y } 
  \]
  commutes up to a uniquely determined 2-isomorphism using the pseudo-functor
  \[ \Delta_1^2 \rightarrow \HH^{\cor, \cov}(\mathcal{M}^{(\new)})(I) \] 
  obtained from $\mu^* \xi \in \mathcal{M} ^{\tw (\Delta_1^2 \times I)}$ via Lemma~\ref{LEMMAPF}, 1.\@ and the equivalence of categories from Lemma~\ref{LEMMAPROPCONSTRSYMMULTI2}, 1.
One checks that this defines a pseudo-functor:
\[ \Fun(I, J) \rightarrow \Fun^{\mathrm{strict}}(\HH^{\cor, \cov}(\mathcal{M}^{(\new)})(J), \HH^{\cor, \cov}(\mathcal{M}^{(\new)})(I)) \]
where $\Fun^{\mathrm{strict}}$ is the 2-category of strict 2-functors, pseudo-natural transformations, and modifications.
\end{DEF}

We can read off the construction: 
\begin{itemize}
\item The objects of $\HH^{\cor, \cov}(\mathcal{M}^{(\new)})(I)$ are morphisms $U \rightarrow X$ in $\mathcal{M}^{\tw I}$ such that $X \in (\mathcal{M}^{(\new)})^{\tw I}$ is admissible and $U \rightarrow X$ is an atlas in the sense of Definition~\ref{DEFATLAS}.
\item Every 1-multimorphism $(U_1 \rightarrow X_1), \dots, (U_n \rightarrow X_n) \rightarrow (U_Y \rightarrow Y)$ is 2-isomorphic to one of the form 
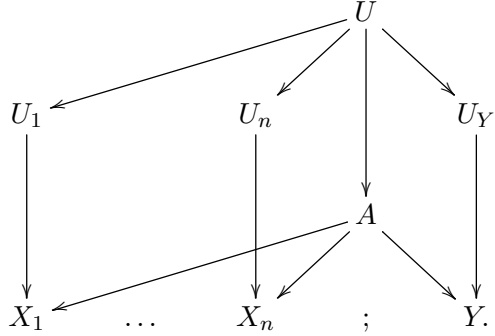
\[ \xymatrix{ & & & U \ar[dd] \ar[llld] \ar[ld] \ar[rd] \\
U_{1}  \ar[dd] &  & U_{n}  \ar[dd]& & U_Y \ar[dd]   \\
 & & & A \ar[llld] \ar[ld] \ar[rd] \\
X_{1} & \dots & X_{n} & ; & Y. }  \]
where the underlying multicorrespondence is as in \ref{BEISPIELMULTICOR} and $U \rightarrow A$ is an atlas in the sense of Definition~\ref{DEFATLAS}.
\item 2-morphisms between the previous multimorphisms are given by 
commutative squares
\[ \xymatrix{ U \ar[r] \ar[d] & U' \ar[d]  \\
A \ar[r] & A'  }\]
in which $A \rightarrow A'$ is a point-wise weak equivalence and $U \rightarrow U'$ is arbitrary, 
compatible with the morphisms to the $(U_i \rightarrow X_i)$ and $(U_Y \rightarrow Y)$. Finally, all morphisms are formally inverted. 
As category of 1-morphsims this gives back --- up to equivalence --- the corresponding homotopy category of multicorrespondences forgetting all atlases. 
\item the action of the symmetric group is given by permuting the morphisms $A \rightarrow X_i$ together with their atlases. 
\end{itemize}

\begin{PROP}\label{PROPEQUIV}
The forgetful strict functor
\[ \HH^{\cor, \cov}(\mathcal{M}^{(\new)}) \rightarrow \HH^{\cor}(\mathcal{M}^{(\new)}) \]
is an equivalence of 2-pre-multiderivators with domain $\Catlf$. 
\end{PROP}
\begin{proof}
The proof of Lemma~\ref{LEMMAPROPCONSTRSYMMULTI2}, 1.\@ shows that the functor 
\[ \HH^{\cor, \cov}(\mathcal{M}^{(\new)})(I) \rightarrow \HH^{\cor}(\mathcal{M}^{(\new)})(I) \]
induces an equivalence of the categories of multimorphisms and also that 
the functor is surjective on objects.
\end{proof}

We will later need the following Lemma.  Recall Definition~\ref{DEFCECH}. 

\begin{LEMMA}\label{LEMMAPROPCECHKGEOM}
Let the setting be as in \ref{PARSETTINGEXT}. 
\begin{enumerate}
\item 
Let $f: U \rightarrow X$ be a morphism in $C$ with $X \in \mathcal{M}^{(\new)}$ and $U \in \mathcal{M}^{(\old)}$. 
Then $\Cech(f) \in (\mathcal{M}^{(\old)})^{\Delta^{\op}}$. 
\item 
 If $X$ and $Y$ are in $\mathcal{M}^{(\old), \Delta^{\op}}$ and $Z \in \mathcal{M}^{(\new)}$, we have  $(X \widetilde{\times}_{Z} Y)_{n,m} \in \mathcal{M}^{(\old)}$. 
\end{enumerate}
\end{LEMMA}
\begin{proof}1.\@ We show by induction that $\Cech_{n}(f)$ is in $\mathcal{M}^{(\old)}$ and that the morphism $\Cech_{n}(f) \to X$ is in $C$. 
Since $\Cech_{0}(f) = U$ this is the case for $n=0$ by assumption. Now we have a homotopy Cartesian diagram
\[ \xymatrix{ 
\Cech_{n+1}(f) \ar[r] \ar[d] & \Cech_{n}(f) \ar[d] \\
U \ar[r] & X   } \]
By induction $\Cech_{n}(f) \rightarrow X$ is in $C$ and  $\Cech_{n}(f)$ is in $\mathcal{M}^{(\old)}$. Therefore by the Axioms~\ref{PARSETTINGEXT}
also $\Cech_{n+1}(f)$ is in $\mathcal{M}^{(\old)}$, the morphism $\Cech_{n+1}(f) \rightarrow U$ is in $C$ and hence so is the composition $\Cech_{n+1}(f) \rightarrow U \rightarrow X$.

2.\@ follows from Lemma~\ref{LEMMAHOFP}, 2.\@ and the axioms~\ref{PARSETTINGEXT}.
\end{proof}

\section{Construction of extensions}\label{SECTMC}

\begin{PAR}\label{PARAXIOMSEXT}
Consider the setting as in \ref{PARSETTINGEXT}. 
Let 
\[ \DD^{(\old)} \to \HH^{\cor}(\mathcal{M}^{(\old)}) \] be an infinite (symmetric) fibered (multi)derivator with domain $\Dirlf$, and with stable, well-generated fibers.
The objective is to extend it to a left (symmetric) fibered (multi)derivator with domain $\Dirlf$
\[ \DD^{(\new)} \to \HH^{\cor, \cov}(\mathcal{M}^{(\new)}). \] 

We will make the following assumption: The restrictions
\[ \DD^{(\old),!}  \rightarrow \mathbb{M}^{(\old)} \quad (\text{resp. }\ \DD^{(\old),*} \rightarrow \mathbb{M}^{(\old),\op} ) \]
have (naive) extensions  $\DD^{!, h} \rightarrow \mathbb{M}$ as fibered derivator (resp.\@  $\DD^{*, h} \rightarrow \mathbb{M}^{\op}$ as fibered (multi)derivator) with stable, well-generated fibers, as well, such that
\begin{itemize}
\item[(D1)] for each morphism $f: X \rightarrow Y$ in $\mathcal{W}$ the functors $f^!$ (resp.\@ $f^*$) are equivalences.
\item[(D2)] For each morphism $f: X \rightarrow Y$ in $C$ (cf.\@ \ref{PARSETTINGEXT}) the following functors are equivalences
\begin{eqnarray*}
 \gamma^*: \DD^{*,h}(\Delta)_{\delta(Y)^{\op}}^{\cocart} &\cong& \DD^{*,h}(\Delta)_{\Cech(f)^{\op}}^{\cocart}  \\
\gamma^!: \DD^{!,h}(\Delta^{\op})_{\delta(Y)}^{\cart} &\cong& \DD^{!,h}(\Delta^{\op})_{\Cech(f)}^{\cart} 
\end{eqnarray*}
where $\gamma$ is the morphism of Lemma~\ref{LEMMAPROPCECH}, 2. 
Note that $\DD^{*,h}(\Delta)_{\delta(Y)^{\op}}^{\cocart} \cong \DD^{*,h}(\cdot)_{Y}$ (and $\DD^{!,h}(\Delta^{\op})_{\delta(Y)}^{\cart} \cong \DD^{!,h}(\cdot)_{Y}$) by (D1) and the acyclicity of $\Delta^{\op}$. 

\item[(CH)] Axiom (CH$(f_1, \dots, f_l)$) (cf.\@ \ref{PARSETTINGLOCALITY2}) holds true  for all tupels $(f_1, \dots, f_l)$ with $f_i \in C \cap \Mor(\mathcal{M}^{(\old)})$, i.e.\@
 for each homotopy Cartesian square in $\mathcal{M}^{(\old)}$
\[ \xymatrix{ X \ar[r]^F \ar[d]_G & X' \ar[d]^g \\ 
Y_1, \dots, Y_l \ar[r]_{(f_i)} & Y'_1, \dots, Y_l'
} \] 
with $f_i \in C$ for all $i$ 
the natural transformations $G^* (f^!_1-, \dots, f_l^!-) \rightarrow F^! g^*(-, \dots, -)$ and (for $l=1$) $F^* g^! \rightarrow G^! f^*$  are isomorphisms (w.r.t.\@ $\DD^{(\old)}$). 
In the non-multi-case $l=1$ is understood as otherwise the statement does not make sense.  
\item[(H1)] Axiom (H1$(f)$) (cf.\@ \ref{PARSETTINGLOCALITYLEFT}) holds true for all $f \in C \cap \Mor(\mathcal{M}^{(\old)})$, i.e.\@ $f^!$, as morphism of derivators, commutes with homotopy colimits  (w.r.t.\@ $\DD^{(\old)}$). 
\end{itemize}
\end{PAR}

\begin{PAR}\label{PARIDEA}
The idea behind the construction of the extension $\DD^{(\new)}$, discussed in the sequel, is the following: Let $\widetilde{X} \in (\mathcal{M}^{(\old)})^{\Delta^{\op}}$ be a presentation of a stack $X \in \mathcal{M}^{(\new)}$ (with fixed morphism $\widetilde{X} \rightarrow \delta(X)$). Here we stick to $\widetilde{X} := \Cech(U \rightarrow X)$ for an atlas $U \rightarrow X$, i.e.\@ a morphism in $C$ with $U \in  \mathcal{M}^{(\old)}$, because the axioms in \ref{PARAXIOMSEXT} have been adapted to this setting. 
The existence of an extension $\DD^{(\new)}$ as envisioned in \ref{PARAXIOMSEXT} would imply (using (D1) and (D2))  equivalences
\[ \DD^{(\new)}(\cdot)_{X} \cong \DD^{(\old)}(\Delta^{\op})_{\widetilde{X}}^{\cart} \cong  \DD^{(\old)}(\Delta)_{\widetilde{X}^{\op}}^{\cocart}.  \]

The idea is to define $\DD^{(\new)}(\cdot)_{X}$ as an ``ambidextrous'' version of the two categories on the right. 
Recall Definition~\ref{DEFFPSOBJ}. Property 4.\@ of Lemma~\ref{LEMMAHOFP} implies that one may regard the object $\widetilde{X}':=(\widetilde{X} \widetilde{\times}_{X} \widetilde{X})_{\bullet,\bullet}$ in $\Fun((\Delta^{\op})^2, \mathcal{M}^{(\old)})$ as an object
 
\[ \Flip(\widetilde{X}') \in \HH^{\cor}(\mathcal{M}^{(\old)})(\Delta^{\op} \times \Delta).  \]
More precisely, pull-back the diagram along the functor
\[ P: \tw(\Delta^{\op} \times \Delta) \rightarrow  (\Delta^{\op})^2 \]
given by mapping 
\[ (\Delta_i \leftarrow \Delta_j, \Delta_k \rightarrow \Delta_l) \mapsto (\Delta_i, \Delta_l)   \]
(cf.\@ also \ref{DEFFLIP}).
Property 4.\@ of Lemma~\ref{LEMMAHOFP} implies that the pull-back is an admissible diagram. 
Then there is a correspondence\footnote{This can be seen as a ``duality morphism'' $ (\Delta^{\op}, \widetilde{X})  \rightarrow (\Delta, \widetilde{X}^{\op}) \cong \mathbf{HOM}( (\Delta^{\op}, \widetilde{X}); (\cdot, \cdot))$ in $\widehat{\Cat}^{\cor}(\HH^{\cor}(\mathcal{M}^{(\old)}))$, cf.\@ Theorem~\ref{SATZDUALITY}. } in $\Cat(\HH^{\cor}(\mathcal{M}^{(\old)}))$:
\[ (\Delta^{\op}, \widetilde{X}) \leftarrow (\Delta^{\op} \times \Delta, \Flip(\widetilde{X}')) \rightarrow (\Delta, \widetilde{X}^{\op}) \]
and it is easy to see that the two morphisms induce equivalences
\[ \DD^{(\old)}(\Delta^{\op})_{\widetilde{X}}^{\cart} \cong \DD^{(\old)}(\Delta^{\op} \times \Delta)_{\Flip(\widetilde{X}')}^{1-\cart, 2-\cocart} \cong \DD^{(\old)}(\Delta)_{\widetilde{X}^{\op}}^{\cocart}.  \]
It is therefore reasonable to define
\[ \DD^{(\new)}(\cdot)_{X} := \DD^{(\old)}(\Delta^{\op} \times \Delta)_{\Flip(\widetilde{X}')}^{1-\cart, 2-\cocart}.  \]
The objective is to show that this is independent of the choice of presentation, and to turn this into a full derivator *,!-formalism (or even six-functor-formalism) again. In particular, one needs to enhance the construction to diagrams of correspondences in $\HH^{\cor}(\mathcal{M}^{(\new)})$. The problem is that the object $\Flip(\widetilde{X}')$ is not in an easy way functorial in $X$, not even in $U \rightarrow X$. For example, consider a correspondence
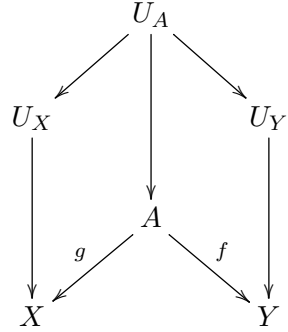
\[ \xymatrix{
& U_A \ar[dd] \ar[ld] \ar[rd] \\
U_X \ar[dd] & & U_Y \ar[dd] \\
& A \ar[ld]_g \ar[rd]^f \\
X & & Y
}\]
of stacks in $\mathcal{M}^{(\new)}$ with  compatible atlases. Setting $\widetilde{X}:=\Cech(U_X \to X)$, $\widetilde{Y}:=\Cech(U_Y \to Y)$, $\widetilde{A}:=\Cech(U_A \to A)$, 
 we have functors
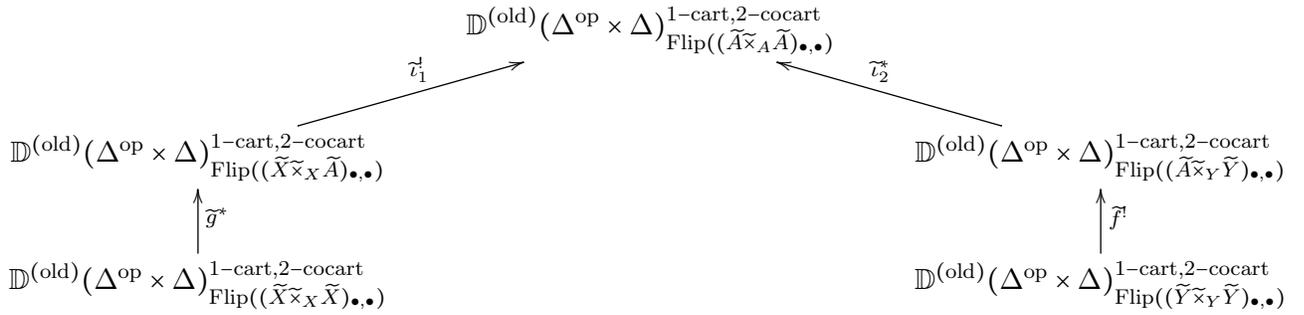
\[ \xymatrix{
 & \DD^{(\old)}(\Delta^{\op} \times \Delta)_{\Flip((\widetilde{A} \widetilde{\times}_{A} \widetilde{A})_{\bullet,\bullet})}^{1-\cart, 2-\cocart} \ar@{<-}[ld]_{\widetilde{\iota}_1^!} \ar@{<-}[rd]^{\widetilde{\iota}_2^*} \\
 \DD^{(\old)}(\Delta^{\op} \times \Delta)_{\Flip((\widetilde{X} \widetilde{\times}_{X} \widetilde{A})_{\bullet,\bullet})}^{1-\cart, 2-\cocart} \ar@{<-}[d]^{\widetilde{g}^*}  & & \DD^{(\old)}(\Delta^{\op} \times \Delta)_{\Flip((\widetilde{A} \widetilde{\times}_{Y} \widetilde{Y})_{\bullet,\bullet})}^{1-\cart, 2-\cocart} \ar@{<-}[d]^{\widetilde{f}^!} \\
\DD^{(\old)}(\Delta^{\op} \times \Delta)_{\Flip((\widetilde{X} \widetilde{\times}_{X} \widetilde{X})_{\bullet,\bullet})}^{1-\cart, 2-\cocart}  &  & \DD^{(\old)}(\Delta^{\op} \times \Delta)_{\Flip((\widetilde{Y} \widetilde{\times}_{Y} \widetilde{Y})_{\bullet,\bullet})}^{1-\cart, 2-\cocart} 
}\]
Axiom (CH) implies that the functors actually preserve the full subcategories of objects that are simultaneously 1-Cartesian and 2-coCartesian, as indicated. Furthermore, we will show that $\widetilde{\iota}_1^!$ and $\widetilde{\iota}_2^*$ are equivalences of categories, and that 
$\widetilde{f}^!$ has a left adjoint $\Box_!^{(1)}\widetilde{f}_!$. Therefore we can define a push-forward along this correspondence as
\[ \Box_!^{(1)} \widetilde{f}_! \ \circ\ (\widetilde{\iota}_2^*)^{-1}\ \circ\ \widetilde{\iota}_1^! \ \circ\ \widetilde{g}^*.  \]
The objective is to make this construction sufficiently coherent to get a left fibered (multi)derivator with domain $\Dirlf$
\[ \DD^{(\new)} \rightarrow \HH^{\cor, \cov}(\mathcal{M}^{(\new)}). \] 
The fact that $\HH^{\cor, \cov}(\mathcal{M}^{(\new)})$ is equivalent as 2-pre-multiderivator (with domain $\Catlf$)  to $\HH^{\cor}(\mathcal{M}^{(\new)})$ shows independence of the atlases.
\end{PAR}

\begin{LEMMA}\label{LEMMAVALIDITY}
Let $\DD \rightarrow \SSS^{\cor}$ be an infinite (symmetric) fibered (multi)derivator with domain $\Dirlf$ which is local w.r.t.\@ the pre-topology on $\mathcal{S}$, with stable, well-generated fibers.

The axioms of (\ref{PARAXIOMSEXT}) are satisfied in Example~\ref{MAINEXEXT1} taking as $\DD^{(\old)}$ an admissible extension of $\DD$ in the sense of Definition~\ref{DEFADMEXT}, 
and in Example~\ref{MAINEXEXT2} taking as $\DD^{(\old)}$ the canonical extension $\DD^{\amalg}$ of $\DD$ to $\SSS^{\amalg, \cor}$ constructed in \ref{PAREXTAMALG}.  
\end{LEMMA}
\begin{proof}
In both cases the restrictions $\DD^{(\old),*}$ and $\DD^{(\old),!}$ have naive extensions to $\mathcal{M}$ which are just the naive extensions (cf.\@ Theorems~\ref{SATZEXTNAIVELEFT} and \ref{SATZEXTNAIVERIGHT}) of $\DD$ to simplicial pre-sheaves $\mathcal{SET}^{\mathcal{S}^{\op} \times \Delta^{\op}}$ (and its restriction to $\mathcal{SET}^{\mathcal{S}^{\op} }$, respectively).
Property (D1) holds true by homological descent \ref{PROPHODESC} (resp.\@ by cohomological descent \ref{PROPCOHODESC}). Property (D2) follows also from (co)homological descent as follows:
If $f$ is a homotopy local epimorphism (for instance, by Lemma~\ref{LEMMALOCALEPI}, if $f$ is $k-C$ for some $k$) then we claim that the morphism in Lemma~\ref{LEMMAPROPCECH}, 2.\@ induces an isomorphism in the homotopy category  
\[ \hocolim_{\Delta^{\op}} \Cech(f)  \cong X. \]
To see this, replace $f$ by a projectively \v{C}ech fibration $f': X' \rightarrow Y'$  between projectively \v{C}ech fibrant objects. Then there is a weak equivalence
\[ \Cech(f')_n \cong X' \times_{Y'}  \cdots \times_{Y'} X'  \]
and $f'$ is an actual local epimorphism (as follows straightforwardly from the fact that the image of $\mathcal{S}^{\amalg}$ consists of cofibrant objects in the corresponding model category structure).
That the latter simplicial object is \v{C}ech-weakly equivalent to $Y'$ is a standard argument, see e.g.\@ \cite[Corollary~A.3]{DHI04}.

In example \ref{MAINEXEXT1}, Property (CH) follows from Theorem~\ref{SATZLOCALBC} and (H1) follows from Theorem~\ref{SATZHEREDITYLOC2}.
In example \ref{MAINEXEXT2}, Properties (CH) and (H1) follow immediately from Axioms (CH), and (H1), respectively, (cf.\@ \ref{DEFLOCAL6FU} and  \ref{PARLOCAL}) which hold by the locality of $\DD$ and the Definition~\ref{DEFC} of $C$. 
\end{proof}

\begin{DEF}\label{DEFFLIP}
Let $I$ be a (multi)diagram and
let $F': \twtw{I} \times \Delta^{\op} \times \Delta^{\op} \rightarrow \mathcal{M}^{(\old)}$ be a diagram which is $(1,3,5/2,4,6)$-admissible, that is:
 every commutative square in which the vertical morphisms are of type 1, 3 or 5 and the horizontal morphisms are of type 2,4 or 6 is mapped to a homotopy
Cartesian square. Then there is a diagram
\[ \Flip(F') \in \HH^{\cor}(\mathcal{M}^{(\old)})( \twwc{I} \times \Delta^{\op} \times \Delta ) \]
given by the admissible functor
\[ \Flip(F'): \tw{(\twwc{I} \times \Delta^{\op} \times \Delta)} \rightarrow \mathcal{M}^{(\old)} \]
which we define to be the composition of functors
\[ \xymatrix{ \tw{(\twwc{I} \times \Delta^{\op} \times \Delta)} \ar[r]^-P & \twtw{I} \times \Delta^{\op} \times \Delta^{\op} \ar[r]^-{F'} &  \mathcal{M}^{(\old)}  } \]
$P$ maps a triple of morphisms
\begin{equation}\label{eqflip} \vcenter{ \xymatrix{
i_1 \ar[r] \ar@{.>}[rd] \ar[d] & i_2 \ar[r] \ar@{<-}[d] & i_3 \ar@{.>}[rd] \ar[r] \ar@{<-}[d] & i_4 \ar[d]  & \numcirc{$\Delta_i$} \ar[d] & \Delta_{i'} \ar@{<-}[d] \\
j_1 \ar[r] & j_2 \ar@{.>}[ur] \ar[r] & j_3 \ar[r] & j_4 & \Delta_{j} & \numcirc{$\Delta_{j'}$}
} } \end{equation}
to the triple
\[ \xymatrix{ i_1 \ar[r] & j_2 \ar[r] & i_3 \ar[r] & j_4  & \Delta_{i} & \Delta_{j'} }  \]
(as indicated by the dotted arrows and circles). One checks that $\Flip(F')$ is indeed admissible (in the usual sense) if $F'$ has the above property. 

The same construction works if $I$ is a multidiagram by interpreting the entries (except $j_4$ and the $\Delta$'s) in (\ref{eqflip}) as lists in the appropriate way and the morphisms as morphisms of lists as in \ref{MORLIST}. Then $\twtw{I} \times \Delta^{\op} \times \Delta^{\op}$ is an opmultidiagram and
$\twwc{I} \times \Delta^{\op} \times \Delta$ is a multidiagram. Those are symmetric if $I$ is symmetric and the functors $P$ and $F'$ are functors of symmetric multicategories. 
\end{DEF}

\begin{PAR}
We will need the following variant: 
Let $I$ be a (multi)diagram and
let $X': \tm \Delta_1 \times  \twtw{I} \times \Delta^{\op} \times \Delta^{\op} \rightarrow \mathcal{M}^{(\old)}$ be a diagram which is $(1,3,5,7/2,4,6,8)$-admissible. Then there is a diagram
\[ \Flip(X') \in \HH^{\cor}(\mathcal{M}^{(\old)})( \tw{\Delta_1} \times \twwc{I} \times \Delta^{\op} \times \Delta ) \]
given by the admissible functor
\[ \Flip(X'): \tw{(\tw \Delta_1 \times \twwc{I} \times \Delta^{\op} \times \Delta)} \rightarrow \mathcal{M}^{(\old)} \]
which we define to be the composition of functors
\[ \xymatrix{ \tw{(\tw \Delta_1 \times \twwc{I} \times \Delta^{\op} \times \Delta)} \ar[r]^-P & \tm \Delta_1 \times \twtw{I} \times \Delta^{\op} \times \Delta^{\op} \ar[r]^-{F'} &  \mathcal{M}^{(\old)}  } \]
$P$ maps a quadruple of morphisms
\begin{equation*} \xymatrix{
i_1 \ar[r] \ar[d] \ar@{.>}[rd] & i_2 \ar[d] & i_3 \ar[r] \ar@{.>}[rd] \ar[d] & i_4 \ar[r] \ar@{<-}[d] & i_5 \ar@{.>}[rd] \ar[r] \ar@{<-}[d] & i_6 \ar[d]  & \numcirc{$\Delta_i$} \ar[d] & \Delta_{i'} \ar@{<-}[d] \\
j_1 \ar[r] & j_2 & j_3 \ar[r] & j_4 \ar@{.>}[ur] \ar[r] & j_5 \ar[r] & j_6 & \Delta_{j} & \numcirc{$\Delta_{j'}$}
} \end{equation*}
to the quadruple
\[ \xymatrix{ i_1 \ar[r] & j_2 & i_3 \ar[r] & j_4 \ar[r] & i_5 \ar[r] & j_6  & \Delta_{i} & \Delta_{j'} }  \]
(as indicated by the dotted arrows and circles). One checks that $\Flip(X')$ is indeed admissible (in the usual sense) if $X'$ has the above property. 
Actually, this construction can also be obtained by letting $\Flip(X'):=\varepsilon^* \Flip(\pr_{12}^* X')$ 
\[ \varepsilon: \tw \Delta_1 \times \twwc I \times \Delta^{\op} \times \Delta^{\op}  \rightarrow \twwc (\Delta_1  \times I) \times \Delta^{\op} \times \Delta^{\op}   \]
maps $(i \to j) \in \tw \Delta_1$ with $i, j \in \{0,1\}$ to $(i \to j \to 1 \to 1) \in \twwc \Delta_1$  and
\[ \pr_{12}: \twtw (\Delta_1 \times I) \times \Delta^{\op} \times \Delta  \rightarrow \tm (\Delta_1) \times \twtw I \times \Delta^{\op} \times \Delta   \]
is the projection.
\end{PAR}

\begin{MC}\label{DEFMC}
Let $I$ be a (multi)diagram and
let $F: \tw I \rightarrow \mathcal{M}^{(\new)}$ be an admissible diagram 
together with an atlas, i.e.\@ a morphism $f: U \rightarrow F$ with $U$ in $(\mathcal{M}^{(\old)})^{\tw I}$, point-wise in $C$. ($U$ does not need to be admissible.)

We construct a diagram
\[ \Flip(F') \in \HH^{\cor}(\mathcal{M}^{(\old)})( \twwc{I} \times \Delta^{\op} \times \Delta ) \]
in which $F'$ is defined as follows: Consider (in the non-multi-case) the three projections
\[ \pi_{13}, \pi_{23}, \pi_{24}: \twtw I \rightarrow \tw I \]
and define 
\[ F'(\mu, \Delta_n, \Delta_m) := (\Cech(f)(\mu_{13}) \widetilde{\times}_{F(\mu_{23})} \Cech(f)(\mu_{24}))_{n,m}.  \]
Here $(-\widetilde{\times}-)_{m,n}$ is the bisimplicial fiber product constructed in Definition~\ref{DEFFPSOBJ}. 
Note that $\Cech(f)(\mu_{13})$ and $\Cech(f)(\mu_{24})$ come equipped with a morphism to $\delta(F(\mu_{23}))$ by construction. 

In the multi-case, $\mu_{24}(o)$ and  $\mu_{23}(o)$ for $o \in \twtw I$ are morphisms of lists as in \ref{MORLIST}.
Then the formula has to be interpreted in the following way 
\[ F(\mu_{23}) := \prod_i F(\mu_{23}^{(i)}) \] 
for a list of morphisms $\mu_{23} = (\mu_{23}^{(1)}, \dots, \mu_{23}^{(n)})$. Similarly 
\[ \Cech(f)(\mu_{24}) := \prod_i \Cech(f)(\mu_{24}^{(i)}) =  \Cech( \prod_i U(\mu_{24}^{(i)}) \rightarrow \prod_i F(\mu_{24}^{(i)})). \] 

Similarly, for a morphism $\xi: (U_1 \rightarrow F_1) \rightarrow (U_2 \rightarrow F_2)$ in $\mathcal{M}^{\tw I}$, where $f_i: U_i \rightarrow F_i$ are as before, and in which $F_1 \rightarrow F_2$ is a point-wise weak equivalence, define a diagram
\[ \Flip(X') \in \HH^{\cor}(\mathcal{M}^{(\old)})( \tw{\Delta_1} \times \twwc{I} \times \Delta^{\op} \times \Delta ) \]

in which $X'$ is defined as follows: 
\begin{eqnarray*} 
X'(0 \to 0, \mu, \Delta_n, \Delta_m) &:=& (\Cech(f_1)(\mu_{13}) \widetilde{\times}_{F_1(\mu_{23})} \Cech(f_1)(\mu_{24}))_{n,m} \\
X'(0 \to 1, \mu, \Delta_n, \Delta_m) &:=& (\Cech(f_1)(\mu_{13}) \widetilde{\times}_{F_2(\mu_{23})} \Cech(f_2)(\mu_{24}))_{n,m} \\
X'(1 \to 1, \mu, \Delta_n, \Delta_m) &:=& (\Cech(f_2)(\mu_{13}) \widetilde{\times}_{F_2(\mu_{23})} \Cech(f_2)(\mu_{24}))_{n,m} 
\end{eqnarray*} 
\end{MC}

\begin{LEMMA}
The construction in \ref{DEFMC} defines an admissible diagram 
\[ \Flip(F') \in \HH^{\cor}(\mathcal{M}^{(\old)})( \twwc{I} \times \Delta^{\op} \times \Delta ). \]
resp.\@
\[ \Flip(X') \in \HH^{\cor}(\mathcal{M}^{(\old)})( \tw \Delta_1 \times \twwc{I} \times \Delta^{\op} \times \Delta ). \]
\end{LEMMA}
\begin{proof}One could check the condition in Definition~\ref{DEFFLIP}. However, we can also prove this directly: 
Consider a composition of three morphisms in $\twwc{I} \times \Delta^{\op} \times \Delta$:
\[ \xymatrix{
i_1 \ar[r] \ar[d] & i_2 \ar[r] \ar@{<-}[d] & i_3 \ar[r] \ar@{<-}[d] & i_4 \ar[d]  & \Delta_{i} \ar@{<-}[d] & \Delta_{i'} \ar[d] \\
j_1 \ar[r] \ar[d] & j_2 \ar[r] \ar@{<-}[d] & j_3 \ar[r] \ar@{<-}[d] & j_4 \ar[d]  & \Delta_{j} \ar@{<-}[d] & \Delta_{j'} \ar[d] \\
k_1 \ar[r] \ar[d] & k_2 \ar[r] \ar@{<-}[d] & k_3 \ar[r] \ar@{<-}[d] & k_4 \ar[d]  & \Delta_{k} \ar@{<-}[d] & \Delta_{k'} \ar[d] \\
l_1 \ar[r] & l_2 \ar[r] & l_3 \ar[r] & l_4 & \Delta_{l} & \Delta_{l'}
} \]
We have to check that the square
\[ \footnotesize \xymatrix{
(\Cech(f)(i_1\rightarrow i_3) \widetilde{\times}_{F(l_2 \rightarrow i_3)} \Cech(f)(l_2 \rightarrow l_4))_{i,l'}  \ar[r] \ar[d] & (\Cech(f)(j_1\rightarrow j_3) \widetilde{\times}_{F(l_2 \rightarrow j_3)} \Cech(f)(l_2 \rightarrow l_4))_{j,l'}  \ar[d] \\
(\Cech(f)(i_1\rightarrow i_3) \widetilde{\times}_{F(k_2 \rightarrow i_3)} \Cech(f)(k_2 \rightarrow k_4))_{i,k'}  \ar[r] & (\Cech(f)(j_1\rightarrow j_3) \widetilde{\times}_{F(k_2 \rightarrow j_3)} \Cech(f)(k_2 \rightarrow k_4))_{j, k'} 
} \]
is homotopy Cartesian. This follows from the properties of $\widetilde{\times}$ established in Lemma~\ref{LEMMAHOFP} and the fact that 
 the square 
\begin{equation}\label{eqcartsquareF} \vcenter{ \xymatrix{
F(l_2 \rightarrow i_3) \ar[r] \ar[d] & F(l_2 \rightarrow j_3) \ar[d] \\
F(k_2 \rightarrow i_3) \ar[r] & F(k_2 \rightarrow j_3)
} } \end{equation}
is homotopy Cartesian because $F: \tw I \rightarrow \mathcal{M}$
 is {\em admissible} in the sense of Definition~\ref{DEFADMISSIBLE}.

If $I$ is a multidiagram, the square (\ref{eqcartsquareF}) is the product of the squares  
\[ \xymatrix{
F(l_2^{(i)} \rightarrow i_3^{(i)}) \ar[r] \ar[d] & F(l_2^{(i)} \rightarrow j_3^{(i)}) \ar[d] \\
F(k_2^{(i)} \rightarrow i_3^{(i)}) \ar[r] & F(k_2^{(i)} \rightarrow j_3^{(i)})
} \]
where now $l_2^{(i)}$ is a single object and $k_2^{(i)},  i_3^{(i)},  j_3^{(i)}$ are the respective lists of preimages. $k_2^{(i)}$ can still be a list $(k_2^{(i,1)}, \dots, k_2^{(i,n)})$.
Thus the square looks like
\[ \xymatrix{
F(l_2^{(i)} \rightarrow i_3^{(i)}) \ar[r] \ar[d] & F(l_2^{(i)} \rightarrow j_3^{(i)}) \ar[d] \\
\prod_{j} F(k_2^{(i,j)} \rightarrow i_3^{(i,j)}) \ar[r] & \prod_{j}  F(k_2^{(i,j)} \rightarrow j_3^{(i,j)})
} \]
(Still the morphisms in the argument can be multi-morphisms but are now elements in $\tw M$, i.e.\@ where the first list consists of exactly one element.) 
This diagram is again homotopy Cartesian because $F$ is admissible. 
Lemma~\ref{LEMMAPROPCECHKGEOM} together with Lemma~\ref{LEMMAHOFP}, 1.\@ implies that $F'$ has values in $\mathcal{M}^{(\old)}$. 
The case of $\Flip(X')$ is done similarly. 
\end{proof}

\begin{PAR}\label{COMPONENTS}
Let $\tau$ be a tree and
\[ F: \tw (\tau \times I) \rightarrow \mathcal{M}^{(\new)} \]
 an admissible diagram with atlas $U \rightarrow F$. 
The Construction~\ref{DEFMC} associates with it the object
\[ \Flip(F') \in \HH^{\cor}(\mathcal{M}^{(\old)})(\twwc (\tau \times I) \times \Delta^{\op} \times \Delta). \]
Constructing the partial underlying diagram (Lemma~\ref{LEMMAPF}, 1.\@), we get an associated pseudo-functor 
\[ \Flip(F'): \twwc \tau \rightarrow \HH^{\cor}(\mathcal{M}^{(\old)})(\twwc I \times \Delta^{\op} \times \Delta). \]

Let $\alpha$ be a multimorphism in $\twwc \tau$ of type 4 (resp.\@ type 3, resp.\@ type 2, resp.\@ type 1). Denote by
\[ \begin{array}{rrcl} 
 \widetilde{g}: &\widetilde{A} &\rightarrow& \widetilde{S}_1, \dots, \widetilde{S}_k \\
 \widetilde{\iota_1}: &\widetilde{A}' &\rightarrow& \widetilde{A} \\
 \widetilde{\iota_2}: &\widetilde{A}' &\rightarrow& \widetilde{A}'' \\
 \widetilde{f}:  &\widetilde{A}'' &\rightarrow& \widetilde{T} 
\end{array} \]
respectively, their images $\Flip(F')(\alpha)$ in
\[ \HH^{\cor}(\mathcal{M}^{(\old)})(\twwc I \times \Delta^{\op} \times \Delta). \]
\end{PAR}

\begin{BEISPIEL}[$\tau = \Delta_{1,k}$] \label{EXCOMPONENTS}
In this case $U \rightarrow F$ constitutes a diagram in $\Fun(\tw  I, \mathcal{M}^{(\new)})$ of the form

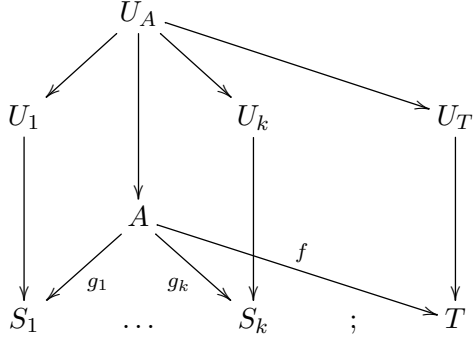
\[ \xymatrix{
& U_A \ar[ld] \ar[dd] \ar[rd]  \ar[rrrd] \\
U_1 \ar[dd]  &  &   U_k \ar[dd]  &   &   U_T \ar[dd] \\
& A \ar[ld]^{g_1} \ar[rd]_{g_k}  \ar[rrrd]^{f} \\
S_1  & \dots &   S_k &  ; &   T
} \]
with $g=(g_1, \dots, g_k)$ type-1-admissible, and $f$ type-2-admissible, and where the vertical morphisms are atlases. 

It induces the following diagram $F'$ of shape $\twtw \Delta_{1,k}$ in $\Fun(\twtw  I \times \Delta^{\op} \times \Delta^{\op}, \mathcal{M}^{(\old)})$: 
{\footnotesize \[ \xymatrix{
  (\Cech(U_A \to A)(\mu_{13}) \widetilde{\times}_{A(\mu_{23})} (\Cech(U_A \to A)(\mu_{24})))_{\bullet,\bullet}   \ar[d]_{\widetilde{\iota}_1}  \ar[rd]^{\widetilde{\iota}_2} &  \\
   (\Cech(U_A \to A)(\mu_{13}) \widetilde{\times}_{\prod_i S_i(\mu_{23})} (\prod_i \Cech(U_{S_i} \to S_i)(\mu_{24})))_{\bullet,\bullet} \ar[d]^{(\widetilde{g}_1, \dots, \widetilde{g}_k)}  & 
     (\Cech(U_T \to T)(\mu_{24}) \widetilde{\times}_{T(\mu_{23})} (\Cech(U_A \to A)(\mu_{13})))_{\bullet,\bullet}   \ar[d]^{\widetilde{f}}  \\
{ \substack{ (\Cech(U_1 \to S_1)(\mu_{13}) \widetilde{\times}_{ S_1(\mu_{23})} (\Cech(U_1 \to S_1)(\mu_{24})))_{\bullet,\bullet}  \\ \vdots \\  (\Cech(U_k \to S_k)(\mu_{13}) \times_{ S_k(\mu_{23})} (\Cech(U_k \to S_k)(\mu_{24})))_{\bullet,\bullet} } } &
  (\Cech(U_T \to T)(\mu_{13}) \widetilde{\times}_{T(\mu_{23})} (\Cech(U_T \to T)(\mu_{24})))_{\bullet,\bullet} 
} \]}%
Here $\mu \in \twtw  I$ is the parameter. 
The goal of the main construction of this section is to construct a fibered (multi)derivator $\DD^{(\new)} \rightarrow \HH^{\cor, \cov}(\mathcal{M}^{(\new)})$ in such a way that objects live in a certain subcategory 
of the fiber of $\DD^{(\old)}(\twwc I \times \Delta^{\op} \times \Delta)$ over  $\Flip(F')$ (evaluated at the corresponding object) and 
 push-forward along $U \rightarrow F$, considered as morphism in  $\HH^{\cor, \cov}(\mathcal{M}^{(\new)})(I)$, is given (up to unique isomorphism) by
\[ \Box_!^{(5)}   \widetilde{f}_!  \ \circ \  (\widetilde{\iota}_2^*)^{-1}  \ \circ \  \widetilde{\iota}_1^!   \ \circ \  (\widetilde{g}_1^* - \otimes \cdots \otimes \widetilde{g}_k^* -) \]
and $\widetilde{\iota}_2^*$ and $\widetilde{\iota}_1^!$ are equivalences (when restricted to these subcategories). 
See Definition/Lemma~\ref{LEMMAEXISTENCE3FUNCTORS} for the discussion of the functors. 
\end{BEISPIEL}

\begin{DEF}\label{DEFCOBJECT}
 A simplicial object $X \in \mathcal{M}^{\Delta^{\op}}$ is called a {\bf $C$-object}, if all 
{\em injective} morphisms in $\Delta$ are mapped to a morphism in $C$ (cf.\@ \ref{PARSETTINGEXT}). 
\end{DEF}

\begin{LEMMA}\label{LEMMAPRESERVE1}
Let $g: X \rightarrow Y_1, \dots, Y_k$  be a homotopy Cartesian (Definition~\ref{DEFCARTMORPH}) multimorphism of $C$-objects (Definition~\ref{DEFCOBJECT}) in $(\mathcal{M}^{(\old)})^{\Delta^{\op}}$ (i.e.\@ such that 
$X \rightarrow \prod_i Y_i$ is homotopy Cartesian). 
In particular, we may consider $g$ as an object in 
\[ \HH^{\cor}(\mathcal{M}^{(\old)})(\Delta_{1,k} \times \Delta^{\op}) \]
such that the (multi)morphisms in $\Delta_{1,k}$-direction go to type 1 and the morphisms in $\Delta^{\op}$-direction go to type 2. 
Then
\[ g^*:  \DD^{(\old)}(\Delta^{\op})_{Y_1} \times \dots \times \DD^{(\old)}(\Delta^{\op})_{Y_k} \rightarrow \DD^{(\old)}(\Delta^{\op})_X \]
preserves Cartesian objects. 
\end{LEMMA}

For $k=1$ there is an obvious dual statement that we leave to the reader to state. 

\begin{proof}
Let $\alpha: \Delta_m \hookrightarrow \Delta_n$ be injective. We get a diagram
\[ \xymatrix{
\DD^{(\old)}(\cdot)_{Y_{1,m}}  \times \cdots \times \DD^{(\old)}(\cdot)_{Y_{k,m}}  \ar[r]^{(Y_i(\alpha)^!)}  \ar[d]_{g^*_m} &  \DD^{(\old)}(\cdot)_{Y_{1,n}} \times \cdots \times  \DD^{(\old)}(\cdot)_{Y_{k,n}}  \ar[d]^{g^*_n} \\
\DD^{(\old)}(\cdot)_{X_m} \ar[r]_{X(\alpha)^!}  &  \DD^{(\old)}(\cdot)_{X_n}  \\
}\]
The underlying square is homotopy Cartesian and the $Y_i(\alpha)$ are in $C$ by assumption. Thus the exchange
\[ X(\alpha)^! g^*_m (-, \dots, -)  \rightarrow g^*_n ( Y_1(\alpha)^! -, \dots,   Y_k(\alpha)^! -)  \]
is an isomorphism by the Axiom (CH) of \ref{PARAXIOMSEXT}.
Thus if the $\mathcal{E}_i \in \DD^{(\old)}(\Delta^{\op})_{Y_{i}}$ are Cartesian objects then in (the underlying diagram of) $\mathcal{F}:=g^*( \mathcal{E}_1, \dots, \mathcal{E}_k)$ injective morphisms $\alpha: \Delta_m \hookrightarrow \Delta_n$ are mapped to Cartesian morphisms. 
However, let $\mathcal{F} \in \DD^{(\old)}(\Delta^{\op})_{X}$ be a Cartesian object.
If $\alpha: \Delta_m \twoheadrightarrow \Delta_n$ is surjective, we find an injective $\beta: \Delta_n \hookrightarrow \Delta_m$ such that $\alpha \beta = \id$.
If furthermore 
\[  \mathcal{E}(\beta): \mathcal{E}_m \rightarrow  \mathcal{E}_n  \]
is Cartesian then, since
\[ \mathcal{E}(\beta) \circ \mathcal{E}(\alpha) = \id  \]
 also $\mathcal{E}(\alpha)$ is Cartesian (by an easy argument that we omit). Therefore $\mathcal{F}$ is Cartesian. 
\end{proof}

\begin{LEMMA}\label{LEMMA34}
Let $I$ be a diagram and $F: \tw I \rightarrow \mathcal{M}^{(\new)}$ be an admissible diagram with atlas $U \rightarrow F$ and let 
\[ \Flip(F') \in \HH^{\cor}(\mathcal{M}^{(\old)})(\twwc I \times \Delta^{\op} \times \Delta) \] be the Main Construction \ref{DEFMC}.
A morphism $\alpha: i \rightarrow j$ of type 3 in $\twwc I$ induces for each $n \in \N$ a morphism 
\[ \iota_1:=\Flip(F')(\alpha, -, n): A:= \Flip(F')(i, -, n) \rightarrow  B := \Flip(F')(j, -, n)   \]
of simplicial objects which induces an equivalence
\[ \iota_1^!: \DD^{(\old)}(\Delta^{\op})_B^{\cart} \cong \DD^{(\old)}(\Delta^{\op})_A^{\cart} \]
Also every pull-back of $\iota_1$ has this property. 

There is an obvious dual statement for type 4 morphisms. 
\end{LEMMA}
\begin{proof}
Since all involved diagrams are in the image of $\Cat(\mathcal{M}^{(\old)})$ this is purely a statement about the naive extension
$\DD^{!,h}$ of $\DD^{(\old),!}$ (cf.\@ \ref{PARAXIOMSEXT}). In other words, we have to show that the morphism
\[ \iota_1: (\Delta^{\op}, A) \rightarrow (\Delta^{\op}, B) \]
is a strong $\DD^{!,h}$-equivalence. 
The morphism $\iota_1$  looks like 
\begin{gather*}
(\Cech(U \to F)(i_1\rightarrow \mathbf{i_3}) \widetilde{\times}_{F(i_2 \rightarrow \mathbf{i_3})} \Cech(U \to F)(i_2 \rightarrow i_4))_{m,n} \\
\longrightarrow (\Cech(U \to  F)(i_1\rightarrow \mathbf{j_3}) \widetilde{\times}_{F(i_2 \rightarrow \mathbf{j_3})} \Cech(U \to F)(i_2 \rightarrow i_4))_{m,n} 
\end{gather*}
with fixed $n$. We can form a diagram in  $\mathcal{M}$ with Cartesian squares in which all morphisms are fibrations and which the $'$ indicate appropriate weakly equivalent objects:

\[ \xymatrix{
\Box  \ar[r]^H \ar[d]  & \Box  \ar[r]^{\widetilde{F}} \ar[d]  & W  \ar[r] \ar[d]  & \ar[d] \Cech(U \to F)(i_2 \rightarrow i_4)_n' \\
U(i_1 \rightarrow i_3)' \ar[r]^h  & \Box \ar[r]^F \ar[d]  & F(i_1 \rightarrow i_3)' \ar[r] \ar[d] & F(i_2 \rightarrow i_3)' \ar[d] \\
& U(i_1 \rightarrow j_3)' \ar[r]^f & F(i_1 \rightarrow j_3)' \ar[r] & F(i_2 \rightarrow j_3)'
} \]
Note that the square involving the objects $F(\cdots)$ is homotopy Cartesian by admissibility of $F$.

Using Lemma~\ref{LEMMAHOFP}, 2., the morphism $\iota_1$ is in the homotopy Category of $\mathcal{M}^{\Delta^{\op}}$ isomorphic to 
\[ \Cech(h F) \times_{F(i_2 \rightarrow i_3)'} \Cech(U \to F)(i_2 \rightarrow i_4)_n'  \rightarrow \Cech(f) \times_{F(i_2 \rightarrow j_3)'} \Cech(U \to F)(i_2 \rightarrow i_4)_n'    \]
which is by Proposition~\ref{LEMMAPROPCECH}, 4.\@  isomorphic to 
\[ \Cech(H \widetilde{F})  \rightarrow \Cech(\widetilde{F}).   \]

By (D1) we are reduced to show that
\[  (\Delta^{\op}, \Cech(H \widetilde{F})) \rightarrow (\Delta^{\op}, \Cech(\widetilde{F})) \]
is a strong $\DD^{!,h}$-equivalence. However, we have a 2-commutative diagram
\[ \xymatrix{ \DD^{!,h}(\Delta^{\op})_{\Cech(\widetilde{F})}^{\cart} \ar@{<-}[rd]_-{\sim} \ar[rr] & & \DD^{!,h}(\Delta^{\op})_{\Cech(H \widetilde{F})}^{\cart} \ar@{<-}[ld]^-{\sim} \\
 &  \DD^{!,h}(\Delta^{\op})_{\delta(W)}^{\cart} } \]
 in which the diagonal functors are equivalences by (D2).  For a pull-back of $\iota_1$ the same argument works.
\end{proof}

\begin{LEMMA}\label{LEMMACOCARTPROJ56}
Let $I$ be a diagram,  let $F: \tw I \rightarrow \mathcal{M}^{(\new)}$ be an admissible diagram with atlas $U \rightarrow F$ and let 
\[ \Flip(F') \in \HH^{\cor}(\mathcal{M}^{(\old)})(\twwc I \times \Delta^{\op} \times \Delta) \] be the Main Construction \ref{DEFMC}.
\begin{enumerate}
\item The inclusion
\[ \DD^{(\old)}(\twwc I \times \Delta^{\op} \times \Delta)^{5-\cart}_{\Flip(F')} \rightarrow \DD^{(\old)}(\twwc I \times \Delta^{\op} \times \Delta)_{\Flip(F')}  \]
has a left adjoint $\Box_!^{(5)}$ which is computed point-wise in $\twwc I \times  \Delta$ and preserves the conditions of being 2, 4, 6-coCartesian. 
\item The inclusion
\[ \DD^{(\old)}(\twwc I \times \Delta^{\op} \times \Delta)^{6-\cocart}_{\Flip(F')} \rightarrow \DD^{(\old)}(\twwc I \times \Delta^{\op} \times \Delta)_{\Flip(F')}  \]
has a right adjoint $\Box_*^{(6)}$ which is computed point-wise in $\twwc I \times  \Delta^{\op}$ and preserves the conditions of being 1,3,5-Cartesian. 
\end{enumerate}
\end{LEMMA}
\begin{proof}
The functors $\Box_!^{(5)}$ and $\Box_*^{(6)}$ exist by \cite[Theorem~4.3.4]{Hor15}. 
We will show 1.\@,  the other statement being dual. Consider the projection 
\[ \pi_{12346}: \twwc I \times \Delta^{\op} \times \Delta \rightarrow \twwc I \times \Delta. \] 
Since it is clearly an opfibration, by \cite[Lemma~3.5.11]{Hor15}, the exchange morphism 
\[ \iota^* \Box_!^{(5)} \rightarrow \Box_! \iota^*  \]
is an isomorphism, where $\iota: \{\mu\} \times \Delta^{\op} \times \{\Delta_n\} \rightarrow \twwc I \times \Delta^{\op} \times \Delta$ is the inclusion of a fiber. 
$\Box_!^{(5)}$ is therefore computed point-wise in $\twwc I \times \Delta$. Consider any 
morphism $\alpha: \mu_1 \rightarrow \mu_2$ in $\twwc I$ of type 2 or type 4 and $\rho: \Delta_n \rightarrow \Delta_m$ and consider the 2-commutative diagram
\[ \xymatrix{
\DD^{(\old)}(\Delta^{\op})_{\Flip(F')(\mu_1,-,n)} \ar@{<-^{)}}[r] \ar@{<-}[d]_{\Flip(F')(\alpha,-,\rho)_*} &   \DD^{(\old)}(\Delta^{\op})_{\Flip(F')(\mu_1,-,n)}^{\cart} \ar@{<-}[d]^{\Flip(F')(\alpha,-,\rho)_*} \\
\DD^{(\old)}(\Delta^{\op})_{\Flip(F')(\mu_2,-,m)} \ar@{<-^{)}}[r] & \DD^{(\old)}(\Delta^{\op})_{\Flip(F')(\mu_2,-,m)}^{\cart}
} \] 
(Note that $\Flip(F')(\alpha,-,\rho)_*$ clearly preserves Cartesianity). 
Its left adjoint is the diagram 
\[ \xymatrix{
\DD^{(\old)}(\Delta^{\op})_{\Flip(F')(\mu_1,-,n)} \ar[r]^{\Box_!} \ar[d]_{\Flip(F')(\alpha,-,\rho)^*} &   \DD^{(\old)}(\Delta^{\op})_{\Flip(F')(\mu_1,-,n)}^{\cart} \ar[d]^{\Flip(F')(\alpha,-,\rho)^*} \\
\DD^{(\old)}(\Delta^{\op})_{\Flip(F')(\mu_2,-,m)} \ar[r]_{\Box_!} & \DD^{(\old)}(\Delta^{\op})_{\Flip(F')(\mu_2,-,m)}^{\cart}
} \] 
because $\Flip(F')(\alpha,-,\rho)^*$ preserves Cartesianity by Lemma~\ref{LEMMAPRESERVE1}.
$\Box_!^{(5)}$ thus preserves the condition of being 2,4,6-coCartesian. 
\end{proof}

\begin{LEMMA}\label{LEMMAPRESERVE2}
Let 
\[ \xymatrix{ 
X' \ar[r]^-G \ar[d]_-F & Y'_1, \dots, Y_k' \ar[d]^{f_1, \dots, f_k} \\
X \ar[r]_-g & Y_1, \dots, Y_k
} \]
be a diagram of $C$-objects (cf.\@ Definition~\ref{DEFCOBJECT}) in $(\mathcal{M}^{(\old)})^{\Delta^{\op}}$, point-wise homotopy Cartesian
and such that $g: X \rightarrow Y_1, \dots, Y_k$ (and hence $G: X' \rightarrow Y'_1, \dots, Y'_k$) are homotopy Cartesian morphisms (\ref{DEFCARTMORPH})\footnote{Here this boils down to $X \rightarrow \prod_i Y_i$ being homotopy Cartesian.}. In other words, the square above induces an object in 
\[ \HH^{\cor}(\mathcal{M}^{(\old)})(\Delta_1 \times \Delta_{1,k} \times \Delta^{\op})  \]
such that the morphisms in $\Delta_{1,k}$ go to type 1 and the morphisms in $\Delta_1 \times \Delta^{\op}$ go to type 2. 
Assume that the $f^!_i$ induce equivalences on Cartesian objects and also $\widetilde{F}_i^!$ for arbitrary pull-backs $\widetilde{F}_i$ (and hence also $F^!$ as a composition of such).  

Then we have that the exchange morphism (of the base change)
\[  G^* (f^!_1-, \dots, f_k^!-)  \rightarrow F^! g^*(-, \dots, -)   \]
is an isomorphism on Cartesian objects. 
\end{LEMMA}
For $k=1$ there is an obvious dual statement of the Lemma with the same proof whose formulation is left to the reader. 
\begin{proof}We can split up the pull-backs $G^*$ and $g^*$ and have to show that all squares in the diagram are 2-commutative (squares $\numcirc{1}$ and $\numcirc{3}$ via the natural exchange morphisms): 
\[ \xymatrix{ 
  \DD^{(\old)}(  \Delta^{\op} )^{\cart}_{Y_1}  \times \dots \times  \DD^{(\old)}(  \Delta^{\op} )^{\cart}_{Y_k} \ar@{}[rd]|{\numcirc{1}}  \ar[r]^{(f_i^!)_i} \ar[d]_{\boxtimes}  &   \DD^{(\old)}(  \Delta^{\op} )^{\cart}_{Y'_1}\times \dots\times  \DD^{(\old)}(  \Delta^{\op} )^{\cart}_{Y'_k} \ar[d]^{\boxtimes} \\
  \DD^{(\old)}(  (\Delta^{\op})^k )^{\cart}_{\prod_i Y_i}   \ar[r]^{(\prod f_i)^!} \ar[d]_{\Delta^*}   \ar@{}[rd]|{\numcirc{2}} &   \DD^{(\old)}( ( \Delta^{\op})^k )^{\cart}_{\prod_i Y_i'} \ar[d]^{\Delta^*} \\
  \DD^{(\old)}(  \Delta^{\op} )^{\cart}_{\prod_i Y_i}   \ar[r]^{(\prod f_i)^!} \ar[d]_{g^*}  \ar@{}[rd]|{\numcirc{3}} &   \DD^{(\old)}( \Delta^{\op} )^{\cart}_{\prod_i Y_i'} \ar[d]^{G^*} \\
  \DD^{(\old)}(  \Delta^{\op} )^{\cart}_{X}  \ar[r]_{F^!}   &   \DD^{(\old)}(  \Delta^{\op} )^{\cart}_{X'}  
} \]
Note that $g^*$ preserve Cartesian objects by Lemma~\ref{LEMMAPRESERVE1}. For $\boxtimes$ a similar argument shows the same. 
For square $\numcirc{2}$ the 2-commutativity is clear because $(\prod f_i)^!$ is computed point-wise. 
Furthermore all horizontal functors are equivalences: For the third row note that  $(\prod f_i)^!$ is a composition of $\widetilde{F}_i^!$ where $\widetilde{F}_i$ is a pull-back of $f_i$. For the second row note that an adjoint of 
 $(\id, \dots, \id, f_i, \id, \dots, \id)^!$ is given by $\Box_!^{(i)} (\id, \dots, \id, f_i, \id, \dots, \id)_!$ and that is computed point-wise in the other $\Delta^{\op}$'s (for $\Box_!^{(i)}$ see the argument in the proof of Lemma~\ref{LEMMACOCARTPROJ56}). Therefore it is an equivalence because it is point-wise in the other $\Delta^{\op}$'s the case. 

2-commutation of square $\numcirc{3}$: We have the 2-commutative diagram: 
\[ \xymatrix{ 
  \DD^{(\old)}(  \Delta^{\op} )^{\cart}_{\prod Y_i}    \ar[r]^{(\prod f_i)^!} \ar@{<-}[d]_{g_*}  &   \DD^{(\old)}(  \Delta^{\op} )^{\cart}_{\prod Y_i'}  \ar@{<-}[d]^{G_*} \\
  \DD^{(\old)}(  \Delta^{\op} )^{\cart}_{X}  \ar[r]_{F^!}   &   \DD^{(\old)}(  \Delta^{\op} )^{\cart}_{X'}  
} \]
in which the horizontal functors are equivalences and in which $g_*$ and $G_*$ clearly preserve Cartesian objects. 
Its exchange is the square $\numcirc{3}$ and thus again 2-commutative.

2-commutation of square $\numcirc{1}$: We have the 2-commutative diagram: 
\[ \xymatrix{ 
  \DD^{(\old)}(  \Delta^{\op} )_{Y_1}  \times \dots \times  \DD^{(\old)}(  \Delta^{\op} )_{Y_k}  \ar@{<-}[r]^{(f_{i,!})_i} \ar[d]_{\boxtimes}  &   \DD^{(\old)}(  \Delta^{\op} )^{\cart}_{Y'_1}\times \dots\times  \DD^{(\old)}(  \Delta^{\op} )_{Y'_k} \ar[d]^{\boxtimes} \\
  \DD^{(\old)}(  (\Delta^{\op})^k )_{\prod_i Y_i}   \ar@{<-}[r]_{(\prod f_i)_!}    &   \DD^{(\old)}( ( \Delta^{\op})^k )^{\cart}_{\prod_i Y_i'} 
 } \]
The diagram
\[ \xymatrix{ 
  \DD^{(\old)}(  \Delta^{\op} )_{Y_1}^{\cart}  \times \dots \times  \DD^{(\old)}(  \Delta^{\op} )_{Y_k}^{\cart}  \ar@{<-}[r]^{(\Box_! f_{i,!})_i} \ar[d]_{\boxtimes}  &   \DD^{(\old)}(  \Delta^{\op} )^{\cart}_{Y'_1}\times \dots\times  \DD^{(\old)}(  \Delta^{\op} )^{\cart}_{Y'_k} \ar[d]^{\boxtimes} \\
  \DD^{(\old)}(  (\Delta^{\op})^k )_{\prod_i Y_i}^{\cart}   \ar@{<-}[r]_{ \Box_! (\prod f_i)_!}    &   \DD^{(\old)}( ( \Delta^{\op})^k )^{\cart}_{\prod_i Y_i'} 
 } \]
noting that $\Box_! = \Box_!^{(1)} \circ \cdots \circ \Box_!^{(n)}$ in the second row, is therefore 2-commutative if $\boxtimes$ commutes with the left Cartesian projector in the $i$-th slot. It suffices to see this for $k=2$ and therefore
to see that 
\[ \xymatrix{ 
  \DD^{(\old)}(  \Delta^{\op} )_{A}^{\cart}  \times  \DD^{(\old)}(  \Delta^{\op} )_{B}^{}  \ar@{<-}[r]^{(\Box_!, \id)} \ar[d]_{\boxtimes}  &   \DD^{(\old)}(  \Delta^{\op} )^{}_{A} \times \DD^{(\old)}(  \Delta^{\op} )^{}_{B} \ar[d]^{\boxtimes} \\
  \DD^{(\old)}(  \Delta^{\op} \times \Delta^{\op})_{A\times B}^{1-\cart}   \ar@{<-}[r]_{ \Box_!^{(1)}}    &   \DD^{(\old)}( \Delta^{\op} \times  \Delta^{\op})_{A \times B} 
 } \]
is 2-commutative. Noting that $\Box_!^{(1)}$ is computed point-wise in the second entry it suffices to see that this holds point-wise. In other words, we have to see that for $\mathcal{E}_n \in \DD^{(\old)}(  \Delta^{\op} )^{}_{B_n}$
\[ \xymatrix{ 
  \DD^{(\old)}(  \Delta^{\op} )_{A}^{\cart}    \ar@{<-}[r]^{\Box_!} \ar[d]_{L:=\boxtimes \mathcal{E}_n}  &   \DD^{(\old)}(  \Delta^{\op} )^{}_{A} \ar[d]^{\boxtimes \mathcal{E}_n} \\
  \DD^{(\old)}(  \Delta^{\op} )_{A\times B_n}^{\cart}   \ar@{<-}[r]_{ \Box_!}    &   \DD^{(\old)}( \Delta^{\op} )_{A \times B_n} 
 } \]
 is 2-commutative. However we have
 \[ m^* L( \mathcal{F}) = \pr_1^* m^* \mathcal{F} \otimes_{A_m \times B_n}  \pr_2^* \mathcal{E}_n .  \]
 $L$ has a right adjoint $R$ which is also computed point-wise with
 \[  m^* R(\mathcal{G}) = \pr_{1,*} \mathcal{HOM}_{A_m \times B_n}( \pr_2^* \mathcal{E}_n; m^* \mathcal{G})   \]
We see that $R$ preserves the subcategory of Cartesian objects and hence $L$ commutes with the left Cartesian projector. 
\end{proof}

\begin{DEFLEMMA}\label{LEMMAEXISTENCE3FUNCTORS}
For the morphisms defined in \ref{COMPONENTS}, depending on a multimorphism $\alpha$ in $\twwc \tau$, we have: 
\begin{enumerate}
\item If $\alpha$ is of type 4,
\[ \Flip(F')(\alpha)_{\#} := \widetilde{g}^*:  \DD^{(\old)}(\twwc I \times \Delta^{\op} \times \Delta)_{\widetilde{S}_1}  
\times  \cdots \times \DD^{(\old)}(\twwc I \times \Delta^{\op} \times \Delta)_{\widetilde{S}_k}  \rightarrow  \DD^{(\old)}(\twwc I \times \Delta^{\op} \times \Delta)_{\widetilde{A}}  \]
where $\widetilde{g}^*$ denotes a chosen multi-push-forward along the multi-morphism $(\widetilde{g}_1, \dots, \widetilde{g}_k)$
preserves the conditions of being simultaneously 3,5-Cartesian and 2,4,6-coCartesian. 
\item If $\alpha$ is of type 3, 
\[ \Flip(F')(\alpha)_{\#} := \widetilde{\iota}_1^!:  \DD^{(\old)}(\twwc I \times \Delta^{\op} \times \Delta)_{\widetilde{A}} \rightarrow  \DD^{(\old)}(\twwc I \times \Delta^{\op} \times \Delta)_{\widetilde{A}'}  \]
preserves the conditions of being simultaneously 3,5-Cartesian and 2,4,6-coCartesian and it induces an equivalence of the subcategories of such objects with inverse given by $\Box_!^{(5)} \widetilde{\iota}_{1,!}$ where $\Box_!^{(5)}$ is the left Cartesian projector from Lemma~\ref{LEMMACOCARTPROJ56}.
\item If $\alpha$ is of type 2, 
\[ \widetilde{\iota}_2^*:  \DD^{(\old)}(\twwc I \times \Delta^{\op} \times \Delta)_{\widetilde{A}''} \rightarrow  \DD^{(\old)}(\twwc I \times \Delta^{\op} \times \Delta)_{\widetilde{A}'}  \]
preserves the conditions of being simultaneously 3,5-Cartesian and 2,4,6-coCartesian and it induces an equivalence of the subcategories of such objects with inverse given by 
\[ \Flip(F')(\alpha)_{\#} := \Box_*^{(6)} \widetilde{\iota}_{2,*} \]
where $\Box_*^{(6)}$ is the right coCartesian projector from Lemma~\ref{LEMMACOCARTPROJ56}.
\item If $\alpha$ is of type 1, 
\[ \Flip(F')(\alpha)_{\#} := \Box_!^{(5)} \widetilde{f}_!:  \DD^{(\old)}(\twwc I \times \Delta^{\op} \times \Delta)_{\widetilde{A}''} \rightarrow  \DD^{(\old)}(\twwc I \times \Delta^{\op} \times \Delta)_{\widetilde{T}}  \]
where $\Box_!^{(5)}$ is the left Cartesian projector from Lemma~\ref{LEMMACOCARTPROJ56}, preserves the conditions of being simultaneously 3,5-Cartesian and 2,4,6-coCartesian.

Furthermore $\Box_!^{(5)}$ is computed point-wise in $\twwc I \times \Delta$. 
\end{enumerate}
\end{DEFLEMMA}

\begin{proof}
1.\@ $\widetilde{g}^*$ clearly preserves any type of coCartesianity.  
That 5-Cartesianity  is preserved follows from Lemma~\ref{LEMMAPRESERVE1}. 
Let $\alpha: o_1 \rightarrow o_2$ be a morphism of type 3 in $\twwc I$ and consider for $n \in \N$ the diagram
\[ \xymatrix{
\DD^{(\old)}(\Delta^{\op})_{\widetilde{S}_1(o_2,-,n)}^{\cart} \times \cdots \times \DD^{(\old)}(\Delta^{\op})_{\widetilde{S}_k(o_2,-,n)}^{\cart} \ar[r]^-{\widetilde{g}^*} \ar[d]_{\widetilde{S}_1(\alpha,-,n)^!, \dots, \widetilde{S}_k(\alpha,-,n)^!} &   \DD^{(\old)}(\Delta^{\op})_{\widetilde{A}(o_2,-,n)}^{\cart} \ar[d]^{\widetilde{A}(\alpha,-,n)^!} \\
\DD^{(\old)}(\Delta^{\op})_{\widetilde{S}_1(o_1,-,n)}^{\cart} \times \cdots \times \DD^{(\old)}(\Delta^{\op})_{\widetilde{S}_k(o_1,-,n)}^{\cart}  \ar[r]_-{\widetilde{g}^*} & \DD^{(\old)}(\Delta^{\op})_{\widetilde{A}(o_1,-,n)}^{\cart}
} \] 
The preservation of Cartesianity for type 3 morphisms follows if this diagram is 2-commutative (via the exchange morphism). This follows
from Lemma~\ref{LEMMAPRESERVE2} because the vertical functors are equivalences (Lemma~\ref{LEMMA34}). 
 
2.\@  $\widetilde{\iota}_1^!$ clearly preserves any type of Cartesianity. That coCartesianity of type 6 morphisms is preserved follows from Lemma~\ref{LEMMAPRESERVE1}.
Let $\alpha: o_1 \rightarrow o_2$ be a morphism in $\twwc I$ of type 4 or type 2 and consider for $n \in \N$ the diagram
\begin{equation*} 
\xymatrix{
\DD^{(\old)}(\Delta^{\op})_{\widetilde{A}(o_1,-,n)}^{\cart} \ar[r]^{\widetilde{\iota}_1^!} \ar[d]_{\widetilde{A}(\alpha,-,n)^*} &   \DD^{(\old)}(\Delta^{\op})_{\widetilde{A}'(o_1,-,n)}^{\cart} \ar[d]^{\widetilde{A}'(\alpha,-,n)^*} \\
\DD^{(\old)}(\Delta^{\op})_{\widetilde{A}(o_2,-,n)}^{\cart} \ar[r]_{\widetilde{\iota}_1^!} & \DD^{(\old)}(\Delta^{\op})_{\widetilde{A}'(o_2,-,n)}^{\cart}
} 
\end{equation*} 
in which $\widetilde{A}'(\alpha,-,n)^*$ preserves Cartesianity by Lemma~\ref{LEMMAPRESERVE1}.
The preservation of coCartesianity for type 4 and type 2 morphisms follows if this diagram is 2-commutative (via the exchange morphism). This follows
from Lemma~\ref{LEMMAPRESERVE2} because the {\em horizontal} functors are equivalences (Lemma~\ref{LEMMA34}).
 
We claim that  $\widetilde{\iota}_1^!$ induces an equivalence
\[  \widetilde{\iota}_1^!:  \DD^{(\old)}(\twwc I \times \Delta^{\op} \times \Delta)_{\widetilde{A}}^{5-\cart} \rightarrow  \DD^{(\old)}(\twwc I \times \Delta^{\op} \times \Delta)_{\widetilde{A}'}^{5-\cart}.    \]
Indeed, we have seen that this is point-wise (in $\twwc I \times \Delta$) the case. An adjoint is given by $\Box_!^{(5)} \widetilde{\iota}_{1,!}$. By Lemma~\ref{LEMMACOCARTPROJ56} it is computed point-wise (in $\twwc I \times \Delta$) as well. Therefore the fact that unit and counit are isomorphisms can be checked point-wise. 

$\Box_!^{(5)} \widetilde{\iota}_{1,!}$ preserves the coCartesianity conditions by Lemma~\ref{LEMMACOCARTPROJ56}. 
We claim that $\Box_!^{(5)} \widetilde{\iota}_{1,!}$ preserves the conditions of being 3,5-Cartesian as well (and also the condition of 1-Cartesian although this is not needed). 5-Cartesian objects are preserved by construction. 
Hence let $\alpha: \mu_1 \rightarrow \mu_2$ be a morphism of type 3 or type 1 and consider the diagram
 \[ \xymatrix{
\DD^{(\old)}(\Delta^{\op})_{\widetilde{A}'(o_1,-,n)}^{\cart} \ar[r]^{\Box_! \widetilde{\iota}_{1,!}} \ar[d]_{\widetilde{A}'(\alpha,-,n)^!} &   \DD^{(\old)}(\Delta^{\op})_{\widetilde{A}(o_1,-,n)}^{\cart} \ar[d]^{\widetilde{A}(\alpha,-,n)^!} \\
\DD^{(\old)}(\Delta^{\op})_{\widetilde{A}'(o_2,-,n)}^{\cart} \ar[r]_{\Box_! \widetilde{\iota}_{1,!}} & \DD^{(\old)}(\Delta^{\op})_{\widetilde{A}(o_2,-,n)}^{\cart}
} \] 
The preservation of Cartesianity for type 3 and type 1 morphisms follows if this diagram is 2-commutative (via the exchange morphism). This is, however, the exchange of the 2-commutative square
 \[ \xymatrix{
\DD^{(\old)}(\Delta^{\op})_{\widetilde{A}'(o_1,-,n)}^{\cart} \ar@{<-}[r]^-{\widetilde{\iota}_{1}^!} \ar[d]_{\widetilde{A}'(\alpha,-,n)^!} &   \DD^{(\old)}(\Delta^{\op})_{\widetilde{A}(o_1,-,n)}^{\cart} \ar[d]^{\widetilde{A}(\alpha,-,n)^!} \\
\DD^{(\old)}(\Delta^{\op})_{\widetilde{A}'(o_2,-,n)}^{\cart} \ar@{<-}[r]_-{\widetilde{\iota}_{1}^!} & \DD^{(\old)}(\Delta^{\op})_{\widetilde{A}(o_2,-,n)}^{\cart}
} \] 
in which the horizontal functors are equivalences (Lemma~\ref{LEMMA34}).

Therefore $\widetilde{\iota}_1^!$ is also an equivalence   
\[  \DD^{(\old)}(\twwc I \times \Delta^{\op} \times \Delta)_{\widetilde{A}}^{3,5-\cart, 2,4,6-\cocart} \rightarrow  \DD^{(\old)}(\twwc I \times \Delta^{\op} \times \Delta)_{\widetilde{A}'}^{3,5-\cart, 2,4,6-\cocart}    \]
with inverse $\Box_!^{(5)} \widetilde{\iota}_{1,!}$.

3.\@ is shown completely analogously to 2.\@

4.\@ From Lemma~\ref{LEMMACOCARTPROJ56} it follows that the projector is
computed point-wise in $\twwc I \times \Delta$. 
$\widetilde{f}_!$ clearly preserves 2,4,6-coCartesianity. $\Box_!^{(5)}$ preserves it by Lemma~\ref{LEMMACOCARTPROJ56}.
It remains to see that $\Box_!^{(5)} \widetilde{f}_!$ preserves the conditions of being 3,5-Cartesian. 5-Cartesian objects are preserved by construction. 
Hence let $\alpha: o_1 \rightarrow o_2$ be a morphism of type 3 and consider the diagram
 \[ \xymatrix{
\DD^{(\old)}(\Delta^{\op})_{\widetilde{A}''(o_1,-,n)}^{\cart} \ar[r]^{\Box_!^{(5)}\widetilde{f}_!} \ar[d]_{\widetilde{T}(\alpha,-,\rho)^!} &   \DD^{(\old)}(\Delta^{\op})_{\widetilde{T}(o_1,-,n)}^{\cart} \ar[d]^{\widetilde{T}(\alpha,-,\rho)^!} \\
\DD^{(\old)}(\Delta^{\op})_{\widetilde{A}''(o_2,-,m)}^{\cart} \ar[r]_{\Box_!^{(5)}\widetilde{f}_!} & \DD^{(\old)}(\Delta^{\op})_{\widetilde{T}(o_2,-,m)}^{\cart}
} \] 
The preservation of Cartesianity for type 3 morphisms follows if this diagram is 2-commutative (via the exchange morphism). This is, however, the exchange of the diagram
 \[ \xymatrix{
\DD^{(\old)}(\Delta^{\op})_{\widetilde{A}''(o_1,-,n)}^{\cart} \ar@{<-}[r]^-{\widetilde{f}^!} \ar[d]_{\widetilde{T}(\alpha,-,\rho)^!} &   \DD^{(\old)}(\Delta^{\op})_{\widetilde{T}(o_1,-,n)}^{\cart} \ar[d]^{\widetilde{T}(\alpha,-,\rho)^!} \\
\DD^{(\old)}(\Delta^{\op})_{\widetilde{A}''(o_2,-,m)}^{\cart} \ar@{<-}[r]_-{\widetilde{f}^!} & \DD^{(\old)}(\Delta^{\op})_{\widetilde{T}(o_2,-,m)}^{\cart}
} \] 
in which the vertical functors are equivalences (Lemma~\ref{LEMMA34}).
\end{proof}

\begin{PROP}\label{PROPPROPERTIESCORCOMPMDIA}
Let $\tau$ be a tree, $I \in \Dirlf$ a diagram, and let $F: \tw (\tau \times I) \rightarrow \mathcal{M}^{(\new)}$ be an admissible diagram with atlas $U \rightarrow F$ and let 
\[ \Flip(F') \in \HH^{\cor}(\mathcal{M}^{(\old)})(\twwc (\tau \times I) \times \Delta^{\op} \times \Delta) \] be the Main Construction~\ref{DEFMC} considered as pseudo-functor
\[ \Flip(F'): \twwc \tau \rightarrow \HH(\mathcal{M}^{(\old)})(\twwc I \times \Delta^{\op} \times \Delta). \]

Then the functors constructed in \ref{LEMMAEXISTENCE3FUNCTORS} commute. More precisely, 
consider a diagram of the form:
\[ \xymatrix{
w_1, \dots, w_n \ar[r]^-d \ar[d]_c & z_1,\dots,z_n \ar[d]^{a} \\
y \ar[r]_-b & x
} \]
in $^{\downarrow\downarrow\downarrow\downarrow} \tau$  [sic!] 
where $a$ and $c$ are of some type 1,2,3 or 4 (1-ary in case of type $\not= 4$) and $b$ and $d$ are of some type 1,2, or 3, respectively.  

Then there is a canonical isomorphism\footnote{In each case an exchange (or its inverse) of a natural commutation given by functoriality of $\DD^{(\old)}$ we we leave to the reader to construct}
\begin{equation} \label{eqexchange} \Flip(F')(a')_{\#} \Flip(F')(d')_{\#} \rightarrow \Flip(F')(b')_{\#} \Flip(F')(c')_{\#}   \end{equation}
where $a' \in \twwc \tau$ denotes the morphism corresponding to $a$, etc.
In case $a$ and $c$ are of type 4 and $n$-ary $\Flip(F')(a')_{\#}$ and $\Flip(F')(c')_{\#}$ are multivalued functors and $\Flip(F')(d')_{\#}$ denotes the corresponding $n$-tupel of functors (of which only one is not an identity). 
\end{PROP}

\begin{proof}
This follows as in the proof of Lemma~\ref{LEMMAEXISTENCE3FUNCTORS}.
\end{proof}

\begin{PAR}\label{TWGENREL}
Let $\Xi \in \{\uparrow, \downarrow\}^l$ be an $l$-tupel of arrow directions as in \ref{PARTW}. The category ${}^{\Xi} \Delta_n$ can be generated by morphisms (of type $i$) of the form
\[ \xymatrix{ n_1 \ar[r] \ar@{=}[d] & \cdots \ar[r] &  n_i \ar[r]\ar@{<->}[d] & \cdots \ar[r] &   n_l \ar@{=}[d] &   \\
n_1 \ar[r] & \cdots \ar[r] &  n_{i} \pm 1 \ar[r] & \cdots \ar[r] &   n_l &   } \]
with $n_1, \dots, n_l \in \Delta_n$ (such that $n_{i} \pm 1 \in \Delta_n$ as well). 

These generators are subject to the relations requiring that squares
\begin{equation} \label{squaretype} \vcenter{ \xymatrix{
w \ar[r] \ar[d] & z \ar[d] \\
x \ar[r] & y
} } \end{equation}
in which the vertical and horizontal morphisms are generators as above, are commutative. W.l.o.g. the type of the horizontal morphisms is different from the type of the vertical morphisms. 

More generally, let $\tau$ be a tree. The category ${}^{\Xi} \tau$ can be similarly generated by multimorphisms of type $l$
\[ \xymatrix{ [S_{1,1}^{(1)}, \dots, S_{1,n_1}^{(1)}, \dots] \ar[r] \ar@{=}[d] & \cdots \ar[r] &  [S_{i,1}^{(1)}, \dots, S_{i,n_i}^{(1)}, \dots] \ar[r]\ar@{=}[d] & \cdots \ar[r] &   [S_{l}^{(1)}, \dots, S^{(n)}_{l}] \ar@{->}[d] &   \\
 [S_{i,1}^{(1)}, \dots, S_{i,n_i}^{(1)}, \dots] \ar[r] & \cdots \ar[r] &  [S_{i,1}^{(1)}, \dots, S_{i,n_i}^{(1)}, \dots] \ar[r] & \cdots \ar[r] &   [T_{l}] &   } \]
where the morphism $S_{l}^{(1)}, \dots, S^{(n)}_{l} \rightarrow T_l$ is a generating morphism of $\tau$
and morphisms of type $i$
\[ \xymatrix{ [S_{1,1}, \dots, S_{1,n_1}] \ar[r] \ar@{=}[d] & \cdots \ar[r] &  [S_{i,1}, \dots, S_{i,n_i}] \ar[r]\ar@{<->}[d] & \cdots \ar[r] &   [S_{l}] \ar@{=}[d] &   \\
 [S_{1,1}, \dots, S_{1,n_1}] \ar[r] & \cdots \ar[r] &  [T_{i,1}, \dots, T_{i,n_i'}] \ar[r] & \cdots \ar[r] &   [S_{l}] &   } \]
in which the morphism of lists $[S_{i,1}, \dots, S_{i,n_i}] \leftrightarrow [T_{i,1}, \dots, T_{i,n_i'}]$ consists of {\em one} generating morphism of $\tau$ and identities otherwise. 
These generators are subject to the relations requiring that squares
\begin{equation} \label{squaretype} \vcenter{ \xymatrix{
w_1, \dots, w_n \ar[r] \ar[d] & z \ar[d] \\
x_1, \dots, x_n \ar[r] & y,
} } \end{equation}
in which the vertical and horizontal morphisms are generators as above, are commutative. Necessarily also only one of the left vertical morpisms is not an identity. 

In the non-multi-case we do not have any non-trivial relation-squares in which the horizontal and vertical morphisms are of the same type. Otherwise this may happen for type $<l$. 
\end{PAR}

Any multimorphism $\alpha$ in  ${}^{\downarrow\downarrow\downarrow\downarrow} \tau$ can be decomposed {\em canonically} into a sequence 
of functors $\alpha = abcd$ where $a$ is of type 1, $b$ of type 2 etc., and where $a$, $b$ and $c$ are 1-ary. Taking the cubical structure of ${}^{\downarrow\downarrow\downarrow\downarrow} \tau$ and the previous discussion into account, Proposition~\ref{PROPPROPERTIESCORCOMPMDIA} may be restated as follows
\begin{PROP}\label{PROPPF}
The association 
\[ \alpha = abcd \mapsto \Flip(F')(\alpha)_{\#} := \Flip(F')(a')_{\#} \Flip(F')(b')_{\#} \Flip(F')(c')_{\#} \Flip(F')(d')_{\#}, \] 
where $a'$ is the morphism of $\twwc \tau$ corresponding to $a$ in ${}^{\downarrow\downarrow\downarrow\downarrow} \tau$, etc.\@,  
 defines a well-defined pseudo-functor on ${}^{\downarrow\downarrow\downarrow\downarrow} \tau$ such that the isomorphism (\ref{eqexchange}) becomes the one
induced by the pseudo-functoriality.
\end{PROP}

\begin{DEF}\label{DEFDER6FU1} Consider the setting in \ref{PARAXIOMSEXT} and recall Definition~\ref{DEFCCOV}.
Let $\tau$ be a tree, and let $I \in \Dirlf$. We define a category with weak equivalences
\[ E_I(\tau) \quad (\text{resp. } E_I(\tau^S))   \]
with objects pairs $(U \rightarrow F, \mathcal{F})$ with $U \rightarrow F$ in $\Cor^{\cov}_I(\tau)$ (cf.\@ Definition~\ref{DEFCCOV}), and an object
\begin{equation} \label{eqprincipal}  \mathcal{F} \in \Fun(\twwc \tau, \DD^{(\old)}(\twwc I \times \Delta^{\op} \times \Delta))^{3,5-\cart, 2,4,6-\cocart}_{ \Flip(F')}   \end{equation}
(functors of multicategories) where 
$\Flip(F') \in \HH^{\cor}(\mathcal{M}^{(\old)})(\twwc (\tau \times I)  \times \Delta^{\op} \times \Delta)$ is the Main Construction~\ref{DEFMC} applied to $U \rightarrow F$.  $\Flip(F')$ denotes also its partial underlying diagram $\twwc \tau \rightarrow \HH^{\cor}(\mathcal{M}^{(\old)})(\twwc I  \times \Delta^{\op} \times \Delta)$. Note that $\twwc I \in \Dirlf$ by Lemma~\ref{LEMMACATLF}. 

Morphisms in $\mathcal{F} \rightarrow \mathcal{F}'$  are  the morphisms $\xi: (U_1 \rightarrow F_1) \Rightarrow (U_2 \rightarrow F_2)$ in $\Cor^{\cov}_I(\tau)$ together with an isomorphism class of diagrams
\[ \mathcal{X} \in \Fun(\tw \Delta_1 \times \twwc \tau, \DD^{(\old)}(\twwc I \times \Delta^{\op} \times \Delta))^{3,5-\cart, 2,4,6-\cocart}_{ \Flip(X')} \]
such that $\mathcal{X}(0 \to 0) = \mathcal{F}$, $\mathcal{X}(1 \to 1) = \mathcal{F}'$ and such that the morphism 
$(0 \to 0) \rightarrow (0 \to 1)$ goes to a coCartesian morphism and $(1 \to 1) \rightarrow (0 \to 1)$ goes to a Cartesian morphism.
Finally, we say that a morphism is a {\em weak equivalence}, if the underlying morphism in $\Cor^{\cov}_I(\tau)$ is a weak equivalence, i.e.\@ if the morphism $F_1 \rightarrow F_2$ is  a weak equivalence. 
\end{DEF}

\begin{BEM}
In (\ref{eqprincipal}) we counted the $\Delta^{\op}$ as 5\textsuperscript{th} entry and $\Delta$ as 6\textsuperscript{th} entry. The condition means that in the underlying diagram 
$\twwc (\tau \times I) \times \Delta^{\op} \times \Delta \rightarrow \DD^{(\old)}(\cdot)$ 
of the objects
\begin{enumerate}
\item each morphism of type 3 in $\twwc( \tau \times I)$ is mapped to a Cartesian morphism;
\item each morphism of type 2 or 4 in $\twwc(\tau \times I)$ to a coCartesian morphism;
\item each morphism in $\Delta^{\op}$ is mapped to a Cartesian morphism;
\item each morphism in $\Delta$ is mapped to a coCartesian morphism. 
\end{enumerate}
Furthermore, the diagram $\Flip(F')$ is constructed in such a way that ``Cartesian'' means
that some morphism
\[ \mathcal{F}_1 \rightarrow f^! \mathcal{F}_2 \]
is an isomorphism and ``coCartesian'' means that some morphism
\[ f^* \mathcal{F}_1 \rightarrow \mathcal{F}_2   \]
is an isomorphism. The adjoints $f_!$ and $f_*$ do not appear. 
\end{BEM}

The following follows immediately from the fact that (co)Cartesian morphisms over identities are just isomorphisms: 
\begin{LEMMA}\label{LEMMASIXFUMOR}
The set of morphisms $\mathcal{F} \to \mathcal{F}'$ in the fiber 
\[ E_I(\tau)_{(U \rightarrow F)}  \] 
over $(U \rightarrow F) \in \Cor^{\cov}_I(\tau)$ is isomorphic to 
\[ \mathrm{Iso}_{\Fun( \twwc \tau, \DD^{(\old)}(\twwc I \times \Delta^{\op} \times \Delta))^{3,5-\cart, 2,4,6-\cocart}_{ \Flip(F')}}(\mathcal{F}, \mathcal{F}'). \]
\end{LEMMA}

\begin{KEYLEMMA}
\label{LEMMACOMPSIXFU}
\begin{enumerate}
\item Let $\tau$ be a tree, and consider an object $U \rightarrow F$ in $\Cor^{\cov}_I(\tau)$.
Let $(\mathcal{E}_o)_{o \in \tau}$ be a collection of objects with $\mathcal{E}_o \in \DD^{(\old)}(\twwc I \times \Delta^{\op} \times \Delta)^{3,5-\cart, 2,4,6-\cocart}_{ \Flip(F')_o}$, where
$\Flip(F')_o$ is the value of $\Flip(F')$ at $\twwc o$. 
Then the fiber\footnote{the fiber over $(U \rightarrow F, \{\mathcal{E}_o\}_{o \in \tau})$ for the functor $E_I(\tau) \rightarrow \Cor^{\cov}_I(\tau) \times \prod_o E_I(\cdot)$  }
\[ E_I(\tau)_{(U \rightarrow F)}(\{\mathcal{E}_o\}_{o \in \tau})  \] 
 is equivalent to a set.

\item For a concatenation $\tau = \tau_2 \circ_i \tau_1$, where $i$ is a source object of $\tau_2$ which will be identified with the final object of $\tau_1$, 
the square
\[ \xymatrix{
E_I(\tau)_{(U \rightarrow F)}((\mathcal{E}_o)_{o \in \tau}) \ar[r] \ar[d] & E_I(\tau_1)_{(U_1 \to F_1)}((\mathcal{E}_o)_{o \in \tau_1}) \ar[d] \\
E_I(\tau_2)_{(U_2 \to F_2)}((\mathcal{E}_o)_{o \in \tau_2}) \ar[r] & \cdot
} \]
is 2-Cartesian. Hence, if we consider the $E_I$ as sets, we have
\[ E_I(\tau)_{(U \rightarrow F)}((\mathcal{E}_o)_{o \in \tau}) \cong E_I(\tau_1)_{(U_1 \rightarrow F_1)}((\mathcal{E}_o)_{o \in \tau_{1}}) \times E_I(\tau_2)_{(U_2 \rightarrow F_2)}((\mathcal{E}_o)_{o \in \tau_{2}}). \]

\item For $\tau=\Delta_{1,n}$, there is a canonical isomorphism of sets
\[ E_I(\Delta_{1,n})_{(U \to F)}(\mathcal{E}_1, \dots, \mathcal{E}_n;  \mathcal{E}_{n+1}) \cong \Hom_{\DD^{(\old)}(\twwc I \times \Delta^{\op} \times \Delta)_{\widetilde{T}}}( \Flip(F')({}^{\downarrow\downarrow\downarrow\downarrow} (0 \to 1))_{\#} (\mathcal{E}_1, \dots, \mathcal{E}_n); \mathcal{E}_{n+1}).  \]
Here $\widetilde{T}$ is the restriction of $\Flip(F')$ to $\twwc \{n+1\}$. 
Note that $\Flip(F')({}^{\downarrow\downarrow\downarrow\downarrow} (0 \to 1))_{\#}$ is the composition of functors as in Example~\ref{EXCOMPONENTS}.

An object $\mathcal{F}$ on the left hand side is mapped to an isomorphism if  each morphism of type 1 $\alpha: o_1 \rightarrow o_2$ in $\twwc \tau$  also induces an isomorphism
\[ \Box_!^{(5)} \widetilde{f}_! \mathcal{F}(o_1) \rightarrow \mathcal{F}(o_2).  \]

\item  The forgetful functor
\[ E_I(\tau)(\{\mathcal{E}_o\}_{o \in \tau}) \rightarrow \Cor^{\cov}(\tau)(\{U_o \to F_o\}_{o \in \tau}) \]
is a bifibration with discrete fibers. 
\end{enumerate}
\end{KEYLEMMA}

\begin{proof}

Given an object $i = (i_1 \to i_2 \to i_3 \to i_4)$ in $\twwc \tau$
we have a sequence of morphisms in $\twwc \tau$ writing $i_1 = [i_1^{(1)}, \dots, i_1^{(n)}]$: 

\[\xymatrix{
(i_1^{(1)} \to i_1^{(1)} \to i_1^{(1)} \to i_1^{(1)}) \ar@{-}[rd]& \cdots & (i_1^{(n)} \to i_1^{(n)} \to i_1^{(n)} \to i_1^{(n)}) \ar@{-}[ld]  \\
& \ar[d]^{\alpha_{i,4}} \\
& (i_1 \to i_1 \to i_1 \to i_4) \ar@{<-}[d]^{\alpha_{i,3}} \\
& (i_1 \to i_1 \to i_3 \to i_4) \ar@{<-}[d]^{\alpha_{i,2}} \\
& (i_1 \to i_2 \to i_3 \to i_4) 
}   \]
in which $\alpha_{i,m}$ is of type $m$.  
Note that any of the objects $\twwc i_1^{(k)}$ lies either in $\twwc \tau_1$ or $\twwc \tau_2$. 

1.\@ 
In view of the observation above and Lemma~\ref{LEMMASIXFUMOR}, the statement follows from the uniqueness (up to unique isomorphism) of
the source (resp.\@ target) of a Cartesian (resp.\@ coCartesian) morphism. 

2.\@ Given restrictions of $\mathcal{F}$ to the union of $\twwc \tau_1$ and $\twwc \tau_2$, we construct
an extension to $\twwc (\tau_2 \circ  \tau_1)$ as follows: 
Every morphism $\alpha_{i,m}$ as above has a corresponding morphism $\alpha_{i,m}' \in  {}^{\downarrow\downarrow\downarrow\downarrow} \tau$ and
we define an extension on objects by setting for $i \not\in \twwc \tau_1 \cup \twwc \tau_2$
\[ \mathcal{F}(i) := \Flip(F')(\alpha_{i,2}'\alpha_{i,3}'\alpha_{i,4}')_\#(\mathcal{F}(\twwc i_1^{(1)}), \cdots, \mathcal{F}(\twwc i_1^{(n)}))  \]

Observe that for $i \in \twwc \tau_1$ and $i \in \twwc \tau_2$ there is already a {\em canonical} isomorphism
\begin{equation} \label{eqrewrite} \mathcal{F}(i) \cong \Flip(F')(\alpha_{i,2}'\alpha_{i,3}'\alpha_{i,4}')_\#(\mathcal{F}(\twwc i_1^{(1)}), \cdots, \mathcal{F}(\twwc i_1^{(n)}))  \end{equation}
because of the (co)Cartesianity conditions on $\mathcal{F}|_{\twwc \tau_1}$ and $\mathcal{F}|_{\twwc \tau_2}$. 

For a generating morphism $\beta: i \rightarrow j$ of type 2, 3, or 4 observe that there is an isomorphism coming from the pseudo-functoriality of $\Flip(F')(\cdot)_\#$ (cf.\@ Proposition~\ref{PROPPF})
\[ \Flip(\beta')_\# \Flip(F')(\alpha_{i,2}'\alpha_{i,3}'\alpha_{i,4}')_\# \cong \Flip(F')(\alpha_{j,2}'\alpha_{j,3}'\alpha_{j,4}')_\# \]
and thus  a coCartesian morphism (for type 2 and 4) resp.\@ a Cartesian morphism (for type 3)
\[ \Flip(F')(\alpha_{i,2}'\alpha_{i,3}'\alpha_{i,4}')_\# (\mathcal{F}(\twwc i_{1}^{(1)}), \cdots, \mathcal{F}(\twwc i_{1}^{(n)})) \rightarrow \Flip(F')(\alpha_{j,2}'\alpha_{j,3}'\alpha_{j,4}')_\# (\mathcal{F}(\twwc i_{1}^{(1)}), \cdots, \mathcal{F}(\twwc i_{1}^{(n)}))   \]
over $\Flip(F')(\beta)$ which we define to be $\mathcal{F}(\beta)$. For a {\em generating} morphism $\beta: i \rightarrow j$ of type 1 we get the following diagram writing $j_1 = [j_1^{(1)}, \dots, j_1^{(n)}]$ and in which $i_1^{(1)}, \dots, i_1^{(n')}$ are possibly still  lists of objects: 
\[  \xymatrix{
{ \substack{ (i_1^{(1)} \to j_1^{(1)} \to j_1^{(1)} \to j_1^{(1)})  \\ \cdots \\  (i_1^{(n')} \to j_1^{(n')} \to j_1^{(n')} \to j_1^{(n')}) } } \ar[d]^{\gamma_{4}}  \ar[r]^{\substack{ \beta_1^{(1)}\\  \dots\\  \beta_1^{(n')}}}  & 
{ \substack{ (j_1^{(1)} \to j_1^{(1)} \to j_1^{(1)} \to j_1^{(1)}) \\ \cdots \\ (j_1^{(n')} \to j_1^{(n')} \to j_1^{(n')} \to j_1^{(n')})  }}  \ar[d]^{\alpha_{j,4}}  \\
 (i_1 \to j_1 \to j_1 \to i_4) \ar@{<-}[d]^{\gamma_3}&  (j_1 \to j_1 \to j_1 \to i_4) \ar@{<-}[d]^{\alpha_{j,3}} \\
 (i_1 \to j_1 \to i_3 \to i_4) \ar@{<-}[d]^{\gamma_2}&  (j_1 \to j_1 \to i_3 \to i_4) \ar@{<-}[d]^{\alpha_{j,2}} \\
 (i_1 \to i_2 \to i_3 \to i_4)  \ar[r]^\beta & (j_1 \to i_2 \to i_3 \to i_4)  
}   \]
Because $\beta$ is generating, only one of the $\beta_1^{(k)}$ is not an identity and lies entirely in $\twwc \tau_1$ or $\twwc \tau_2$. 
$\mathcal{F}(\beta_1^{(k)})$ corresponds to a morphism 
\[ \Flip(F')((\beta_1^{(k)})')_{\#} \mathcal{F} (i_1^{(k)} \to j_1^{(k)} \to j_1^{(k)} \to j_1^{(k)})  \rightarrow  \mathcal{F}( \twwc j_1^{(k)} ) . \]
Applying $\Flip(F')(\alpha_{j,2}'\alpha_{j,3}'\alpha_{j,4}')_{\#}$ to it (and the various identities), we get, using the pseudo-functoriality of $\Flip(F')(\cdot)_\#$ (cf.\@ Proposition~\ref{PROPPF}) a morphism
\begin{gather*}
 \Flip(F')(\beta)_{\#} \Flip(F')(\gamma_2'\gamma_3'\gamma_4')_{\#}  \mathcal{F} (i_1^{(k)} \to j_1^{(k)} \to j_1^{(k)} \to j_1^{(k)})    \\ \rightarrow
  \Flip(F')(\alpha_{j,2}'\alpha_{j,3}'\alpha_{j,4}')_{\#} \mathcal{F}( \twwc j_1^{(k)}).
\end{gather*}
Using  (\ref{eqrewrite}), and the pseudo-functoriality again, the left hand side might be rewritten as 
\[ \Flip(F')(\beta)_{\#} \Flip(F')(\alpha_{i,3}'\alpha_{i,2}'\alpha_{i,1}')_{\#}  \mathcal{F} (\twwc i_1^{(k)})  \rightarrow  \Flip(F')(\alpha_{j,2}'\alpha_{j,3}'\alpha_{j,4}')_{\#} \mathcal{F}(\twwc j_1^{(k)}) . \]
This yields a morphism
\[ \Flip(F')(\beta)_{\#} \mathcal{F}(i) \rightarrow \mathcal{F}(j) \]
which corresponds to a morphism 
\[  \mathcal{F}(i) \rightarrow \mathcal{F}(j) \]
over $\Flip(F')(\beta)$ which we define to be $\mathcal{F}(\beta)$. 
To get a valid diagram $\mathcal{F}$ one has to check that the relation squares are mapped to commutative diagrams. This follows from the pseudo-functoriality of $\Flip(F')(\cdot)_\#$ (cf.\@ Proposition~\ref{PROPPF}). 

3.\@ is clear. 

4. Recall that a morphism $\xi:  (U_1 \to F_1) \rightarrow (U_2 \to F_2)$ in $\Cor^{\cov}_I(\tau)$ induces a diagram 
\[ \Flip(X') \in \HH^{\cor}(\mathcal{M}^{(\old)})(\tw \Delta_1 \times \twwc I \times \Delta^{\op} \times \Delta) \] 
(Main Construction \ref{DEFMC}). Denote the morphisms in $\tw \Delta_1$ as follows
\[ \xymatrix{ & (0 \to 1) \ar[ld]_{\iota_1} \ar[rd]^{\iota_2}  \\
(0 \to 0) & & (1 \to 1)
}\]
We first show the following analogue of Lemma~\ref{LEMMA34}: For each $n \in \N$ (resp.\@ $m \in \N$) and object $\mu \in \twwc \Delta_n \times I$, the morphisms
\[ f:=\Flip(X')(\iota_1,\mu, m, -): A:= \Flip(F_{1}')(\mu, m, -) \rightarrow  B := \Flip(X')(0\to 1,\mu, m, -)   \]
\[ g:=\Flip(X')(\iota_2,\mu, -, n): C:= \Flip(X')(0 \to 1,\mu, -, n) \rightarrow  D := \Flip(F_{2}')(\mu, -, n)   \]
of simplicial objects induce equivalences
\[ f^*: \DD^{(\old)}(\Delta)_B^{\cocart} \cong \DD^{(\old)}(\Delta)_A^{\cocart}. \]
\[ g^!: \DD^{(\old)}(\Delta^{\op})_D^{\cart} \cong \DD^{(\old)}(\Delta^{\op})_C^{\cart}. \]
Then one proceeds as in the proof of Lemma~\ref{LEMMAEXISTENCE3FUNCTORS} to check that 
$f^*$ and $g^!$ as well as  
and $\Box_*^{(6)} f_*$, and  $\Box_!^{(5)} g_!$,
 preserve the conditions of being simultaneously 3,5-Cartesian and 2,4,6-coCartesian and are computed point-wise in $\twwc (\tau \times I) \times \Delta^{\op}$ (resp.\@ point-wise in $\twwc (\tau \times I) \times \Delta$). Since they are equivalences point-wise
they are in each of the two cases mutually inverse equivalences. Therefore for each 
\[ \mathcal{F}_1 \in \Fun(\twwc \tau, \DD^{(\old)}(\twwc I \times \Delta^{\op} \times \Delta))^{3,5-\cart, 2,4,6-\cocart}_{ \Flip(F_1')} \]

 there is a unique diagram 
\[ \mathcal{X} \in \Fun(\tw \Delta_1 \times \twwc \tau, \DD^{(\old)}(\twwc I \times \Delta^{\op} \times \Delta))^{3,5-\cart, 2,4,6-\cocart}_{ \Flip(X')}, \]
 up to unique isomorphism, with $\mathcal{X}(\tw 0) = \mathcal{F}_1$ and similarly for a given $\mathcal{F}_2$. In combination with 1., this shows the statement. 
For the used analogue of Lemma~\ref{LEMMA34} note that 
\[ A(m) = (\Cech(U_1 \to F_1)(\mu_{13}) \widetilde{\times}_{F_1(\mu_{23})} \Cech(U_1 \to F_1)(\mu_{24}))_{m,n} \]
\[ B(m) = (\Cech(U_1 \to F_1)(\mu_{13}) \widetilde{\times}_{F_2(\mu_{23})} \Cech(U_2 \to F_2)(\mu_{24}))_{m,n} \]
By the same argument as in the proof of Lemma~\ref{LEMMA34} using (D1) and (D2) we see that $f^*$ is an equivalence. 
Observe that the diagram
\[ \xymatrix{F_1(\mu_{24}) \ar[r] \ar[d]& F_1(\mu_{23}) \ar[d] \\
F_2(\mu_{24}) \ar[r] & F_2(\mu_{23})
   } \]
is homotopy Cartesian because, the vertical morphisms are weak equivalences. 
Similarly, we have
\[ C(n) = (\Cech(U_1 \to F_1)(\mu_{13}) \widetilde{\times}_{F_2(\mu_{23})} \Cech(U_2 \to F_2)(\mu_{24}))_{m,n} \]
\[ D(n) = (\Cech(U_2 \to F_2)(\mu_{13}) \widetilde{\times}_{F_2(\mu_{23})} \Cech(U_2 \to F_2)(\mu_{24}))_{m,n} \]
By the same argument as in the proof of Lemma~\ref{LEMMA34}, using (D1) and (D2), we see that $g^!$ is an equivalence. 
\end{proof}

\begin{LEMMA}\label{LEMMADER6FU3}
The strict 2-functor 
\[ \Delta_n \mapsto E_I(\Delta_n) \quad (\text{resp. } \tau \mapsto E_I(\tau), \quad \text{resp. } \tau^S \mapsto E_I(\tau^S))   \]
satisfies the properties of \cite[Proposition~A.1]{Hor17}, resp.\@ of \cite[Proposition~A.3]{Hor17}.
\end{LEMMA}
\begin{proof}
The surjectivity on objects (Axiom 1) is clear and the equivalence in Axiom 2 follows from Lemma~\ref{LEMMACOMPSIXFU}, 2. and 4.\@ 
\end{proof}

\begin{DEF}\label{DEFDER6FU2}With the setting as in \ref{PARAXIOMSEXT},
we construct a morphism of (symmetric) 2-pre-(multi)derivators with domain $\Dirlf$
\begin{eqnarray*}
 \DD^{(\new)} &\rightarrow& \HH^{\cor, \cov}(\mathcal{M}^{(\new)}).
 \end{eqnarray*}
For a diagram $I \in \Dirlf$, let $\DD^{(\new)}(I)$ be the (symmetric) 2-(multi)category obtained by applying \cite[Proposition~A.1]{Hor17} (resp.\@ \cite[Proposition~A.3]{Hor17}) to the strict functor of Definition~\ref{DEFDER6FU1}
\[ \Delta_n \mapsto E_I(\Delta_n) \quad (\text{resp. } \tau \mapsto E_I(\tau), \quad \text{resp. } \tau^S \mapsto E_I(\tau^S)).   \]
It comes equipped with an obvious morphism to $\HH^{\cor, \cov}(\mathcal{M}^{(\new)})(I)$ which was constructed applying the same Proposition to 
$\tau^S \mapsto \Cor_I^{\cov}(\tau^S)$. 
For a functor $\alpha: I \rightarrow J$ in $\Cat$ we define the pullback $\alpha^*=\DD^{(\new)}(\alpha)$
to be $\DD^{(\old)}(\twwc \alpha \times \id \times \id)$. Note that $\alpha$ induces a functor $\twwc \alpha: \twwc I \rightarrow \twwc J$ and that $\DD^{(\old)}(\twwc \alpha \times \id \times \id)$ preserves the relevant (co)Cartesianity conditions. 
The 2-pre-multiderivator $\DD^{(\new)}(I)$ is defined on natural transformations as follows. A natural transformation $\alpha \Rightarrow \beta$ can be seen as a functor $\Delta_1 \times I \rightarrow J$.
The pullback of a diagram in $\mathcal{E} \in \DD^{(\old)}(\twwc J \times \Delta^{\op} \times \Delta)$ and taking partial underlying diagram gives a functor in   $\Fun(\twwc \Delta_1, \DD^{(\old)}(\twwc I \times \Delta^{\op} \times \Delta))$ which satisfies the correct
 (co)Cartesianity conditions. It is, by definition, a morphism \[ \alpha^* \mathcal{E} = e_0^* \mathcal{F} \rightarrow  \beta^* \mathcal{E} = e_1^* \mathcal{F} \]  in $\DD^{(\new)}(I)$ which we 
define to be the natural transformation $\DD^{(\new)}(\mu)$ at $\mathcal{E}$. 
We proceed to describe the pseudo-naturality constraint of $\DD^{(\new)}(\mu)$ and restrict to 1-ary morphisms for simplicity. 
 A 1-ary 1-morphism given by $\xi \in E_I(\Delta_1)$ (note that a general 1-ary 1-morphism in $\DD^{\new}(I)$ is freely generated by those) consists of an object  $U \to F$ in $\Cor_I^{\cov}(\Delta_1)$ plus an object
  \[ \mathcal{F} \in \Fun(\twwc \Delta_1,  \DD^{(\old)}(\twwc I \times \Delta^{\op} \times \Delta))^{3,5-\cart, 2,4,6-\cocart}_{ \Flip(F')}   \]
  It is a morphism from $\mathcal{E}_1 := e_0^*\mathcal{F}$ to $\mathcal{E}_2 := e_1^*\mathcal{F}$.
We get an object $\mu^*\mathcal{F}$ and after taking partial underlying diagrams an object 
 \[ (\mu^* \mathcal{F})' \in  \Fun(\twwc \Delta_1^2,  \DD( \twwc   I))^{4-\cocart, 3-\cocart^*, 2-{\cart}}_{ \pi_{234}^*(\mu^* \widetilde{X}^{\op})}.  \]

The two (non-degenerate) embeddings $\Delta_2 \rightarrow \Delta_1^2$ yield two 2-isomorphisms
\[ (\mu^* \mathcal{E}_1) \circ  (\alpha^* \mathcal{F}) \Rightarrow ( \twwc \delta )^* (\mu^* \mathcal{F})' \Leftarrow  (\beta^* \mathcal{F})  \circ (\mu^* \mathcal{E}_2).      \]
where $\delta: \Delta_1 \rightarrow \Delta_1^2$ denotes the diagonal. This endows $\DD^{\new}(\mu): \alpha^* \Rightarrow \beta^*$ with the structure of pseudo-natural transformation and one checks that this construction yields a pseudo-functor:
\[ \Fun(I, J) \rightarrow \Fun^{\mathrm{strict}}(\DD^{\new}(J), \DD^{\new}(I)). \]
where $\Fun^{\mathrm{strict}}$ is the 2-category of strict 2-functors, pseudo-natural transformations, and modifications.
\end{DEF}

The goal of this construction is to establish that the 
morphism of (symmetric) 2-pre-(multi)derivators
\[ \DD^{(\new)} \rightarrow \HH^{\cor, \cov}(\mathcal{M}^{(\new)})  \]
is a (symmetric) left fibered (multi)derivator with domain $\Dirlf$. The next Lemma shows that it satisfies (FDer0 left). In the next section it will be  established that 
it satisfies (FDer3--4 left).  

\begin{LEMMA}\label{LEMMA6FUOPFIB}
The functor of 2-categories
\begin{eqnarray*}
 \DD^{(\new)}(I) &\rightarrow& \HH^{\cor, \cov}(\mathcal{M}^{(\new)})(I) 
 \end{eqnarray*}
constructed in Definition~\ref{DEFDER6FU2} is a 1-opfibration and 2-bifibration with 1-categorical fibers. 
\end{LEMMA}
\begin{proof}
The functors have 1-categorical fibers and are 2-bifibered by Lemma~\ref{LEMMADER6FU3}.
By Lemma~\ref{LEMMACOMPSIXFU}, 3.\@, weakly coCartesian morphisms exist and a push-forward functor is given by
\[  \Flip(F')({}^{\downarrow\downarrow\downarrow\downarrow} (0 \to 1))_{\#}. \]
Therefore Proposition~\ref{PROPPF} implies that 
the composition of weakly coCartesian morphisms is weakly coCartesian and hence the functor is a 1-opfibration \cite[Proposition~2.7]{Hor15b}.
\end{proof}

\section{Properties of the main construction}

\begin{LEMMA}[{cf.\@ also \cite[Lemma~11.1]{Hor17}}]\label{LEMMAOPFIB}
Let $\alpha: I \rightarrow J$ be an opfibration, and consider the sequence of functors:
\[ \xymatrix{ \twwc{I} \ar[rrr]^-{q_1=(\twwc{\alpha},\pi_{123})} &&& \twwc{J} \times_{({}^{\downarrow \uparrow \uparrow}{J})} {}^{\downarrow \uparrow \uparrow}{I} \ar[rr]^-{q_2=\id \times \pi_1} && \twwc{J} \times_J I. } \]
\begin{enumerate}
\item The functor $q_1$ is an opfibration. The fiber of $q_1$ over a pair $j_1 \rightarrow j_2 \rightarrow j_3 \rightarrow j_4$ and $i_1 \rightarrow i_2 \rightarrow i_3$ is
\[ i_4 \times_{/I_{j_4}} I_{j_4}\]
where $i_4$ is the target of a coCartesian arrow over $j_3 \rightarrow j_4$ with source $i_3$. 
\item The functor $q_2$ is a fibration. The fiber of $q_2$ over a pair $j_1 \rightarrow j_2 \rightarrow j_3 \rightarrow j_4$ and $i_1$ (lying over $j_1$) is 
\[ (i_2 \times_{/I_{j_2}} I_{j_2} \times_{/I_{j_3}} I_{j_3})^{\op} \]
where $i_2$ is the target of a coCartesian arrow over $j_1 \rightarrow j_2$ with source $i_1$ and the second comma category is 
constructed via the functor $I_{j_2} \rightarrow I_{j_3}$ being the coCartesian push-forward along $j_2 \rightarrow j_3$.  
\end{enumerate}
\end{LEMMA}

\begin{proof}Straightforward. \end{proof}

\begin{LEMMA}\label{LEMMAKAN2}With the setting as in \ref{PARAXIOMSEXT}.  Let $J \in \Dirlf$ be a diagram,
let $F \in (\mathcal{M}^{(\new)})^{\tw J}$ be an admissible diagram with atlas $U \rightarrow F$ in $\mathcal{M}^{\tw J}$, and let $\Flip(F') \in \HH^{\cor}(\mathcal{M}^{(\old)})(\twwc J \times \Delta^{\op} \times \Delta)$ be the Main Construction~\ref{DEFMC} applied to it. 

If $\alpha: I \rightarrow J$ is an opfibration  then
the functors
\[ \xymatrix{
 \DD^{(\old)}(\twwc{J} \times_J I \times \Delta^{\op} \times  \Delta)_{\Flip(F')}^{3,5-\cart, 2,4,6-\cocart} \ar[d]^-{q_2^*} \\
 \DD^{(\old)}(\twwc{J} \times_{({}^{\downarrow \uparrow \uparrow}{J})} {}^{\downarrow \uparrow \uparrow}{I} \times \Delta^{\op} \times  \Delta)^{3,5-\cart, 2,4,6-\cocart}_{\Flip(F')}  \ar[d]^-{q_1^*} \\
   \DD^{(\old)}(\twwc{I} \times \Delta^{\op} \times  \Delta   )^{3,5-\cart, 2,4,6-\cocart}_{\Flip(F')}  } \]
are equivalences.
\end{LEMMA}
\begin{proof}
For $q_1$, resp.\@ $q_2$ the requirements of Lemma~\ref{LEMMAKAN3} are met by virtue of Lemma~\ref{LEMMAOPFIB}. The resulting equivalences 
obviously restrict to the subcategories of $3,5$-Cartesian and $2,4,6$-coCartesian objects. 
\end{proof}

\begin{LEMMA}\label{LEMMAKAN4}
Let $\SSS$ be a 2-pre-derivator,
let $\DD \rightarrow \SSS$ be a fibered derivator with domain $\Dia$, let $I \in \Dia$ be a diagram (such that also $\twwc I \in \Dia$), and let $S \in \SSS(\tm I)$
then the projection: $\pi_{14}: \twwc I \rightarrow  \tm I$ induces an equivalence 
\[ \xymatrix{ \DD(\tm I)_{S}^{\pi_1-\cocart} \ar[r]^-{\pi_{14}^*} & \DD(\twwc I)_{\pi_{14}^* S}^{\pi_{1}-\cocart}  } \]
\end{LEMMA}
\begin{proof}
Denote by $f$ the natural morphism 
\[ f: \pi_1^* \Delta^*  S \rightarrow S.  \]
Consider the diagram
\[ \xymatrix{
 \DD(\tm I)_{S}^{\pi_1-\cocart} \ar[r]^-{\pi_{14}^*}  \ar@/^5pt/[dd]^{\Delta^*}  &  \DD(\twwc I)_{\pi_{14}^* S}^{\pi_{1}-\cocart} \ar[d]^{(\pi_{14}^*f)^\bullet}  \\
 &  \DD(\twwc I)_{\pi_{1}^* \Delta^* S}^{\pi_{14}-\cocart} \ar[ld]_{\pi_{1,*}^{(\Delta^* S)}}   \\
 \DD(I)_{\Delta^* S} \ar[r]^-{\pi_{1}^*}  \ar@/^5pt/[uu]^{f_\bullet \pi_1^{\tm I,*}}  &  \DD(\twwc I)_{\pi_{1}^* \Delta^* S}^{\pi_{1}-\cocart}    \ar@{^{(}->}[u]
} \]
The left vertical functors are equivalences (inverse to each other) and, in fact, the functor
\[ \Box_* := f_\bullet \pi_1^{\tm I,*} \Delta^*: \DD(\tm I)_{S} \rightarrow \DD(\tm I)_{S}^{\pi_1-\cocart}  \]
is a right coCartesian projector. Therefore the functor $\pi_{14}^*$ in the first row has the right adjoint 
$\Box_* \pi_{14,*}^{(S)}$. Note that 
\[ \Box_* \pi_{14,*}^{(S)} = f_\bullet \pi_1^{\tm I,*} \pi_{1,*}^{\tm I,(\Delta^*S)} f^\bullet \pi_{14,*}^{(S)} \cong f_\bullet \pi_1^{\tm I,*} \pi_{1,*}^{(\Delta^*S)} (\pi_{14}^*f)^\bullet.  \]
We have to show that the unit and counit are isomorphisms on $\pi_1$-coCartesian objects. 

It follows from a similar argument as in Lemma~\ref{LEMMAKAN2}  that $\pi_1^*$ in the last row is an equivalence. 
The right adjoint $\pi_{1,*}^{(\Delta^* S)}$ is an inverse equivalence on $\pi_{1}$-coCartesian objects. This cannot be used directly though because  $(\pi_{14}^*f)^\bullet$ does not preserve $\pi_{1}$-coCartesian objects.

We first claim that for $i \in I$:
\[ i^* \pi_{1,*}^{(\Delta^* S)} \cong  (\twwc i)^* \]
{\em on $\pi_{14}$-coCartesian objects}. Since $\pi_1$ is a fibration with fibers ${}^{\uparrow \uparrow \downarrow} (i \times_{/I} I)$ we are left to show that\footnote{We keep the numbering of the projections from the larger diagram $\twwc I$, hence $\pi_1$ is the constant functor to $i$ and $\pi_4$ is the projection to the last entry. }
\[  p_*: \DD({}^{\uparrow \uparrow \downarrow} (i \times_{/I} I))_{ p^* S(\tm i) }^{\pi_{4}-\cocart} \rightarrow  \DD( \cdot )_{ S(\tm i) }      \]
is evaluation at $\twwc i$. We can factor $p_*$:
\[ \xymatrix{  \DD({}^{\uparrow \uparrow \downarrow} (i \times_{/I} I))_{ p^* S(\tm i) }^{\pi_{4}-\cocart} \ar[r]^-{\pi_{4,*}} &
 \DD(i \times_{/I} I)_{ p^* S(\tm i) } \ar[r]^-{\pi_*} &
  \DD( \cdot )_{ S(\tm i) }    }   \]
 By Lemma~\ref{LEMMAKAN3}, the functor $\pi_4^*$ is an equivalence between the first two categories with $\pi_{4,*}$ quasi-inverse, and $\pi_{*}$ is evaluation at $\id_i$ (because this is an initial object). The claim follows. 

Hence for $\mu: i \rightarrow i' \in \tm I$: 
\[ \mu^* f_\bullet \pi_1^{\tm I,*} \pi_{1,*}^{(\Delta^*S)} (\pi_{14}^*f)^\bullet  \cong  f(\mu)_\bullet i^*   \pi_{1,*}^{(\Delta^*S)} (\pi_{14}^*f)^\bullet  \cong f(\mu)_\bullet (\twwc{i})^* (\pi_{14}^*f)^\bullet  \cong   f(\mu)_\bullet \twwc{i}^*.  \]
On $\pi_1$-coCartesian objects this is isomorphic to 
$ \widetilde{\mu}^* $
for any $\widetilde{\mu}$ with $\pi_{14}(\widetilde{\mu}) = \mu$. From this it follows that unit and counit are isomorphisms on $\pi_{1}$-coCartesian objects.  
\end{proof}

\begin{PROP}\label{PROPKAN6}
With the setting as in \ref{PARAXIOMSEXT}. Let $I \in \Dirlf$ be a diagram. 
\begin{enumerate}
\item 
Let $F := \pi^*_1 S: \tw I \rightarrow \mathcal{M}^{(\new)}$ with $S: I \rightarrow \mathcal{M}^{(\new)}$ (i.e.\@ a diagram of stacks instead of a diagram of correspondences of stacks) and let $U \rightarrow S$ be an atlas, and $\Flip(F') \in \HH^{\cor}(\mathcal{M}^{(\old)})(\twwc I \times \Delta^{\op} \times \Delta)$ the Main Construction~\ref{DEFMC} applied to $(\pi^*_1U \rightarrow F)$. 
Then there is an equivalence (functorial in $I, S$, i.e.\@ of fibered derivators):
\[ \DD^{(\old)}(\twwc{I} \times \Delta^{\op} \times  \Delta   )^{3,5-\cart, 2,4,6-\cocart}_{\Flip(F')}   \cong  \DD^{!,h}(I)_{S}  \]
where $\DD^{!,h} \rightarrow \mathbb{M}$ is the naive extension (cf.\@ \ref{PARAXIOMSEXT}) of the restriction of $\DD^{(\old)}$ to $\mathbb{M}^{(\old)}$. 

\item
Let $F := \pi^*_2 S: \tw I \rightarrow \mathcal{M}^{(\new)}$ with $S: I^{\op} \rightarrow \mathcal{M}^{(\new)}$ and let $U \rightarrow S$ be an atlas, and $\Flip(F') \in \HH^{\cor}(\mathcal{M}^{(\old)})(\twwc I \times \Delta^{\op} \times \Delta)$ the Main Construction~\ref{DEFMC} applied to $(\pi^*_2 U \rightarrow F)$. 
Then there is an equivalence (functorial in $I, S$, i.e.\@ of fibered derivators):
\[  \DD^{(\old)}(\twwc{I} \times \Delta^{\op} \times  \Delta   )^{3,5-\cart, 2,4,6-\cocart}_{\Flip(F')}   \cong \DD^{*,h}(I)_{S^{\op}}  \]
where $\DD^{*,h} \rightarrow \mathbb{M}^{\op}$ is the naive extension (cf.\@ \ref{PARAXIOMSEXT}) of the restriction of $\DD^{(\old)}$ to $\mathbb{M}^{(\old), \op}$. 
\end{enumerate}
\end{PROP}

\begin{proof}Recall \ref{PARFPSOBJ} and Definitions~\ref{DEFFPS}--\ref{DEFFPSOBJ}.
We show 1.\@ the other assertion being dual. 

By the Main Construction~\ref{DEFMC}:
\[ F'(\mu, \Delta_m, \Delta_n) = (\Cech(U \to S)(\mu_{1}) \widetilde{\times}_{S(\mu_{2})} \Cech(U \to S)(\mu_{2}))_{m,n}.  \]
Define 
\begin{eqnarray*}
 F_2(\mu, \Delta_m, \Delta_n) &:=& (\Cech(U \to S)(\mu_{1}) \widetilde{\times}_{S(\mu_{2})} \delta(S(\mu_{2})))_{m,n}  \\
 G(\mu, \Delta_m) &:=& \Cech(U \to S)(\mu_{1})_m
\end{eqnarray*}

Note that, despite the fact that $\delta(S(\mu_{2}))$ is not in $\mathcal{M}^{(\old)}$, the object in $F_2(\mu, \Delta_m, \Delta_n)$ is in $\mathcal{M}^{(\old)}$ because it is weakly equivalent to $\Cech(U \to S)(\mu_{1})_m$ which
is in $\mathcal{M}^{(\old)}$.  

We have obvious morphisms 
\[ \xymatrix{ F' \ar[r]^-{\alpha} & F_2 \ar[r]^-{\gamma}  & G  }. \]
Here $\gamma$ consist of weak equivalences and $\alpha$ is type-1,3,5-admissible. We claim that $\alpha$ induces an equivalence
\[ \alpha^*: \DD^{(\old)}(\twwc I \times \Delta^{\op} \times \Delta)_{\Flip(F_2)}^{3,5-\cart, 2,4,6-\cocart}  \rightarrow 
\DD^{(\old)}(\twwc I \times \Delta^{\op} \times \Delta)_{\Flip(F')}^{3,5-\cart, 2,4,6-\cocart} \]

Indeed,
point-wise in $\twwc I \times \Delta^{\op}$ the functor $\alpha^*$ is an equivalence of the form 
\[ \DD^{(\old)}(\Delta)_{(\Cech(U_X \to X) \widetilde{\times}_{Y} \Cech(U_Y \to Y))_{m,\bullet}}^{\cocart} \cong \DD^{(\old)}(\Delta)_{(\Cech(U_X \to X) \widetilde{\times}_{Y} \delta(Y))_{m,\bullet}}^{\cocart}  \]
as follows from the axioms (D1) and (D2) as in Lemma~\ref{LEMMA34}.
Therefore, as in Lemma~\ref{LEMMAEXISTENCE3FUNCTORS}, 3.\@, the functors $\alpha^*$ and $\Box_*^{(6)}\alpha_{*}$ preserve the conditions of being simultaneously 3,5-Cartesian and 2,4,6-coCartesian. 
Since they are both computed point-wise in $\twwc I \times \Delta^{\op}$  the claim follows. 

As in Lemma~\ref{LEMMAKAN2} the functor $\pi_{15}^*$ induces an equivalence
\[ \DD^{(\old)}(I \times \Delta^{\op})_{S}^{2-\cart}  \rightarrow \DD^{(\old)}(\twwc I \times \Delta^{\op} \times \Delta)_{\pi_{15}^*S = \Flip(G)}^{3,5-\cart, 2,4,6-\cocart}.      \]

By assumption (cf.\@ \ref{PARAXIOMSEXT}) the restriction of $\DD^{(\old)}$ to $\mathbb{M}^{(\old)}$ has a naive extension $\DD^{!,h}$ to $\mathbb{M}$ and we have by (D1) and (D2):
\begin{eqnarray*}
   \DD^!(I \times \Delta^{\op})_{\Cech(U \to S)}^{2-\cart} &\cong& \DD^{!,h}(I)_{S} .
\end{eqnarray*}
\end{proof}

\begin{LEMMA}~\label{LEMMAKAN5}
With the setting as in \ref{PARAXIOMSEXT}, let $\DD^{(\new)} \rightarrow \HH^{\cor, \cov}(\mathcal{M}^{(\new)})$ be the morphism of (symmetric) 2-pre-(multi)derivators from Definition~\ref{DEFDER6FU2}. 
Let $\alpha: I \rightarrow J$ be an opfibration and $U \rightarrow F$ an element of $\HH^{\cor, \cov}(\mathcal{M}^{(\new)})(J)$. Then $\alpha^*: \DD^{(\new)}(J)_{(U \to F)} \rightarrow \DD^{(\new)}(I)_{\alpha^*(U \to F)} $ has a left adjoint $\alpha_!^{(U \to F)}$.
\end{LEMMA}
\begin{proof}
We have to show that
\[ (\twwc{\alpha})^*: \DD^{(\old)}(\twwc J \times \Delta^{\op} \times \Delta)^{3,5-\cart, 2,4,6-\cocart}_{\Flip((\alpha^*F)')} \rightarrow \DD^{(\old)}(\twwc I \times \Delta^{\op} \times \Delta)^{3,5-\cart, 2,4,6-\cocart}_{\Flip(F')} \]
has a left adjoint. The right hand side category is by Lemma~\ref{LEMMAKAN2} equivalent to  
\[ \DD^{(\old)}((\twwc{J}) \times_J I \times \Delta^{\op} \times \Delta)^{3,5-\cart, 2,4,6-\cocart}_{\Flip(F')}, \]
hence we have to show that
\[ \pr_1^*: \DD^{(\old)}(\twwc{J} \times \Delta^{\op} \times \Delta )^{ 3,5-\cart, 2,4,6-\cocart}_{\Flip(F')} \rightarrow \DD^{(\old)}((\twwc{J}) \times_J I \times \Delta^{\op} \times \Delta)^{3,5-\cart, 2,4,6-\cocart}_{ \alpha^* \Flip(F')} \]
has a left adjoint. By assumption, the functor has  a left adjoint $\pr_{1,!}$ forgetting the coCartesianity conditions. 
It suffices to show that $\pr_{1,!}$ preserves the conditions of being 3,5-Cartesian and 2,4,6-coCartesian. 

{\em Case of type 2,4 and 6:} Let $\kappa: j \rightarrow j'$ be a morphism of some type 2,4, or 6 in $\twwc{J} \times \Delta^{\op} \times \Delta$. 
Denote \[  f := \Flip(F')(\kappa)  \] 
the corresponding morphism in $\mathcal{M}^{(\old)}$.
We have to show that the induced map
\[ f^* j^* \pr_{1,!} \rightarrow (j')^* \pr_{1,!} \]
is an isomorphism on $2,4,6$-coCartesian objects.
Note that $\pr_1$ is an opfibration with fibers over $j$ and $j'$ isomorphic to $I_{j_1}$. Let 
 $e_j: I_{j_1} \hookrightarrow (\twwc{J}) \times_J I \times \Delta^{\op} \times \Delta$ (resp.\@ $e_{j'}: I_{j_1} \hookrightarrow (\twwc{J}) \times_J I \times \Delta^{\op} \times \Delta$) be the inclusions.
It suffices to show that the natural morphism
\[  f^* \hocolim_{I_{j_1}} e_{j}^* \rightarrow \hocolim_{I_{j_1}} e_{j'}^*  \]
is an isomorphism on $2,4,6$-coCartesian objects. 
However, $f^*$ being a left adjoint commutes with homotopy colimits, this is to say that
\[ \hocolim_{I_{j_1}}  f^* e_{j}^* \rightarrow \hocolim_{I_{j_1}} e_{j'}^*  \]
has to be an isomorphism. 
However, the natural morphism
\[ f^*  e_{j}^* \rightarrow e_{j'}^*  \]
is already an isomorphism on $2,4,6$-coCartesian objects by definition. 

{\em Case of type 5:} 
Let $\kappa: j \rightarrow j'$ be a morphism of type 5 in $\twwc{J} \times \Delta^{\op} \times \Delta$. 
Denote \[ f := \Flip(F')(\kappa)  \] 
the corresponding morphism in $\mathcal{M}^{(\old)}$.
We have to show that the induced map
\[ f^! j^* \pr_{1,!} \rightarrow (j')^* \pr_{1,!} \]
is an isomorphism on $3,5$-Cartesian objects.
If $\kappa^{\op}: \Delta_{n'} \rightarrow \Delta_{n}$ is an {\em injective} morphism in $\Delta$ then $f$ is in $C$ (cf.\@ \ref{PARSETTINGEXT}) and thus $f^!$ preserves homotopy colimits (Axiom (H1) of \ref{PARAXIOMSEXT}).
Hence
\[ \hocolim_{I_{j_1}}  f^! e_{j}^* \rightarrow \hocolim_{I_{j_1}} e_{j'}^*  \]
has to be an isomorphism. 
However, the natural morphism
\[ f^! e_{j}^* \rightarrow e_{j'}^*  \]
is already an isomorphism on $3,5$-Cartesian objects by definition. As in the proof of Lemma~\ref{LEMMAPRESERVE1} the case of injective morphisms in $\Delta$ is sufficient. 

{\em Case of type 3:} 
Let $\kappa: j \rightarrow j'$ be a morphism of type 3 in $\twwc{J} \times  \Delta$. 
Denote \[  f := \Flip(F')(\kappa)  \] 
the corresponding morphism in $(\mathcal{M}^{(\old)})^{\Delta^{\op}}$.
We have to show that the induced map
\[ f^! j^* \pr_{1,!} \rightarrow (j')^* \pr_{1,!} \]
is an isomorphism on $3,5$-Cartesian objects.
Consider the diagram: 
 \[ \xymatrix{
\DD^{(\old)}(I_{j_1} \times \Delta^{\op})_{\pr_1^*\Flip(F')(j')}^{5-\cart} \ar[r]^-{\pr_{1,!}} \ar[d]_{f^!} &   \DD^{(\old)}(\Delta^{\op})_{\Flip(F')(j')}^{\cart} \ar[d]^{f^!} \\
\DD^{(\old)}(I_{j_1} \times \Delta^{\op})_{\pr_1^*\Flip(F')(j)}^{5-\cart} \ar[r]_-{\pr_{1,!}} & \DD^{(\old)}(\Delta^{\op})_{\Flip(F')(j)}^{\cart}
} \] 
Preservation of Cartesianity for type 3 morphisms follows if this diagram is 2-commutative (via the exchange morphism). Note that we already showed that $\pr_{1,!}$ preserves 5-Cartesian objects. 
This is, however, the exchange of the diagram
 \[ \xymatrix{
\DD^{(\old)}(I_{j_1} \times \Delta^{\op})_{\pr_1^*\Flip(F')(j')}^{5-\cart} \ar@{<-}[r]^-{\pr_1^*} \ar[d]_{f^!} &   \DD^{(\old)}(\Delta^{\op})_{\Flip(F')(j')}^{\cart} \ar[d]^{f^!} \\
\DD^{(\old)}(I_{j_1} \times \Delta^{\op})_{\pr_1^*\Flip(F')(j)}^{5-\cart} \ar@{<-}[r]_-{\pr_1^*} & \DD^{(\old)}(\Delta^{\op})_{\Flip(F')(j)}^{\cart}
} \] 
in which the vertical functors are equivalences. The first diagram is therefore 2-commutative. 

Hence
\[ \hocolim_{I_{j_1}}  f^! e_{j}^* \rightarrow \hocolim_{I_{j_1}} e_{j'}^*  \]
has to be an isomorphism. 
However, the natural morphism
\[ f^! e_{j}^* \rightarrow e_{j'}^*  \]
is already an isomorphism on $3,5$-Cartesian objects by definition. 
\end{proof}

\begin{BEM}\label{REMKANEXT}
For an opfibration $\alpha$ the proof shows that 
we can actually just take
\[ \alpha_!^{(U \to F)} = (\twwc \alpha)_!^{(\Flip(F'))}\]
where the right hand side is the relative left Kan extension in $\DD^{(\old)}$. 
Indeed by construction (we omit the bases of the rel.\@ Kan extensions for $\DD^{(\old)}$)
\[ \alpha_!^{(U \to F)} = \pr_{1,!} q_{2,*} q_{1,!}. \]
However, applying this to an object in the category
\[  \DD^{(\old)}(\twwc{I} \times \Delta^{\op} \times \Delta)^{3,5-\cart, 2,4,6-\cocart}_{\Flip(F')}  \]
 $q_{2,*}$ receives an object which lies in the essential image of $q_{2}^*$ and we have proven
\[ \id \cong q_{2,*} q_{2}^*. \]
Hence, taking adjoints, we also have
\[ q_{2,!} q_{2}^* \cong \id . \]
Hence on the essential image of $q_{2}^*$, we have an isomorphism $q_{2,!} \cong q_{2,*}$.
\end{BEM}

\section{Conclusion}

In the following discussion the words ``multi'' and ``symmetric'' have been set in parenthesis to emphasize that all statements are true with or without the multi-structure or symmetric multi-structure. 
That is, extensions of a $*,!$-formalism are constructed, for which the multi-structure (e.g.\@ the tensor product in a symmetric six-functor-formalism) does not play any role. 
However, if it exists, then the extension carries the same structure. 

\begin{SATZ}\label{SATZSTEP} Let the setting be as in \ref{PARAXIOMSEXT}.
The morphism of (symmetric) 2-pre-(multi)derivators
\[ \DD^{(\new)} \rightarrow \HH^{\cor, \cov}(\mathcal{M}^{(\new)})  \]
constructed in Definition~\ref{DEFDER6FU2} is an infinite (symmetric) left fibered (multi)derivator with domain $\Dirlf$.
\end{SATZ}
\begin{proof}The 2-pre-multiderivator $\DD^{(\new)}$, as defined in Definition~\ref{DEFDER6FU2}, satisfies axioms (Der1${}^\infty$) and (Der2) because $\DD^{(\old)}$ satisfies them. The first part of axiom (FDer0 left) was shown in Lemma~\ref{LEMMA6FUOPFIB} and the second part (point-wise computation of the push-forward) follows from Lemma~\ref{LEMMAEXISTENCE3FUNCTORS}.
Instead of Axioms (FDer3--4 left) it is sufficient to show Axioms (FDer3--4 left') (cf.\@ \cite[Theorem 4.2]{Hor16}). (FDer3 left') is Lemma~\ref{LEMMAKAN5}, and
axiom (FDer4 left') follows from its proof.  
(FDer5 left): This is only needed for $n$-ary morphisms for $n \ge 2$, otherwise it follows from the other axioms. The axiom also holds for a composition as soon as it holds for the individual factors. 
It is also sufficient to show the axiom for $J=\cdot$. This boils it down to a statement about the naive extension $\DD^{*,h}$ as fibered {\em multi}\,derivator (cf.\@ \ref{PARAXIOMSEXT}) in which (FDer5 left) holds by definition. 
\end{proof}

\begin{LEMMA}\label{LEMMAGENERATION}
In the situation of Theorem~\ref{SATZSTEP},
if $\DD^{(\old)} \rightarrow \HH^{\cor}(\mathcal{M}^{(\old)})$ has stable, well-generated (resp.\@ compactly generated) fibers then 
{\em the fibers of} $\DD^{(\new)} \rightarrow  \HH^{\cor, \cov}(\mathcal{M}^{(\new)})$ are right derivators with domain (at least) $\mathrm{Posf}$, stable and well-generated (resp.\@ compactly generated). 
\end{LEMMA}
\begin{proof}
By Lemma~\ref{LEMMAKAN2} there is an equivalence (compatible with pull-backs in $J$):
\[ \DD^{(\new)}(I \times J)_{\pr_1^* (U \to F)} \cong \DD^{(\old)}((\twwc I) \times \Delta^{\op} \times \Delta \times J)^{3,5-\cart, 2,4,6-\cocart}_{\pr_1^* \Flip(F')}.  \]
Hence the statement follows if we can show that for $\alpha: J_1 \rightarrow J_2$ in $\mathrm{Posf}$, the right Kan extension functor
\[ (\id \times \alpha)_*: \DD^{(\old)}((\twwc I)  \times \Delta^{\op} \times \Delta \times J_1)_{\pr_1^* \Flip(X')} \rightarrow \DD^{(\old)}((\twwc I) \times \Delta^{\op} \times \Delta \times J_2)_{\pr_1^* \Flip(X')} \]
respects the conditions of being $3,5$-Cartesian and $2,4,6$-coCartesian.
This follows because the commutation with homotopy colimits implies that all functors $g^*$  involved in the 
definitions of coCartesian are {\em exact} and hence commute also with {\em homotopy limits of shape $\mathrm{Posf}$} (actually {\em homotopy finite} is sufficient, cf.\@ \cite[Theorem 7.1]{PS14}). 
By \cite[Lemma~4.1.6]{Hor15}, to see that the left fibered derivator $\DD^{(\new)} \rightarrow \HH^{\cor, \cov}(\mathcal{M}^{(\new)})$ has again well-generated (resp.\@ compactly generated) fibers it suffices to see that 
the fibers for $I=\cdot$ are well-generated (resp.\@ compactly generated). By Proposition~\ref{PROPKAN6} those are equal to the corresponding fiber of the naive extension $\DD^{!,h}$ (or just as well $\DD^{*,h}$). By assumption (cf.\@ \ref{PARAXIOMSEXT}) the naive extension $\DD^{!,h}$ (or just as well $\DD^{*,h}$) has well-generated fibers. 
\end{proof}

\begin{HAUPTSATZ}\label{HAUPTSATZ}Let $\mathcal{S}$ be a small category with finite limits and Grothendieck pre-topology. Assume that $\mathcal{S}$ is geometric (cf.\@ Definition~\ref{DEFGEOMETRIC}).

Let $\DD \rightarrow \SSS^{\cor}$ be an infinite (symmetric) fibered  (multi)derivator (i.e.\@ a derivator *,!-formalism, resp.\@ six-functor-formalism) 
with domain\footnote{With domain $\Dirlf$ would be sufficient.} $\Cat$, with stable, well-generated fibers which is local w.r.t.\@ the Grothendieck pre-topology (cf.\@ Definition~\ref{DEFLOCAL6FU}).
Recall from \ref{DEFSHCOR} the 2-pre-(multi)derivator $\SSS^{\hcor, (\infty)}$ of correspondences of higher geometric stacks.  

There exists an infinite (symmetric) fibered (multi)derivator (i.e.\@ a derivator *,!-formalism, resp.\@ six-functor-formalism)
\[ \DD^{(\infty)} \rightarrow \SSS^{\hcor, (\infty)}  \]
with domain $\Cat$, with stable, well-generated fibers which is an {\em admissible extension}, i.e.\@ which satisfies:

\begin{enumerate}
\item The pull-backs of $\DD^{(\infty)}$ via
\[ \mathbb{M}^{(\infty)} \rightarrow \SSS^{\hcor, (\infty)} \quad \text{resp.\@} \quad (\mathbb{M}^{(\infty)})^{\op} \rightarrow \SSS^{\hcor, (\infty)} \]
are equivalent to the naive extensions (cf.\@ Theorems~\ref{SATZEXTNAIVELEFT}--\ref{SATZEXTNAIVERIGHT}) of $\DD^!$, and $\DD^*$, respectively. In particular, we have equivalences
\[ \DD_X^{(\infty)} \cong \DD_X^{!,h} \cong \DD_X^{*,h}  \]
of fibers for every $k$-geometric stack $X$. 
\item The pull-back of $\DD^{(\infty)}$ via
\[ \SSS^{\cor} \rightarrow \SSS^{\hcor, (\infty)} \]
is equivalent to $\DD$. 
\end{enumerate}
\end{HAUPTSATZ}
\begin{proof}
Let $(\mathcal{M}, \Fib, \mathcal{W})$ be a category of fibrant objects representing the \v{C}ech localization of simplicial pre-sheaves on $\mathcal{S}$ as in \ref{PARSETTINGTV}. 
Recall from Definition~\ref{DEFKGEOMETRIC} the homotopy strictly full subcategories $\mathcal{M}^{(k)} \subset \mathcal{M}$ of $k$-geometric stacks.  

The construction will be done in the following steps:
\begin{enumerate}
\item Extend $\DD$ to a (symmetric) fibered (multi)derivator
\[ \DD^{(-1)} \rightarrow \SSS^{\amalg, \cor}  \]
i.e.\@ to the coproduct completion, with domain $\Cat$.
\item Extend $\DD$ to a (symmetric) {\em left} fibered (multi)derivator
\[ \DD^{(0)} \rightarrow \HH^{\cor}(\mathcal{SH}^{(0)}(\mathcal{S})) \cong \HH^{\cor}(\mathcal{M}^{(0)})  \]
i.e.\@ to ordinary sheaves on $\mathcal{S}$ which are coproducts of representables, with domain $\Dirlf$. 
\item Extend $\DD^{(0)}$ to $\Cat$ using the theory of \cite{Hor17b} and establish that $\DD^{(0)}$ is also a right fibered (multi)derivator using Brown representability. 
\item Verify that $\DD^{(0)}$ is an {\em admissible} extension of $\DD$ (cf.\@ Definition~\ref{DEFADMEXT}).
\end{enumerate}

After these steps we dispose of an infinite (symmetric) fibered (multi)derivator
\[ \DD^{(0)} \rightarrow  \HH^{\cor}(\mathcal{M}^{(0)})  \]
with domain $\Cat$ which is an admissible extension of $\DD$ and we extend it to $\HH^{\cor}(\mathcal{M}^{(k)})$ by induction on $k \ge 1$  using the following analogous steps. 

\begin{enumerate}
\item[5.] Extend $\DD^{(k-1)}$ to an infinite (symmetric) {\em left} fibered (multi)derivator
\[ \DD^{(k)} \rightarrow  \HH^{\cor}(\mathcal{M}^{(k)})  \]
with domain $\Dirlf$. 
\item[6.] Extend $\DD^{(k)}$ to $\Cat$ using the theory of \cite{Hor17b} and establish that $\DD^{(k)}$ is also a right fibered (multi)derivator using Brown representability. 
\item[7.] Verify that $\DD^{(k)}$ is again an {\em admissible} extension of $\DD$ (cf.\@ Definition~\ref{DEFADMEXT}).
\end{enumerate}

We proceed to describe the steps in detail: 

{\bf Step 1.\@} The extension has been constructed in \ref{PAREXTAMALG}. We denote it by $\DD^{(-1)} \rightarrow \SSS^{\amalg, \cor}$.

{\bf Step 2.\@}  We apply Theorem~\ref{SATZSTEP} to the setting of Example~\ref{MAINEXEXT2} (using that $\mathcal{S}$ is geometric by assumption), i.e.\@ we set
$\mathcal{M} := \mathcal{SET}^{\mathcal{S}^{\op}}$  the category of  ordinary pre-sheaves (with {\em trivial} structure as simplicial category of fibrant objects),
$\mathcal{M}^{(\old)} := \mathcal{S}^{\amalg}$ and $\mathcal{M}^{(\new)} := \iota(\mathcal{SH}^{(0)})$ and $C$ is
the class of local isomorphisms. Then, by assumption and Lemma~\ref{LEMMAVALIDITY}, $\DD^{(\old)} := \DD^{(-1)} \rightarrow \HH^{\cor}(\mathcal{M}^{(\old)}) = \SSS^{\amalg, \cor}$ satisfies the axioms of \ref{PARAXIOMSEXT}. 
Theorem~\ref{SATZSTEP} thus gives an infinite (symmetric) left fibered (multi)derivator with domain $\Dirlf$
\[ \DD^{(\new)} \rightarrow \HH^{\cor, \cov}(\mathcal{M}^{(\new)}).   \]
By Proposition~\ref{PROPEQUIV} there is an equivalence of symmetric 2-pre-multiderivators
\[  \HH^{\cor, \cov}(\mathcal{M}^{(\new)})  \cong \HH^{\cor}(\mathcal{M}^{(\new)})    \]
and the latter is, by Lemma~\ref{LEMMAFUNCTHOCOR} in view of Lemma~\ref{LEMMAEQUIVSHEAVES}, equivalent to
$\HH^{\cor}(\mathcal{M}^{(0)})$.
This allows to pull-back $\DD^{(\new)}$ to a (symmetric) left fibered (multi)derivator (still with domain $\Dirlf$) 
\[ \DD^{(0)} \rightarrow \HH^{\cor}(\mathcal{M}^{(0)})  \]
(cf.\@ \cite[Remark~12.2]{Hor17} for a discussion of strictification).

{\bf Step 3.\@}  By \cite[Corollary 1.3]{Hor17b}, $\DD^{(0)} \rightarrow \HH^{\cor}(\mathcal{M}^{(0)})$ extends to $\Cat$. Note that
$\HH^{\cor}(\mathcal{M}^{(0)})$ is already defined on $\Cat$. The reader may check that the techniques of \cite{Hor17b} go through for the case of 2-pre-multiderivators instead of pre-multiderivators, because the formal properties of $\Dia^{\cor}(\SSS)$ used in \cite{Hor17b} for a 2-pre-multiderivator $\SSS$ remain exactly the same by \cite[Section~3]{Hor16}. Also observe that it is sufficient that the input has domain $\Dirlf$, because the functor $N$ of \cite[Proposition~3.7]{Hor17b}
factors $N: \Cat \rightarrow \Dirlf \subset \Dir$ (because factoring a morphism in $N(I)$ into non-identities boils down to factoring the corresponding map between simplicies). 
Lemma~\ref{LEMMAGENERATION} shows that  the fibers of $\DD^{(0)} \rightarrow \HH^{\cor}(\mathcal{M}^{(0)})$ are right derivators (with domain $\Posf$) which are stable and well-generated\footnote{this straightforwardly extends to fibers over diagrams of shape $I \in \Cat$ because the subcategories of Cartesian objects used to define the extensions are preserved by homotopy limits. }. 
Finally, Brown representability (applied as in the proof of \cite[Theorem 6.3]{Hor16}) implies that  $\DD^{(0)} \rightarrow \HH^{\cor}(\mathcal{M}^{(0)})$ is also a right fibered (multi)derivator.

{\bf Step 4.\@}  We construct a morphism of 2-pre-multiderivators
\[ \HH^{\cor}(\mathcal{M}^{(\old)}) \rightarrow \HH^{\cor,\cov}(\mathcal{M}^{(\new)})  \]
mapping an admissible $U \in \mathcal{M}^{(\old), \tw I}$ to $U \rightarrow F:=\iota L X$ (cf.\@ \ref{PARINTROGEOM} for the notation) in $\mathcal{M}^{\tw I}$, which we shall call the canonical atlas. For the first statement of admissibility, it suffices to show that
the pull-back of $\DD^{(\new)}$ along this morphism is equivalent to $\DD^{(-1)}$. 
Let $\Flip(F') \in \HH^{\cor}(\mathcal{M}^{(\old)})(\twwc I \times \Delta \times \Delta^{\op})$ the Main Construction~\ref{DEFMC} applied to the canonical atlas.

Claim: There is an equivalence (compatible with push-forward and pull-back functors)
\[ \xymatrix{
 \DD^{(-1)}(I)_{F} \ar[d]^-{\sim} \\
 \DD^{(-1)}(\twwc{I} \times \Delta^{\op} \times  \Delta   )^{3,5-\cart, 2,4,6-\cocart}_{\Flip(F')}  } \]

We have by the Main Construction~\ref{DEFMC}:
\[ F'(\mu, \Delta_n, \Delta_m) = (\Cech(U \rightarrow F)(\mu_{13}) \widetilde{\times}_{F(\mu_{23})} \Cech(U \rightarrow F)(\mu_{24}))_{n,m} \]
for $\mu \in \twtw I$. 

Recall that on $\mathcal{M}$ the trivial structure of simplicial category of fibrant objects has been chosen, hence $\Cech(U \rightarrow F)_n = U \times_F \cdots \times_F U$. 
Therefore there are morphisms 
\[ U \rightarrow \Cech(U \rightarrow F) \]
where $U$ is considered constant. Define

\begin{eqnarray*}
 F_2(\mu, \Delta_n, \Delta_m) &:=& U(\mu_{13}) \times_{F(\mu_{23})} U(\mu_{24})  \\
 F_3(\mu, \Delta_n, \Delta_m) &:=& U(\mu_{13}) \times_{U(\mu_{23})} U(\mu_{24})  \\
 F_4(\mu) &:=& U(\mu_{14})  
\end{eqnarray*}

We have morphisms 
\[ \xymatrix{ F' & \ar[l]_-{\alpha_1}  F_2 \ar[r]^-{\alpha_1} &  F_3 & F_4 \ar[l]_-{\gamma} } \]
and it is clear that all morphisms are local isomorphisms.  Therefore by (co)homological descent (suffices here in the form of axiom (D1) \ref{PARAXIOMSEXT})
we have
\[ \xymatrix{
 \DD^{(-1)}(\twwc{I} \times \Delta^{\op} \times  \Delta  )^{3,5-\cart, 2,4,6-\cocart}_{\Flip(F_4)=\pi_{14}^*\widetilde{U}} \ar[r]^-{\sim} &
 \DD^{(-1)}(\twwc{I} \times \Delta^{\op} \times  \Delta   )^{3,5-\cart, 2,4,6-\cocart}_{\Flip(F')}  } \]
 where $\widetilde{U}$ is given by the composition 
 \[ \xymatrix{ \tw ({}^{\downarrow  \downarrow} I)  \ar[r]^-{P} & \tw I \ar[r]^-U  & \mathcal{M}  }  \]

From Lemma~\ref{LEMMAKAN4} follows that $\Delta^*$ and $\pi_{14}^*$ induce equivalences
\[ \xymatrix{  \DD^{(-1)}(I)_{U \cong \Delta^* \widetilde{U}}  & \ar[l]_-{\Delta^*}^-\sim   \DD^{(-1)}(\tm I)^{\pi_{1}-\cocart}_{\widetilde{U}} \ar[r]^-{\pi_{14}^*}_-\sim & \DD^{(-1)}(\twwc I)_{\Flip(F_4)=\pi_{14}^*\widetilde{U}}^{\pi_{1}-\cocart}.   }   \]

Furthermore, we have to show that for a diagram in $S \in (\mathbb{M}^{(\new)})^{I}$ (resp $S \in (\mathbb{M}^{(\new), \op})^{I}$) the value
\[ \DD^{(\new)}(I)_S \]
is just the value of the naive extension of $\DD^!$ (resp.\@ of $\DD^*$) in a functorial way.
This is Proposition~\ref{PROPKAN6}.

{\bf Step 5.\@}  We apply Theorem~\ref{SATZSTEP} to the setting of Example~\ref{MAINEXEXT1}, i.e.\@ we set
$(\mathcal{M}, \Fib, \mathcal{W})$ as in \ref{PARSETTINGTV},   
$\mathcal{M}^{(\old)} := \mathcal{M}^{(k-1)}$, $\mathcal{M}^{(\new)} := \mathcal{M}^{(k)}$ and $C$ is
the class $(k-1)-C$. Then by induction and Lemma~\ref{LEMMAVALIDITY} $\DD^{(\old)} := \DD^{(k-1)} \rightarrow \HH^{\cor}(\mathcal{M}^{(\old)}) = \HH^{\cor}(\mathcal{M}^{(k-1)}) $ satisfies the axioms of \ref{PARAXIOMSEXT}.
Theorem~\ref{SATZSTEP} thus gives an infinite (symmetric) left fibered (multi)derivator with domain $\Dirlf$
\[ \DD^{(\new)} \rightarrow \HH^{\cor, \cov}(\mathcal{M}^{(k)}).   \]
By Proposition~\ref{PROPEQUIV} there is an equivalence of symmetric 2-pre-multiderivators
\[  \HH^{\cor, \cov}(\mathcal{M}^{(k)})  \cong \HH^{\cor}(\mathcal{M}^{(k)}).    \]
This allows to pull-back $\DD^{(\new)}$ to a (symmetric) left fibered (multi)derivator with domain $\Dirlf$
\[ \DD^{(k)} \rightarrow \HH^{\cor}(\mathcal{M}^{(k)})  \]
(cf.\@ \cite[Remark~12.2]{Hor17} for a discussion of strictification).

{\bf Step 6.\@}  This is the same as Step 3.\@

{\bf Step 7.\@}  Like in Step 4, we have a morphism of 2-pre-multiderivators
\[ \HH^{\cor}(\mathcal{M}^{(k-1)}) \rightarrow \HH^{\cor,\cov}(\mathcal{M}^{(k)})  \]
mapping a diagram $U \in \mathcal{M}^{(k-1), \tw I}$ to the identity $U=F:=U$  in $\mathcal{M}^{(k),\tw I}$, which we shall call the canonical atlas. For the first statement of admissibility it suffices to show that
the pull-back of $\DD^{(\new)}$ along this morphism is equivalent to $\DD^{(k-1)}$. 
Let $\Flip(F') \in \HH^{\cor}(\mathcal{M}^{(k-1)})(\twwc I \times \Delta \times \Delta^{\op})$ the Main Construction~\ref{DEFMC} applied to the canonical atlas.

Claim: There is an equivalence (compatible with push-forward and pull-back functors)
\[ \xymatrix{
 \DD^{(k-1)}(I )_{U} \ar[d]^-{\sim} \\
 \DD^{(k-1)}(\twwc{I} \times \Delta^{\op} \times  \Delta   )^{3,5-\cart, 2,4,6-\cocart}_{\Flip(F')}  } \]

We have by the Main Construction~\ref{DEFMC}:
\[ F'(\mu, \Delta_n, \Delta_m) = (\Cech(U \rightarrow F)(\mu_{13}) \widetilde{\times}_{F(\mu_{23})} \Cech(U \rightarrow F)(\mu_{24}))_{n,m} \]
for $\mu \in \twtw I$. Define

\begin{eqnarray*} 
 F_2(\mu, \Delta_n, \Delta_m) &:=& (\Cech(U \rightarrow F)(\mu_{13}) \widetilde{\times}_{F(\mu_{23})} \delta(F(\mu_{24})))_{n,m}  \\
 F_3(\mu, \Delta_n, \Delta_m) &:=& (\delta(F(\mu_{13})) \widetilde{\times}_{F(\mu_{23})} \delta(F(\mu_{24})))_{n,m}  \\
 G_1(\mu) &:=& F(\mu_{13}) \widetilde{\times}_{F(\mu_{23})} F(\mu_{24})  \\
 G_2(\mu) &:=& F(\mu_{14}) \widetilde{\times}_{F(\mu_{14})} F(\mu_{14})  \\
 H(\mu) &:=& F(\mu_{14})  
\end{eqnarray*} 

We have morphisms (obviously all point-wise weak equivalences)
\[ \xymatrix{ F' \ar[r]^-{\alpha_1} & F_2 \ar[r]^-{\alpha_1} &  F_3 & G_1 \ar[l]_-{\gamma}  &  G_2   \ar[l]_-{\beta} &  H \ar[l]_-{\delta} }. \]
The morphism $\gamma$ was defined in Lemma~\ref{LEMMAHOFP}, 2.
Therefore by (co)homological descent (suffices here in the form of axiom (D1) \ref{PARAXIOMSEXT})
we have
\[ \xymatrix{
 \DD^{(k-1)}(\twwc{I} \times \Delta^{\op} \times  \Delta  )^{3,5-\cart, 2,4,6-\cocart}_{\Flip(F_4)=\pi_{14}^*\widetilde{U}} \ar[r]^-{\sim} &
 \DD^{(k-1)}(\twwc{I} \times \Delta^{\op} \times  \Delta   )^{3,5-\cart, 2,4,6-\cocart}_{\Flip(F')}  } \]
 where $\widetilde{U}$ is given by the composition 
 \[ \xymatrix{ \tw ({}^{\downarrow  \downarrow} I)  \ar[r]^-{P} & \tw I \ar[r]^-U  & \mathcal{M}.  }  \]

From Lemma~\ref{LEMMAKAN4} follows that $\Delta^*$ and $\pi_{14}^*$ induce equivalences
\[ \xymatrix{  \DD^{(k-1)}(I)_{U \cong \Delta^* \widetilde{U}}  & \ar[l]_-{\Delta^*}^-\sim   \DD^{(k-1)}(\tm I)^{\pi_{1}-\cocart}_{\widetilde{U}} \ar[r]^-{\pi_{14}^*}_-\sim & \DD^{(k-1)}(\twwc I)_{\Flip(F_4)=\pi_{14}^*\widetilde{U}}^{\pi_{1}-\cocart}.   }   \]

Furthermore, we have to show that for a diagram in $X \in (\mathbb{M}^{(k)})^I$ (resp $X \in (\mathbb{M}^{(k), \op})^I$) the value
\[ \DD^{(k)}(I)_X \]
is just the value of the naive extension of $\DD^!$ (resp.\@ of $\DD^*$) in a functorial way. This is Proposition~\ref{PROPKAN6}.
\end{proof}

\section{Example: The stable homotopy categories and categories of motives}\label{SECTIONEXAMPLE}

Let $\mathcal{S}$ be a small category with finite limits and Grothendieck pre-topology. 
As input to Main Theorem~\ref{HAUPTSATZ}
a derivator six-functor-formalism over $\mathcal{S}$ is needed, i.e.\@ is a fibered multiderivator 
\[ \DD \rightarrow \SSS^{\cor}. \]
In this section, we recall from \cite[Section~15]{Hor17} that the algebraic derivators $\mathbb{SH}(...)$ of Ayoub which comprise
the stable homotopy categories and various kinds of Morel-Voevodsky motives give rise to such a six-functor-formalism. 

\begin{PAR}
To discuss the examples of Ayoub, we take $\mathcal{S} := \mathcal{SCH}_S$ the category of quasi-projective schemes over a base scheme $S$.
The setting of Ayoub has been generalized to more general schemes over $S$. Since it suffices for the purpose of this article to construct
the input derivator six-functor-formalisms on affine schemes of finite type over $S$ this is irrelevant here.
\end{PAR}

Recall \cite[Definition~4.4.23]{Ayo07II}:
\begin{DEF}\label{DEFCATCOEFF}
A {\bf category of coefficients} is a model category $\mathcal{M}$ with the following properties
\begin{enumerate}
\item $\mathcal{M}$ is left proper, cofibrantly generated, and stable;
\item the weak equivalences are stable under finite coproducts;
\item there is a subset $\mathcal{E} \subset \mathcal{M}$ of homotopically compact\footnote{\cite[Definition 4.4.22]{Ayo07I}} objects which generate $h(\mathcal{M})$ under arbitrary coproducts. 
\end{enumerate}
\end{DEF}

\begin{PAR}\label{PARSETTINGAYOUB}
Consider a triple $(\tau, \mathcal{M}, T)$ as in Ayoub \cite[Section~4.5]{Ayo07II} in which
\begin{itemize}
\item $\tau$ is either the etale or Nisnevich topology on $\mathcal{SCH}_S$.
\item $\mathcal{M}$ is a category of coefficients in the sense of Definition~\ref{DEFCATCOEFF}.
\item $T$ is a projectively cofibrant object of $\mathrm{PreShv}(\mathrm{Sm}/S, \mathcal{M})$ with the condition in~\cite[4.5.18]{Ayo07II}.
\end{itemize}
\end{PAR}

\begin{SATZ}\label{KORAYOUB}
There is a symmetric derivator six-functor-formalism (i.e.\@ a symmetric fibered multiderivator)
\[ \mathbb{SH}_{\mathcal{M}}^T \rightarrow \SSS^{\cor} \]
with domain $\Cat$ such that 
\begin{enumerate}
\item for a diagram $F: I \rightarrow \mathcal{S}^{\op}$ of quasi-projective schemes over $S$ (embedded via inclusion $\SSS^{\op} \rightarrow \SSS^{\cor}$), we have an equivalence of monoidal categories
\[ \mathbb{SH}_{\mathcal{M}}^T(I)_F \cong \mathbb{SH}_{\mathcal{M}}^T(F^{\op}, I^{\op}) \]
where the right hand side is the ``algebraic derivator'' defined by Ayoub \cite[D\'efinition~4.2.21]{Ayo07II};
\item the push-forward along a multicorrespondence 
\[ \xymatrix{ & & & A  \ar[llld]_{g_1}\ar[ld]^{g_n}\ar[rd]^f \\
S_1 & \cdots & S_n & ; & T 
 }\]
 in $\SSS^{\cor}(\cdot)$ is given up to unique isomorphism by
\[ f_! ( L g^*_1 - \overset{L}{\otimes} \cdots \overset{L}{\otimes} Lg_n^* -) \]
with the functors $f_!$ of \cite[Proposition~1.6.46]{Ayo07I} and the $g_i^*$ as in \cite[Th\'eor\`eme~4.5.23]{Ayo07II}, cf.\@ also \cite[Scholie~1.4.2]{Ayo07I};
\item The fibered multiderivator is infinite (i.e.\@ satisfies (Der1${}^\infty$)) and has stable, well-generated fibers. 
\end{enumerate}
\end{SATZ}
\begin{proof}
\cite[Corollary~15.9]{Hor17}.
\end{proof}

\begin{BEM}
By \cite[Corollary~15.9]{Hor17}, there is also a {\em proper} symmetric derivator six-functor-formalism in the sense of \cite[Definition~10.2]{Hor17} (encoding also the morphism $f_! \rightarrow f_*$, which is an isomorphism for projective $f$), i.e.\@ an oplax left fibered multiderivator, resp.\@ a lax right fibered multi-derivator 
\[  \mathbb{SH}_{\mathcal{M}}^T \rightarrow \SSS^{\cor,0,\oplax} \qquad \mathbb{SH}_{\mathcal{M}}^T \rightarrow \SSS^{\cor,0,\lax}.  \]
However, those will not be used (yet) for the extension to stacks envisioned in this article. 
\end{BEM}

Essential for the extension to stacks is the following:

\begin{SATZ}[Ayoub]\label{SATZAYOUB3}
The symmetric derivator six-functor-formalism 
\[ \mathbb{SH}_{\mathcal{M}}^T \rightarrow \SSS^{\cor} \]
of Corollary~\ref{KORAYOUB} is {\em strongly local} (cf.\@ Definition~\ref{DEFSTRONGLYLOCAL}) w.r.t.\@ 
to the ``Nisnevich-smooth'' pre-topology, in which coverings are collections $\{U_i \rightarrow X\}$ of smooth morphisms which have (as collection) a Nisnevich refinement (and thus are, in particular, jointly surjective). 
\end{SATZ}
\begin{proof}
(O1) Follows from the isomorphism $f^! 1 \cong \mathrm{Th}(\Omega_f)(1)$ for $f$ smooth \cite[Scholie~1.4.2, 3]{Ayo07I} using \cite[Section~2.3.3]{Ayo07I} and the fact that $\mathrm{Th}(\Omega_f)$ is an equivalence.

(O2) is stated in \cite[Section~2.3.4]{Ayo07I}.

(O3) follows from Nisnevich conservativity \cite[Corollaire~1.4.4]{Ayo07I}. 

(O4) follows from the isomorphism $f^! 1 \cong \mathrm{Th}(\Omega_f)(1)$ \cite[Scholie~1.4.2, 3]{Ayo07I} together with \cite[Theorem~1.5.9]{Ayo07I}.
\end{proof}

\begin{BEISPIEL}\label{EXDA}
For example take $\mathcal{M}$ to be (unbounded) complexes of $\Lambda$-modules for a commutative ring $\Lambda$, $T = (\PP^1_S, \infty_S) \otimes \Lambda$, and $\tau$ the etale topology. Then \cite[\S 3]{Ayo14}: 
\[ \mathbb{SH}_\mathcal{M}^T(\cdot)_X \cong \mathcal{DA}^{et}(X, \Lambda)  \]
where the right hand side is the category of etale Voevodsky motives (without transfers) on $X$.
Denote by $\mathbb{DA}^{et}_{S, \Lambda}$ the corresponding derivator six-functor-formalism. 
 For this example, one has even the following strengthening: 
\end{BEISPIEL}

\begin{SATZ}[Ayoub]\label{SATZAYOUB4}
The symmetric derivator six-functor-formalism 
\[ \mathbb{DA}^{et}_{S, \Lambda} \rightarrow \SSS^{\cor} \]
of Example~\ref{EXDA} is {\em strongly local} (cf.\@ Definition~\ref{DEFSTRONGLYLOCAL}) w.r.t.\@ 
to the smooth pre-topology, in which coverings are collections $\{U_i \rightarrow X\}$ of smooth morphisms which are jointly surjective.
\end{SATZ}
\begin{proof}
The verification (O1) and (O3--4) has been done in \ref{SATZAYOUB3} for all smooth morphisms already. (O2) is \cite[Lemme~3.4]{Ayo14} taking into account that a smooth covering has an \'etale refinement \cite[Tag~055V]{SP}. 
\end{proof}

Let $\mathcal{S}$ the category of quasi-affine schemes of finite type over $S$. Note that  $\mathcal{S}$ has a small skeleton, hence w.l.o.g.\@ $\mathcal{S}$ is small. 
Equip $\mathcal{S}$ with the smooth pre-topology, or with the Nisnevich-smooth pre-topology, respectively (according to whether axiom (O3) holds for Nisnevich coverings or for all etale (hence smooth) coverings). 
Recall that we denoted $\SSS^{\hcor, (\infty)} := \HH^{\cor, (\infty)}(\mathcal{M})$ the 2-pre-multiderivator associated with
 a category of fibrant objects $\mathcal{M}$  representing the \v{C}ech localization of simplicial pre-sheaves on $\mathcal{SCH}_S$, restricted to higher geometric stacks. Actually, this category is the same whether we replace $\mathcal{S}$ with all schemes of finite type over $S$, or --- in the case of the smooth pre-topology --- even with algebraic spaces of finite type over $S$. 

\begin{SATZ}\label{SATZEXTAYOUB}
There exists a symmetric derivator six-functor-formalism with domain $\Cat$
\[ \mathbb{SH}_\mathcal{M}^{T, (\infty)} \rightarrow \SSS^{\hcor, (\infty)}\]
on higher geometric stacks which is an admissible extension of the symmetric derivator six-functor-formalism
\[ \mathbb{SH}_\mathcal{M}^{T} \rightarrow \SSS^{\cor} \]
from Theorem~\ref{SATZAYOUB3}. 
\end{SATZ}
Recall that ``admissible extension'' means, in particular, that $\mathbb{SH}_\mathcal{M}^{T, (\infty)}$ restricts to $\mathbb{SH}_\mathcal{M}^{T} \rightarrow \SSS^{\cor}$, i.e.\@ on quasi-affine schemes of finite type over $S$, and that the fibers in the underlying category, as well as the pull-back and push-forward functors between them, are given by 
the ``naive'' extensions constructed via homological and cohomological descent (the latter existing even on {\em all} simplicial pre-sheaves).

\appendix

\section{Geometric sites}\label{APPENDIXGEOMETRIC}

\begin{PAR}\label{PARINTROGEOM}
Let $\mathcal{S}$ be a small category with finite limits and Grothendieck (pre)topology. 
Let $\mathcal{M} = \mathcal{SET}^{\mathcal{S}^{\op} \times \Delta^{\op}}$ equipped with the \v{C}ech weak equivalences $\mathcal{W}$.
Let $\mathcal{SH}$ be the category of ordinary sheaves of sets on $\mathcal{S}$.
We have  the adjunctions
\[ \xymatrix{  \mathcal{SET}^{\mathcal{S}^{\op}\times \Delta^{\op}}  \ar@<-2pt>[rr]_{\pi_0} && \ar@<-2pt>[ll]_{c} \mathcal{SET}^{\mathcal{S}^{\op}}  \ar@<-2pt>[rr]_{L} && \ar@<-2pt>[ll]_{\iota} \mathcal{SH}  }\] 
in which $\pi_0$ and  $L$ (sheafification) are the left adjoints. There is the
(extension of the) Yoneda embedding
\[ h: \mathcal{S}^{\amalg} \rightarrow \mathcal{SET}^{\mathcal{S}^{\op}}. \]
We denote by $\mathcal{SH}^{(0)}$, and $\mathcal{M}^{(0)}$, respectively, the (homotopy) strictly full subcategories of $\mathcal{SH}$, and $\mathcal{M}$, respectively, containing the objects in the image of 
the Yoneda embedding $ch$ (resp.\@ $Lh$). Consider $\mathcal{M}^{(0)}$ as category with weak equivalences with the restriction of the \v{C}ech weak equivalences $\mathcal{W}^{(0)}$ . 
\end{PAR}

\begin{LEMMA}\label{LEMMAEQUIVSHEAVES}
The functors $L \pi_0$ and $c \iota$ restrict to equivalences of categories with weak equivalences
\[ (\mathcal{M}^{(0)}, \mathcal{W}^{(0)}) \cong (\mathcal{SH}^{(0)}, \mathrm{Iso}(\mathcal{SH}^{(0)})) . \]
\end{LEMMA}
\begin{proof}
By definition $\mathcal{M}^{(0)}$ is the full subcategory of objects weakly equivalent to those of the form $ch(Y)$ for $Y \in \mathcal{S}^{\amalg}$. 
We know that $\pi_0$ maps \v{C}ech weak equivalences  to local isomorphisms hence $L \pi_0$ maps those to isomorphisms. Hence $L \pi_0$ and $c \iota$ are functors of categories with weak equivalences.  

The functors $c$ and $\iota$ are fully faithful, hence both counits are isomorphisms.
We claim that the unit
\[  X \rightarrow  c \iota L \pi_0 X  \]
is a \v{C}ech weak equivalence for $X$ in $\mathcal{M}^{(0)}$. Indeed, the unit $ch(Y) \rightarrow c \pi_0 c h(Y)$ is an isomorphism and
the morphism $Y \rightarrow \iota LY$ is mapped by $c$ to a \v{C}ech weak equivalence for all $Y \in \mathcal{SET}^{\mathcal{S}^{\op}}$ \cite[Theorem A.6]{DHI04}. Hence the unit is a \v{C}ech weak equivalence for all $X \in ch(\mathcal{S}^{\amalg})$. Because $L \pi_0$ and $c \iota$ preserve  weak equivalences this is the case also for all $X \in  \mathcal{M}^{(0)}$. 
\end{proof}

\begin{DEF}\label{DEFGEOMETRIC}
In the setting of \ref{PARINTROGEOM}
we will say that $\mathcal{S}$ is {\bf geometric}, if it is subcanonical, and the following equivalent conditions hold
\begin{enumerate}
\item For every $X \in \mathcal{SH}^{(0)}$ the diagonal $X \rightarrow X \times X$ is representable, i.e.\@ for all $Y \in \mathcal{S}$ and morphisms $Lh(Y) \rightarrow X \times X$ the fiber product $Lh(Y) \times_{X \times X} X$ is representable. 
\item For every $X \in \mathcal{M}^{(0)}$ the diagonal $X \rightarrow X \times X$ is representable, i.e.\@ for all $Y \in \mathcal{S}$ and morphisms $ch(Y) \rightarrow X \times X$ the homotopy fiber product $ch(Y) \widetilde{\times}_{X \times X} X$ is representable. 
\end{enumerate}
\end{DEF}
The equivalence of the two statement follows immediately from Lemma~\ref{LEMMAEQUIVSHEAVES}.

\begin{LEMMA}\label{LEMMAWEM0}
In the setting of \ref{PARINTROGEOM} assume that $\mathcal{S}$ is geometric. 
\begin{enumerate}
\item For $I \in \Inv$ with finite matching diagrams.  Every object $X \in (\mathcal{SH}^{(0)})^I$ (considered as subcategory of $\mathcal{SET}^{\mathcal{S}^{\op} \times \Delta^{\op} \times I}$), and hence every object in $(\mathcal{M}^{(0)})^I$, is weakly equivalent to an object in $(\mathcal{S}^{\amalg})^I$.
\item
Every morphism in $\mathcal{SH}^{(0)}$ (considered as subcategory of $\mathcal{SET}^{\mathcal{S}^{\op}\times \Delta^{\op}}$), and hence every morphism in $\mathcal{M}^{(0)}$, is weakly equivalent to a morphism in $\mathcal{S}^{\amalg}$. 
\item $\mathcal{M}^{(0)}$ is closed under homotopy pull-back. 
\end{enumerate}
\end{LEMMA}
\begin{proof}
1.\@ The unit $X \rightarrow c \iota L \pi_0 X$ is a \v{C}ech weak equivalence for all $X \in \mathcal{M}^{(0)}$, we may thus assume w.l.o.g.\@ that $X \in (\mathcal{SH}^{(0)})^I$.
Now apply Proposition~\ref{PROPCOVERDIRECTED} to ($\mathcal{S}$, $\mathcal{S}_0$, $\overline{C}$) in which $\mathcal{S}$ is the full subcategory of $\mathcal{SET}^{\mathcal{S}^{\op}}$ consisting of geometric pre-sheaves $X$ (i.e.\@ with the property that $X \rightarrow X\times X$ is representable) which have the property that there is a local isomorphism
$\widetilde{X} \rightarrow X$ with $\widetilde{X} \in \mathcal{S}^\amalg$, in which $\mathcal{S}_0 := \mathcal{S}^{\amalg}$, and $\overline{C}$ is the class consisting of the local isomorphisms. 
By assumption $\mathcal{SH}^{(0)} \subset \mathcal{S}$.
We have to see that $\mathcal{S}$ admits pull-back along 
morphisms in $\overline{C}$. Indeed, $\mathcal{S}$ is closed under pull-back. 
For consider a diagram in $\mathcal{SET}^{\mathcal{S}^{\op}}$  with Cartesian squares in which $\widetilde{Z}$ and $\widetilde{X}$ are in $\mathcal{S}^{\amalg}$:
\[ \xymatrix{ \widetilde{W} \ar[r]^{\mathcal{W}} \ar[d]_{\mathcal{W}} & \Box \ar[d]^{\mathcal{W}} \ar[r] & \widetilde{Z} \ar[d]^{\mathcal{W}}  \\
\Box \ar[r]_{\mathcal{W}} \ar[d] & W \ar[r] \ar[d] & Z \ar[d]\\
 \widetilde{X} \ar[r]_{\mathcal{W}} & X \ar[r] & Y  } \]
 Then $\widetilde{W}$ is in $\mathcal{S}^{\amalg}$ because $Y$ is geometric. One checks that also geometricity is closed under pull-back and thus $W \in \mathcal{S}$. 

2.\@ follows from 1.\@ applied to $I = \rightarrow$.

3.\@ follows from 1.\@ applied to $I = \righthalfcup$ and the fact that $ch: \mathcal{S}^\amalg \hookrightarrow \mathcal{M}$ commutes with (homotopy) fiber products. 
\end{proof}

\begin{LEMMA}
In the setting of \ref{PARINTROGEOM} assume that the topology is subcanonical.
Let $\kappa$ be a family of Cardinals (closed under decomposition) and assume that $\mathcal{S}$ is $\kappa$-superextensive, i.e.\@
\begin{enumerate}
\item  Every covering in $\mathcal{S}$ contains a subcovering with Cardinality in $\kappa$.
\item  $\mathcal{S}$ is $\kappa$-extensive and the pre-topology is finer than the $\kappa$-extensive topology (i.e.\@ coproduct covers of Cardinality in $\kappa$ are coverings). 
\end{enumerate}
Then $\mathcal{S}$ is geometric. 
\end{LEMMA}
\begin{proof}We will show property 1.\@ of the definition of geometric. 
Let $f_0: Lh(Y_{0}) \rightarrow L(\coprod_i h(X_i))$ and $f_1: Lh(Y_{1}) \rightarrow L(\coprod_i h(X_i))$ be two morphisms from a representable into a coproduct of representables. 
By definition of sheafification and 1., there are coverings
$\{ U_j \rightarrow Y_i \}$, of Cardinality in  $\kappa$, such that the morphisms $h(U_j) \rightarrow L(\coprod_i h(X_i))$ factor through one of the $h(X_i)$. By assumption 2.\@ we can form the coproduct over the corresponding indices and obtain
\[ L(\coprod_i h(X_i)) \cong L((\coprod_{i \in \mathcal{I}_0} h(X_i)) \amalg h(\coprod_{i \in \mathcal{I}_1} X_i) ) \]
because, by 2., $Lh: \mathcal{S} \rightarrow \mathcal{SH}$ commutes with coproducts of Cardinality in $\kappa$. 
Obviously the $f_i$ factor through the inclusion 
\[ Lh(\coprod_{i \in \mathcal{I}_1} X_i) \hookrightarrow L(\coprod_i h(X_i))    \]
which is a monomorphism before (and hence also after) sheafification by 2.\@ Hence
\[ Lh(Y_0) \times_{L(\coprod_i h(X_i))} Lh(Y_1) \cong L( h(Y_0) \times_{h(\coprod_{i \in \mathcal{I}_1} X_i)} h(Y_1)) = Lh(Y_0 \times_{\coprod_{i \in \mathcal{I}_1} X_i} Y_1).    \]
\end{proof}

\section{Constructions on inverse diagram categories}\label{APPENDIXINVERSE}

This sections contains some fairly simple constructions of inverse diagrams in a category. This is used, for instance, to construct coherent atlases of higher geometric  stacks in Section~\ref{SECTCOHATLAS}. 

\begin{PAR}We fix the following setting for this section. 
Let $I$ be an inverse diagram (cf.\@ \ref{EXDIA}) and denote by $I_{<n}$, resp.\@ $I_{\le n}$ the full subcategory of objects of degree $< n$, resp.\@ $\le n$.
For $i \in I$ we denote my $M_i$ the matching diagram at $i$, i.e.\@ the subcategory of $i \times_{/I} I$ of non-identity morphisms.  
We assume that $I$ has furthermore {\em finite} matching diagrams.

Let $\mathcal{S}$ be a category with finite products and $F: I \rightarrow \mathcal{S}$ a functor. 
We denote by $M_i(F):=\lim_{M_i} \iota^* F$  where $\iota: M_i \rightarrow I$ is the forgetful functor. We need also a relative matching object $M_i(\mu)$ for a morphism
$\mu: F \Rightarrow G$. If $M_i(F)$ and $M_i(G)$ exist then this is defined by the pullback square
\[ \xymatrix{ 
 M_i(\mu) \ar[r] \ar[d] & G_i \ar[d] \\
 M_i F \ar[r] & M_i G.  } \]
Otherwise denote by 
\[ M_i' := \nabla \left( \vcenter{\xymatrix{ & \cdot \ar@{<-}[d] \\ M_i \ar@{<-}[r] & M_i  }} \right) = \int \left( \vcenter{\xymatrix{  M_i \ar[r]&  i \times_{/I} I }} \right) \]
A morphism $\mu$ defines an obvious functor $\mu': M_i' \rightarrow \mathcal{S}$ and we define $M_i(\mu) := \lim_{M_i'} \mu'$. 
It it is clear that this coincides with the previous definition, if $M_i(F)$ and $M_i(G)$ exist. 
\end{PAR}

\begin{DEF}\label{DEFCMORPH}
Let $\mathcal{S}$ be a category with finite products (thus also with a final object) and
let $C$ be a class of morphisms that is stable under pull-back (i.e.\@ arbitrary pull-backs along morphisms in $C$ exist and are in $C$ again). 

A functor $F: I \rightarrow \mathcal{S}$ is called a {\bf $C$-diagram} if all $M_i(F)$ exist and
the induced morphisms $F_i \rightarrow M_i(F)$ are in $C$. 
A morphism (natural transformation) $\mu: F \Rightarrow G$ is called a {\bf $C$-morphism}, if all
$M_i(\mu)$ exist and the induced morphism
\[ F_i \rightarrow M_i(\mu) \]
is in $C$. A functor $F$ is thus a $C$-diagram if and only if the morphism $F \Rightarrow \cdot$ is a $C$-morphism. 
\end{DEF}

\begin{LEMMA}
Let $\iota: M_i \rightarrow I$ be the forgetful functor, and let $\mu: F \Rightarrow G$ be a $C$-morphism in $\mathcal{S}^I$. Then also $\iota^* F \Rightarrow \iota^* G$ is a $C$-morphism. 
\end{LEMMA}
\begin{proof}
The diagram $M_{i \rightarrow j}'$ for $M_i$ is isomorphic to $M_j'$ and hence we
have isomorphisms $M_{i \rightarrow j} (\iota^* \mu) \cong M_{j}(\mu)$. The statement follows.  
\end{proof}

\begin{LEMMA}\label{LEMMAACYCLXI}
For $n \in \N_0$ the morphism 
\[ \xi_n: \Xi_{n}:= \nabla \left( \vcenter{ \xymatrix{ 
& \coprod_{|i|=n} \{i\} \ar@{<-}[d]  \\
 I_{<n} \ar@{<-}[r] & \coprod_{|i|=n} M_i  } } \right) \rightarrow I_{\le n} \]
is acyclic.
\end{LEMMA}
\begin{proof}
Let $i' \in I$ and 
consider the diagram $\Xi_{n} \times_{/I_{\le n}} i'$. If $|i'|=n$ then $\Xi_{n} \times_{/I_{\le n}} i'$ consists just of the object $i'$ sitting in the first row (together with $\id_{i'}$). This is contractible. 

If $|i'|<n$ then $\Xi_{n} \times_{/I_{\le n}} i'$ is the diagram 
\[ \nabla \left( \vcenter{ \xymatrix{ 
& \coprod_{|i|=n, i \rightarrow i'} \{i\} \ar@{<-}[d]  \\
I_{<n} \times_{/I} i' \ar@{<-}[r] & \coprod_{|i|=n, i \rightarrow i'} M_{i}  \times_{/I} i' } } \right)  \]

There is an obvious adjunction between this diagram and 
\[ \nabla \left( \vcenter{ \xymatrix{ 
\{i' = i'\} \ar@{<-}[r] & \coprod_{|i|=n, i \rightarrow i'} \{ i \rightarrow i' = i' \}} } \right)  \]
And the latter has an initial object so it is contractible. 
\end{proof}

\begin{PROP}\label{PROPLIMDIRECTED}
Let $\mathcal{S}$ be a category with finite products and
let $C$ be a class of morphisms that is stable under pull-back. 
Let $I$ be a finite inverse diagram and $\mu: F \Rightarrow G$ is a $C$-morphism then pull-backs along 
\[ \lim F \rightarrow \lim G \]
exist and are in $C$ and also all components $F_i \rightarrow G_i$ are in $C$. 
It is not assumed or claimed that $\lim F$ and $\lim G$ exist. If they do not exist the statement means more precisely that for $H$ with $p_I^*H \Rightarrow G$ the limit in
\[ \lim_{\int(I \rightarrow I^{\triangleleft})} (H, \mu) \rightarrow H \]
exists and the morphism is in $C$.  
\end{PROP}

\begin{BEM}
We have
\[ \int(I \rightarrow I^{\triangleleft}) = \nabla  \left( \vcenter{\xymatrix{ & \cdot \ar@{<-}[d] \\ I \ar@{<-}[r] & I  }} \right)  \]
and $(H, \mu)$ denotes the diagram in $\Cat^{\op}(\mathcal{S})$: 
\[ \nabla  \left( \vcenter{\xymatrix{ & (\cdot, H) \ar@{<-}[d] \\ (I, F) \ar@{<-}[r] & (I, G)  }} \right).  \]
\end{BEM}

\begin{proof}
By induction on the length $n$ of $I$. 
Define a diagram in $\Cat^{\op}(\mathcal{S})$ 

\[ Q:= \nabla \left( \vcenter{ \xymatrix{ &  (\righthalfcup, D) \ar@{<-}[d] \\  (\Xi_n, \xi_n^*F) \ar@{<-}[r] & (\Xi_n, \xi_n^*G)   }} \right) \]
where
\[  D:= \left( \vcenter{ \xymatrix{ & \prod_{|i|=n} G_i \ar[d] \\ H \ar[r] & \prod_{|i|=n} G_i  }} \right).  \]
and in which $H$ is any object equipped with a morphism $p_I^* H \Rightarrow G$. 

Lemma~\ref{LEMMAACYCLXI} implies that 
\[ \lim Q = \lim_{\int (I \rightarrow I^{\triangleleft})} (H, \mu).   \]
On the other hand, switching the order of construction of $Q$, we have
\[ \lim Q = \lim \left( \vcenter{ \xymatrix{ & \prod_{|i|=n} F_i \ar[d]   \\ \lim_{\int (I_{<n} \rightarrow I_{<n}^{\triangleleft})} (H,\mu_{< n}) \ar[r] &  \prod_{|i|=n} M_i(\mu)     } } \right) \]

By induction on $n$ we have that $\lim_{\int (I_{<n} \rightarrow I_{<n}^{\triangleleft})} (H,\mu_{< n})$ exists and
the morphism  \[ \lim_{\int (I_{<n} \rightarrow I_{<n}^{\triangleleft})} (H,\mu_{< n}) \rightarrow H\] is in $C$. By assumption the $M_i(\mu)$ exist and the morphism $F_i \rightarrow M_i(\mu)$ is in $C$. Therefore also the composition 
\[  \lim_{\int (I \rightarrow I^{\triangleleft})} (H, \mu) \rightarrow H \]
is in $C$. Since the morphism $F_i \rightarrow M_i(\mu)$ is in $C$ and we just saw that $M_i(\mu) \rightarrow G_i$ is in $C$ (applying the Lemma to the morphism $\iota^*\mu$ in $\mathcal{S}^{M_i}$ and $H:=G_i$) the morphism $F_i \rightarrow G_i$ is in $C$. 
\end{proof}

\begin{PROP}\label{PROPCOVERDIRECTED}
Let $\mathcal{S}$ be a category with finite products and
let $C$ be a class of morphisms that is stable under pull-back\footnote{in the sense that pull-backs of morphisms in $C$ along arbitrary morphisms exist and are in $C$ again. }. 
Let $I$ be an inverse diagram with finite matching diagrams and assume
there is a subcategory $\mathcal{S}_0$ such that for every object $X \in \mathcal{S}$ there is a $X_0 \rightarrow X$ in $C$ with $X_0$ in $\mathcal{S}_0$. 
Then
\begin{enumerate} 
\item for any object $G \in \mathcal{S}^I$ there is $C$-morphism $F \Rightarrow G$ with $F \in \mathcal{S}_0^I$;
\item a $C$-morphism $F' \Rightarrow \iota^* G$ on any final subdiagram $\iota: I' \hookrightarrow I$ can be extended to a $C$-morphism $F \Rightarrow G$;
\item any two $C$-morphisms $F \Rightarrow G, F' \Rightarrow G$ with $F,F' \in \mathcal{S}_0$ can be dominated by a third of this kind.
\end{enumerate}
\end{PROP}
\begin{proof}
1.\@ By induction on the length of $I$. If the length is 1, $I$ is discrete and the statement is clear. Otherwise, construct $\mu: F \rightarrow G$ by induction up to degree $k-1$. For any $i$ of degree $n$ by Proposition~\ref{PROPLIMDIRECTED} applied to the diagram $M_i$ the morphism $\iota^* \mu$ and $H:=G_i$ the diagram $M_i (\mu)$ exists.
Now choose on object $F_i \in \mathcal{S}_0$ together with a morphism $F_i \rightarrow M_i (\mu)$ in $C$. By construction $M_i (\mu)$ comes equipped with a morphism $p_{M_i}^* (M_i (\mu)) \rightarrow \iota^*F$. This allows to extend $F$ to degree $\le n$. 

2.\@ follows from the construction in 1\@.

3.\@ For two $C$-morphisms $F \Rightarrow G$ and $F' \Rightarrow G$ also the pull-back $F' \times_G F \rightarrow G$ exists (because 
the morphism $F \rightarrow G$ is point-wise in $C$ by Proposition~\ref{PROPLIMDIRECTED}) and is a $C$-morphism. We can find a $C$-morphism
$F'' \rightarrow F' \times_G F$ with $F'' \in \mathcal{S}_0^I$ and the compositon with $F' \times_G F \rightarrow G$ does the job.  
\end{proof}

\newpage
\bibliographystyle{abbrvnat}
\bibliography{6fu}

\begin{thebibliography}{40}
\providecommand{\natexlab}[1]{#1}
\providecommand{\url}[1]{\texttt{#1}}
\expandafter\ifx\csname urlstyle\endcsname\relax
  \providecommand{\doi}[1]{doi: #1}\else
  \providecommand{\doi}{doi: \begingroup \urlstyle{rm}\Url}\fi

\bibitem[SGA(1972{\natexlab{a}})]{SGAIV1}
\emph{Th\'eorie des topos et cohomologie \'etale des sch\'emas. {T}ome 1:
  {T}h\'eorie des topos}.
\newblock Lecture Notes in Mathematics, Vol. 269. Springer-Verlag, Berlin-New
  York, 1972{\natexlab{a}}.
\newblock S{\'e}minaire de G{\'e}om{\'e}trie Alg{\'e}brique du Bois-Marie
  1963--1964 (SGA 4), Dirig{\'e} par M. Artin, A. Grothendieck, et J. L.
  Verdier. Avec la collaboration de N. Bourbaki, P. Deligne et B. Saint-Donat.

\bibitem[SGA(1972{\natexlab{b}})]{SGAIV2}
\emph{Th\'eorie des topos et cohomologie \'etale des sch\'emas. {T}ome 2}.
\newblock Lecture Notes in Mathematics, Vol. 270. Springer-Verlag, Berlin-New
  York, 1972{\natexlab{b}}.
\newblock S{\'e}minaire de G{\'e}om{\'e}trie Alg{\'e}brique du Bois-Marie
  1963--1964 (SGA 4), Dirig{\'e} par M. Artin, A. Grothendieck et J. L.
  Verdier. Avec la collaboration de N. Bourbaki, P. Deligne et B. Saint-Donat.

\bibitem[SGA(1973)]{SGAIV3}
\emph{Th\'eorie des topos et cohomologie \'etale des sch\'emas. {T}ome 3}.
\newblock Lecture Notes in Mathematics, Vol. 305. Springer-Verlag, Berlin-New
  York, 1973.
\newblock S{\'e}minaire de G{\'e}om{\'e}trie Alg{\'e}brique du Bois-Marie
  1963--1964 (SGA 4), Dirig{\'e} par M. Artin, A. Grothendieck et J. L.
  Verdier. Avec la collaboration de P. Deligne et B. Saint-Donat.

\bibitem[SP(2022)]{SP}
\emph{The stacks project}.
\newblock 2022.
\newblock available at:
  \href{https://stacks.math.columbia.edu/}{https://stacks.math.columbia.edu/}.

\bibitem[Ayoub(2007{\natexlab{a}})]{Ayo07I}
J.~Ayoub.
\newblock Les six op\'erations de {G}rothendieck et le formalisme des cycles
  \'evanescents dans le monde motivique. {I}.
\newblock \emph{Ast\'erisque}, \penalty0 (314):\penalty0 x+466 pp.,
  2007{\natexlab{a}}.

\bibitem[Ayoub(2007{\natexlab{b}})]{Ayo07II}
J.~Ayoub.
\newblock Les six op\'erations de {G}rothendieck et le formalisme des cycles
  \'evanescents dans le monde motivique. {II}.
\newblock \emph{Ast\'erisque}, \penalty0 (315):\penalty0 vi+364 pp.,
  2007{\natexlab{b}}.

\bibitem[Ayoub(2014)]{Ayo14}
J.~Ayoub.
\newblock La r\'{e}alisation \'{e}tale et les op\'{e}rations de {G}rothendieck.
\newblock \emph{Ann. Sci. \'{E}c. Norm. Sup\'{e}r. (4)}, 47\penalty0
  (1):\penalty0 1--145, 2014.

\bibitem[Ayoub et~al.(2020)Ayoub, Gallauer, and Vezzani]{AGV20}
J.~Ayoub, M.~Gallauer, and A.~Vezzani.
\newblock The six-functor formalism for rigid analytic motives.
\newblock arXiv: \href{https://arxiv.org/abs/2010.15004}{2010.15004}, 2020.

\bibitem[Bhatt and Scholze(2015)]{BS15}
B.~Bhatt and P.~Scholze.
\newblock The pro-\'{e}tale topology for schemes.
\newblock \emph{Ast\'{e}risque}, \penalty0 (369), 2015.

\bibitem[Chowdhury(2021)]{Cho21}
C.~Chowdhury.
\newblock Motivic homotopy theory of algebraic stacks.
\newblock Dissertation, University Duisburg-Essen, Germany, 2021.

\bibitem[Cisinski(2003)]{Cis03}
D.-C. Cisinski.
\newblock Images directes cohomologiques dans les cat\'egories de mod\`eles.
\newblock \emph{Ann. Math. Blaise Pascal}, 10\penalty0 (2):\penalty0 195--244,
  2003.

\bibitem[Cisinski(2004)]{Cis04}
D.-C. Cisinski.
\newblock Le localisateur fondamental minimal.
\newblock \emph{Cah. Topol. G\'eom. Diff\'er. Cat\'eg.}, 45\penalty0
  (2):\penalty0 109--140, 2004.

\bibitem[Cisinski(2008)]{Cis08}
D.-C. Cisinski.
\newblock Propri\'et\'es universelles et extensions de {K}an d\'eriv\'ees.
\newblock \emph{Theory Appl. Categ.}, 20:\penalty0 No. 17, 605--649, 2008.

\bibitem[Cisinski and D\'{e}glise(2019)]{CD19}
D.-C. Cisinski and F.~D\'{e}glise.
\newblock \emph{Triangulated categories of mixed motives}.
\newblock Springer Monographs in Mathematics. Springer, Cham, 2019.

\bibitem[Drew(2018)]{Dre18}
B.~Drew.
\newblock Motivic {H}odge modules.
\newblock arXiv: \href{http://arxiv.org/abs/1801.10129}{1801.10129}, 2018.

\bibitem[Drew and Gallauer(2020)]{DG20}
B.~Drew and M.~Gallauer.
\newblock The universal six-functor formalism.
\newblock arXiv: \href{http://arxiv.org/abs/2009.13610}{2009.13610}, 2020.

\bibitem[Dugger et~al.(2004)Dugger, Hollander, and Isaksen]{DHI04}
D.~Dugger, S.~Hollander, and D.~C. Isaksen.
\newblock Hypercovers and simplicial presheaves.
\newblock \emph{Math. Proc. Cambridge Philos. Soc.}, 136\penalty0 (1):\penalty0
  9--51, 2004.

\bibitem[Fausk et~al.(2003)Fausk, Hu, and May]{FHM}
H.~Fausk, P.~Hu, and J.~P. May.
\newblock Isomorphisms between left and right adjoints.
\newblock \emph{Theory Appl. Categ.}, 11:\penalty0 No. 4, 107--131, 2003.

\bibitem[Gaitsgory and Rozenblyum(2020)]{GR16}
D.~Gaitsgory and N.~Rozenblyum.
\newblock A study in derived algebraic geometry.
\newblock book project in progress:
  \href{http://www.math.harvard.edu/~gaitsgde/GL/}{http://www.math.harvard.edu/\~{}gaitsgde/GL/},
  2020.

\bibitem[Gulotta et~al.(2022)Gulotta, Hansen, and Weinstein]{GHW22}
D.~Gulotta, D.~Hansen, and J.~Weinstein.
\newblock An enhanced six-functor formalism for diamonds and v-stacks.
\newblock arXiv: \href{https://arxiv.org/abs/2202.12467}{2202.12467}, 2022.

\bibitem[H\"ormann(2017{\natexlab{a}})]{Hor15}
F.~H\"ormann.
\newblock Fibered multiderivators and (co)homological descent.
\newblock \emph{Theory Appl. Categ.}, 32\penalty0 (38):\penalty0 1258--1362,
  2017{\natexlab{a}}.

\bibitem[H\"ormann(2017{\natexlab{b}})]{Hor16}
F.~H\"ormann.
\newblock Derivator {S}ix-{F}unctor-{F}ormalisms --- {D}efinition and
  {C}onstruction {I}.
\newblock arXiv: \href{http://arxiv.org/abs/1701.02152}{1701.02152},
  2017{\natexlab{b}}.

\bibitem[H\"ormann(2018)]{Hor15b}
F.~H\"ormann.
\newblock Six-{F}unctor-{F}ormalisms and {F}ibered {M}ultiderivators.
\newblock \emph{Sel. Math.}, 24\penalty0 (4):\penalty0 2841--2925, 2018.

\bibitem[H\"ormann(2019)]{Hor17}
F.~H\"ormann.
\newblock Derivator {S}ix-{F}unctor-{F}ormalisms --- {C}onstruction {II}.
\newblock arXiv: \href{http://arxiv.org/abs/1902.03625}{1902.03625}, 2019.

\bibitem[H\"{o}rmann(2020)]{Hor17b}
F.~H\"{o}rmann.
\newblock Enlargement of (fibered) derivators.
\newblock \emph{J. Pure Appl. Algebra}, 224\penalty0 (3):\penalty0 1023--1063,
  2020.

\bibitem[H\"ormann(2021)]{Hor21c}
F.~H\"ormann.
\newblock Higher stacks as diagrams.
\newblock arXiv: \href{https://arxiv.org/abs/2105.08479}{2105.08479}, 2021.

\bibitem[Hoyois(2017)]{Hoy17}
M.~Hoyois.
\newblock The six operations in equivariant motivic homotopy theory.
\newblock \emph{Adv. Math.}, 305:\penalty0 197--279, 2017.

\bibitem[Khan and Ravi(2021)]{KR21}
A.~A. Khan and C.~Ravi.
\newblock Generalized cohomology theories for algebraic stacks.
\newblock arXiv: \href{https://arxiv.org/abs/2106.15001}{2106.15001}, 2021.

\bibitem[Laszlo and Olsson(2008{\natexlab{a}})]{LO08I}
Y.~Laszlo and M.~Olsson.
\newblock The six operations for sheaves on {A}rtin stacks. {I}. {F}inite
  coefficients.
\newblock \emph{Publ. Math. Inst. Hautes \'Etudes Sci.}, \penalty0
  (107):\penalty0 109--168, 2008{\natexlab{a}}.

\bibitem[Laszlo and Olsson(2008{\natexlab{b}})]{LO08II}
Y.~Laszlo and M.~Olsson.
\newblock The six operations for sheaves on {A}rtin stacks. {II}. {A}dic
  coefficients.
\newblock \emph{Publ. Math. Inst. Hautes \'Etudes Sci.}, \penalty0
  (107):\penalty0 169--210, 2008{\natexlab{b}}.

\bibitem[Lipman and Hashimoto(2009)]{LH09}
J.~Lipman and M.~Hashimoto.
\newblock \emph{Foundations of {G}rothendieck duality for diagrams of schemes},
  volume 1960 of \emph{Lecture Notes in Mathematics}.
\newblock Springer-Verlag, Berlin, 2009.

\bibitem[Liu and Zheng(2012{\natexlab{a}})]{ZL14}
Y.~Liu and W.~Zheng.
\newblock Enhanced six operations and base change theorem for {A}rtin stacks.
\newblock arXiv: \href{http://arxiv.org/abs/1211.5948}{1211.5948},
  2012{\natexlab{a}}.

\bibitem[Liu and Zheng(2012{\natexlab{b}})]{ZL14b}
Y.~Liu and W.~Zheng.
\newblock Gluing restricted nerves of $\infty$-categories.
\newblock arXiv: \href{http://arxiv.org/abs/1211.5294}{1211.5294},
  2012{\natexlab{b}}.

\bibitem[Ponto and Shulman(2016)]{PS14}
K.~Ponto and M.~Shulman.
\newblock The linearity of traces in monoidal categories and bicategories.
\newblock \emph{Theory Appl. Categ.}, 31:\penalty0 Paper No. 23, 594--689,
  2016.

\bibitem[Pridham(2013)]{Pri13}
J.~P. Pridham.
\newblock Presenting higher stacks as simplicial schemes.
\newblock \emph{Adv. Math.}, 238:\penalty0 184--245, 2013.

\bibitem[Schn\"{u}rer(2018)]{Sch15}
O.~M. Schn\"{u}rer.
\newblock Six operations on dg enhancements of derived categories of sheaves.
\newblock \emph{Selecta Math. (N.S.)}, 24\penalty0 (3):\penalty0 1805--1911,
  2018.

\bibitem[Scholze(2017)]{Sch17}
P.~Scholze.
\newblock Etale cohomology of diamonds.
\newblock arXiv: \href{http://arxiv.org/abs/1709.07343}{1709.07343}, 2017.

\bibitem[To\"{e}n and Vezzosi(2008)]{TV08}
B.~To\"{e}n and G.~Vezzosi.
\newblock Homotopical algebraic geometry. {II}. {G}eometric stacks and
  applications.
\newblock \emph{Mem. Amer. Math. Soc.}, 193\penalty0 (902):\penalty0 x+224,
  2008.

\bibitem[Verdier(1977)]{Ver77}
J.-L. Verdier.
\newblock Cat\'egories d\'eriv\'ees.
\newblock In \emph{Cohomologie Etale (SGA 4$\frac{1}{2}$}, volume 569 of
  \emph{Lecture Notes in Mathematics}, pages 262--311. Springer, 1977.

\bibitem[Zheng(2015)]{Zh10}
W.~Zheng.
\newblock Six operations and {L}efschetz-{V}erdier formula for
  {D}eligne-{M}umford stacks.
\newblock \emph{Sci. China Math.}, 58\penalty0 (3):\penalty0 565--632, 2015.

\end{thebibliography}

\end{document}